\documentclass[letterpaper]{amsart}
\pdfoutput=1
\allowdisplaybreaks{}
\usepackage{mathtools}
\usepackage[T1]{fontenc}
\usepackage{microtype}
\usepackage[shortlabels]{enumitem}
\usepackage[british]{babel}
\usepackage{tikz-cd}
\usepackage{hyperref}
\usepackage{booktabs}
\usepackage{caption} 
\usepackage{mathrsfs}
\numberwithin{equation}{section}
\usepackage[margin=1.3in]{geometry}
\def\subsectionstyle{\@startsection{subsection}{2}%
	\z@{.5\linespacing\@plus.7\linespacing}{-.5em}%
	{\normalfont\bfseries}}

\newtheorem{thm}{Theorem}[section]
\newtheorem{cor}[thm]{Corollary}
\newtheorem{lem}[thm]{Lemma}
\newtheorem{prop}[thm]{Proposition}
\newtheorem*{thm*}{Theorem}
\newtheorem*{prop*}{Proposition}

\newcounter{theoremalph}

\newtheorem{thmAlph}[theoremalph]{Theorem}
\newtheorem{corAlph}[theoremalph]{Corollary}

\theoremstyle{definition}
\newtheorem{defn}[thm]{Definition}
\newtheorem*{defn*}{Definition}
\theoremstyle{remark}
\newtheorem{ques}[thm]{Question}
\newtheorem{rem}[thm]{Remark}
\newtheorem*{rem*}{Remark}
\newtheorem{exmp}[thm]{Example}

\newtheorem{notat}[thm]{Notation}

\DeclareMathOperator{\Ann}{Ann}
\DeclareMathOperator{\Aut}{Aut}
\DeclareMathOperator{\CAT}{CAT}

\DeclareMathOperator{\Ext}{Ext}

\DeclareMathOperator{\Homfd}{Hom_\textrm{fd}}
\DeclareMathOperator{\Hom}{Hom}

\DeclareMathOperator{\Tor}{Tor}

\DeclareMathOperator{\asdim}{asdim}
\DeclareMathOperator{\ccd}{ccd}
\DeclareMathOperator{\cd}{cd}

\DeclareMathOperator{\coker}{coker}
\DeclareMathOperator{\diam}{diam}

\DeclareMathOperator{\id}{id}
\DeclareMathOperator{\im}{im}

\DeclareMathOperator{\lf}{lf}

\DeclareMathOperator{\rank}{rank}

\DeclareMathOperator{\supp}{supp}
\DeclareMathOperator{\vcd}{vcd}
\newcommand{\Haus}{d_\textrm{Haus}}
\newcommand{\coarse}{H_\textrm{coarse}}
\newcommand{\hcoarse}{H^\textrm{coarse}}

\newcommand{\bB}{\mathbf{B}}
\newcommand{\bC}{\mathbf{C}}
\newcommand{\bD}{\mathbf{D}}
\newcommand{\bE}{\mathbf{E}}

\newcommand{\bL}{\mathbf{L}}
\newcommand{\bM}{\mathbf{M}}
\newcommand{\bN}{\mathbf{N}}
\newcommand{\bP}{\mathbf{P}}

\newcommand{\bZ}{\mathbf{Z}}

\newcommand{\bbF}{\mathbb{F}}

\newcommand{\bbN}{\mathbb{N}}

\newcommand{\bbR}{\mathbb{R}}
\newcommand{\bbZ}{\mathbb{Z}}

\newcommand{\cB}{\mathcal{B}}
\newcommand{\cC}{\mathcal{C}}

\newcommand{\cF}{\mathcal{F}}

\newcommand{\cM}{\mathcal{M}}
\newcommand{\cN}{\mathcal{N}}

\newcommand{\cZ}{\mathcal{Z}}

\author{Alexander Margolis}
\title{Coarse homological invariants of metric spaces}
\address{Alexander Margolis, Department of Mathematics, The Ohio State University,  Mathematics Tower,  231 W 18th Ave,  Columbus,  OH  43210, USA}
\email{margolis.93@osu.edu}

\subjclass[2020]{20J05, 20J06, 20F65, 51F30}

\begin{document}
\begin{abstract}
	Inspired by group cohomology, we define  several coarse topological invariants of metric spaces. We define the coarse cohomological dimension of a metric space, and demonstrate that if $G$ is a countable group, then the coarse cohomological dimension of $G$ as a metric space coincides with the cohomological dimension of $G$ as a group whenever the latter is finite. Extending a result of Sauer, it is shown that coarse cohomological dimension is monotone under coarse embeddings, and hence is invariant under coarse equivalence. We characterise  unbounded quasi-trees as quasi-geodesic metric spaces of coarse cohomological dimension one.
	A classical theorem of Hopf and Freudenthal states that if $G$ is a finitely generated group, then the number of ends of $G$ is either $0$, $1$, $2$ or $\infty$. We prove a higher-dimensional analogue of this result, showing that if $\bbF$ is a field, $G$ is countable, and  $H^k(G,\bbF G)=0$ for $k<n$, then $\dim H^n(G,\bbF G)=0,1$ or $\infty$, significantly extending a result of Farrell from 1975. Moreover, in the case $\dim H^n(G,\bbF G)=1$, then $G$ must be a coarse Poincar\'e duality group. We prove an  analogous result   for metric spaces.
\end{abstract}
\maketitle
\section{Introduction}
In this article,  we define  several coarse homological invariants of metric spaces, inspired by and analogous to  concepts in  group cohomology; see Table~\ref{table:dictionary}. We first discuss an application of this machinery.
A foundational result in the theory of ends, proven independently by Freudenthal and Hopf, states that the number of ends of a finitely generated group is  equal to either $0,1,2$ or $\infty$.
We recall the number of ends of an infinite group $G$ can be characterised cohomologically as the dimension  of $H^1(G,\bbF_2 G)$ minus one. We can thus reformulate the Freudenthal--Hopf theorem as follows:
\begin{thm}[\cite{hopf1944enden,freudenthal1945uber}]\label{thm:hopf_ends_cohom}
	If $G$ is a  finitely generated group, then exactly one of the following holds:
	\begin{enumerate}
		\item $\dim_{\bbF_2} H^1(G,\bbF_2 G)=0$.
		\item $\dim_{\bbF_2}  H^1(G,\bbF_2 G)=1$, in which case $G$ is virtually infinite cyclic.
		\item  $\dim_{\bbF_2}  H^1(G,\bbF_2 G)=\infty$.
	\end{enumerate}
\end{thm}
We prove the following  higher-dimensional analogue of Theorem~\ref{thm:hopf_ends_cohom}, significantly extending work of Farrell~\cite{farrell1975pdgroups}.
\newcommand{\maingroup}{Let $n\in \bbN$,  let $\bbF$ be a field and let $G$ be a countable group. If $H^k(G,\bbF G)=0$ for $k<n$, then exactly one of the following holds:
	\begin{enumerate}
		\item $\dim_{\bbF} H^n(G,\bbF G)=0$.
		\item $\dim_{\bbF} H^n(G,\bbF G)=1$, in which case $G$ is a coarse Poincar\'e duality group of dimension $n$ over  $\bbF$.
		\item  $\dim_{\bbF} H^n(G,\bbF G)=\infty$.
	\end{enumerate}
	Moreover, if $H^n(G,\bbF G)$ has countable dimension and a finite-dimensional $G$-invariant subspace, then $\dim_{\bbF} H^n(G,\bbF G)=1$.}
\begin{thmAlph}[Theorem~\ref{thm:main_field}]\label{thm:main_field_intro}
	\maingroup{}
\end{thmAlph}

\emph{Poincar\'e duality groups} were introduced by Johnson--Wall and Bieri~\cite{johnsonwall1972groups,bieri1972gruppen}, and are an important and well-studied class of groups that possess a  form of duality between group homology and cohomology. In particular, fundamental groups of closed aspherical manifolds are Poincar\'e duality groups.
\emph{Coarse Poincar\'e duality groups of dimension $n$ over a ring $R$}, henceforth called \emph{coarse $PD_n^R$} groups, were introduced by Kapovich--Kleiner~\cite{kapovich2005coarse}. They  are a generalisation of Poincar\'e duality groups in which the dualising maps between chains and cochains need not  be equivariant; see Remark~\ref{rem:coarse_pdn}. Such groups have the same large-scale topology as $\bbR^n$. A coarse $PD_n^R$ group is a Poincar\'e duality group if and only if it has finite cohomological dimension (Proposition~\ref{prop:pdngroup_properties}).  We note that for $n=1$, a finitely generated group is a coarse $PD_1^\bbF$ group if and only if it is virtually infinite cyclic.

Informally, Theorem~\ref{thm:main_field_intro} can be seen as a precise articulation of the fact that for a countable group $G$, setting $n\coloneqq\inf\{k\in \bbN\mid H^k(G,\bbF  G)\neq 0\}$ and assuming $n<\infty$, either:
\begin{itemize}
	\item $G$ is highly singular  ---  $\dim H^n(G,\bbF G)=\infty$; or
	\item $G$ is highly non-singular ---  $G$  has the same large-scale topology as $\bbR^n$.
\end{itemize}
An analogous dichotomy was shown in the case of a hyperbolic group $G$, whose boundary $\partial G$ is either  a highly non-Euclidean
``fractal'' space or a sphere~\cite[Theorem 4.4]{kapovichbenakli2002boundaries}. There is a third more exotic possibility  that $H^n(G,\bbF  G)= 0$ for all $n$, which occurs when $G$ is Thompson's group $F$~\cite{browngeoghegan1984infinitedimensional}.

In 1975, Farrell obtained a similar result to Theorem~\ref{thm:main_field_intro} under --- in Farrell's own words --- ``rather strong finiteness conditions'', specifically that the group $G$ is of type $FP_n^\bbF$  and has finite cohomological dimension over $\bbF$~\cite{farrell1975pdgroups}; see also~\cite{farrell1974second} for earlier work in the case $n=2$. In Theorem~\ref{thm:main_field_intro}, we impose  no such assumptions. Rather, we can actually  deduce that in the $\dim H^n(G,\bbF G)=1$ case of Theorem~\ref{thm:main_field_intro},  $G$ must necessarily be of type $FP_\infty^\bbF$, and must have \emph{finite coarse cohomological dimension}, a concept we introduce in this article and will discuss shortly.

We recall the following question of Serre from 1974:
\begin{ques}[\cite{farrell1974second}]\label{ques:serre}
	Let $n\in \bbN$,  let $\bbF$ be a field and let $G$ be a finitely presented group. Does every $G$-invariant subspace of $H^n(G,\bbF G)$ have dimension 0, 1 or $\infty$?
\end{ques}
Theorem~\ref{thm:main_field_intro} builds on work of Farrell~\cite{farrell1974second,farrell1975pdgroups} to obtain many more instances in which Question~\ref{ques:serre} has a positive answer, specifically when $n=\inf\{k\in \bbN\mid H^k(G,\bbF  G)\neq 0\}<\infty$ and $H^n(G,\bbF G)$ has countable dimension, the latter condition always holding  if $G$ is of type $FP_n^\bbF$ by Proposition~\ref{prop:ctble_cohomology_fingen}. Moreover, only countability, not  finite presentability, is needed. However,  Question~\ref{ques:serre} in full generality is still wide open.

We observe the main conclusion  of Theorem~\ref{thm:main_field_intro} --- that $\dim H^n(G,\bbF G)$ is either $0$, $1$ or $\infty$ --- is a statement purely about group cohomology. Therefore, one might naturally expect that such a statement  can be proved using group cohomology, i.e.\ via an algebraic argument using  $\bbF G$-modules and $G$-equivariant maps between them. However, this is not the approach we take to prove Theorem~\ref{thm:main_field_intro}, and the author strongly believes such an approach  is impossible without additional hypotheses.
Instead, we prove Theorem~\ref{thm:main_field_intro} by appealing to one of the key concepts of  geometric group theory: we equip a countable group with a metric and then forget about the algebraic structure of a group, focussing solely on the large-scale geometric and topological properties of the group as a metric space.

In order to  prove Theorem~\ref{thm:main_field_intro}, we work not in the category of $RG$-modules,  but in the category of \emph{$R$-modules over a metric space} that we introduce in Section~\ref{sec:rmod_metric}. A large part of this article is therefore spent developing  the necessary language and machinery of \emph{homological algebra over metric spaces}. Indeed, this machinery is of independent interest, and was developed in ongoing projects of  the author to prove quasi-isometric rigidity of splittings and of commensurated subgroups.
To illustrate the philosophy adopted in the article, we state a variant of Theorem~\ref{thm:main_field_intro} that applies to metric spaces rather than to groups:
\newcommand{\mainspace}{{Let $R$ be a PID and  let $n\geq 1$. Suppose $X$ is a coarsely homogeneous proper metric space that is coarsely uniformly $(n-1)$-acyclic over $R$. Assume that $\coarse^k(X;R)=0$ for $k<n$. Then either:
			\begin{enumerate}
				\item $\coarse^n(X;R)=0$.
				\item $\coarse^n(X;R)\cong R$, in which case $X$ is a coarse $PD_n^R$ space.
				\item $\coarse^n(X;R)$ is not finitely generated as an $R$-module.
			\end{enumerate}}}
\begin{thmAlph}[Corollary~\ref{cor:main_finiteness_space}]\label{thm:main_finitenss_space_intro}
	\mainspace{}
\end{thmAlph}
Theorem~\ref{thm:main_finitenss_space_intro} is another generalisation of the Hopf--Freudenthal theorem on ends.
Although there is significant overlap between Theorems~\ref{thm:main_field_intro} and~\ref{thm:main_finitenss_space_intro}, neither one implies the other. Indeed, Theorem~\ref{thm:main_field_intro} makes no assumptions that the group is  coarsely uniformly $(n-1)$-acyclic (see Theorem~\ref{thm:type_FPn}),  whilst Theorem~\ref{thm:main_finitenss_space_intro} holds for metric spaces, and is valid for any PID, not just a field. Additional variations of these results are proved in Section~\ref{sec:coarsePDn}.

We briefly discuss the terms in the statement of Theorem~\ref{thm:main_finitenss_space_intro}.
A metric space $X$ is \emph{coarsely homogeneous} if for all points $x,y\in X$, there is a coarse equivalence $\phi$ with uniform control functions such that $\phi(x)=y$. Such an assumption is clearly necessary for Theorem~\ref{thm:main_finitenss_space_intro} --- in the case $n=1$, there exist metric spaces that are not coarsely homogeneous, but possess any finite number of ends, a fact that would contradict the conclusion of Theorem~\ref{thm:main_finitenss_space_intro}.
The property of being coarsely uniformly $(n-1)$-acyclic is a metric property that is analogous to a group being of type $FP_n$. Indeed,  Theorem~\ref{thm:type_FPn} says the two properties coincide when $G$ is a countable group equipped with  a proper left-invariant metric.

Theorem~\ref{thm:main_finitenss_space_intro} is formulated using $\coarse^*(X;R)$, the \emph{coarse cohomology} of a metric space $X$ with coefficients in $R$. For a countable group $G$ equipped with a proper left-invariant metric, $\coarse^*(G;R)$ coincides with group cohomology $H^*(G,RG)$ with group ring coefficients.
Roe first defined a version of coarse cohomology~\cite{roe1993coarse}. It is well-known to capture many aspects of the topology at infinity of a space or of a group. Building on work of Kapovich--Kleiner~\cite{kapovich2005coarse}, we develop a more general coarse cohomology theory that is better suited to  quantitative and effective arguments on the chain and cochain level.

We  define coarse cohomology modules of the form $\coarse^*(X;\bM)$, where the coefficient module $\bM$ is an  \emph{$R$-module over $X$}.  Our theory is thus a better analogue of group cohomology, in which cohomology modules are of the form $H^*(G,M)$ for an $RG$-module $M$.
In this article, many phenomena that occur  within the  category of $RG$-modules are reimagined  within the more general category of $R$-modules over a metric space. To that end, Table~\ref{table:dictionary} gives us a  dictionary between these two categories.
\begin{table}[bht]
	\centering
	\begin{tabular}{p{0.45\textwidth} p{0.45\textwidth}}
		\toprule
		A group $G$                                                        & A metric space $X$                                                                                                                                            \\
		\midrule
		An $RG$-module $M$                                                 & An $R$-module $\bM$ over $X$                                                                                                                                  \\
		\midrule A projective $RG$-module $M$                              & A projective $R$-module $\bM$ over $X$                                                                                                                        \\
		\midrule
		A $G$-equivariant map $M\to N$ between $RG$-modules                &
		A finite displacement map $\bM\to \bN$ between $R$-modules over $X$                                                                                                                                                                \\
		\midrule
		Group cohomology $H^n(G,M)$                                        & Coarse cohomology $\coarse^n(X;\bM)$                                                                                                                          \\
		\midrule
		A finitely generated $RG$-module                                   & A finite-height $R$-module over $X$                                                                                                                           \\
		\midrule
		A projective resolution $C_\bullet$ of the trivial $RG$-module $R$ & A projective $R$-resolution $\bC_\bullet$ over $X$                                                                                                            \\

		\midrule
		The following are equivalent:
		\begin{itemize}[leftmargin=*]
			\item $H^k(G,M)=0$ for all $RG$-modules $M$ and  $k>n$.
			\item There exists a  $C_\bullet$ as above such that $C_k=0$ for $k>n$.
		\end{itemize}
		When either  hold, we say $\cd_R(G)\leq n$.
		                                                                   & The following are equivalent:
		\begin{itemize}[leftmargin=*]
			\item $\coarse^k(X;\bM)=0$ for every $R$-module $\bM$ over $X$  and every $k>n$.
			\item There exists a  $\bC_\bullet$ as above  such that $\bC_k=0$ for $k>n$.
		\end{itemize}
		When either  hold, we say $\ccd_R(X)\leq n$.                                                                                                                                                                                       \\
		\midrule
		$G$ is of type $FP_n^R$ if there exists a $C_\bullet$ as above such that $C_k$ is a finitely generated $RG$-module for $k\leq n$.
		                                                                   & $X$ is coarsely uniformly $(n-1)$-acyclic over $R$  if and only if there exists a $\bC_\bullet$ as above such that $\bC_k$ is of finite height for $k\leq n$. \\
		\midrule
		Type $FP$ over $R$                                                 & Coarse finite type over $R$                                                                                                                                   \\
		\midrule
		A Poincar\'e duality group of dimension $n$ over $R$               & A coarse $PD_n^R$ space                                                                                                                                       \\
		\bottomrule
	\end{tabular}

	\caption{A dictionary relating concepts in group cohomology in the left-hand column, to  analogous concepts in the coarse cohomology of  metric spaces in the right-hand column.}\label{table:dictionary}
\end{table}
Table~\ref{table:dictionary} is intended as a rough guide for the casual reader who is already familiar with the group cohomology concepts in the left-hand column, and desires   to quickly gain an informal appreciation of the vocabulary of $R$-modules over metric spaces in the right-hand column. However, we  do not yet assert any precise relationship between items in the left-hand and right-hand  columns, referring the reader to the main text for precise definitions and formal statements of results.

It is important to emphasise that even if one is only interested in the group cohomology of a countable group $G$, there are reasons to work in the category of $R$-modules over $G$. Indeed, there are many instances in which it is unnecessarily restrictive  to limit ourselves to $G$-equivariant maps between $G$-modules. Passing to the more flexible framework of non-equivariant maps between  $R$-modules over a metric space is crucial in proving  results such as Theorems~\ref{thm:main_field_intro} and~\ref{thm:main_finitenss_space_intro}.

These ideas are further illustrated by the  \emph{coarse cohomological dimension} of a metric space, which we introduce in Section~\ref{sec:coarse_dim}. We recall that the classical cohomological dimension $\cd_R(G)$ of a group $G$ with respect to a coefficient ring $R$ is one of the most  important  and  well-studied invariants of infinite groups~\cite{brown1982cohomology,bieri1981homological}. Unfortunately (at least from the viewpoint of the author) $\cd_R(G)$ is quite sensitive to the precise algebraic structure of the group.

The situation is the most extreme  in the case of groups with torsion: any group $G$ containing a torsion element satisfies $\cd(G)=\infty$. We consider the illuminating case where $G=D_\infty$ is the infinite dihedral group and $H\leq G$ is an infinite cyclic subgroup of index 2, which satisfies $1=\cd(H)\leq\cd(G)=\infty$. We suppose that \[\cdots\to C_2\to C_1\to C_0\to \bbZ\to 0\] is a resolution of the trivial $\bbZ G$-module $\bbZ$ by free $\bbZ G$-modules. The condition $\cd(G)\leq 1$ is equivalent to the existence of a $G$-equivariant retraction $r:C_1\to C_1$ onto $K\coloneqq \ker(C_1\to C_0)$. Indeed, the existence of such a map $r$ allows us to $G$-equivariantly truncate the above  resolution to a finite length projective resolution
\[0\to \ker(r)\to C_0\to \bbZ\to 0\] of the trivial $\bbZ G$-module $\bbZ$.
Of course, since $G=D_\infty$ has torsion, $\cd(G)=\infty$, and so  no such retraction $r:C_1\to C_1$ exists. However, since $\cd(H)=1$, there is an $H$-equivariant retraction $r:C_1\to C_1$ onto $K\coloneqq \ker(C_1\to C_0)$.

When we pass to a more general category that allows non-equivariant maps, such obstructions to finite dimensionality disappear, and we obtain a robust notion of dimension that is invariant under coarse equivalence. In particular, even though the $H$-equivariant retraction $r:C_1\to C_1$ considered above is not $G$-equivariant, it is nonetheless a finite displacement map between $\bbZ$-modules over $G$; see Section~\ref{sec:rmod_metric} for definitions of these terms. We can therefore  (non-equivariantly) truncate the resolution $C_\bullet$ to obtain a finite length resolution $0\to \ker(r)\to C_0\to \bbZ\to 0$ by projective $R$-modules over $G$, which shows that the coarse cohomological dimension of $G$ is equal to 1; see Section~\ref{sec:coarse_dim}.

This approach is very fruitful, and yields a robust notion of  dimension for arbitrary metric spaces. We let $\ccd_R(X)$ denote the \emph{coarse cohomological dimension} of $X$ with coefficients in $R$, and in the case $R=\bbZ$, we denote $\ccd_\bbZ$ by $\ccd$. We show $\ccd_R(X)$ satisfies the following properties:

\begin{thmAlph}[Theorems~\ref{thm:ccd_monotonicity} and~\ref{thm:gp_ccd}]\label{thm:ccd_intro}
	For each  commutative ring $R$, the following hold:
	\begin{enumerate}
		\item If $X$ and $Y$ are metric spaces and $f:X\to Y$ is a coarse embedding, then $\ccd_R(X)\leq \ccd_R(Y)$. In particular, if $X$ and $Y$ are coarsely equivalent, then $\ccd_R(X)=\ccd_R(Y)$.
		\item If $G$ is a countable group equipped with a proper left-invariant metric and $\cd_R(G)<\infty$, then $\cd_R(G)=\ccd_R(G)$.
	\end{enumerate}
\end{thmAlph}
\begin{rem}
	For readers more accustomed to quasi-isometries than to coarse equivalences, we note that coarse equivalences and quasi-isometries coincide for quasi-geodesic metric spaces; see Lemma~\ref{lem:gro_triviallem}. In particular, two finitely generated groups are coarsely equivalent if and only if they are quasi-isometric.
\end{rem}
Applying Theorem~\ref{thm:ccd_intro} to the class of countable groups with finite cohomological dimension, we recover the following result of Sauer:
\begin{corAlph}[{\cite[Theorem 1.2]{sauer2006cdqi}}]\label{cor:cd_intro}
	Let $R$ be a commutative ring and let $G$ and $H$ be countable groups. If $f:H\to G$ is a coarse embedding and $\cd_R(H)$ is finite, then $\cd_R(H)\leq \cd_R(G)$.
	In particular, if $G$ and $H$ are finitely generated quasi-isometric groups and $\cd_R(H),\cd_R(G)<\infty$, then $\cd_R(G)=\cd_R(H)$.
\end{corAlph}
However, Theorem~\ref{thm:ccd_intro} is a stronger and more powerful result than Corollary~\ref{cor:cd_intro}, as it applies to metric spaces rather than just groups. Moreover, while $\cd_R(-)$ is only a quasi-isometry invariant when restricted to groups of finite cohomological dimensional, $\ccd_R(-)$ is a genuine quasi-isometry invariant amongst all finitely generated groups, agreeing with $\cd_R(-)$ when the latter is finite.

It should be remarked that our proof of  Corollary~\ref{cor:cd_intro} takes a rather different approach than that used by Sauer. In~\cite{sauer2006cdqi}, given a coarse embedding $H\to G$ and an $H$-module $M$,  an induced $G$-module $\overline{I}(M)$ is constructed, which allows one to relate the cohomology of $G$ to the cohomology of $H$.  See also~\cite{li2018dynamic}. Our approach is to bypass group cohomology and $G$-modules entirely, instead relating the coarse cohomological dimensions of $G$ and $H$ when considered as metric spaces.

A (virtually) torsion-free hyperbolic group $G$  has finite (virtual) cohomological dimension, and Bestvina--Mess proved the famous formula~\cite{bestvina1991boundary} \[\vcd(G)=\dim(\partial G)+1,\] where $\dim(\partial G)$ is the Lebesgue covering dimension of the Gromov boundary of $G$. Using work of Bestvina, Dranishnikov and Moran on $\cZ$-boundaries~\cite{bestvina1996zboundary,dranishnikov2006bestvinamess,moran2016finite}, we extend the Bestvina--Mess formula to all hyperbolic and $\CAT(0)$ groups:
\newcommand{\bdry}{Suppose that either:
	\begin{enumerate}
		\item $G$ is a hyperbolic group and $X$ is a Cayley graph of $G$; or
		\item $G$ is a $\CAT(0)$ group acting geometrically on a proper $\CAT(0)$ space $X$.
	\end{enumerate}
	Then  $\ccd(G)=\dim(\partial X)+1<\infty$.}
\begin{thmAlph}[{Corollary~\ref{cor:bdry}}]\label{thm:bdry_intro}
	\bdry{}
\end{thmAlph}
Here, $\partial X$ is either the Gromov or the visual boundary of $X$. It is a notorious open problem to determine whether all hyperbolic groups have a finite-index torsion-free subgroup. This problem is equivalent to determining whether all hyperbolic groups are residually finite~\cite{wise96thesis,olshanskiui2000basslubotzky}. One consequence of a positive solution to this problem would be that all hyperbolic groups have finite virtual cohomological dimension, and then the original Bestvina--Mess formula would hold. Theorem~\ref{thm:bdry_intro} circumvents this issue by  passing to the more robust notion of coarse cohomological dimension, which is finite for all hyperbolic groups, regardless of whether they are virtually torsion-free.

In contrast, there  are  examples of $\CAT(0)$ groups that are not virtually torsion-free. Such groups  have infinite (virtual) cohomological dimension, but by Theorem~\ref{thm:bdry_intro}, the Bestvina--Mess formula holds when formulated in terms of coarse cohomological dimension.  For example, Wise constructed a group $G$ which acts geometrically on a $\CAT(0)$ square complex, but has no finite-index torsion-free subgroup~\cite{wise96thesis}. In fact, this example  is both biautomatic and a $C(4)-T(4)$ small-cancellation group. I am  grateful to Jingyin Huang for alerting the author to this  example.

Recall that a \emph{quasi-tree} is a metric space that is quasi-isometric to a simplicial tree. Work of  Bestvina--Bromberg--Fujiwara has emphasised the importance  of quasi-trees in geometric group theory~\cite{bestvinabrombergfujiwara2015constructing}. Using Manning's \emph{bottleneck criterion}~\cite{manning2005pseudo}, we  give  a coarse topological characterisation of quasi-trees. We restrict ourselves to unbounded quasi-trees, since for each $R$,  a metric space $X$ is bounded if and only if $\ccd_R(X)=0$ (Proposition~\ref{prop:bdd_0dim}).
\begin{thmAlph}[Theorem~\ref{thm:ccd1}]\label{thm:ccd1_intro}
	Let $X$ be a quasi-geodesic metric space and let $R$ be a commutative ring. The following are equivalent:
	\begin{enumerate}
		\item $\ccd_R(X)=1$;
		\item $X$ is an  unbounded quasi-tree.
	\end{enumerate}
\end{thmAlph}
Since a  finitely generated group that is quasi-isometric to a simplicial tree is  virtually free~\cite{dunwoody1985accessibility}, Theorem~\ref{thm:ccd1_intro}  is a generalisation of the following theorem of Stallings (in the case $R=\bbZ$) and Dunwoody (for general $R$).
\begin{cor}[{\cite{stallings1968torsionfree,dunwoody1979accessibility}}]\label{cor:fg_free_intro}
	Suppose $R$ is a commutative ring and $G$ is a finitely generated group.  If  $\cd_R(G)=1$, then $G$ is virtually free.
\end{cor}
\begin{rem}
	In fact, work of Swan and Dunwoody shows Corollary~\ref{cor:fg_free_intro} holds without the hypothesis that $G$ is finitely generated~\cite{swan1969codim1,dunwoody1979accessibility}. However, this stronger statement is not generalised by Theorem~\ref{thm:ccd1_intro}. We note also that Theorem~\ref{thm:ccd1_intro} does \emph{not} give another proof of the fact that a finitely generated  group quasi-isometric to a tree is virtually free; this requires  work of Stallings and Dunwoody~\cite{stallings1968torsionfree,dunwoody1985accessibility}.
\end{rem}

There  is another large-scale notion of dimension of a metric space $X$, due to Gromov, called the \emph{asymptotic dimension} and denoted $\asdim(X)$. In Proposition~\ref{prop:ccd_vs_asdim}, we apply  work of Dranishnikov~\cite{dranishnikov2009cohom} to show that $\ccd(X)\leq \asdim(X)$ provided $\ccd(X)<\infty$ and $X$ is coarsely uniformly acyclic. However, the hypothesis that $X$ is coarsely uniformly acyclic cannot be dropped due to the following example of Denis Osin~\cite[Example 1.3]{fujiwarawhyte2007note}. For any non-trivial finite group $F$, the wreath product $G=F\wr \bbZ$ is finitely generated, not  finitely presented, and satisfies $\asdim(G)=1$. In particular, $G$ is not virtually free, so  by Theorem~\ref{thm:ccd1_intro}, we have $\ccd(G)>1=\asdim(G)$.

Outside the world of discrete countable groups, the tools developed in this article can be used to study various non-discrete topological groups.
It is known that many uncountable groups,  such as locally compact compactly generated groups~\cite{cornulierdlH2016metric} and more generally $CB$ generated groups~\cite{rosendal2022coarse}, can be endowed with a metric unique up to quasi-isometry. Thus, $\ccd_R(G)$ provides a well-defined  dimension of such groups.
For instance, if $T_\infty$ is the regular tree of countably infinite valence, then Rosendal showed that $\Aut(T_\infty)$ is quasi-isometric to $T_\infty$~\cite[Example 6.34]{rosendal2022coarse}. Therefore, Theorem~\ref{thm:ccd1_intro} implies that $\ccd(\Aut(T_\infty))=1$.   One potential future application of coarse cohomological dimension in particular, and the theory of $R$-modules over metric spaces in general,  would be to distinguish \emph{big groups}, such as big mapping class groups and homeomorphism groups of manifolds, up to quasi-isometry.

\subsection*{Comparison to other coarse cohomology theories}
We briefly compare and contrast our methods to earlier approaches in the study of coarse algebraic topology of groups and metric spaces.  Coarse topology was developed and used by many authors throughout the last 30–40 years~\cite{blockweinberger1992aperiodic,gromov1993asymptotic,gersten1993quasi,roe1993coarse,schwartz95r1lattices,farbschwartz96,farbmosher1998bs1,eskinfarb1997quasiflats,mosher2003quasi,mosher2011quasiactions}. Early work of Gersten~\cite{gersten1993quasi} showed that when restricted to groups of type $FP_\infty$, $H^i(G,\bbZ G)$ is a quasi-isometry invariant, and that when restricted to groups of type $FP$, $\cd(G)$ is a quasi-isometry invariant. These ideas were significantly developed by Kapovich--Kleiner in both the main article and appendix of~\cite{kapovich2005coarse},  where they developed  the coarse algebraic topology of \emph{coarsely uniformly acyclic metric spaces}.
The work of Kapovich--Kleiner has  several applications, including work of Farb--Mosher~\cite{farbmosher2000abelianbycyclic,farbmosher2002surfacebyfree}, Mosher--Sageev--Whyte~\cite{mosher2003quasi,mosher2011quasiactions}, Papasoglu~\cite{papasoglu2007group} and earlier works of the author~\cite{margolis2018quasi,margolisxer2021geometry,margolis2019codim1}.
The framework and language used in this article is principally inspired by these works. However, one of the key novelties in our approach is that by using filtrations,  we are able to prove results for general metric spaces with no uniform acyclicity hypothesis.

Another approach to coarse algebraic topology is due to Roe~\cite{roe1993coarse}. The coarse cohomology modules $\coarse^k(X;R)$ defined in this article coincide with Roe's coarse cohomology (Proposition~\ref{prop:roe_iso}), although we consider more general coarse cohomology modules of the form $\coarse^k(X;\bM)$ where $\bM$ is an $R$-module over $X$.
Although Roe's theory is well suited to problems in index theory,  the main difficulty with Roe's theory is that it is harder to carry out  arguments at the cochain level that are needed, for instance,  to prove results such as Theorems~\ref{thm:main_field_intro} and~\ref{thm:main_finitenss_space_intro}. As Kapovich--Kleiner state, ``By passing to the limit \ldots{} one inevitably loses quantitative information which
is essential in many applications of coarse topology to quasi-isometries
and geometric group theory''. Another key feature of our work,  not fully developed by either Roe or Kapovich--Kleiner, is a homological algebraic approach to manipulating \emph{chain complexes of $R$-modules over a metric space} as objects in their own right and not merely as a means to defining homology and cohomology. This perspective is central to this article, especially   in  the theory of coarse cohomological dimension.

\subsection*{Outline of article}
In Section~\ref{sec:prelim}, we review basic concepts in coarse geometry and group cohomology. In Sections~\ref{sec:rmod_metric} and~\ref{sec:projective}, we define  $R$-modules over a metric space and projective $R$-modules over a metric space. Section~\ref{sec:duality} establishes the basic properties of $\Hom$-sets of maps between $R$-modules over a metric space.

In Section~\ref{sec:projres}, we finally define the main tool of this paper: projective $R$-resolutions over a metric space. These are analogues of projective resolutions of the trivial $RG$-module $R$. After doing this, it is relatively straightforward to define and prove basic properties about coarse cohomology and homology, coarse cohomological dimension, and the theory of uniformly acyclic spaces, which we do in Sections~\ref{sec:coarse_cohom},~\ref{sec:coarse_dim} and~\ref{sec:unif_acyc} respectively. In Section~\ref{sec:ccd1}, we characterise quasi-trees in terms of coarse cohomological dimension, using Manning's bottleneck criterion~\cite{manning2005pseudo}.

In Section~\ref{sec:cupcap} we develop a notion of cup and cap products for proper metric spaces. Section~\ref{sec:coarsePDn} is the culmination of the article, in which we use virtually all the machinery developed so far to further extend the theory of coarse $PD_n$ spaces, originally developed by Kapovich--Kleiner~\cite{kapovich2005coarse}. In particular, we prove Theorems~\ref{thm:main_field_intro} and~\ref{thm:main_finitenss_space_intro}. To do this, we make use of some results from Appendix~\ref{sec:inverse_limits} regarding the cardinality of the cohomology of inverse and derived limits of cochain complexes.

\subsection*{Acknowledgements}
I would like to thank Jingyin Huang for many helpful discussions and suggestions.

\section{Preliminaries}\label{sec:prelim}
All rings considered in this article will possess a multiplicative identity, and all homomorphisms between such rings will be assumed to preserve the multiplicative identity. Most rings we consider will be commutative, and will typically be denoted by either $R$ or $S$. The only potentially non-commutative   rings that we consider are group rings, which will be denoted $RG$, $SH$ etc.
We assume $0$ is a natural number.

We will frequently work with maps $f:X\to Y$ between partially ordered sets.  It is always assumed that $\bbN$ and $\bbR$ are equipped with their standard orders, and that  $\bbN\times \bbN$ is equipped with the  product partial order, i.e.\ $(m,n)\leq (m',n')$ if and only if $m\leq m'$ and $n\leq n'$. A map $f:X\to Y$ between partially ordered sets is \emph{increasing} if $f(x)\leq f(y)$ whenever $x\leq y$. In particular, increasing functions need not be strictly increasing.
The set of all maps $f:X\to Y$ between partially ordered sets will itself be equipped with a partial order such that $f\leq f'$ when $f(x)\leq f'(x)$ for all $x\in X$. In the case $Y$ is totally ordered, any non-empty finite collection of functions $X\to Y$ will have a maximum element.

\subsection{Coarse geometry of  metric spaces}
Let $X$ and $Y$ be metric spaces.  Given subsets $Z,W\subseteq X$ and $r\geq 0$, we write \[N_r(Z)\coloneqq \{x\in X\mid d(x,Z)\leq r\}\] for the closed metric $r$-neighbourhood of $Z$, and $\Haus(Z,W)$ for the Hausdorff distance between $Z$ and $W$.
Given $A\geq 0$, two maps $f,f':X\to Y$ are said to be \emph{$A$-close} if $d(f(x),f'(x))\leq A$ for all $x\in X$. We say $f,f':X\to Y$ are \emph{close} if they are $A$-close for some $A$. An $A$-\emph{coarse inverse} of $f:X\to Y$ is a map $g:Y\to X$ such that $g\circ f$ and $f\circ g$ are $A$-close to $\id_X$ and $\id_Y$ respectively. A \emph{coarse inverse} of $f$ is an $A$-coarse inverse of $f$ for some $A\geq 0$.
\begin{defn}
	Let $(X,d_X)$ and $(Y,d_Y)$ be metric spaces, and let $\Upsilon,\Lambda:\bbR_{\geq 0}\to \bbR_{\geq 0}$ be proper increasing functions. We say that a function $f:X\to Y$ is:
	\begin{itemize}
		\item \emph{$\Upsilon$-bornologous} if for all $x,x'\in X$, $d_Y(f(x),f(x'))\leq \Upsilon(d_X(x,x'))$.
		\item A \emph{$(\Lambda,\Upsilon)$-coarse embedding} if for all $x,x'\in X$, \[\Lambda\left(d_X(x,x')\right)\leq d_Y(f(x),f(x'))\leq \Upsilon\left(d_X(x,x')\right).\]
		\item A \emph{$(\Lambda,\Upsilon)$-coarse equivalence} if $f$ is a $(\Lambda,\Upsilon)$-coarse embedding, and has an $\Upsilon(0)$-coarse inverse that is also a $(\Lambda,\Upsilon)$-coarse embedding.
	\end{itemize}
	We call $\Upsilon$ and $\Lambda$  the \emph{control functions} of $f$.
	We say that $f$ is \emph{bornologous} if it is $\Upsilon$-bornologous for some proper increasing function $\Upsilon$, and we define the notion of a  \emph{coarse embedding} and a \emph{coarse equivalence} similarly.
\end{defn}

If $\Upsilon:\mathbb{R}_{\geq 0}\rightarrow \mathbb{R}_{\geq 0}$ is a proper increasing function, we define the \emph{weak inverse} of $\Upsilon$ to be the proper increasing function $\widetilde {\Upsilon}:\mathbb{R}_{\geq 0}\rightarrow \mathbb{R}_{\geq 0}$ given by \[\widetilde{\Upsilon}(R)\coloneqq \begin{cases}
		\sup(\Upsilon^{-1}([0,R])) & \textrm{if}\;\Upsilon(0)\leq R, \\
		0                          & \textrm{if}\; \Upsilon(0)> R.
	\end{cases}\]  The weak inverse satisfies the following properties:
\begin{itemize}
	\item If $\Upsilon(S)\leq R$, then $S\leq \widetilde \Upsilon(R)$.
	\item If $R<\Upsilon(S)$, then $\widetilde \Upsilon(R)\leq S$.
\end{itemize}
Using weak inverses, is straightforward to verify a coarse embedding $f:X\to Y$ is a coarse equivalence if and only if $Y=N_A(f(X))$ for some $A$.

A map $f:X\to Y$ is \emph{coarse Lipschitz} (resp.\ a \emph{quasi-isometric embedding}, a \emph{quasi-isometry}) if it is bornologous (resp.\ a coarse embedding, a coarse equivalence) with affine control functions. A metric space is said to be \emph{coarsely  geodesic} (resp.\ \emph{quasi-geodesic}) if it is coarsely equivalent (resp.\ quasi-isometric) to a geodesic metric space.

The following lemma is due to Gromov:
\begin{lem}[{\cite[\S0.2.D, Lemma]{gromov1993asymptotic}}]\label{lem:gro_triviallem}
	Let $X$ and $Y$ be metric spaces, and let $f:X\to Y$ be bornologous.
	\begin{enumerate}
		\item If $X$ is quasi-geodesic, then $f$ is coarse Lipschitz.
		\item If $X$ and $Y$ are quasi-geodesic and $f$ is a coarse equivalence, then $f$ is a quasi-isometry.
	\end{enumerate}
\end{lem}

Each countable group $G$ admits a proper left-invariant metric, and this metric is unique up to coarse equivalence; see for instance~\cite[\S1]{nowakyu2023large}. For example, if $G$ is finitely generated, the word metric with respect to any finite generating set is such a metric, and is in fact a quasi-geodesic metric.  In this article, we always assume every countable group $G$ is equipped with such a metric. Moreover, in the case $G$ is finitely generated, we assume this metric is the word metric with respect to a finite generating set, hence is quasi-geodesic.

Given a metric space $X$ and a parameter $r\geq 0$, the \emph{Rips complex}  $P_r(X)$  of $X$ with parameter $r$ is the  simplicial complex with vertex set $X$, and where each finite subset $\{x_0,\dots, x_k\}\subseteq X$ spans a simplex of $P_r(X)$  if $d(x_i,x_j)\leq r$ for all $0\leq i,j\leq k$. For each subspace $Y\subseteq X$ and  $s\leq r$, we always identify $P_s(Y)$ with its image in $P_r(X)$ under the natural inclusion $P_s(Y)\to P_r(X)$.

In the following definition, if $Z$ is a simplicial complex and $R$ is a commutative ring, then $\widetilde{H}_k(Z;R)$ denotes the $k$th reduced simplicial homology of $Z$ with coefficients in $R$.
\begin{defn}\label{defn:unif_acyc}
	Let  $\lambda:\mathbb{R}_{\geq 0}\rightarrow \mathbb{R}_{\geq 0}$ and $\mu:\mathbb{R}_{\geq 0}\times \mathbb{R}_{\geq 0} \rightarrow \mathbb{R}_{\geq 0}$  be increasing functions such that $\lambda(i)\geq i$ and $\mu(i,r)\geq r$ for all $i,r\in \mathbb{R}$.  Given a commutative ring $R$, we say that a metric space $X$ is \emph{$(\lambda,\mu)$-coarsely uniformly $n$-acyclic over $R$} if for every $k\leq n$, $i,r\geq 0$ and $x\in X$, the map
	\[\widetilde{H}_k(P_i(N_r(x));R)\to \widetilde{H}_k(P_{\lambda(i)}(N_{\mu(i,r)}(x));R),\] induced by inclusion, is zero. A metric space is \emph{coarsely uniformly $n$-acyclic over $R$} if there exist $\lambda$ and $\mu$ as above, and a metric space $X$ is \emph{coarsely uniformly acyclic over $R$} if it is coarsely uniformly $n$-acyclic over $R$ for each $n$.
\end{defn}
Coarse uniform $n$-acyclicity over $R$ is invariant under coarse equivalence~\cite[Proposition 2.15]{margolis2018quasi}.

\subsection{Group cohomology}
We recall  some key concepts from group cohomology. The reader is referred to  the books of Bieri and Brown for more details~\cite{bieri1981homological,brown1982cohomology}. We fix a commutative ring $R$.
Suppose $G$ is a group and \[\cdots \to P_n\to P_{n-1}\to \cdots \to P_0\to R\to 0\] is a projective resolution of  the trivial $RG$-module $R$. If $N$ is any $RG$-module, we define the \emph{group cohomology}   $H^k(G,N)$ to be the $k$th cohomology $H^k(\Hom_{RG}(P_\bullet,N))$ of the cochain complex $\Hom_{RG}(P_\bullet,N)$ obtained by applying the functor $\Hom_{RG}(-,N)$  to $P_\bullet$.

The \emph{cohomological dimension} $\cd_R(G)$ of $G$ over $R$ is the least $n$ such that there exists a length $n$ projective resolution  \[\cdots\to 0 \to P_n\to P_{n-1}\to \cdots \to P_0\to R\to 0\] of  the trivial $RG$-module $R$. We set $\cd_R(G)=\infty$ if no such $n$ exists. If $G$ acts freely on a contractible $n$-dimensional CW complex, then $\cd_R(G)\leq n$ for all $R$. We write $\cd_\bbZ(G)$ as $\cd(G)$. If $G$ has torsion, then $\cd(G)=\infty$. If $G$ has a torsion-free finite-index subgroup $G'$, then the \emph{virtual cohomological dimension} $\vcd_R(G)$ is defined to be $\cd_R(G')$. A theorem of Serre ensures that $\vcd_R(G)$ is independent of the choice of $G'$. If no such $G'$ exists, then $\vcd_R(G)=\infty$.

Given a commutative ring $R$,  a group $G$ is of type $FP_n^R$ if the trivial $RG$-module $R$ admits a projective resolution \[\cdots \to P_n\to P_{n-1}\to \cdots \to P_0\to R\to 0\] with $P_i$ a finitely generated $RG$-module for each $i\leq n$. In the case $R=\bbZ$, we write type $FP_n^\bbZ$ simply as $FP_n$. Being of type $FP_1$ coincides with being finitely generated, whilst being of type $FP_2$ generalises the notion of being finitely presented; see~\cite{bestvina1997morse,geoghegan2008topological} for more details. A group that acts freely on a contractible locally finite CW complex with finitely many $G$-orbits of $k$-cells for $k\leq n$ is necessarily of type $FP_n^{R}$.

The following can be deduced from work of Alonso~\cite{alonso1994finiteness} and Kapovich--Kleiner~\cite{kapovich2005coarse}; see also~\cite[Theorem 3.2]{margolis2018quasi} for an explicit statement.
\begin{thm}[{\cite{alonso1994finiteness,kapovich2005coarse},\cite[Theorem 3.2]{margolis2018quasi}}]\label{thm:type_FPn}
	A countable group $G$ is of type $FP_n^R$ if and only if it is coarsely uniformly $(n-1)$-acyclic over $R$.
\end{thm}

\section{Modules over metric spaces}\label{sec:rmod_metric}
\emph{We define an $R$-module over a metric space in Definition~\ref{defn:rmod_overX}.  We discuss and develop  the basic theory of $R$-modules over a metric space, including finite displacement maps, maps satisfying the uniform preimage property, pullbacks and pushforwards.}
\vspace{.3cm}

Henceforth, we let $R$ be a commutative ring unless explicitly stated otherwise.

A  \emph{filtration} of a set $B$ is a collection $\{B(i)\}_{i\in \bbN}$ of subsets of $B$ satisfying the following properties:
\begin{enumerate}
	\item If $i\leq j$, then $B(i)\subseteq B(j)$.
	\item $B=\cup_{i\in \bbN} B(i)$.
\end{enumerate}
As all filtrations considered in this article will have indexing set $\bbN$, we will write $\{B(i)\}_{i\in \bbN}$ as either $\{B(i)\}_{i}$ or as $\{B(i)\}$ when convenient.

\begin{defn}\label{defn:rmod_overX}
	Let $X$ be a metric space and let $R$ be  a commutative ring. An \emph{$R$-module over $X$} consists of the data $\bM=(M,B,\delta,p,\cF)$ where:
	\begin{enumerate}
		\item $M$ is an $R$-module, called the \emph{underlying $R$-module}.
		\item $B$ is an arbitrary indexing set.
		\item $\delta$ is a function $B\to M^*=\Hom_R(M,R)$. We denote the image of $b\in B$ under $\delta$ as $\delta_b$. We call  each $\delta_b$ a \emph{coordinate}.
		\item $p:B\to X$ is a function called the \emph{control function}.
		\item For each $m\in M$, $\delta_b(m)=0$ for all $b\in B$ if and only if $m=0$.\label{item:moddef_zero}
		\item $\cF=\{B(i)\}_{i\in \bbN}$ is a filtration of $B$  such that for each $m\in M$, there exists some $i\in \bbN$ such that $\delta_b(m)=0$ for all $b\in B\setminus B(i)$.\label{item:moddef_filtr}
	\end{enumerate}
	For each $m\in M$, we write \[\supp_\bM(m)\coloneqq \{p(b)\mid b\in B, \delta_b(m)\neq 0\},\] and for each $i\in \bbN$ we write, \[\bM(i)\coloneqq \{m\in M\mid \delta_b(m)=0 \text{ for all }b\in B\setminus B(i)\}.\]
\end{defn}
We note that $\{\bM(i)\}_i$ is a filtration of $\bM$ by submodules. For every $R$-module $\bM$ over $X$, the product map $\Pi_{b\in B}\delta_b:M\to \Pi_{b\in B}R$ is an embedding. This places restrictions on  $R$-modules over $X$; for instance, they are always   torsion-free (or more accurately, their underlying $R$-modules are torsion-free).  We frequently write $\{\delta_b\}_{b\in B}$ or even $\{\delta_b\}$ instead of $\delta$ when convenient.
An  $R$-module $\bM$ over $X$ is frequently identified with its underlying $R$-module $M$, so we will write expressions such as $m\in \bM$ and $m\in M$ interchangeably. We will use bold typeface $\bM$, $\bN$ etc.\ to denote   $R$-modules over metric spaces, and will use regular typeface $M$, $N$ etc.\ to denote $R$-modules that are not equipped with the structure of an $R$-module over a metric space.
When unambiguous, we write $\supp$ instead of $\supp_\bM$.

Before developing the general theory, we first give some important examples of $R$-modules over a metric space.
\begin{exmp}\label{exmp:first_exam}
	Let $X$ be a metric space and let $M$ be the free abelian group with basis the set $B=\{[x,y]\mid x,y\in X\}$ of all ordered pairs in $X$. We define $\delta:B\to M^*$ so that $\{\delta_b\}_{b\in B}$ is the dual basis to $B$, i.e.\ \[\delta_{[x,y]}([x',y'])=\begin{cases}
			1 & \textrm{if}\; x=y\;\textrm{and}\; y=y', \\
			0 & \textrm{otherwise}.\end{cases}\] We define the map $p:B\to X$ by $[x,y]\mapsto x$.
	Finally, we define a filtration $\{B(i)\}$ of $B$ by \[B(i)\coloneqq\{[x,y]\in B\mid d(x,y)\leq i\}\] for each $i\in \bbN$.

	We claim that $\bM=(M,B,\delta,p,\{B(i)\})$ is a $\bbZ$-module over $X$. Indeed, every $m\in M$ can be written as a sum \[m=\sum_{b\in B}\delta_b(m)b\] in which almost all (i.e.\ all but finitely many) terms are zero. Thus either $m=0$, or $\delta_b(m)\neq 0$ for some $b\in B$, proving (\ref{item:moddef_zero}). Moreover, as almost all terms in the above sum are zero, there is some $i$ such that $\delta_b(m)=0$ for all $b\in B\setminus B(i)$, proving (\ref{item:moddef_filtr}).

	A  slogan that should  guide the reader's intuition   about   elements of a general $R$-module $\bN$ over a metric space $X$ is the following:
	\begin{quote}
		\emph{If $n\in \bN(i)$, then $n$ should be visualized as lying close to $\supp_\bN(n)$, up to some error depending on $i$.}
	\end{quote}
	We examine this statement for $\bN=\bM$ as defined above. Suppose $x,y,z,w$ are distinct points in $X$ and $m=[x,y]+3[z,w]\in \bM$. We have $\delta_{[x,y]}(m)=1$, $\delta_{[z,w]}(m)=3$, and $\delta_b(m)=0$ for all $b\in B\setminus \{[x,y],[z,w]\}$. Thus $\supp_\bM(m)=\{x,z\}$. We now note that $m\in \bM(i)$ if and only if $i\geq \max(d(x,y),d(z,w))$. Thus if $m\in \bM(i)$ for some $i$, we have $\{x,y,z,w\}\subseteq N_i(\supp_\bM(m))$. This holds more generally: if $m'\in \bM(j)$ and $L\subseteq X$ is the set of all elements that appear when we express $m'$ as a sum of elements of $B$, then $L\subseteq N_j(\supp_\bM(m'))$.

\end{exmp}

\begin{exmp}\label{exmp:orbit}
	Let $G$ be a countable group equipped with a proper left-invariant metric. Let $M$ be a free $\bbZ G$-module with countably infinite $\bbZ G$-basis $C=\{c_0,c_1,\ldots \}$.
	We let $B\coloneqq GC_0$, which is a free $\bbZ$-basis of $M$. Let $\{\delta_b\}_{b\in B}$ be the dual basis of $B$. We define $p:B\to G$ by $p(gc_i)=g$ for all $i\in \bbN$ and $g\in G$, and we set $B(i)=\{gc_k\mid g\in G, k\leq i \}$. Then \[\bM=(M,B,\delta,p,\{B(i)\})\] is a $\bbZ$-module over $G$.
	For each $m\in \bM$, knowing $\supp_\bM(m)$ and that $m\in \bM(i)$ tells us a lot of information about $m$. For example, if $m\in \bM(1)$ and $\supp_\bM(m)=\{g,k\}$, we know that $m$ must lie in the $\bbZ$-span of $\{gc_0, gc_1, kc_0, kc_1\}$.
\end{exmp}

Examples~\ref{exmp:first_exam} and~\ref{exmp:orbit} are both instances of \emph{free $R$-modules over a metric space}, which we discuss in more detail in Section~\ref{sec:projective}. These specific examples will be generalised and developed further in Proposition~\ref{prop:standard_proj_res} and Lemma~\ref{lem:ginduced_produce} respectively.

We now return to the general theory of $R$-modules over metric spaces.
\begin{lem}\label{lem:support_sum}
	Let $\bM$ be an $R$-module over $X$, and let $m_1, m_2\in \bM$ with $r_1,r_2\in R$. Then $\supp_\bM(r_1m_1+r_2m_2)\subseteq \supp_\bM(m_1)\cup \supp_\bM(m_2)$.
\end{lem}
\begin{proof}
	Suppose $\bM=(M,B,\delta,p,\cF)$ and  $x\in \supp_\bM(r_1m_1+r_2m_2)$. Then there exists $b\in B$ with $p(b)=x$ and $\delta_b(r_1m_1+r_2m_2)\neq 0$. Then one of $\delta_b(m_1)$ or $\delta_b(m_2)$ is non-zero, and so either $x\in \supp(m_1)$ or $x\in \supp(m_2)$.
\end{proof}

A \emph{homomorphism} $\phi:\bM\to \bN$ between $R$-modules over $X$ and $Y$ respectively, is  simply a homomorphism $\phi:M\to N$ between underlying $R$-modules.
However, two homomorphisms  $\phi:\bM\to \bN$ and $\psi:\bM'\to \bN'$ are regarded as being equal when $\bM=\bM'$, $\bN=\bN'$ and the underlying homomorphism $M\to N$ are equal. Formally, we  should think of  a homomorphism from $\bM$ to $\bN$ as a triple $(\bM,\bN,\phi)$, where $\phi:M\to N$ is a homomorphism between underlying $R$-modules. However, it is notationally cumbersome to distinguish between the triple $(\bM,\bN,\phi)$ and the map $\phi$.
In particular, we note that  if $\bM$ and $\bM'$ are different $R$-modules over $X$, then $\phi:\bM\to \bN$ and $\psi:\bM'\to \bN$  should always be thought of as distinct since they have different ``domains'', even in the case when $\bM$ and $\bM'$ have the same underlying $R$-module $M$  and the underlying homomorphisms $M\to N$ between $R$-modules are equal.

\begin{defn}\label{defn:submodule}
	Let $X$ be a metric space and let $\bM=(M,B,\delta,p,\cF)$ be an $R$-module over $X$. A \emph{geometric submodule} $\bN$ of $\bM$, denoted  $\bN\leq \bM$, is an $R$-module over $X$ of the form $\bN=(N,B,\epsilon,p,\cF)$ where $N$ is a submodule of $M$ and for  each $b\in B$, the coordinate $\epsilon_b$ of $\bN$ is the restriction of the coordinate $\delta_b$ of  $\bM$.
\end{defn}
If $\bN\leq \bM$, then for each $n\in \bN$ and $i\in \bbN$, we have $\supp_\bN(n)=\supp_\bM(n)$ and  $\bN(i)=\bM(i)\cap N$. In particular, there is no ambiguity in writing $\supp(n)$.
It is clear that there is a bijective correspondence between geometric submodules of $\bM$ and submodules of the underlying $R$-module $M$.
We emphasise that the notation $\bN\leq \bM$ means $\bN$ is a geometric submodule, which is stronger than merely requiring that the underlying $R$-module of $\bN$ be a submodule of the underlying $R$-module of $\bM$, which we  denote as $N\leq M$.

We now define the appropriate notion of morphism between $R$-modules over a metric space:
\begin{defn}\label{defn:findisp}
	Let $\bM$ and $\bN$ be $R$-modules over $X$ and $Y$ respectively. Given a map $f:X\to Y$ and a  homomorphism $\phi:\bM\to \bN$, we say $\phi$ has \emph{finite displacement}  over $f$ if there exists an increasing function $\Phi:\bbN\to\bbN$ such that the following hold for all $i\in \bbN$:
	\begin{enumerate}
		\item $\phi(\bM(i))\subseteq \bN({\Phi(i)})$;
		\item for all $m\in \bM(i)$, $\supp_{\bN}(\phi(m))\subseteq N_{\Phi(i)}(f(\supp_{\bM}(m)))$.
	\end{enumerate}
	When the above holds, we say that $\phi$ has \emph{displacement $\Phi$ over $f$}.
	In the case where $X=Y$ and no $f:X\to X$ is explicitly mentioned, we say  $\phi$ has \emph{finite displacement} when it has finite displacement over the identity map $\id_X:X\to X$.
\end{defn}
We note that if $f,g:X\to Y$ are close, then $\phi:\bM\to \bN$ has finite displacement over $f$ if and only if it has finite displacement over $g$.

\begin{defn}\label{defn:iso}
	A map $\phi:\bM\to \bN$ between two  $R$-modules  over $X$ is a \emph{geometric isomorphism} if it is an isomorphism between the underlying $R$-modules such that both $\phi$ and $\phi^{-1}$ have finite displacement over $\id_X$.
\end{defn}

\begin{rem}
	We do \emph{not} typically identify geometrically isomorphic $R$-modules over $X$. This is because $\supp_{\bM}(m)$ is not necessarily equal to $\supp_{\bN}(\phi(n))$ when $\phi:\bM\to \bN$ is a geometric isomorphism. For quantitative arguments in this paper, it is vital that $\supp_{\bM}(m)$ has a definite meaning,  and this  is lost upon replacing  $\bM$ with its isomorphism class. Nonetheless,  $\supp_{\bM}(m)$ and $\supp_{\bN}(\phi(m))$  are related up to some  error depending on the displacement of $\phi$ and $\phi^{-1}$.
\end{rem}

We will often consider two $R$-modules $\bM$ and $\bN$ over metric spaces $X$ and $Y$ that have the same underlying $R$-module $M$. In this situation, the \emph{canonical identification} is the map $\iota:\bM\to \bN$ that is the identity on $M$. In the case $X=Y$, we say $\bM$ and $\bN$ are \emph{canonically isomorphic} if the canonical identification is a geometric isomorphism, i.e.\ both $\iota$ and $\iota^{-1}$ have finite displacement over $\id_X$. For transparency, we explicitly state when $R$-modules over a metric space are canonically isomorphic:

\begin{lem}\label{lem:canon_isom}
	Let $X$ be a metric space and let $\bM$ and $\bM'$ be $R$-modules over $X$ with the same underlying $R$-module. Let $\iota:\bM\to \bM'$ be the canonical identification. Then $\bM$ and $\bM'$ are canonically isomorphic if and only if there exists an increasing function $\Psi:\bbN\to \bbN$ such that the following hold:
	\begin{enumerate}
		\item $\iota(\bM(i))\subseteq \bM'(\Psi(i))$ and $\bM'(i)\subseteq \iota(\bM(\Psi(i)))$ for all $i$;
		\item For all $m\in \bM(i)$, $\Haus(\supp_\bM(m),\supp_{\bM'}(\iota(m)))\leq \Psi(i)$.
	\end{enumerate}
\end{lem}
If $\bN\leq \bM$ is a geometric submodule, then the inclusion $\bN\to \bM$ has finite displacement.
If $\phi:\bN\to \bM$ has finite displacement, then we  think of  $\ker(\phi)\leq \bN$ and $\im(\phi)\leq \bM$ as geometric submodules. Moreover, the induced map $\bN\to \im(\phi)$ through which $\phi$ factors also has finite displacement.
\begin{rem}
	If $\bN$ and $\bM$ are $R$-modules over $X$ with $\bN\leq \bM$ a geometric submodule, then $\bM/\bN$ is \emph{not} naturally endowed with the structure of an $R$-module over $X$. In particular, in the category $\cC$ whose objects are $R$-modules over $X$ and whose morphisms are finite displacement maps over $\id_X$, cokernels do not typically exist. Thus $\cC$ is not an abelian category, and there is no general nonsense approach to carrying out homological algebraic methods in $\cC$.
\end{rem}

We now show that the composition of finite displacement maps is frequently also  of finite displacement:
\begin{lem}\label{lem:compos_findisp}
	Suppose $\bL$, $\bM$ and $\bN$ are   $R$-modules over  $X$, $Y$ and $Z$.
	Suppose also that  $\phi:\bL\to \bM$ and $\psi:\bM\to \bN$ both have displacement $\Psi$ over $f:X\to Y$ and $g:Y\to Z$. We also  suppose that $g$ is $\Upsilon$-bornologous. Then $\psi\circ \phi$ has  displacement $\Psi'$ over $g\circ f$, where $\Psi'$ depends only on $\Psi$ and $\Upsilon$.
\end{lem}
\begin{proof}
	Suppose $l\in \bL(i)$. Then $\psi(\phi(l))\in \bN(\Psi^2(i))$ and
	\begin{align*}
		\supp_\bN(\psi\circ\phi(l)) & \subseteq N_{\Psi^2(i)}(g(\supp_{\bM}(\phi(l))))                     \\
		                            & \subseteq
		N_{\Psi^2(i)}(g( N_{\Psi(i)}(f(\supp_{\bL}l)) ))                                                   \\
		                            & \subseteq N_{\Psi^2(i)+\Upsilon(\Psi(i))}(gf(\supp_{\bL}l)).\qedhere
	\end{align*}
\end{proof}

\begin{defn}
	Let $\phi:\bM\to \bN$ be a homomorphism  between $R$-modules over $X$. We say $\phi$ has \emph{uniform preimages} if there is an increasing function  $\Omega:\bbN\times \bbN\to \bbN$  such that for every $i\in \bbN$ and $n\in \bN(i)\cap \im(\phi)$,  there is some $m\in \phi^{-1}(n)$ such that:
	\begin{itemize}
		\item $m\in \bM({\Omega(i,D)})$ and
		\item $\supp_\bM(m)\subseteq N_{\Omega(i,D)}(\supp_\bN(n))$,
	\end{itemize}
	where $D\coloneqq \lceil\diam(\supp_\bN(n))\rceil$.
	For $\Omega$ as above, we say $\phi$ has \emph{$\Omega$-uniform preimages} when the above condition is satisfied.
\end{defn}
\begin{rem}
	The uniform preimage property is defined only for maps $\bM\to\bN$ between $R$-modules over the same metric space.
\end{rem}

We now discuss pullbacks and pushforwards of modules over metric spaces.
\begin{defn}
	Let $\bM=(M,B,\delta,p,\cF)$ be an $R$-module over $X$. If $f:X\to Y$ is an arbitrary map, the \emph{pushforward} of $\bM$ is the $R$-module $f_*\bM=(M,B,\delta,f\circ p,\cF)$ over $Y$.
\end{defn}
Since $\bM$ and $f_*\bM$ have the same underlying $R$-module, there is a  canonical identification $\iota:\bM\to f_*\bM$.
It is clear that for each $m\in \bM$, $f(\supp_\bM(m))=\supp_{f_*\bM}(\iota (m))$. This implies the canonical identification $\bM\to f_*\bM$ has finite displacement over $f$.  Moreover,  $(\id_X)_*\bM=\bM$ and $g_*f_*\bM=(gf)_*\bM$ for any $g:Y\to Z$.
\begin{lem}\label{lem:induced_map_compos}
	Let $X$ and $Y$ be metric spaces and let $f:X\to Y$ be a coarse embedding. Suppose  $\bM$ and $\bN$ are $R$-modules over $X$ and $\bL$ is an $R$-module over $Y$. Let $\iota:\bN\to f_*\bN$ be the canonical identification. Suppose $\phi:\bM\to \bL$ has finite displacement over $f$ and  $\psi:\bL\to f_*\bN$ has finite displacement over $\id_Y$.  Then $\iota^{-1}\circ \psi\circ\phi:\bM\to \bN$ has finite displacement over $\id_X$.
	\begin{figure}[htb]
		\begin{tikzcd}
			X && X && {\mathbf{M}} && {\mathbf{N}} \\
			Y && Y && {\mathbf{L}} && {f_*\mathbf{N}}
			\arrow["{\text{id}_X}", from=1-1, to=1-3]
			\arrow["f"', from=1-1, to=2-1]
			\arrow["f", from=1-3, to=2-3]
			\arrow["{\iota^{-1}\circ\psi\circ\phi}", from=1-5, to=1-7]
			\arrow["\phi"', from=1-5, to=2-5]
			\arrow["\iota", from=1-7, to=2-7]
			\arrow["{\text{id}_Y}", from=2-1, to=2-3]
			\arrow["\psi"', from=2-5, to=2-7]
		\end{tikzcd}
		\caption{Each map in the right-hand  diagram has finite displacement over the corresponding map in the left-hand diagram.}
	\end{figure}
\end{lem}

\begin{proof}
	Suppose that  $f$ is a $(\Lambda, \Upsilon)$-coarse embedding, and that $\phi$ and $\psi$ have displacement $\Phi$ over $f$ and $\id_Y$ respectively. We define an increasing function  $\Psi:\bbN\to \bbN$ by $\Psi(i)\coloneqq  \max\left(\widetilde{\Lambda}\left(\Phi(i)+\Phi^2(i)\right),\Phi^2(i)\right)$, where $\widetilde{\Lambda}$ is the weak inverse of $\Lambda$ as in Section~\ref{sec:prelim}. We show $\iota^{-1}\circ \psi\circ\phi:\bM\to \bN$ has displacement $\Psi$ over $\id_X$.

	Fix $i\in \bbN$ and set $i_k=\Phi^k(i)$.
	It is clear that \begin{align*}
		(\iota^{-1}\circ \psi\circ \phi)(\bM(i))\subseteq (\iota^{-1}\circ \psi)(\bL(i_1))\subseteq \iota^{-1}(f_*\bN(i_2))=\bN(i_2)\subseteq \bN(\Psi(i)).
	\end{align*}
	Now suppose $m\in \bM(i)$. Then since $\phi(m)\in \bL(i_1)$, we have:
	\begin{align*}
		f(\supp_\bN((\iota^{-1}\circ\psi\circ\phi)(m))) & =\supp_{f_*\bN}((\psi\circ\phi)(m)) \\&\subseteq N_{i_2}(\supp_\bL(\phi(m))) \\&\subseteq N_{i_2+i_1}(f(\supp_\bM(m))).
	\end{align*}
	It follows that $\supp_\bN((\iota^{-1}\circ\psi\circ\phi)(m))\subseteq N_{\Psi(i)}(\supp_\bM(m))$ as required.
\end{proof}
\begin{lem}\label{lem:induced_close}
	If $\bM$ is an $R$-module over $X$ and $f,g:X\to Y$ are close, then $f_*\bM$ and $g_*\bM$ are canonically isomorphic. In particular, if $f:X\to Y$ is a coarse equivalence with coarse inverse $g:Y\to X$, then $g_*f_*\bM$ is canonically isomorphic to $\bM$.
\end{lem}
\begin{proof}
	The condition that $f$ and $g$ are close ensures there is a constant $A$ such that $f(\supp_\bM(m))$ and $g(\supp_\bM(m))$ have Hausdorff distance at most $A$. By Lemma~\ref{lem:canon_isom}, this implies the canonical identification  $f_*\bM\to g_*\bM$  is a geometric isomorphism.
\end{proof}
The notion of a pullback is more subtle than that of a pushforward:
\begin{defn}\label{defn:pullback}
	Let $f:X\to Y$ be a coarse embedding.
	We define a  map $\pi:Y\to X$  as follows: for each $y\in Y$, pick $\pi(y)\in X$ with $d(f(\pi(y)),y)$ minimal. In other words, $\pi$ is the composition of a closest-point projection $Y\to f(X)$ and a section $f(X)\to X$ of $f$.
	If   $\bM=(M,B,\delta,p,\{B_i\})$ is an $R$-module over $Y$, we define the \emph{pullback} $R$-module $f^*\bM$ over $X$ by \[f^*\bM=(M,B,\delta,\pi\circ p,\{f^*B(i)\}),\]
	where $f^*B(i)\coloneqq \{b\in B(i)\mid d(f(X),p(b))\leq i\}$.
\end{defn}
Although $\pi$ is not necessarily unique,  $f^*\bM$ is nonetheless well-defined up to canonical isomorphism:
\begin{lem}
	There is a function $\Psi:\bbN\to\bbN$ such that for all  $y\in N_i(f(X))$, whenever  $x,x'\in X$ are chosen such that $d(f(x),y)=d(f(x'),y)$ is minimal, then $d(x,x')\leq \Psi(i)$. Consequently, up to canonical isomorphism, $f^*\bM$ is independent of the choice of $\pi:Y\to X$.
\end{lem}
\begin{proof}
	Suppose $f$ is a $(\Lambda,\Upsilon)$-coarse embedding and $\widetilde{\Lambda}$ is the weak inverse of $\Lambda$.
	Since $d(f(x),f(x'))\leq 2i$,  we have $d(x,x')\leq \widetilde{\Lambda}(2i)$.
	Thus if $\pi,\pi':Y\to X$ are two functions as in Definition~\ref{defn:pullback}, then the restrictions of $\pi$ and $\pi'$ to $N_i(f(X))$ are $\widetilde \Lambda(2i)$-close.
	This ensures that for each $b\in f^*B(i)$,  $d(\pi(p(b)),\pi'(p(b)))\leq \widetilde \Lambda(2i)$. By Lemma~\ref{lem:canon_isom}, the resulting pullback modules defined using $\pi$ and $\pi'$ are canonically isomorphic.
\end{proof}

\begin{lem}\label{lem:induced_mod}
	Let $f:X\to Y$ be a coarse embedding between metric spaces, let  $\bM$ be an $R$-module over $Y$, and let $\iota:f^*\bM\to \bM$ be the canonical identification. Then:
	\begin{enumerate}
		\item $\iota(f^*\bM(i))=\{m\in \bM(i)\mid \supp_\bM(m)\subseteq N_i(f(X))\}$.\label{lemitem:induced_mod1}
		\item If $m\in f^*\bM(i)$, then $\Haus(\supp_\bM(\iota(m)),f(\supp_{f^*\bM}(m)))\leq i$.\label{lemitem:induced_mod2}%
	\end{enumerate}
	In particular, $\iota:f^*\bM\to \bM$ has finite displacement over $f$.
\end{lem}
\begin{proof}
	(\ref{lemitem:induced_mod1}) is an immediate consequence of  the definitions. For (\ref{lemitem:induced_mod2}), we fix $m\in f^*\bM(i)$. First, suppose that $x\in \supp_{f^*\bM}(m)$.  Pick $b\in f^*B(i)$ with $\pi(p(b))=x$ and $\delta_b(m)\neq 0$. It follows that $p(b)\in \supp_\bM(\iota(m))$. Since $b\in f^*B(i)$, we have $p(b)\in N_i(f(X))$, and so the definition of $\pi$ ensures $d(f(x),p(b))\leq i$. Thus $f(\supp_{f^*\bM}(m))\subseteq N_i(\supp_\bM(\iota(m)))$.

	Conversely, suppose $y\in \supp_\bM(\iota(m))$. Then there exists some $b\in B$ with $p(b)=y$ and $\delta_b(m)\neq 0$. Since $m\in f^*\bM(i)$ and $\delta_b(m)\neq 0$, we know $b\in f^*B(i)$ and so $d(p(b),f(X))\leq i$. It follows that $\pi(y)\in \supp_{f^*\bM}(m)$ and $d(f(\pi(y)),y)\leq i$, which shows $\supp_\bM(\iota(m))\subseteq N_i(f(\supp_{f^*\bM}(m)))$.
\end{proof}
Let $\bM$ and $\bN$ be $R$-modules over $Y$, and let $f:X\to Y$ be a coarse embedding.   Every map $\phi:\bM\to \bN$ corresponds to a  map $f^*\phi:f^*\bM\to f^*\bN$ by identifying $\bM$ and $\bN$ with $f^*\bM$ and $f^*\bN$ respectively via their canonical identifications.
\begin{prop}\label{prop:induced_mod_maps}
	Let $f:X\to Y$ be a coarse embedding between metric spaces, and let $\bM$ and $\bN$ be $R$-modules over $Y$.
	\begin{enumerate}
		\item If $\phi:\bM\to \bN$ has finite displacement over $\id_Y$, then $f^*\phi:f^*\bM\to f^*\bN$ has finite displacement over $\id_X$.\label{propitem:induced_mod_maps1}
		\item  If $\phi:\bM\to \bN$ has uniform preimages, then so does $f^*\phi$.\label{propitem:induced_mod_maps2}
	\end{enumerate}
\end{prop}
\begin{proof} Suppose $f$ is a $(\Lambda, \Upsilon)$-coarse embedding and $\widetilde{\Lambda}$ is the weak inverse of $\Lambda$. We make use of both parts of Lemma~\ref{lem:induced_mod} freely throughout the subsequent proof.

	(\ref{propitem:induced_mod_maps1}): Suppose $\phi$ has displacement $\Phi$ over $\id_X$ and $m\in f^*\bM(i)$. Then $\iota(f^*\phi(m))=\phi(\iota(m))\subseteq \phi(\bM(i))\subseteq \bN(\Phi(i))$ and
	\begin{align*}\supp_\bN(\phi(\iota(m)))&\subseteq N_{\Phi(i)}(\supp_\bM(\iota(m)))\subseteq N_{\Phi(i)+i}(f(\supp_{f^*\bM}(m)))\subseteq N_{\Phi(i)+i}(f(X)).\end{align*} Thus $f^*\phi(m)\in f^*\bN(\Phi(i)+i)$. Moreover,
	\begin{align*}
		f(\supp_{f^*\bN}(f^*\phi(m))) & \subseteq
		N_{\Phi(i)+i}(\supp_{\bN}(\phi(\iota(m))))                                        \subseteq N_{2\Phi(i)+2i}(f(\supp_{f^*\bM}(m))),
	\end{align*} which ensures that $\supp_{f^*\bN}(f^*\phi(m))\subseteq N_{\widetilde{\Lambda}(2\Phi(i)+2i)}(\supp_{f^*\bM}(m))$. This shows $f^*\phi$ has finite displacement.

	(\ref{propitem:induced_mod_maps2}): 	Suppose $\phi$ has $\Omega$-uniform preimages. Suppose also that $n\in \im(f^*\phi)\cap f^*\bN(i)$ and  $D\coloneqq \lceil\diam(\supp_{f^*\bN}(n)) \rceil$. Then  Lemma~\ref{lem:induced_mod} ensures $\iota(n)\in \im(\phi)\cap \bN(i)$ and $\supp_\bN(\iota(n))\subseteq N_i(f(X))$.
	Moreover, if $y,y'\in \supp_\bN(\iota(n))$, then there exist $x,x'\in \supp_{f^*\bN}(n)$ such that $d(f(x),y)\leq i$ and $ d(f(x'),y')\leq i$. Thus $d(y,y')\leq 2i+d(f(x),f(x'))\leq \Upsilon(D)+2i$, so that  $\diam(\supp_{\bN}(\iota(n)))\leq \Upsilon(D)+2i$.

	Setting $j\coloneqq \Omega(i,\lceil\Upsilon(D)+2i\rceil)$,  there exists some $\iota(m)\in \bM(j)$ with $\phi(\iota(m))=\iota(n)$ and \[\supp_\bM(\iota(m))\subseteq N_j(\supp_\bN(\iota(n)))\subseteq N_{i+j}(f(X)).\] Thus $m\in f^*\bM(i+j)$, $(f^*\phi)(m)=n$ and
	\begin{align*}
		f(\supp_{f^*\bM}(m)) & \subseteq N_{i+j}(\supp_\bM(\iota(m)))\subseteq N_{i+2j}(\supp_\bN(\iota(n))) \subseteq N_{2i+2j}(f(\supp_{f^*\bN}(n))).
	\end{align*} Thus $\supp_{f^*\bM}(m)\subseteq N_{\widetilde{\Lambda}(2i+2j)}(\supp_{f^*\bN}(n))$. This demonstrates that $f^*\phi$ has uniform preimages.
\end{proof}

When dealing with countable families of $R$-modules over metric spaces, it will be convenient to assume the following condition:
\begin{defn}\label{defn:refinement}
	Let $\cM\coloneqq \{\bM_j\}_{j\in \bbN}$, where each \[\bM_j=(M_j,B_j,\delta^j,p_j,\{B_j(i)\}_i)\] is an $R$-module over a metric space $X_j$. We say $\cM$ is \emph{staggered} if for each $i$, the set $\{j\in \bbN\mid B_j(i)\neq \emptyset \}$ is finite.
\end{defn}
The following remark shows that there is no loss of generality in working with staggered collections:
\begin{rem}\label{rem:staggering}
	For any countable collection $\cM\coloneqq \{\bM_j\}_{j\in \bbN}$ of  $R$-modules over metric spaces, we can always modify the filtration on each $\bM_j$ to  ensure that $\cM$ is staggered. In detail, if $\bM_j=(M_j,B_j,\delta^j,p_j,\{B_j(i)\}_i)$, we set \[B'_j(i)=\begin{cases}
			0      & \textrm{if}\; j>i,     \\
			B_j(i) & \textrm{if}\; j\leq i.
		\end{cases}\] We let $\bM'_j\coloneqq(M_j,B_j,\delta^j,p_j,\{B_j'(i)\}_i)$.  We note that $\bM_j$ and $\bM'_j$ are canonically isomorphic.  It is clear from the construction that $\cM'=\{\bM_j'\}$ is staggered.
\end{rem}
The utility of working with staggered collections is the following:
\begin{prop}\label{prop:uniform_displacement}
	Let $\cM\coloneqq \{\bM_j\}_{j\in \bbN}$ and $\cN\coloneqq \{\bN_j\}_{j\in \bbN}$ be staggered, where each $\bM_j$ and $\bN_j$ is an  $R$-module over  $X_j$ and $Y_j$ respectively.
	Suppose   $\{\phi_j:\bM_j\to \bN_j\}_{j\in \bbN}$ is a family of  maps.
	\begin{enumerate}
		\item If each $\phi_j$ has finite displacement over some $f_j:X_j\to Y_j$, then there exists a $\Psi$ such that each $\phi_j$ has displacement $\Psi$ over $f_j$.\label{item:unif_disp_map}
		\item Suppose $X_j=Y_j$ for each $j$ and  each  $\phi_j$ has uniform preimages. Then there exists an $\Omega$ such that each $\phi_j$ has $\Omega$-uniform preimages.\label{item:unif_unif_preimages}
	\end{enumerate}
\end{prop}
\begin{proof}  Suppose \begin{align*}
		\bM_j=(M_j,B_j,\delta^j,p_j,\{B_j(i)\}) \;\text{and}\; \bN_j=(N_j,C_j,\epsilon^j,q_j,\{C_j(i)\}).
	\end{align*}
	Since $\cM$ and $\cN$ are staggered, we can pick an increasing function $\gamma:\bbN\to \bbN$ such that $B_j(i)=C_j(i)=\emptyset$ for all $i,j\in \bbN$ satisfying $j>\gamma(i)$. In particular, this implies $\bM_j(i)=\bN_j(i)=0$ for all $i,j\in \bbN$ satisfying $j>\gamma(i)$.

	(\ref{item:unif_disp_map}): Suppose each $\phi_j:\bM_j\to \bN_j$  has displacement $\Psi_j$ over $f_j$. 	 We set $\Psi(i)\coloneqq \max_{j\leq \gamma(i)}\Psi_j(i)$ and claim that each $\phi_j$ has displacement $\Psi$ over $f_j$.
	To see this, let $i\in \bbN$. In the case $\phi_j(\bM_j(i))=0$ there is nothing to show, so assume that $\phi_j(\bM_j(i))\neq 0$. This implies that $j\leq \gamma(i)$. It follows that
	\[\phi_j(\bM_j(i))\subseteq \bN_j(\Psi_j(i))\subseteq  \bN_j(\Psi(i)),\] and for each $m\in \bM_j(i)$, we have \[\supp_{\bN_j}(\phi_j(m))\subseteq N_{\Psi_j(i)}(f_j(\supp_{\bM_j}(m)))\subseteq N_{\Psi(i)}(f_j(\supp_{\bM_j}(m))).\] This shows each $\phi_j$ has displacement $\Psi$ as required.

	(\ref{item:unif_unif_preimages}): Assume each $\phi_j$ has $\Omega_j$-uniform preimages.  Set $\Omega(i,D)\coloneqq \max_{j\leq \gamma(i)}\Omega_j(i,D)$. We claim each $\phi_j$ has $\Omega$-uniform preimages.
	Suppose $n\in \im(\phi_j)\cap \bN_j(i)$ with $D=\lceil \diam(\supp_{\bN_j}(n))\rceil$.  If $n=0$ there is nothing to show, so we assume $n\neq 0$. Since $\bN_j(i)\neq 0$, we have  $j\leq \gamma(i)$,   and so $\Omega_j(i,D)\leq \Omega(i,D)\eqqcolon i'$. As $\phi_j$ has $\Omega_j$-uniform preimages, it follows there exists $m\in \bM_j(i')$ with $\supp_{\bM_j}(m)\subseteq N_{i'}(\supp_{\bN_j}(n))$ and $\phi_j(m)=n$. This proves that $\phi_j$ has $\Omega$-uniform preimages.
\end{proof}
\begin{rem}\label{rem:uniformdisplacement}
	It is clear from the proof of Proposition~\ref{prop:uniform_displacement} that  $\Psi$ and $\Omega$ depend only on $\cM$, $\cN$ and the functions $\{\Psi_i\}_i$ and $\{\Omega_i\}_i$ respectively, not on the actual maps $\{\phi_i\}$.
\end{rem}
We can take the direct sum  of $R$-modules over a metric space as follows. We use the following notation: if $f:X\to Z$ and $g:Y\to Z$ are functions, then $f\sqcup g:X\sqcup Y\to Z$ is the function agreeing with $f$ on $X\subseteq X\sqcup Y$ and with $g$ on $Y\subseteq X\sqcup Y$. The proof of Proposition~\ref{prop:dirsum_overX} is routine, so will be omitted.
\begin{prop}\label{prop:dirsum_overX}
	Let $\bM=(M,B,\delta,p,\{B(i)\})$ and $\bN=(N,C,\epsilon,q,\{C(i)\})$ be $R$-modules over $X$. We identify the codomain $\Hom_R(M,R)$ of $\delta$  with a submodule of $\Hom_R(M\oplus N,R)\cong \Hom_R(M,R)\oplus \Hom_R(N,R)$ in the obvious way, and similarly for $\epsilon$, so that $\delta\sqcup \epsilon$ is a map $B\sqcup C\to \Hom(M\oplus N,R)$.
	\begin{enumerate}
		\item The direct sum  \[\bM\oplus \bN\coloneqq(M\oplus N,B\sqcup C, \delta\sqcup \epsilon, p\sqcup q,\{B(i)\sqcup C(i)\}_i)\] is an  $R$-module over $X$.
		\item $(\bM\oplus\bN)(i)=\bM(i)\oplus \bN(i)$.
		\item $\supp_{\bM\oplus\bN}(m\oplus n)= \supp_\bM(m)\cup \supp_\bN(n)$ for all $m\in \bM$ and $n\in \bN$.
	\end{enumerate}
\end{prop}

We now introduce two  finiteness properties that an $R$-module $\bM$ over  $X$ may possess:
\begin{defn}\label{defn:finiteness_geom}
	Let  $\bM=(M,B,\delta,p,\{B(i)\})$ be an $R$-module over $X$. We say that:
	\begin{enumerate}
		\item $\bM$ is  \emph{proper} if,  for each $i\in \bbN$ and bounded set $F\subseteq X$, the intersection $p^{-1}(F)\cap B(i)$ is finite.
		\item $\bM$ has \emph{finite height} if  there exists an $i\in \bbN$ such that $\bM=\bM(i)$. The minimal such $i$ is the \emph{height} of $\bM$.

	\end{enumerate}
\end{defn}

The following proposition follows readily from the definitions:
\begin{prop}\label{prop:finiteness_operations_first}
	Let $\bM$ be an $R$-module over $X$.
	\begin{enumerate}
		\item If $\bN\leq \bM$ is a geometric submodule and $\bM$ is proper (resp.\ of finite height), then so is $\bN$.
		\item If $f:X\to Y$ is a coarse embedding and  $\bM$ is proper (resp.\ of finite height), then so is $f_*\bM$.
		\item If $f:Y\to X$ is a coarse embedding and  $\bM$ is proper, then so is $f^*\bM$. \item If $f:Y\to X$ is a coarse equivalence and  $\bM$ is of finite height, then so is $f^*\bM$. (This is not typically true if $f$ is only a coarse embedding.)
		\item If $\bN$ is an $R$-module over $X$ and both $\bM$ and $\bN$ are proper (resp.\ of finite height), then so is $\bM\oplus \bN$.
	\end{enumerate}
\end{prop}

\section{Projective modules over metric spaces}\label{sec:projective}
\emph{We define projective $R$-modules over a metric space, and we prove in Proposition~\ref{prop:projective} that they satisfy a lifting property. In Lemma~\ref{lem:ginduced_produce}, we show that projective $RG$-modules induce projective $R$-modules over $G$. In Propositions~\ref{prop:ginduced} and~\ref{prop:ginv_unifpreim}, we show $G$-equivariant maps between these induced modules have finite displacement and uniform preimages. Lemma~\ref{lem:comp} shows that the finite displacement retraction of a projective $R$-module over $X$ is also a projective $R$-module over $X$, which plays an important role in the theory of coarse cohomological dimension, and is the main reason why it is necessary  to work with projective $R$-modules over $X$ rather than free $R$-modules over $X$.}
\vspace{.3cm}

Recall that, given an $R$-module $M$ and an arbitrary indexing set $B$, a \emph{projective basis} of $M$ is a collection $\{(m_b,\delta_b)\}_{b\in B}$  satisfying the following. For each $b\in B$, $m_b\in M$, $\delta_b\in M^*=\Hom_R(M,R)$, and  for all $m\in M$, \[\sum_{b\in B}\delta_b(m)m_b=m,\] where   $\delta_b(m)=0$ for all but finitely many $b\in B$. We call each $m_b$ a \emph{basis element} and each $\delta_b$ a \emph{coordinate}. The existence of a projective basis of $M$ is equivalent to $M$ being a projective $R$-module~\cite[Proposition I.8.2]{brown1982cohomology}. We define  a \emph{projective  $R$-module $\bM$ over $X$} as follows:
\begin{defn}\label{defn:projective}
	Let $X$ be a metric space and let $\bM=(M,B,\delta,p,\{B(i)\})$ be an $R$-module over $X$.
	We say a collection of elements $\{m_b\}_{b\in B}$ of $M$ is a  \emph{projective basis} of $\bM$ if there exists an increasing function $\Psi:\bbN\to \bbN$ such that  the following hold:
	\begin{itemize}
		\item $\{(m_b,\delta_b)\}_{b\in B}$ is a projective basis of $M$;
		\item  $m_b\in \bM(\Psi(i))$ for each $b\in B(i)$;
		\item  $\supp_
			      \bM(m_b)\subseteq N_{\Psi(i)}(p(b))$ for each $b\in B(i)$.
	\end{itemize}
	When the above holds, we say that \emph{$\{m_b\}_{b\in B}$ has displacement $\Psi$}. We say that $\bM$ is a \emph{projective $R$-module over $X$} if it has a projective basis, and we say that $\bM$ has \emph{displacement $\Psi$} if it has a projective basis of displacement $\Psi$.
\end{defn}
We will shortly see in Proposition~\ref{prop:projective} that projective $R$-modules over $X$  satisfy a lifting property.
An important special type of projective $R$-module over $X$ is the following:
\begin{defn}
	An $R$-module $\bM$ over $X$ is a \emph{free $R$-module over $X$} if it admits a projective basis $\{m_b\}_{b\in B}$ that is $R$-linearly independent, in which case we say that $\{m_b\}_{b\in B}$ is a \emph{basis}.
\end{defn}

It is important to note that since the coordinates $\{\delta_b\}_b$ are built into the structure of $\bM$ as an $R$-module over $X$, the choice of a projective basis of $\bM$ is extremely limited. For instance, in the case $\bM$ is a free $R$-module over $X$, the basis $\{m_b\}$ is uniquely determined by the coordinates $\{\delta_b\}_b$, in which case $\{\delta_b\}_{b\in B}$ is the dual basis of $\{m_b\}_{b\in B}$.

When working with a free $R$-module $\bM=(M,B,\delta,p,\cF)$ over $X$, we may identify the indexing set $B$ and the basis $\{m_b\}_{b\in B}$. We can thus say that $B$ is a basis of $\bM$ and write $m=\sum_{b\in B}\delta_b(m)b$ for each $m\in \bM$. This identification is not typically possible when working with a general projective $R$-module over $X$.

In practice, we often start with an $R$-module and specify the data defining the $R$-module structure over $X$ and the projective basis at once. It will thus be convenient to write  \[\bM=(M,B,\delta,p,\{B(i)\},\{m_b\})\] to indicate that $\bM$ is the $R$-module over $X$ given by $(M,B,\delta,p,\{B(i)\})$, and that $\{m_b\}$ is a projective basis of $\bM$.

\begin{rem}\label{rem:rmod_conditions}
	Given the data \[\bM=(M,B,\delta,p,\{B(i)\},\{m_b\})\] as above such that $\{(m_b,\delta_b)\}$ is a projective basis of $M$, conditions (\ref{item:moddef_zero}) and (\ref{item:moddef_filtr})  in Definition~\ref{defn:rmod_overX} are automatically satisfied, and so $\bM$ is an $R$-module over $X$.
\end{rem}

The following proposition complements Proposition~\ref{prop:uniform_displacement}:
\begin{prop}\label{prop:uniform_displacement_projective}
	Suppose $\cM\coloneqq \{\bM_j\}_{j\in \bbN}$  is staggered, where each $\bM_j$ is a projective  $R$-module over  $X_j$.
	Then there exists a $\Phi$ such that each $\bM_j$ is of displacement $\Phi$.
\end{prop}
\begin{proof}
	Suppose $\bM_j=(M_j,B_j,\delta^j,p_j,\{B_j(i)\})$.
	We can pick an increasing function $\gamma:\bbN\to \bbN$ such that $B_j(i)=\emptyset$ for all $i,j\in \bbN$ satisfying $j>\gamma(i)$.
	Suppose that each $\bM_j$ has displacement $\Phi_j$.
	Set $\Phi(i)\coloneqq \max_{j\leq \gamma(i)}\Phi_j(i)$. We claim each $\bM_j$ has displacement $\Phi$. Indeed, if $b\in B_j(i)$, then $j\leq \gamma(i)$. Thus $m_b\in \bM_j(\Phi_j(i))\subseteq \bM_j(\Phi(i))$ and  $\supp_{\bM_j}(m_b)\subseteq N_{\Phi_j(i)}(p(b))\subseteq N_{\Phi(i)}(p(b))$ as required.
\end{proof}

The following lemma furnishes us with  an important criterion for constructing finite displacement maps between $R$-modules over a metric space.
\begin{lem}\label{lem:define_fin_disp}
	Let $X$ and $Y$ be metric spaces and suppose $\bM= (M,B,\delta,p,\{B(i)\})$ is a projective $R$-module over $X$ with projective basis $\{m_b\}_{b\in B}$. Suppose also  $\bN$ is an  $R$-module over $Y$ and  $f:X\to Y$ is an arbitrary map.
	Assume there is an increasing function $\Psi:\bbN\to\bbN$ and a collection $\{q_b\}_{b\in B}$ of elements of $\bN$ such that for all $i\in \bbN$ and  $b\in B(i)$: \begin{enumerate}
		\item $q_b\in \bN(\Psi(i))$; and
		\item $\supp_\bN(q_b)\subseteq N_{\Psi(i)}(f(p(b)))$.
	\end{enumerate}  Then $m\mapsto  \sum_{b\in B}\delta_b(m)q_b$ defines a  map $\phi:\bM\to \bN$ of displacement $\Psi$ over $f$.
\end{lem}

\begin{proof}
	Since each $\delta_b$ is a homomorphism and for each $m\in M$, $\delta_b(m)=0$ for almost all $b$, we deduce that $\phi$ is a well-defined $R$-module  homomorphism.  If $m\in \bM(i)$, then all non-zero summands of $\sum_{b\in B}\delta_b(m)q_b$ are contained in $\bN({\Psi(i)})$, hence $\phi(m)\in \bN({\Psi(i)})$. It remains to show that if $y\in\supp_\bN(\phi(m))$, then $y\in N_{\Psi(i)}(f(\supp_\bM(m)))$.  Indeed, Lemma~\ref{lem:support_sum} implies that $y\in \supp_\bN(q_b)$ for some $b\in B(i)$ with $\delta_b(m)\neq 0$, so that $p(b)\in \supp_\bM(m)$. Thus  $y\in N_{\Psi(i)}(f(p(b)))\subseteq  N_{\Psi(i)}(f(\supp_\bM(m)))$ as required.
\end{proof}

We now show  projective  $R$-modules over a metric space satisfy the following lifting property:
\begin{prop}\label{prop:projective}
	Let $X$ and $Y$ be  metric spaces and let $f:X\to Y$ be an $\Upsilon$-bornologous map.
	Let $\bP$ be a  projective $R$-module over $X$ of displacement $\Phi$, and let $\bM$ and $\bN$ be $R$-modules over $Y$. Suppose $\psi:\bM\to \bN$  has  $\Omega$-uniform preimages.
	If $\phi:\bP\to \bN$ has  displacement $\Phi$ over $f$ and $\im(\phi)\subseteq \im(\psi )$, then there exists a map $\xi:\bP\to \bM$ of  displacement $\Psi$ over $f$ such that $\psi\circ \xi=\phi$. Moreover, $\Psi$ depends only on $\Upsilon$, $\Phi$ and $\Omega$.

	\begin{figure}[htp]
		\[
			\begin{tikzcd}
				& \bP \arrow[d, "\phi"] \arrow[ld, "\xi"', dashed] \\
				\bM \arrow[r, "\psi"'] & \bN
			\end{tikzcd}
		\]
		\caption{The lifting property for  projective $R$-modules over a metric space.}
	\end{figure}
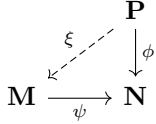
\end{prop}

\begin{proof}
	We suppose $\bP=(P,B,\delta,p,\{B(i)\})$ and that $\{p_b\}_{b\in B}$ is a projective basis of $\bP$ with displacement $\Phi$.
	Let $i\in \bbN$, and set $i_1=\Phi(i)$ and $i_2\coloneqq \Phi(i_1)+\Upsilon(i_1)$. Fixing $b\in B(i)$, we see $p_b\in \bP(i_1)$ and $\supp_\bP(p_b)\subseteq N_{i_1}(p(b))$.  Therefore $\phi(p_b)\in \bN(i_2)$ and
	\[\supp_\bN(\phi(p_b))\subseteq N_{\Phi(i_1)}(f(\supp_\bP(p_b)))\subseteq N_{i_2}(f(p(b))).\]   In particular, $\supp_\bN(\phi(p_b))$ has diameter at most $2i_2$. Setting $i_3\coloneqq \Omega(i_2,2i_2)+i_2$, we see there exists some $m_b\in \bM(i_3)$ with $\supp_\bM(m_b)\subseteq N_{i_3}(f(p(b)))$ and $\psi(m_b)=\phi(p_b)$.

	We can thus apply Lemma~\ref{lem:define_fin_disp} to define a map $\xi:\bP\to \bM$ given by  $p\mapsto  \sum_{b\in B}\delta_b(p)m_b$. By the choice of $\{m_b\}$ and Lemma~\ref{lem:define_fin_disp},  $\xi$ is  of  displacement $\Psi$  over $f$, where  $\Psi$ depends only on $\Upsilon$, $\Phi$ and $\Omega$. For each $p\in \bP$, we have
	\[\psi(\xi(p))=\psi\left(\sum_{b\in B} \delta_b(p) m_b\right)=\sum_{b\in B} \delta_b(p)\psi (m_b)=\phi\left(\sum_{b\in B} \delta_b(p)p_b\right)=\phi (p)\]
	as required.
\end{proof}

We now show that projective $RG$-modules induce projective $R$-modules over $G$.
As always, we assume that whenever $G$ is a countable group, it is equipped with a proper left-invariant metric.
\begin{defn}
	Let $G$ be a countable group.
	A \emph{$G$-induced projective $R$-module over $G$} is a projective $R$-module $\bM=(M,B,\delta,p,\{B(i)\},\{m_b\})$ over $G$ satisfying the following:
	\begin{itemize}
		\item $M$ is an $RG$-module;
		\item $G$ acts on $B$ so that $m_{gb}=gm_b$ and $\delta_{gb}(m)=\delta_b({g^{-1}m})$ for all $b\in B$ and $g\in G$;
		\item $p:B\to G$ is $G$-equivariant;
		\item For each $i$, $B(i)=GB(i)$ and $B(i)$ consists of finitely many $G$-orbits.
	\end{itemize}
\end{defn}
It is clear that if $\bM$ is a $G$-induced projective $R$-module over $G$, then $\bM(i)$ is an  $RG$-submodule of $\bM$, and $\supp_\bM(gm)=g\supp_\bM(m)$ for all $b\in B$ and $m\in \bM$. Moreover, the underlying $RG$-module $M$ is necessarily countably generated.

\begin{lem}\label{lem:ginduced_produce}
	Every countably generated projective $RG$-module $M$ can be endowed with the structure of a $G$-induced projective $R$-module $\bM$ over $G$.
\end{lem}
\begin{proof}
	We pick a countable generating set $C$ of $M$ as an $RG$-module, and let $\{C(i)\}$ be a filtration of $C$ with each $C(i)$ finite. Let $F_C$ be the free $RG$-module with $RG$-basis $\{e_c\mid c\in C\}$, and let  $\pi:F_C\to M$ be the surjection  given by      $\pi(e_c)=c$. Since $M$ is a projective $RG$-module, there exists a $G$-equivariant section $s:M\to F_C$ of $\pi$.

	Let  $B=\{ge_c\mid g\in G, c\in C\}\subseteq F_C$, which is a basis of $F_C$ as a free $R$-module, and let $\{\hat\delta_b\}_b$ be the dual basis to $B$, i.e.\ \[\delta_b(b')=\begin{cases}
			1 & \textrm{if}\;b=b', \\0& \textrm{if}\; b\neq b'.
		\end{cases}\] For each $b\in B$, define $\delta_b:M\to R$ by $\hat\delta_b\circ s$.
	For each $m\in M$, we have \[s(m)=\sum_{b\in B}\hat\delta_b(s(m))b=\sum_{b\in B}\delta_b(m)b.\]  Setting $m_b\coloneqq \pi(b)$ for each $b\in B$, we see that $\cB=\{(m_b,\delta_b)\}_{b\in B}$ is a projective basis of $M$ as an $R$-module. Since $s$ and $\pi$ are $G$-equivariant, we  note that  $m_{gb}=gm_b$ and $\delta_b(g^{-1}m)=\delta_{gb}(m)$.

	Set $B(i)\coloneqq \{ge_c\mid g\in G, c\in C(i)\}$, noting that $\cB=\{B(i)\}$ is a filtration of $B$. We define $p:B\to G$ by $p(ge_c)=g$ for each $g\in G$ and $c\in C$.
	We define $\bM$ to be $(M,B,\delta,p,\{B(i)\})$. All that remains to be shown is that $\{m_b\}$  has displacement $\Phi$ for some $\Phi$.
	For each $i\in \bbN$, as $C(i)$ is finite, we can pick $n_i$ such that $m_{e_c}\in \bM(n_i)$ and $\supp_\bM(m_{e_c})\subseteq N_{n_i}(1)$ for all $c\in C(i)$. By $G$-equivariance, we see that $m_b\in \bM(n_i)$ and $\supp(m_b)\subseteq N_{n_i}(p(b))$ for all $b\in B(i)$.  Defining $\Phi(i)\coloneqq \max\{n_j\mid j\leq i \}$, we see that $\{m_b\}$ is a projective basis of displacement $\Phi$.
\end{proof}
\begin{rem}\label{rem:staggering_ginv}
	Let $G$ be a countable group and let $\{M_j\}_{j\in \bbN}$ be a collection of countably generated $RG$-modules. The proof of Lemma~\ref{lem:ginduced_produce} ensures that we can define a staggered family  $\{\bM_j\}$ of $G$-induced projective modules over $G$ in which each $\bM_j$ has underlying $RG$-module $M_j$.
\end{rem}

We now discuss some  properties of $G$-induced modules:
\begin{prop}\label{prop:ginduced}
	Let $G$ be a countable group.
	\begin{enumerate}
		\item Suppose $M$ is a countably generated projective $RG$-module and that $\bM$ and $\bM'$ are two $G$-induced projective $R$-modules over $G$ with the same underlying $RG$-module $M$. Then $\bM$ and $\bM'$ are canonically isomorphic.\label{item:ginduced1}
		\item Suppose $\bM$ and $\bN$ are  $G$-induced projective $R$-modules over $G$. Then any $G$-equivariant map $\phi:\bM\to\bN$  has finite displacement over $\id_G$.\label{item:ginduced2}
	\end{enumerate}
\end{prop}
\begin{proof}
	(\ref{item:ginduced1}): This is a special case of (\ref{item:ginduced2}), since the identity map  $\id_M$ is $G$-equivariant with $G$-equivariant inverse.

	(\ref{item:ginduced2}): Suppose $\bM=(M,B,\delta,p,\{B(i)\},\{m_b\})$, and that $B(i)=GC(i)$ for some finite $C(i)\subseteq B(i)$.  Since $C(i)$ is finite, we may choose an increasing function $\Psi:\bbN\to\bbN$ such that for all $i\in \bbN$, the following hold:
	\begin{itemize}
		\item $\phi(m_c)\in \bN(\Psi(i))$ for all $c\in C(i)$;
		\item $\supp_{\bN}(\phi(m_c))\subseteq N_{\Psi(i)}(p(c))$ for all $c\in C(i)$.
	\end{itemize}

	Thus for each $b\in B(i)$, we have $b=gc$ for some $c\in C(i)$, so that $\phi(m_b)=g\phi(m_c)\in \bN(\Psi(i))$ and $\supp(\phi(m_b))\subseteq gN_{\Psi(i)}(p(c))=N_{\Psi(i)}(p(b))$. It follows from  Lemma~\ref{lem:define_fin_disp} that $\phi$ has displacement $\Psi$ over $\id_G$.
\end{proof}

\begin{prop}\label{prop:ginv_loc_fin_type}
	Let $G$ be a countable group and let $\bM$ be a $G$-induced projective $R$-module over $G$. Then $\bM$ is proper, and is of finite height if and only if it is finitely generated as an $RG$-module.
\end{prop}
\begin{proof}
	Suppose $\bM=(M,B,\delta, p, \{B(i)\})$. The fact that $G$ is proper follows from the fact that $B(i)$ consists of finitely many $G$-orbits, $p$ is $G$-equivariant and  $G$ is equipped with a proper metric.
	If $\bM$ is of finite 	height, then $\bM=\bM(i)$ for some $i$, so that $\bM$ is generated by $B(i)$ as an $R$-module. Since $B(i)$ consists of finitely many $G$-orbits, $\bM$ is finitely  generated as an $RG$-module.
	For the converse, we note that $\bM=\cup_i\bM(i)$ is a filtration of $\bM$ by $RG$-modules. Thus if $\bM$ is finitely generated as an $RG$-module, then $\bM=\bM(i)$ for some $i$ sufficiently large.
\end{proof}

The uniform preimage property holds automatically for $G$-equivariant maps under a mild assumption on the ring $R$:
\begin{prop}\label{prop:ginv_unifpreim}
	If $R$ is a Noetherian ring, $G$ is a countable group and $\phi:\bM\to \bN$ is a $G$-equivariant map between $G$-induced projective $R$-modules over $G$, then $\phi$ has uniform preimages.
\end{prop}
\begin{proof}
	Since $\bN$ is $G$-induced, it is proper by Proposition~\ref{prop:ginv_loc_fin_type}. Thus for each $i,D\in \bbN$, the set $\Omega_{i,D}\coloneqq \{n\in\bN(i)\mid \supp_\bN(n)\subseteq N_D(e) \}$ is a  finitely generated $R$-module. Since $R$ is Noetherian,   the submodule $\Omega_{i,D}\cap \im(\phi)$ is also finitely generated.  Hence, there is a number $\Psi(i,D)$ such that for each $n\in \Omega_{i,D}\cap \im(\phi)$, there is some $m\in \bM(\Psi(i,D))$ with $\phi(m)=n$ and $\supp_\bM(m)\subseteq N_{\Psi(i,D)}(e)$.

	Now suppose $n\in \bN(i)\cap \im(\phi)$ with $\diam(\supp(n))\leq D$. If $g\in \supp_\bN(n)$, we see that  $g^{-1}n\in \Omega_{i,D}\cap \im(\phi)$. Then there is some $m\in \bM(\Psi(i,D))$ with $\phi(m)=g^{-1}n$ and $\supp_\bM(m)\subseteq N_{\Psi(i,D)}(e)$. Hence we deduce that $\phi(gm)=n$, that $m\in  \bM(\Psi(i,D))$, and that $\supp_\bM(gm)\subseteq N_{\Psi(i,D)}(g)\subseteq N_{\Psi(i,D)}(\supp_\bN(n))$. This demonstrates that $\phi$ has uniform preimages.
\end{proof}

If $\bM$ is a projective $R$-module over $X$ of height $i_0<\infty$ and $\{m_b\}$ is an arbitrary projective basis of $\bM$, then it is possible that $m_b\neq 0$ for some $b\in B\setminus B(i_0)$. However, we can always pick a projective basis such that this does not occur:
\begin{lem}\label{lem:proj_basis_finheight}
	Let $\bM=(M,B, \delta, p,\{B(i)\})$ be a projective $R$-module over $X$ of displacement $\Phi$ and height $i_0<\infty$. Then there exists a projective basis $\{m_b\}$ of $\bM$ of displacement $\Phi$ such that $m_b=0$ for all $b\in B\setminus B(i_0)$.
\end{lem}
\begin{proof}
	Since $\bM=\bM(i_0)$, we see $\delta_b=0$ for all $b\in B\setminus B(i_0)$. By hypothesis, there exists a projective basis $\{m'_b\}$ of $\bM$ of displacement $\Phi$. Setting \[m_b\coloneqq \begin{cases}
			m'_b & \textrm{if}\;b\in B(i_0), \\ 0 & \textrm{if}\;b\in B\setminus B(i_0),
		\end{cases}\] we see that $\{m_b\}$ is the required projective basis of $\bM$.
\end{proof}

We now show pullbacks, pushforwards and  direct sums preserve projective $R$-modules over a metric space:

\begin{prop}\label{prop:projective_operations_second}
	Let $\bM$ be a projective $R$-module over $X$.
	\begin{enumerate}
		\item If $f:X\to Y$ is a coarse embedding, then $f_*\bM$ is a projective $R$-module over $Y$.\label{item:projective_operations1}
		\item If $f:Y\to X$ is a coarse embedding, then $f^*\bM$ is a projective $R$-module over $Y$.\label{item:projective_operations2}
		\item If $\bN$ is a projective $R$-module over $X$, then  so is $\bM\oplus \bN$.\label{item:projective_operations3}
	\end{enumerate}
\end{prop}
\begin{proof}
	We suppose  $\bM=(M,B,\delta,p,\{B(i)\})$ has a projective basis $\{m_b\}_{b\in B}$ of displacement $\Phi$, and that $f$ is a $(\Lambda, \Upsilon)$-coarse embedding.

	(\ref{item:projective_operations1}): Let $\iota:\bM \to f_*\bM$ be the canonical identification. For each $b\in B(i)$, we have 	$\iota(m_b)\in f_*\bM(\Phi(i))$ and $\supp_{f_*\bM}(\iota(m_b))=f(\supp_\bM(m_b))\subseteq f(N_{\Phi(i)}(p(b)))\subseteq N_{\Upsilon(\Phi(i))}(f(p(b)))$, showing that $\{\iota(m_b)\}$ is a projective basis of $f_*\bM$.

	(\ref{item:projective_operations2}): Let $\iota:f^*\bM \to \bM$ be the canonical identification, and let $n_b\coloneqq \iota^{-1}(m_b)$ for each $b\in B$.  For each $b\in f^*B(i)$, we have $d(p(b),f(Y))\leq i$ and so $m_b\in \bM(\Phi(i))$ and $\supp_\bM(m_b)\subseteq N_{\Phi(i)}(p(b))\subseteq N_{\Phi(i)+i}(f(Y))$. By Lemma~\ref{lem:induced_mod}, this implies $n_b\in f^*\bM(\Phi(i)+i)$ and
	\begin{align*}
		f(\supp_{f^*\bM}(n_b)) & \subseteq N_{\Phi(i)+i}(\supp_\bM(m_b)) \subseteq N_{2\Phi(i)+i}(p(b))\subseteq N_{2\Phi(i)+2i}((f\circ\pi\circ p)(b)),
	\end{align*} where $\pi$ is as in Definition~\ref{defn:pullback}. Thus $\supp_{f^*\bM}(n_b)\subseteq N_{\widetilde \Lambda(2\Phi(i)+2i)}(\pi(p(b)))$, showing that $\{n_b\}$ is a projective basis of $f^*\bM$.

	(\ref{item:projective_operations3}) is straightforward.
\end{proof}
In the following lemma, we assume the product $X\times Y$ of metric spaces is endowed with the $\ell_1$-product metric.
\begin{prop}\label{prop:tensor_overX}
	Let $\bM=(M,B,\delta,p,\{B(i)\})$ and $\bN=(N,C,\epsilon,q,\{C(i)\})$ be projective $R$-modules over $X$ and $Y$ respectively of displacement $\Phi$.
	\begin{enumerate}
		\item The tensor product  \[\bM\otimes \bN\coloneqq(M\otimes N,B\times C, \delta\otimes \epsilon, p\times q,\{B(i)\times C(i)\}_i)\] is a projective $R$-module over $X\times Y$ of displacement $\Phi$, where for each $b\in B$ and $c\in C$, $(\delta\otimes\epsilon)_{(b,c)}$ is the map $M\otimes N\to R$ given by $m\otimes n\mapsto \delta_b(m)\epsilon_c(n)$ for all $m\in M$ and $n\in N$.
		\item  $\supp_{\bM\otimes \bN}(m\otimes n)=\supp_\bM(m)\times \supp_\bN(n)$ for all $m\in \bM$ and $n\in \bN$.
		\item $\bM(i)\otimes \bN(i)\subseteq (\bM\otimes \bN)(i)\subseteq \bM(\Phi(i))\otimes \bN(\Phi(i))$.
		\item If  both $\bM$ and $\bN$ are proper (resp.\ of finite height), then so is $\bM\otimes \bN$.
		\item Suppose $\bP$ and $\bL$ are projective $R$-modules  over $W$ and $Z$ respectively. If $\phi:\bM\to \bP$ and $\psi:\bN\to \bL$ have finite displacement over bornologous maps $f:X\to W$ and $g:Y\to Z$, then $\phi\otimes \psi:\bM\otimes\bN\to \bP\otimes \bL$ has finite displacement over $f\times g:X\times Y\to W\times Z$.
	\end{enumerate}
\end{prop}
\begin{proof}
	(1): Let $\{m_b\}$ and $\{n_c\}$ be projective bases of  $\bM$ and $\bN$ of displacement $\Phi$. If $z\in M\otimes N$, then  $z=\sum_{i=1}^n m_i\otimes n_i$ with $m_i\in M$ and $n_i\in N$ for all $i$. Thus, \begin{align*}
		z & =\sum_{i=1}^n m_i\otimes n_i  =
		\sum_{i=1}^n \left(\sum_{b\in B}\delta_b(m_i)m_b\right)\otimes  \left(\sum_{c\in C}\epsilon_c(n_i)n_c\right) \\
		  & =\sum_{i=1}^n\sum_{(b,c)\in B}\delta_b(m_i)\epsilon_c(n_i)m_b\otimes n_c
		=\sum_{i=1}^n\sum_{(b,c)\in B}(\delta\otimes \epsilon)_{(b,c)}(m_i\otimes n_i)m_b\otimes n_c                 \\
		  & =\sum_{(b,c)\in B}(\delta\otimes \epsilon)_{(b,c)}(z)m_b\otimes n_c,
	\end{align*}
	which verifies $M\otimes N$ has  projective basis $\{(m_b\otimes n_c),(\delta_b\otimes \epsilon_c)\}_{(b,c)\in B\times C}$. Remark~\ref{rem:rmod_conditions} now implies $\bM\otimes \bN$ is an $R$-module over $X\times Y$.
	It remains to show $\{m_b\otimes n_c\}$ has finite  displacement, which will follow from (2) and (3).

	(2): We note that $(x,y)\in \supp_{\bM\otimes \bN}(m\otimes n)$ if and only if there exist $b\in p^{-1}(x)$ and $c\in q^{-1}(y)$ such that $\delta_b(m)\epsilon_c(n)\neq 0$. The latter is true if and only if  $(x,y)\in \supp_\bM(m)\times \supp_\bN(n)$.

	(3): If $z\in \bM(i)\otimes \bN(i)$, then it is clear that $(\delta_b\otimes \epsilon_c)(z)=0$ whenever $(b,c)\notin B(i)\times C(i)$, hence $z\in (\bM\otimes \bN)(i)$.
	If  $z\in (\bM\otimes \bN)(i)$, then  $(\delta_b\otimes \epsilon_c)(z)=0$ whenever $(b,c)\notin B(i)\times C(i)$. Hence $z=\sum_{(b,c)\in B(i)\times C(i)}(\delta\otimes \epsilon)_{(b,c)}(z)m_b\otimes n_c$. For each  $b\in B(i)$ and $c\in C(i)$, we have $m_b\otimes n_c\in \bM(\Phi(i))\otimes  \bN(\Phi(i))$. Thus $z\in \bM(\Phi(i))\otimes  \bN(\Phi(i))$ as required.

	(4): This follows easily from the definitions.

	(5): By increasing $\Phi$ if necessary, we may assume that $\bP$ and $\bL$ have displacement $\Phi$, and that $\phi$ and $\psi$ have displacement $\Phi$ over $f$ and $g$ respectively. We pick $\Upsilon$ such that $f$ and $g$ are $\Upsilon$-bornologous. Fix $i$ and set $i_k=\Phi^k(i)$ for $k\in \bbN$.
	Thus \begin{align*}
		(\phi\otimes \psi)((\bM\otimes \bN)(i)) & \subseteq (\phi\otimes \psi)(\bM(i_1)\otimes \bN(i_1))\subseteq \phi(\bM(i_1))\otimes \psi(\bN(i_1)) \\
		                                        & \subseteq \bP(i_2)\otimes \bL(i_2)\subseteq (\bP\otimes \bL)(i_2).
	\end{align*}
	Moreover, suppose that $v\in (\bM\otimes \bN)(i)$ and that $(w,z)\in \supp_{\bP\otimes \bL}((\phi\otimes \psi)(v))$.
	Since $v=\sum_{(b,c)\in B(i)\times C(i)}(\delta_b\otimes \epsilon_c)(v)m_b\otimes n_c$, Lemma~\ref{lem:support_sum} implies  there exists some $(b,c)\in B(i)\times C(i)$ with $(\delta_b\otimes \epsilon_c)(v)\neq 0$ and $(w,z) \in \supp_{\bP\otimes \bL}((\phi\otimes \psi)(m_b\otimes n_c))$. Therefore \begin{align*}(w,z) & \in\supp_{\bP\otimes \bL}(\phi(m_b)\otimes \psi(n_c))                         \\
                    & =\supp_\bP(\phi(m_b))\times \supp_\bL(\phi(n_c))                              \\
                    & \subseteq N_{i_2}(f(\supp_\bM(m_b)))\times N_{i_2}(g(\supp_\bN(n_c)))         \\
                    & \subseteq N_{i_2}(f(N_{i_1}(p(b))))\times N_{i_2}(f(N_{i_1}(q(c))))           \\
                    & \subseteq N_{i_2+\Upsilon(i_1)}(f(p(b)))\times N_{i_2+\Upsilon(i_1)}(g(q(c))) \\
                    & \subseteq N_{j}(f(p(b)),g(q(c))),\end{align*}
	where $j\coloneqq 2i_2+2\Upsilon(i_1)$. As $(\delta_b\otimes \epsilon_c)(v)\neq0$, we see that $(p(b),q(c))\in \supp_{\bM\otimes \bN}(v)$, and so $\supp_{\bP\otimes \bL}((\phi\otimes \psi)(v))\subseteq N_{j}((f\times g)(\supp_{\bM\otimes \bN}(v)))$. Thus $\phi\otimes \psi$ has finite displacement over $f\times g$.
\end{proof}

Given a  homomorphism $f:R\to S$ between commutative rings and an $R$-module $M$, we have $M\otimes _R S$ is an $S$-module, where $S$ is thought of as a left $R$-module via $r\cdot s=f(r)s$. We can thus convert a projective $R$-module over $X$ to a projective $S$-module over $X$ as follows:
\begin{prop}\label{prop:change_rings}
	Let $f:R\to S$ be a ring homomorphism and let $\bM=(M,B,\delta,p,\{B(i)\})$ be a projective $R$-module over $X$ of displacement $\Phi$.
	\begin{enumerate}
		\item Then \[\bM\otimes _R S =(M\otimes _R S, B,\epsilon,p,\{B(i)\})\] is a projective $S$-module over $X$, where $\epsilon_b$ is given by $\epsilon_b(m\otimes s)=f(\delta_b(m))s$.\label{item:change_rings1}
		\item $\bM(i)\otimes_R S \subseteq (\bM\otimes_R S)(i)\subseteq  \bM(\Phi(i))\otimes_R S$.\label{item:change_rings2}
		\item $\supp_{\bM\otimes_R S}(m\otimes s)\subseteq \supp_\bM(m)$ for all $m\in \bM$ and $s\in S$.\label{item:change_rings3}
		\item If $\bN$ is an $R$-module over $Y$ and $\phi:\bM\to \bN$ has finite displacement over $f$, then $\phi\otimes \id_S:\bM\otimes_R S\to \bN\otimes_R S$ has finite displacement over $f$.\label{item:change_rings4}
	\end{enumerate}
\end{prop}
\begin{proof}
	Suppose that $\{m_b\}$ is a projective basis of $\bM$  and that $\{m_b\}$ has displacement $\Phi$.

	(\ref{item:change_rings1}): For each $z\in \bM\otimes _R S$, we have $z=\sum_{i=1}^n m_i\otimes s_i$, where each  $m_i\in \bM$ and $s_i\in S$. Therefore \begin{align*}
		z & =\sum_{i=1}^n  \left(\sum_{b\in B}\delta_b(m_i)m_b\right)\otimes s_i
		=\sum_{b\in B}\sum_{i=1}^n  f(\delta_b(m_i))s_i (m_b\otimes 1 )          \\
		  & =\sum_{b\in B}\sum_{i=1}^n\epsilon_b(m_i\otimes s_i) (m_b\otimes 1 )
		=\sum_{b\in B}\epsilon_b(z)(m_b\otimes 1 ),
	\end{align*}
	showing that $\{(m_b\otimes 1, \epsilon_b)\}$ is a projective basis of $M\otimes _R S$. Remark~\ref{rem:rmod_conditions} now implies $ \bM\otimes_R S $ is an $S$-module over $X$.
	It remains to show $\{m_b\otimes 1\}$ has finite  displacement,  which will follow from (\ref{item:change_rings2}) and (\ref{item:change_rings3}).

	(\ref{item:change_rings2}): If $m\in \bM(i)$ and $s\in S$, then $\epsilon_b(m\otimes s)=f(\delta_b(m))s=0$ for all $b\in B\setminus B(i)$. This implies $\epsilon_b(\bM(i)\otimes _R S)=0$ for all $b\in B\setminus B(i)$, and so $\bM(i)\otimes_R S\subseteq (\bM\otimes S)(i)$ as required.
	Now suppose $z\in (\bM\otimes S)(i)$. Then $z=\sum_{b\in B(i)}\epsilon_b(z)(m_b\otimes 1)$. Since for each $b\in B(i)$ we know that $m_b\in \bM(\Phi(i))$, we deduce $z\in \bM(\Phi(i))\otimes_R S$ as required.

	(\ref{item:change_rings3}): Suppose $x\in \supp_{\bM\otimes_R S}(m\otimes s)$. Then there is some $b\in p^{-1}(x)$ with $\epsilon_b(m\otimes s)=f(\delta_b(m))s\neq 0$. Thus $\delta_b(m)\neq 0$, and so $x=p(b)\in \supp_{\bM}(m)$.

	(\ref{item:change_rings4}): By increasing $\Phi$, we can assume $\bN$ has displacement $\Phi$ and $\phi$ has displacement $\Phi$ over $f$. We pick $\Upsilon$ so that $f$ is $\Upsilon$-bornologous. Fix $i\in \bbN$ and let $i_k\coloneqq \Phi^k(i)$ for each $k\in \bbN$. Then \begin{align*}
		(\phi\otimes \id_S)((\bM\otimes_R S)(i)) & \subseteq (\phi\otimes \id_S)(\bM(i_1)\otimes_R S)
		\subseteq \bN(i_2)\otimes_R S\subseteq (\bN\otimes_R S)(i_2).
	\end{align*}
	Suppose that $z\in (\bM\otimes_R S)(i)$ and that $y\in \supp_{\bN\otimes_R S}((\phi\otimes\id)(z))$. Since $z=\sum_{b\in B(i)}\epsilon_b(z)(m_b\otimes 1)$, there is some $b\in B(i)$ with $\epsilon_b(z)\neq 0$ and  $y \in \supp_{\bN\otimes_R S}((\phi\otimes\id)(m_b\otimes 1))$. Thus
	\begin{align*} y  \in \supp_{\bN\otimes_R S}(\phi(m_b)\otimes 1) & \subseteq\supp_\bN(\phi(m_b))
               \subseteq N_{i_2}(f(\supp_\bM(m_b)))                                              \\&\subseteq N_{i_2}(f(N_{i_1}(p(b))))\subseteq
               N_{j}(f(p(b))),
	\end{align*}
	where $j\coloneqq i_2+\Upsilon(i_1)$.
	Since $\epsilon_b(z)\neq 0$, we see $p(b)\in \supp_{\bM\otimes_R S}(z)$ and so  $\supp_{\bN\otimes_R S}((\phi\otimes\id)(z))\subseteq N_{j}(f(\supp_{\bM\otimes_R S}(z)))$. This demonstrates that $\phi\otimes \id_S$ has finite displacement over $f$.
\end{proof}

Given an $R$-module $M$ and a submodule $N$, a homomorphism $r:M\to M$ is said to be a \emph{retraction  onto $N$} if $\im(r)=N$ and $r|_{N}=\id_N$.
\begin{lem}\label{lem:comp}
	Let $\bM$ be a projective $R$-module over $X$ and let $\bN\leq \bM$ be a geometric submodule. The following are equivalent:
	\begin{enumerate}
		\item $\bN$ is a projective $R$-module over $X$.\label{item:comp1}
		\item There exists a finite displacement retraction $\phi:\bM\to \bM$ onto $\bN$.\label{item:comp2}
		\item There exists a geometric submodule $\bL\leq \bM$ such that $\bM$ splits as a direct sum $\bM=\bN\oplus\bL$ with finite displacement projections.\label{item:comp3}
	\end{enumerate}
\end{lem}
\begin{rem}
	Condition (\ref{item:comp1})  is significantly more restrictive than merely requiring that the underlying $R$-module $N$ is projective. The key point is that as $\bN\leq \bM$, the coordinates of $\bN$ are obtained by restricting the coordinates of $\bM$. The  condition that $\bN$ is a projective $R$-module thus requires that there is a projective basis of $\bN$  whose coordinates are compatible with the coordinates of $\bM$.
\end{rem}
\begin{proof}
	Suppose $\bM=(M,B,\delta,p,\cF)$ and $\{m_b\}_{b\in B}$ is a projective basis of $\bM$.

	(\ref{item:comp1}) $\implies$ (\ref{item:comp2}): Assume $\{n_b\}_{b\in B}$ is a projective basis of $\bN$. We note that by Definition~\ref{defn:submodule},  the indexing of $\bN$ is identical to the indexing set of $\bM$. We define a  map $\phi:\bM\to \bM$ by $\phi(m)=\sum_{b\in B}\delta_b(m)n_b$. This map has finite displacement because $\{n_b\}$ is a projective basis of $\bN$, and so  Lemma~\ref{lem:define_fin_disp} applies. Since $\im(\phi)=\bN$ and for each $n\in \bN$, we have \[\phi(n)=\phi\left(\sum_{b\in B}\delta_b(n)m_b\right)=\sum_{b\in B}\delta_b(n)n_b=n,\] we conclude that $\phi$ is indeed a finite displacement retraction onto $\bN$.

	(\ref{item:comp2}) $\implies$ (\ref{item:comp3}): Setting $\bL=\ker(\phi)$, we see that $\bM=\bN\oplus \bL$ and the projections to $\bN$ and to $\bL$ are given by $\phi$ and $\id_\bM-\phi$, which clearly have finite displacement.

	(\ref{item:comp3}) $\implies$ (\ref{item:comp1}): Let $\phi:\bM\to \bN$ be the finite displacement projection to $\bN$. For each $b\in B$, set $n_b=\phi(m_b)$. Then for each $n\in \bN$, we have \[\sum_{b\in B}\delta_b(n)n_b=\sum_{b\in B}\delta_b(n)\phi(m_b)=\phi\left(\sum_{b\in B}\delta_b(n)m_b\right)=\phi(n)=n,\] so that $\{(n_b,\delta_b)\}_{b\in B}$ is a projective basis of the $R$-module $N$. Pick $\Phi$ such that both $\{m_b\}$ and $\phi$ have displacement $\Phi$.
	Thus for each $b\in B(i)$ we have \[n_b=\phi(m_b)\in \phi(\bM(\Phi(i)))\subseteq \bN(\Phi^2(i))\] and \[\supp_\bN(n_b)=\supp_\bM(n_b)=\supp_\bM(\phi(m_b))\subseteq N_{\Phi^2(i)}(\supp_\bM(m_b))\subseteq N_{\Phi^2(i)+\Phi(i)}(p(b)),\] showing that $\{n_b\}_{b\in B}$ is a projective basis of $\bN$.
\end{proof}

\section{Hom-sets and duality}\label{sec:duality}
\emph{If $Y$ is a locally finite simplicial complex and $C_k(Y)$ is the module of simplicial $k$-chains of $Y$, there are three modules that can be  associated to $C_k(Y)$:\begin{enumerate}
		\item the module $C^k(Y)$ of all $k$-cochains;
		\item the submodule $C^k_c(Y)\leq C^k(Y)$ of compactly supported $k$-cochains;
		\item the module $C^{\lf}_k(Y)$ of infinite, locally finite chains, which contains $C_k(Y)$ as a submodule.
	\end{enumerate} We refer the reader to~\cite[Part III]{geoghegan2008topological} for a helpful discussion of these topics.\\
	\indent In this section, we construct three analogous modules for each proper projective $R$-module $\bM$ over a metric space $X$, denoted $\bM^*$, $\bM^*_c$ and $\bM^{\lf}$ respectively. We also study the hom-set $\Homfd(\bM,\bN)$ of finite displacement maps between $R$-modules over metric spaces, which will be used to define coarse cohomology of metric spaces. In general, neither  $\bM^*$ nor $\bM^*_c$ is an $R$-module over $X$. However, in the case $\bM$ is a finite-height proper projective $R$-module over $X$, we show in Proposition~\ref{prop:duality_ht} that $\bM^*_c$ is also a finite-height proper projective $R$-module over $X$. This is an analogue of the fact that the dual of a finitely generated projective $RG$-module is also a finitely generated projective $RG$-module, and will play a key role in our study of coarse $PD_n^R$ spaces in Section~\ref{sec:coarsePDn}.
}
\vspace{.3cm}

\begin{defn}\label{defn:homfd}
	Let $\bM$ and $\bN$ be $R$-modules over a metric space $X$, with underlying $R$-modules $M$ and $N$.  Let $\Homfd(\bM,\bN)$ denote the submodule of $\Hom_R(M,N)$ consisting of homomorphisms with finite displacement over $\id_X$.
\end{defn}
\begin{rem}\label{rem:hom_close}
	If $\bM$ and $\bM'$ are canonically isomorphic in the sense of Definition~\ref{defn:iso}, then we see that $\Homfd(\bM,\bN)=\Homfd(\bM',\bN)$ since they are equal as subsets of $\Hom_R(M,N)$. Similarly, if $\bN$ and $\bN'$ are canonically isomorphic, then $\Homfd(\bM,\bN)=\Homfd(\bM,\bN')$.
	In particular, suppose $f,g:X\to Y$ are close, $\bM$ is an $R$-module over $X$, and $\bL$ is an $R$-module over $Y$. Then Lemma~\ref{lem:induced_close} implies that  $\Homfd(\bL,f_*\bM)=\Homfd(\bL,g_*\bM)$ and $\Homfd(f_*\bM,\bL)=\Homfd(g_*\bM,\bL)$
\end{rem}

The following is a reformulation of  Lemma~\ref{lem:induced_map_compos}:
\begin{prop}\label{prop:induced_hom}
	Let $X$ and $Y$ be metric spaces and let $f:X\to Y$ be a coarse embedding. Suppose  $\bM$ and $\bN$ are $R$-modules over $X$, and $\bL$ is an $R$-module over $Y$. Let $\iota:\bN\to f_*\bN$ be the canonical identification.  If $\phi:\bM\to \bL$ has finite displacement over $f$, then $\phi$ induces a homomorphism $\phi^*:\Homfd(\bL,f_*\bN)\to \Homfd(\bM,\bN)$ given by $\psi\mapsto \iota^{-1}\circ\psi\circ \phi$.
	In particular, when $X=Y$ and $f=\id_X$, then $\phi$ induces a homomorphism $\phi^*:\Homfd(\bL,\bN)\to \Homfd(\bM,\bN)$ given by $\psi\mapsto \psi\circ \phi$.
\end{prop}

Another type of hom-set we work with is the following:
\begin{defn}\label{defn:homc}
	Let $X$ be a metric space, let  $\bM$ be an $R$-module over $X$, and let $N$ be an $R$-module. A homomorphism $\phi:M\to N$ has \emph{locally bounded support} if for each $i$,  there exists a bounded subset $F_i\subseteq X$ such that $\phi(m)=0$ for  all $m\in \bM(i)$ with $\supp(m)\subseteq X\setminus F_i$. We then define \[\Hom_c(\bM,N)\coloneqq  \left\{\phi\in \Hom_R(M,N)\;\middle|\;\phi \text{ has locally  bounded support}\right\}.\] In the case $N=R$, we define $\bM^*_c\coloneqq \Hom_c(\bM,R)$ and $\bM^*\coloneqq \Hom_R(\bM,R)$, and we observe that $\bM^*_c\leq \bM^*$.
\end{defn}

We now show that hom-sets of the form $\Hom_c(\bM,N)$  are special cases of hom-sets of the form  $\Homfd(\bM,\bN)$,  ensuring  results proved concerning the latter apply to the former.
\begin{defn}\label{defn:constant}
	An  $R$-module $\bN$ over $X$ is said to be \emph{constant} if:
	\begin{enumerate}
		\item $\supp_\bN(\bN)=\{x_0\}$ for some $x_0\in X$; and
		\item $\bN(i)=\bN$ for all $i\in \bbN$.
	\end{enumerate}
\end{defn}
Any $R$-module $N$ can be endowed with the structure of a constant $R$-module $\bN$ over any (non-empty) metric space.  Applying Lemma~\ref{lem:canon_isom}, we see that all constant $R$-modules over $X$ with the same underlying  $R$-module  are canonically isomorphic.  Consequently, Remark~\ref{rem:hom_close} ensures that if $\bM$ is an $R$-module over $X$ and if $\bN$ and $\bN'$ are two constant $R$-modules over $X$ with the same underlying $R$-module, then  $\Homfd(\bM,\bN)=\Homfd(\bM,\bN')$.
\begin{lem}\label{lem:trivial_module}
	Let $X$ be a metric space, let  $\bM$ be an $R$-module over $X$, and let $N$ be an $R$-module. Then $\Hom_{c}(\bM,N)=\Homfd(\bM,\bN)$, where $\bN$ is some (equivalently any)  constant  $R$-module over $X$ with underlying $R$-module $N$.
\end{lem}
\begin{proof}
	Pick $x_0\in X$ such that $\supp_\bN(\bN)=\{x_0\}$.

	Suppose $\phi:M\to N$ has locally bounded support. For each $i\in \bbN$, choose a bounded $F_i\subseteq X$ such that $\phi(m)=0$ for all $m\in \bM(i)$ with $\supp(m)\subseteq X\setminus F_i$. For each $i$, set $\Phi(i)\coloneqq\lceil \sup_{x\in F_i}d(x_0,x)\rceil$. We claim $\phi:\bM\to \bN$ has displacement $\Phi$.
	Clearly $\phi(\bM(i))\subseteq \bN=\bN(\Phi(i))$. Suppose $m\in \bM(i)$. It remains to show $\supp_{\bN}(\phi(m))\subseteq N_{\Phi(i)}(\supp_\bM(m))$. If $\phi(m)=0$ we are done,  so assume $\phi(m)\neq 0$. Then $\supp_\bM(m)\cap F_i\neq \emptyset$, and so $\supp_\bN(\phi(m))=\{x_0\}\subseteq N_{\Phi(i)}(\supp_\bM(m))$ as required.

	Conversely, suppose $\phi:\bM\to\bN$ has displacement $\Phi$. Set $F_i\coloneqq N_{\Phi(i)}(x_0)$, which is a bounded subset of $X$.  Suppose $m\in \bM(i)$ and $\supp_\bM(m)\subseteq X\setminus F_i$. Then \[\supp_\bN(\phi(m))\subseteq N_{\Phi(i)}(\supp_\bM(m))\subseteq X\setminus \{x_0\}.\] If $\phi(m)\neq 0$, then $\supp_\bN(\phi(m))=\{x_0\}$, which is a contradiction. Thus $\phi(m)=0$, and so $\phi$ has locally bounded support.
\end{proof}

It is easy to see that if $\bL$ is a constant $R$-module over $X$ and $f:X\to Y$ is arbitrary, then $f_*\bL$ is a constant $R$-module over $Y$. Combined with Proposition~\ref{prop:dual_fd}, this implies:
\begin{prop}\label{prop:dual_fd}
	Let $X$ and $Y$ be metric spaces and let $f:X\to Y$ be a coarse embedding. Suppose  $\bM$ and $\bN$ are $R$-modules over $X$ and $Y$ respectively, and  $\phi:\bM\to \bN$ has finite  displacement over $f$. Then for each $R$-module $L$, $\phi$ induces the dual map $\phi^*:\Hom_c(\bN,L)\to \Hom_c(\bM,L)$ via $\alpha\mapsto \alpha\circ \phi$.
\end{prop}

For the remainder of this section, we focus on projective $R$-modules over $X$, which have a richer duality theory than arbitrary $R$-modules over $X$.  We first establish counterparts of the  support function $\supp_\bM$ that apply to  elements of $\Hom_R(\bM,N)$:
\begin{defn}\label{defn:localsupport}
	Let $X$ be a metric space, let  $\bM=(M,B, \delta, p, \{B(i)\})$ be a projective $R$-module over $X$ with projective basis $\{m_b\}$, and let $N$ be an $R$-module. We assume  that $\{m_b\}$ satisfies the conclusion of Lemma~\ref{lem:proj_basis_finheight} when $\bM$ has finite height; by Lemma~\ref{lem:proj_basis_finheight}, there is no loss of generality in assuming this.
	For each $R$-module homomorphism $\alpha:\bM\to N$ and $i\in \bbN$, we define the \emph{local support of $\alpha$ at height $i$} to be
	\[\supp(\alpha,i)=\supp_\bM(\alpha,i)\coloneqq \{p(b)\in X\mid b\in B(i),  \alpha(m_b)\neq 0\}.\]
\end{defn}

It is important to note that $\supp(\alpha,i)$ depends upon the choice of projective basis $\{m_b\}$. However, as we will observe in Remark~\ref{rem:locsupport_indep},  $\supp(\alpha,i)$ is independent of $\{m_b\}$ up to some error depending on the displacement of $\{m_b\}$. We assume that some $\{m_b\}$ as in Definition~\ref{defn:localsupport} is fixed.

\begin{prop}\label{prop:bdd_support}
	Let $X$ be a metric space, let  $\bM$ be a projective $R$-module over $X$,  and let $N$ be an $R$-module. Then $\alpha:\bM\to N$ has locally bounded support if and only if $\supp(\alpha,i)$ is bounded for each $i\in \bbN$.
	In particular, when $\bM$ is proper, $\phi$ has locally bounded support if and only if $\supp(\alpha,i)$ is finite for each $i$.
\end{prop}
\begin{proof}
	Assume that  the projective basis $\{m_b\}$ as in Definition~\ref{defn:localsupport} has displacement $\Phi$. We first suppose $\alpha$ has locally bounded supports. Fix $i\in \bbN$ and pick a bounded set $F\subseteq X$ such that $\alpha(m)=0$ for all $m\in \bM(\Phi(i))$ with $\supp_\bM(m)\subseteq X\setminus F$. Suppose $x\in \supp(\alpha,i)$. Then there is some $b\in B(i)$ with $p(b)=x$ and $\alpha(m_b)\neq 0$. Since $m_b\in \bM(\Phi(i))$ and $\alpha(m_b)\neq 0$, we deduce $\supp(m_b)\cap F\neq \emptyset$.  As $\supp(m_b)\subseteq N_{\Phi(i)}(p(b))=N_{\Phi(i)}(x)$, we deduce that $x\in N_{\Phi(i)}(F)$. Thus $\supp(\alpha,i)\subseteq N_{\Phi(i)}(F)$ is bounded.

	Conversely, suppose that  $\supp(\alpha,i)$ is bounded for every $i\in \bbN$. Fix $i\in \bbN$ and let $F\coloneqq \supp(\alpha,i)$, which is bounded by assumption. Suppose $m\in \bM(i)$ with $\supp(m)\subseteq X\setminus F$. It is sufficient to show $\alpha(m)=0$. To see this, we know that $m=\sum_{b\in B(i)} \delta_b(m)m_b$ with almost all terms zero. If $\alpha(m_b)\neq 0$ for some $b\in B(i)$, then $p(b)\in\supp(\alpha,i)\subseteq F$, and so $\delta_b(m)=0$ because $\supp(m)\subseteq X\setminus F$. Thus for each $b\in B(i)$, we have $\delta_b(m)\alpha(m_b)=0$. This shows that $\alpha(m)=0$ as required.
\end{proof}

\begin{lem}\label{lem:pairing_support_finite}
	Let $X$ be a metric space, let  $\bM$ be a projective $R$-module over $X$,  and let $N$ be an $R$-module. If $\sigma\in \bM(i)$, $\alpha\in \Hom_R(\bM,N)$ and $\alpha(\sigma)\neq 0$, then $\supp(\sigma)\cap \supp(\alpha,i)\neq \emptyset$.
\end{lem}
\begin{proof}
	Since $\sigma\in \bM(i)$, we have $\sigma=\sum_{b\in B(i)}\delta_b(\sigma)m_b$. If $\alpha(\sigma)\neq 0$, then there is some $b\in B(i)$ with $\delta_b(\sigma)\neq 0$ and $\alpha(m_b)\neq 0$. Thus $p(b)\in  \supp(\sigma)\cap \supp(\alpha,i)$.
\end{proof}
\begin{lem}\label{lem:local_support_disp}
	Let $X$ be a metric space, let  $\bM$ and $\bN$ be  projective $R$-modules over $X$ of displacement $\Phi$,  and let $L$ be an $R$-module. Suppose  $\phi:\bM\to \bN$ has displacement $\Phi$ over $\id_X$. For all  $i\in \bbN$ and $\alpha\in \Hom_R(\bN,L)$,  \[\supp_\bM(\phi^*\alpha,i)\subseteq N_{\Phi^2(i)+\Phi(i)}(\supp_\bN(\alpha,\Phi^2(i))),\]
	where $\phi^*$ is as in Proposition~\ref{prop:dual_fd}.
\end{lem}
\begin{proof}
	Suppose $x\in \supp_\bM(\phi^*(\alpha),i)$. Then there is some $b\in B(i)$ with $p(b)=x$ and  $\phi^*(\alpha)(m_b)=\alpha(\phi(m_b))\neq 0$.  Since $m_b\in \bM(\Phi(i))$, we see that $\phi(m_b)\in \bN(\Phi^2(i))$ and that $\supp_\bN(\phi(m_b))\subseteq N_{\Phi^2(i)}(\supp_\bM(m_b))\subseteq N_{\Phi^2(i)+\Phi(i)}(x)$. Lemma~\ref{lem:pairing_support_finite} thus implies that since $\alpha(\phi(m_b))\neq 0$, we have \[\emptyset\neq \supp_\bN(\alpha,\Phi^2(i))\cap \supp_\bN(\phi(m_b))\subseteq \supp_\bN(\alpha,\Phi^2(i))\cap N_{\Phi^2(i)+\Phi(i)}(x). \] It follows that $x\in N_{\Phi^2(i)+\Phi(i)}(\supp_\bN(\alpha,\Phi^2(i)))$.
\end{proof}
\begin{rem}\label{rem:locsupport_indep}
	Applying Lemma~\ref{lem:local_support_disp} in the situation  $\bM=\bN$,  $\phi=\id$, and $\bM$ and $\bN$ are equipped with two different  projective basis, each of displacement $\Phi$, we see that up to an error of at most $\Phi^2(i)+\Phi(i)$,  $\supp_\bM(\alpha,i)$ is independent of the choice of projective basis  $\{m_b\}$.
\end{rem}

For the remainder of this section, we focus on $\bM^*=\Hom_R(\bM,R)$ and $\bM^*_c=\Hom_c(\bM,R)$, where $\bM$ is a projective $R$-module over $X$. We first make the following observation:

\begin{rem}\label{rem:coord_boundedsupp}
	Let $\bM=(M,B,\delta,p,\{B(i)\})$ be an $R$-module over $X$. If $b\in B$, then $\delta_b(m)\neq 0$ only if $p(b)\in \supp_\bM(m)$. Consequently, if $m\in \bM$ with $\supp(m)\subseteq X\setminus \{p(b)\}$, then $\delta_b(m)=0$. This shows $\delta_b\in \bM^*_c$.
\end{rem}
While it is not typically  the case that either $\bM^*$ or $\bM^*_c$ are themselves projective $R$-modules over $X$, we will see  the coordinates $\{\delta_b\}$ of $\bM$ are a  sort of ``projective Schauder  basis'' of $\bM^*$. To make this precise, we introduce the following convention regarding infinite sums:
\begin{notat}\label{not:infsum}
	Let $M$ be an $R$-module, let $B$ be an arbitrary indexing set and $\phi\in M^*=\Hom_R(M,R)$. Suppose that for each $b\in B$, we have elements $\phi_b\in M^*$ and $r_b\in R$. We write the expression \[\phi=\sum_{b\in B}r_b\phi_b,\] to indicate that for  each $m\in M$,  we have $\phi(m)=\sum_{b\in B}r_b\phi_b(m)$ and that $r_b\phi_b(m)=0$ for almost all  $b\in B$.
\end{notat}
In the case $B$ is finite, this notation is clearly consistent with the sum $\sum_{b\in B}r_b\phi_b$ of  finitely many elements of the $R$-module $M^*$.

This yields a useful description of elements of $\bM^*$:
\begin{lem}\label{lem:dual_infsum_second}
	Let $\bM=(M,B,\delta,p,\{B(i)\})$ be  a projective $R$-module over $X$ with  projective basis $\{m_b\}$.
	\begin{enumerate}
		\item For each $\phi\in \bM^*$, \[\phi=\sum_{b\in B}\phi(m_b)\delta_b.\]
		\item For each $i\in \bbN$ and $\phi\in \bM^*$, we have $\phi-\phi_i\in \Ann(\bM(i))$, where $\phi_i\coloneqq \sum_{b\in B(i)}\phi(m_b)\delta_b$ and  $\Ann(\bM(i))\coloneqq\{\phi\in \bM^*_c\mid \phi(\bM(i))=0\}$.
	\end{enumerate}
\end{lem}
\begin{proof}
	Let $m\in \bM$. Since $\delta_b(m)= 0$ for almost all $b\in B$, we have  \[\phi(m)=\phi\left(\sum_{b\in B}\delta_b(m)m_b\right)=\left(\sum_{b\in B}\phi(m_b)\delta_b\right)(m),\] as required.
	If $\phi\in \bM^*$ and $m\in \bM(i)$, then \begin{align*}
		\phi(m)=\phi\left(\sum_{b\in B}\delta_b(m)m_b\right) & =
		\phi\left(\sum_{b\in B(i)}\delta_b(m)m_b\right)                                                                       =\left(\sum_{b\in B(i)}\phi(m_b)\delta_b\right)(m)=\phi_i(m),
	\end{align*}
	so that $\phi-\phi_i\in \Ann(\bM(i))$.
\end{proof}

In the case $\bM$ is a proper projective $R$-module over $X$, we can define an  important submodule of $\Hom_R(\bM^*_c,R)$ as follows:
\begin{defn}\label{defn:lf}
	Let $\bM$ be a proper projective $R$-module over $X$. We set \[\bM^{\lf}\coloneqq \{\psi\in \Hom_R(\bM^*_c ,R)\mid  \psi(\Ann(\bM(i)))=0 \text{ for some $i$} \}.\]
\end{defn}
We now show that $\bM^{\lf}$ can be endowed with the structure  of an $R$-module over $X$:
\begin{prop}\label{prop:induced_proj_lf}
	Let $\bM=(M,B,\delta,p,\cF)$ be a proper projective $R$-module over $X$ with projective basis $\{m_b\}$ of displacement $\Phi$.
	\begin{enumerate}
		\item The double-dual map $\tau:\bM\to \bM^{\lf}$   given by $\tau(m)(\phi)=\phi(m)$ is injective.\label{item:induced_proj_lf0}
	\end{enumerate}
	Henceforth, we identify $\bM$ with its image under $\tau$.
	\begin{enumerate}
		\setcounter{enumi}{1}
		\item\label{item:induced_proj_lf1} For each $\psi\in \bM^{\lf}$ and $i$ is sufficiently large such that $\psi(\Ann(\bM(i)))=0$, we have \[\psi=\sum_{b\in B}\psi(\delta_b)m_b=\sum_{b\in B(i)}\psi(\delta_b)m_b.\]

		\item $\bM^{\lf}$ can be endowed with the structure of an $R$-module over $X$ given by \[\bM^{\lf}=(\bM^{\lf},B,\hat \delta,p,\cF),\] where $\hat\delta_b(\psi)=\psi(\delta_b)$ for all $b\in B$. Moreover,  after identifying $\bM$ with its image under  $\tau$, we have that $\bM$ is a geometric submodule of $\bM^{\lf}$.\label{item:induced_proj_lf2}

		\item\label{item:induced_proj_lf4} For all $i\in \bbN$, we have \begin{align*}
			      \{\psi\in \bM^{\lf}\mid \psi(\Ann(\bM(i)))=0\}  \subseteq \bM^{\lf}(i) \subseteq \{\psi\in \bM^{\lf}\mid \psi(\Ann(\bM(\Phi(i))))=0\}.
		      \end{align*}
	\end{enumerate}
\end{prop}
\begin{proof}
	(\ref{item:induced_proj_lf0}): We first show that for each $m\in \bM$, the element $\tau(m)\in \Hom_R(\bM^*_c,R)$ given by $\tau(m)(\phi)=\phi(m)$ is actually in $\bM^{\lf}$. Indeed, we know that $m\in \bM(i)$ for some $i\in \bbN$. If $\phi\in \Ann(\bM(i))$, we have $\tau(m)(\phi)=\phi(m)=0$. Thus $\tau(m)(\Ann(\bM(i)))=0$, and so $\tau(m)\in \bM^{\lf}$. To see $\tau$ is injective, we observe that if $m\in \bM$ such that $\tau(m)=0$, then \[m=\sum_{b\in B}\delta_b(m)m_b=\sum_{b\in B}\tau(m)(\delta_b)m_b=0.\]

	(\ref{item:induced_proj_lf1}): Let $\psi\in \bM^{\lf}$ with $\psi(\Ann(\bM(i)))=0$, and suppose $\phi\in \bM^*_c$. Note that for $b\in B\setminus B(i)$, we have $\delta_b\in \Ann(\bM(i))$, and so $\psi(\delta_b)=0$. Since $\bM$ is proper, it follows from Proposition~\ref{prop:bdd_support} that $\phi(m_b)=0$ for almost all $b\in B(i)$; thus almost all terms of  $\phi_i\coloneqq \sum_{b\in  B(i)}\phi(m_b)\delta_b$ are zero. Lemma~\ref{lem:dual_infsum_second} implies that $\phi-\phi_i\in \Ann(\bM(i))$, and so $\psi(\phi)=\psi(\phi_i)$. Therefore,
	\begin{align*}
		\psi(\phi)=\psi\left(\sum_{b\in  B(i)}\phi(m_b)\delta_b\right)=\sum_{b\in  B(i)}\phi(m_b)\psi(\delta_b)=\left(\sum_{b\in  B(i)}\psi(\delta_b)\tau(m_b)\right)(\phi).
	\end{align*}
	Since this holds for all $\phi\in \bM^*_c$, we conclude that $\psi=\sum_{b\in  B(i)}\psi(\delta_b)\tau(m_b)$. Since we are identifying  $\bM$ with its image under $\tau$, we have $\psi=\sum_{b\in  B(i)}\psi(\delta_b)m_b$.

	(\ref{item:induced_proj_lf2}):
	To show $(\bM^{\lf},B,\hat \delta,p,\cF)$ is an $R$-module over $X$, we need to verify (\ref{item:moddef_zero}) and (\ref{item:moddef_filtr}) of Definition~\ref{defn:rmod_overX}. Firstly, if $\psi\in \bM^{\lf}$ with $\hat \delta_b(\psi)=\psi(\delta_b)=0$ for all $b\in B$, then $\psi=0$ by 	(\ref{item:induced_proj_lf1}). Secondly, if $\psi\in \bM^{\lf}$, then  $\psi(\Ann(\bM(i)))=0$ for some $i$. Thus for $b\in B\setminus B(i)$,  $\delta_b\in \Ann(\bM(i))$ and so $\hat\delta_b(\psi)=0$. This shows that $\bM^{\lf}$ is indeed an $R$-module over $X$. To see $\tau(\bM)$ is a geometric submodule of $\bM^{\lf}$, we need to show $\hat\delta_b\circ \tau$ agrees with $\delta_b$ for each $b\in B$. This is clear, since for each $m\in M$, we have \begin{align*}
		\hat\delta_b(\tau(m))=\tau(m)(\delta_b)=\delta_b(m)
	\end{align*} as required.

	(\ref{item:induced_proj_lf4}): First suppose $\psi(\Ann(\bM(i)))=0$. For each $b\in B\setminus B(i)$, we have  $\delta_b\in \Ann(\bM(i))$ and so $\hat\delta_b(\psi)=\psi(\delta_b)=0$. It follows that $\psi\in \bM^{\lf}(i)$.

	Now suppose $\psi\in \bM^{\lf}(i)$. Let $\phi\in\Ann(\bM(\Phi(i)))$.  We know from (\ref{item:induced_proj_lf1}) that
	\begin{align}\label{eqn:lf_expand}
		\psi(\phi)=\left(\sum_{b\in  B}\tau(m_b)\psi(\delta_b)\right)(\phi)=\sum_{b\in  B}\psi(\delta_b)\phi(m_b),
	\end{align} where $\psi(\delta_b)\phi(m_b)=0$ for almost all $b\in B$. If $b\in B(i)$, then $m_b\in \bM(\Phi(i))$ and so $\phi(m_b)=0$. If $b\in B\setminus B(i)$, then $\psi(\delta_b)=\hat\delta_b(\psi)=0$ since $\psi\in \bM^{\lf}(i)$. Thus $\psi(\delta_b)\phi(m_b)=0$ for all $b\in B$, and so (\ref{eqn:lf_expand}) implies  $\psi(\phi)=0$. This implies $\psi(\Ann(\bM(\Phi(i))))=0$.
\end{proof}

Henceforth, whenever $\bM$ is a  proper projective $R$-module over $X$, we endow $\bM^{\lf}$ with the structure of an $R$-module over $X$ as in Proposition~\ref{prop:induced_proj_lf}, and identify $\bM$ with its image in $\bM^{\lf}$ under $\tau$.
\begin{lem}\label{lem:lf_induced}
	Let $\bM$ and $\bN$ be proper projective $R$-modules over $X$ and $Y$ respectively. Suppose $f:X\to Y$ is a coarse embedding  and $\phi:\bM\to \bN$ has finite displacement over $f$.
	Then $\phi$ extends to  a finite displacement map $\phi^{\lf}:\bM^{\lf}\to \bN^{\lf}$ given by $\phi^{\lf}(\psi)(\alpha)=\psi(\phi^*\alpha)$ for each $\psi\in \bM^{\lf}$ and $\alpha\in \bN^*_c$.
\end{lem}
\begin{proof}
	Suppose $\bM=(M,B,\delta,p,\{B(i)\})$ and $\bN=(N,C,\epsilon,q,\{C(i)\})$. Pick $\Phi:\bbN\to \bbN$ large enough such that  $\bM$ has displacement $\Phi$ and $\phi$ has displacement $\Phi$ over $f$. Let $\{m_b\}$ be a projective basis of $\bM$ with displacement $\Phi$. Assume also that $f$ is a $(\Lambda,\Upsilon)$-coarse embedding. Fix $i\in \bbN$ and set $i_k=\Phi^k(i)$.

	If $\psi\in \bM^{\lf}(i)$, then $\psi(\Ann(\bM(i_1)))=0$ by (\ref{item:induced_proj_lf4}) of Proposition~\ref{prop:induced_proj_lf}. Since $\phi(\bM(i_1))\subseteq \bN(i_2)$, it follows that $\phi^*(\Ann(\bN(i_2)))\subseteq \Ann(\bM(i_1))$. Thus \[\phi^{\lf}(\psi)(\Ann(\bN(i_2)))\subseteq \psi(\Ann(\bM(i_1)))=0,\] and so (\ref{item:induced_proj_lf4}) of Proposition~\ref{prop:induced_proj_lf} implies $\phi^{\lf}(\psi)\in \bN^{\lf}(i_2)$.

	Now suppose $\psi\in \bM^{\lf}(i)$ and $y\in \supp_{\bN^{\lf}}(\phi^{\lf}(\psi))$. Then there exists some $c\in C$ with $q(c)=y$ and $\hat\epsilon_c(\phi^{\lf}(\psi))\neq 0$. It follows that $\psi(\phi^*\epsilon_c)=\phi^{\lf}(\psi)(\epsilon_c)\neq 0$. As observed above, we have $\psi(\Ann(\bM(i_1)))=0$. By Lemma~\ref{lem:dual_infsum_second} $\phi^*\epsilon_c-\sum_{b\in B(i_1)}(\phi^*\epsilon_c)(m_b)\delta_b\in \Ann(\bM(i_1))$, and so \begin{align*}
		\psi\left(\sum_{b\in B(i_1)}(\phi^*\epsilon_c)(m_b)\delta_b\right)=\psi(\phi^*\epsilon_c)\neq 0.
	\end{align*}
	Thus there is some $b\in B(i_1)$ with $\epsilon_c(\phi(m_b))\hat\delta_b(\psi)\neq 0$. Since $\hat\delta_b(\psi)\neq 0$, we have $p(b)\in \supp_{\bM^{\lf}}(\psi)$. Since $m_b\in \bM(i_2)$ and $\epsilon_c(\phi(m_b))\neq 0$, we have
	\[y\in \supp_\bN(\phi(m_b))\subseteq N_{i_3}(f(\supp_\bM(m_b)))
		\subseteq N_{i_3}(f(N_{i_2}(p(b))))
		\subseteq N_{j}(f(\supp_{\bM^{\lf}}(\psi))),\]
	where $j\coloneqq\Upsilon(i_2)+i_3$.
	Thus $\supp_{\bN^{\lf}}(\phi^{\lf}(\psi))\subseteq N_j(f(\supp_{\bM^{\lf}}(\psi)))$ as required.
\end{proof}

\begin{rem}
	Suppose $\phi$ is as above and $\bM$ has projective basis $\{m_b\}$. Then  each $\psi\in \bM^{\lf}$ can be expressed as an infinite sum $\psi=\sum_{b\in B}\psi(\delta_b)m_b$, and it can be readily verified that \[\phi^{\lf}\left(\sum_{b\in B}\psi(\delta_b)m_b\right)=\sum_{b\in B}\psi(\delta_b)\phi(m_b),\] where all these infinite sums satisfy the criterion in Notation~\ref{not:infsum}. Thus $\phi^{\lf}$ can be thought of as the linear extension of $\phi$ to infinite sums.
\end{rem}

Although  $\bM^*_c$ does not naturally inherit the structure of an $R$-module over $X$ in general, the next proposition demonstrates that  if $\bM^*_c$ is a proper projective $R$-module of finite height, then so is $\bM^*_c$. In light of Proposition~\ref{prop:ginv_loc_fin_type}, this is an analogue of the fact that the dual of a finitely generated projective $RG$-module is also projective~\cite[Proposition I.8.3]{brown1982cohomology}.
\begin{prop}\label{prop:duality_ht}
	Let $\bM=(M,B,\delta,p,\{B(i)\})$ be a proper  projective $R$-module over $X$,  of displacement $\Phi$ and of height $i_0<\infty$. Then $\bM^*_c$ can be endowed with the structure of a  proper  projective $R$-module over $X$, of displacement $\Phi'=\Phi'(\Phi,i_0)$ and of height at most $i_0$. Moreover, the following are  satisfied:
	\begin{enumerate}
		\item\label{item:duality_ht_support_basis}  $\supp_{\bM^*_c}(\delta_b)\subseteq N_{\Phi(i_0)}(p(b))$ for all $b\in B$.
		\item\label{item:duality_ht_support} If $\alpha\in \bM^*_c$ and $m\in \bM$ satisfy $\alpha(m)\neq 0$, then \[\supp_{\bM}(m)\cap N_{\Phi(i_0)}(\supp_{\bM^*_c}(\alpha))\neq \emptyset.\]
		\item\label{item:duality_ht_double}
		      After identifying  $\bM$ with its image under the natural monomorphism $\tau:\bM\to \Hom_R(\bM^*_c,R)$ given by $\tau(m)(\alpha)=\alpha(m)$, we have $\bM=(\bM^*_c)^*_c$.
	\end{enumerate}

	\begin{enumerate}
		\setcounter{enumi}{3}
		\item\label{item:duality_ht_natural} Suppose $f:X\to Y$ is a $(\Lambda, \Upsilon)$-coarse equivalence and $\phi:\bM\to \bN$ has  displacement $\Phi$ over $f$, where $\bN$ is a finite-height proper projective $R$-module over $Y$. Assume  that $\bN^*_c$ is endowed with the structure of a proper projective finite-height $R$-module over $X$, in an analogous manner as $\bM^*_c$. Then $\phi^*:\bN^*_c\to \bM^*_c$ has  displacement $\Phi'$ over an $\Upsilon(0)$-coarse inverse of $f$, where $\Phi'$ depends only on $\Lambda$, $\Upsilon$, $\Phi$ and $\bM$.
	\end{enumerate}
\end{prop}
\begin{proof}
	Since $\bM$ has height $i_0$, we deduce that $\delta_b=0$ for all $b\in B\setminus B(i_0)$. By Lemma~\ref{lem:proj_basis_finheight}, we can pick a projective basis $\{m_b\}$ of $\bM$ with displacement $\Phi$ such that $m_b=0$ for all $b\in B\setminus B(i_0)$.
	Let $\tau:\bM\to \Hom_R(\bM^*_c,R)$ be given by $\tau(m)(\alpha)=\alpha(m)$. We  give $\bM^*_c=\Hom_c(\bM,R)$ the structure of an $R$-module over $X$ by \[\bM^*_c=(\bM^*_c,B,\{\tau(m_b)\},p,\{B(i)\}).\] We  note that  $\bM^*_c=\bM^*_c(i_0)$,  showing $\bM^*_c$ has height at most $i_0$. Moreover, $\bM^*_c$ is proper as $\bM$ is.

	We claim $\{\delta_b\}$ is a projective basis of $\bM^*_c$. Firstly,  $\delta_b\in \bM^*_c$ by Remark~\ref{rem:coord_boundedsupp}.
	Let $\alpha\in \bM^*_c$. By Proposition~\ref{prop:bdd_support}, there are only finitely many $b\in B(i_0)$ such that $\tau(m_b)(\alpha)=\alpha(m_b)\neq 0$. Since $m_b=0$ for $b\in B\setminus B(i_0)$, there are only finitely many $b\in B$ such that $\tau(m_b)(\alpha)=\alpha(m_b)\neq 0$
	Thus for each $m\in \bM=\bM(i_0)$, we have
	\begin{align*}
		\sum_{b\in B}\tau(m_b)(\alpha)\delta_b(m)=\sum_{b\in B}\alpha(m_b)\delta_b(m)=\alpha\left(\sum_{b\in B}\delta_b(m)m_b\right)=\alpha(m),
	\end{align*}
	hence $\alpha=\sum_{b\in B}\tau(m_b)(\alpha)\delta_b$, where almost all  terms of this sum are zero. This shows that $\{(\delta_b,\tau(m_b))\}_{b\in B}$ is a projective basis of $\bM^*_c$ as an $R$-module.

	To see $\bM^*_c$ is a projective $R$-module over $X$, it remains to demonstrate that $\{\delta_b\}$ has finite displacement. Let $b\in B$. If $x\in \supp_{\bM^*_c}(\delta_b)$, then there is some $b'\in B$ with $\tau(m_{b'})(\delta_b)\neq 0$ and $p(b')=x$. Thus $\delta_b(m_{b'})\neq 0$, and so $p(b)\in \supp_{\bM}(m_{b'})\subseteq N_{\Phi(i_0)}(x)$. Therefore $\supp_{\bM^*_c}(\delta_b)\subseteq N_{\Phi(i_0)}(p(b))$, proving (\ref{item:duality_ht_support_basis}).  As $\delta_b\in \bM^*_c(i_0)$ for all $b\in B$, it follows that $\{\delta_b\}$ has displacement depending only on $\Phi$ and $i_0$.

	(\ref{item:duality_ht_support}):  Suppose that $\alpha\in \bM^*_c$ and $m\in \bM$ satisfy $\alpha(m)\neq 0$. We have $\alpha=\sum_{b\in B}\tau(m_b)(\alpha)\delta_b$ and $m=\sum_{b\in B}\delta_b(m)m_b$. Thus there exist $b,b'\in B(i_0)$ with $\tau(m_b)(\alpha)=\alpha(m_b)\neq 0$, $\delta_{b'}(m)\neq 0$ and $\delta_b(m_{b'})\neq 0$. Thus $p(b)\in \supp_\bM(m_{b'})\subseteq N_{\Phi(i_0)}(p(b'))$ so that $d(p(b),p(b'))\leq \Phi(i_0)$. Since $p(b)\in \supp_{\bM^*_c}(\alpha)$ and $p(b')\in \supp_\bM(m)$, we conclude that $p(b')\in \supp_{\bM}(m)\cap N_{\Phi(i_0)}(\supp_{\bM^*_c}(\alpha))$ as required.

	(\ref{item:duality_ht_double}):
	Repeating the above construction with $\bM^*_c$ in place of $\bM$, we have  that \[(\bM^*_c)^*_c=((\bM^*_c)^*_c,B,\theta(\delta_b),p,\{B(i)\})\] is a finite height projective $R$-module over $X$ with projective basis $\{\theta(\delta_b)\}$, where   $\theta(\delta_b)(\nu)=\nu(\delta_b)$ for all $\nu\in(\bM^*_c)^*_c$ and $b\in B$.
	For each $m\in M$, we note \[\{b\in B\mid \tau(m)(\delta_b)\neq 0\}=\{b\in B\mid \delta_b(m)\neq 0\}\subseteq p^{-1}(\supp_\bM(m))\]  is finite because $\bM$ is proper. Thus Proposition~\ref{prop:bdd_support} ensures $\tau(m)\in (\bM^*_c)^*_c$.  To see $\tau$ is injective, we note that  if $\tau(m)=0$, then $m=\sum_{b\in B}\delta_b(m)m_b=\sum_{b\in B}\tau(m)(\delta_b)m_b=0$. If $\nu\in (\bM^*_c)^*_c$, we see that \begin{align*}
		\tau\left(\sum_{b\in B}\nu(\delta_b)m_b\right)=\sum_{b\in B}\nu(\delta_b)\tau (m_b)=\sum_{b\in B}\theta(\delta_b)(\nu)\tau (m_b)=\nu,
	\end{align*}
	where the last equality follows because $(\tau(m_b),\theta(\delta_b))$ is a projective basis of $(\bM^*_c)^*_c$. Thus $\tau$ induces an isomorphism $\bM\to (\bM^*_c)^*_c$ between underlying $R$-modules. Since for each $m\in \bM$ and $b\in B$, we have \[\theta(\delta_b)(\tau(m))=\tau(m)(\delta_b)=\delta_b(m),\] we see that after identifying $\bM$ with $(\bM^*_c)^*_c$ via $\tau$, we have $\theta(\delta_b)=\delta_b$, and so $\bM=(\bM^*_c)^*_c$.

	(\ref{item:duality_ht_natural}): Now suppose $\bN=(N,C, \epsilon,q,\{C(i)\})$ is another proper finite height projective $R$-module over $Y$ with projective basis $\{n_c\}$. As above, we define $\bN^*_c=(\bN^*_c, C,\{\tau(n_c)\},q,\{C(i)\})$ with projective basis $\{\epsilon_c\}$.
	Setting $A=\Upsilon(0)$, there exists a coarse inverse $g:Y\to X$ that is a  $(\Lambda, \Upsilon)$-coarse equivalence such that $g\circ f$ and $f\circ g$ are $A$-close to the identity maps.
	Let $\alpha\in \bN^*_c$ and suppose $x\in \supp_{\bM^*_c}(\phi^*(\alpha))$. There is some $b\in B(i_0)$ with $p(b)=x$ and $0\neq \tau(m_b)(\phi^*(\alpha))=\alpha(\phi(m_b))$. Since $\phi(m_b)=\sum_{c\in C}\epsilon_c(\phi(m_b))n_c$, there is some $c\in C$ with $\tau(n_c)(\alpha)=\alpha(n_c)\neq 0$ and $\epsilon_c(\phi(m_b))\neq 0$.
	Thus $q(c)\in \supp_{\bN^*_c}(\alpha)\cap \supp_\bN(\phi(m_b))$. We have
	\[\supp_\bN(\phi(m_b))\subseteq N_{\Phi(i_0)}(f(\supp_\bM(m_b)))\subseteq N_{\Phi(i_0)+\Upsilon(\Phi(i_0))}(f(x))\]
	Setting $r\coloneqq \Upsilon(\Phi(i_0)+\Upsilon(\Phi(i_0)))+A$, we have $d(x,g(q(c)))\leq r$. This implies that $\supp_{\bM^*_c}(\phi^*(\alpha))\subseteq N_r(g(\supp_{\bN^*_c}(\alpha)))$. Since $\bN^*_c$ and $\bM^*_c$ have finite height, this  proves that $\phi^*:\bN^*_c\to\bM^*_c$ has finite displacement over $g$.
\end{proof}
\begin{rem}\label{rem:duality_ht}
	Henceforth, for $\bM$ satisfying the hypotheses of Proposition~\ref{prop:duality_ht}, we always assume $\bM^*_c$ is endowed with the structure of a projective $R$-module over $X$  as in Proposition~\ref{prop:duality_ht}, and we freely identify $\bM$ and $(\bM^*_c)^*_c$ when convenient.
\end{rem}
\begin{rem}\label{rem:height_support}
	We remark that if $\bM$ is of height $i_0<\infty$ and $i\geq i_0$, then $\supp(\alpha,i)=\supp_{\bM^*_c}(\alpha)$ for all $\alpha\in \bM^*_c$.
\end{rem}
\begin{prop}\label{prop:tensor_dual}
	Let $\bM$ be a finite-height proper projective $R$-module over $X$. If $f:R\to S$ is a ring homomorphism, then there is a geometric isomorphism  $\Psi_\bM=\Psi:\bM^*_c\otimes_R S\to (\bM\otimes_R S)^*_c$ given by $\Psi(\alpha\otimes s)(m\otimes s')=f(\alpha(m))ss'$ for all  $m\in \bM$, $\alpha\in \bM^*_c$ and $s,s'\in S$.

	Moreover, $\Psi_\bM$ in natural in the following sense. If  $\phi:\bM\to \bN$ is a finite displacement between finite-height projective $R$-modules  over $\id_X$, then the following diagram commutes:
	\begin{figure}[htb]
		\[\begin{tikzcd}
				{\bN^*_c\otimes_R S} && {(\bN\otimes_R S)^*_c} \\
				\\
				{\bM^*_c \otimes_R S} && {(\bM\otimes_R S)^*_c}
				\arrow["{\Psi_\mathbf{N}}", from=1-1, to=1-3]
				\arrow["{\phi^*\otimes \id_S}"', from=1-1, to=3-1]
				\arrow["{(\phi\otimes \id_S)^*}"', from=1-3, to=3-3]
				\arrow["{\Psi_\mathbf{M}}", from=3-1, to=3-3]
			\end{tikzcd}\]
		\caption{A commutative diagram showing naturality of $\Psi$.}
	\end{figure}
\end{prop}
\begin{proof}
	Suppose that $\bM=(M,B,{\delta_b},p,\{B(i)\})$ and has projective basis $\{m_b\}$.
	We use the description of projective basis given in Propositions~\ref{prop:change_rings} and~\ref{prop:duality_ht} to deduce:
	\begin{itemize}
		\item $\bM^*_c\otimes_R S$ has projective basis $\{\delta_b\otimes 1\}_{b\in B}$;
		\item $(\bM\otimes _R S)^*_c$ has projective basis $\{\epsilon_b\}_{b\in B}$, where $\epsilon_b:\bM\otimes_R S\to S$ is the map $\epsilon_b(s\otimes m)=f(\delta_b(m))s$.
	\end{itemize}
	Note that for each $b\in B$, $s\in S$ and $m\in \bM$, we have \[\Psi(\delta_b\otimes 1)(m\otimes s)=f(\delta_b(m))s=\epsilon_b(m\otimes s)\] so that $\Psi(\delta_b\otimes 1)=\epsilon_b$ for all $b\in B$.

	We choose $\Phi$ sufficiently large such that  both these projective bases have displacement $\Phi$, i.e.\
	\begin{enumerate}
		\item For each $b\in B(i)$, $\delta_b\otimes 1\in (\bM^*_c\otimes_R S)(\Phi(i))$ and $\supp(\delta_b\otimes 1)\subseteq N_{\Phi(i)}(p(b))$.
		\item For each $b\in B(i)$, $\epsilon_b\in ((\bM\otimes _R S)^*_c)(\Phi(i))$ and $\supp(\epsilon_b)\subseteq N_{\Phi(i)}(p(b))$.
	\end{enumerate}
	Thus Lemma~\ref{lem:define_fin_disp} implies that $\Psi$ has finite displacement over $\id_X$. Moreover,  Lemma~\ref{lem:define_fin_disp} also implies we can define a finite displacement map $\Theta:(\bM\otimes _R S)^*_c\to \bM^*_c\otimes_R S$ such that $\Theta(\epsilon_b)= \delta_b\otimes 1$. Since $\Theta\circ \Psi=\id$ and $\Psi\circ \Theta=\id$, we deduce $\Psi$ is a geometric isomorphism.

	If $\phi:\bM\to \bN$ has finite displacement, then for all  $s,s'\in S$, $m\in \bM$ and $\alpha\in \bN^*_c$, we have \begin{align*}
		\left((\phi\otimes \id_S)^*\circ \Psi_\bN\right)( \alpha\otimes s)(m\otimes s') & =\Psi_\bN(\alpha\otimes s)(\phi(m)\otimes s')=f(\alpha(\phi(m))) ss'    \\
		                                                                                & =f(\phi^*(\alpha)(m))ss'                                                \\
		                                                                                & =\Psi_\bM(\phi^*(\alpha)\otimes s)(m\otimes s')                         \\
		                                                                                & =(\Psi_\bM\circ (\phi^*\otimes \id_S))( \alpha \otimes s)(m\otimes s').
	\end{align*}
	Thus $(\phi\otimes \id_S)^*\circ \Psi_\bN= \Psi_\bM\circ (\phi^*\otimes \id) $ as required.
\end{proof}

\section{Projective resolutions over a metric space}\label{sec:projres}
\emph{In this section, we  define the main objects of study: projective $R$-resolutions over a metric space. We prove they always exist in Proposition~\ref{prop:resln_exist}, and that they are unique up to finite displacement chain homotopy in Corollary~\ref{cor:projres_indpt}. A key result is Proposition~\ref{prop:proj_res}, which is  an analogue of the \emph{Fundamental Lemma of Homological Algebra}~\cite[Lemma I.7.4]{brown1982cohomology}. It asserts that a coarse embedding $f:X\to Y$ induces a finite displacement chain map $f_\#:\bC_\bullet\to \bD_\bullet$ between projective $R$-resolutions over $X$ and $Y$, and this map is  unique up to a finite displacement chain homotopy.
}
\vspace{.3cm}

\begin{defn}\label{defn:r-chain-complex}
	We say that a non-negative chain complex $\bC_\bullet$ is an \emph{$R$-chain complex over $X$} if the following hold:
	\begin{enumerate}
		\item each $\bC_i$ is an $R$-module over $X$;
		\item each boundary map $\partial: \bC_i\to \bC_{i-1}$ has finite displacement over $\id_X$;
		\item $\{\bC_i\}_i$ is staggered.
	\end{enumerate}
\end{defn}
\begin{rem}
	The requirement that $\{\bC_i\}_i$ is staggered is somewhat artificial, and the  subsequent arguments can be easily adapted to remove it. However,  Remarks~\ref{rem:staggering} and~\ref{rem:staggering_ginv} ensure there is no real loss of generality in assuming this condition. Doing so means we can apply Propositions~\ref{prop:uniform_displacement} and~\ref{prop:uniform_displacement_projective}, allowing us  to work with a single function $\Phi:\bbN\to \bbN$ rather than countable families of such functions, thus drastically simplifying subsequent notation by reducing indices needed.
\end{rem}

An $R$-chain complex $\bC_\bullet$ is said to be \emph{proper} (resp.\ of \emph{finite height}) if each $\bC_i$ is proper (resp.\ of finite height).
An $R$-chain complex $\bC_\bullet$ over $X$ has \emph{uniform preimages} (resp.\ \emph{$\Omega$-uniform preimages}) if  each boundary map $\partial: \bC_i\to \bC_{i-1}$ has uniform preimages (resp.\ $\Omega$-uniform preimages). By Proposition~\ref{prop:uniform_displacement}, if $\bC_\bullet$ has uniform preimages, then there exists  an $\Omega$ such that $\bC_\bullet$ has $\Omega$-uniform preimages.

We say a subcomplex $\bD_\bullet\leq \bC_\bullet$ is a \emph{geometric subcomplex} if each $\bD_k$ is a geometric submodule of $\bC_k$. A chain map $f_\#:\bC_\bullet\to \bD_\bullet$ between $R$-chain complexes over $X$ and $Y$ respectively is said to be of \emph{finite displacement} (resp.\ \emph{displacement $\Phi$}) over $f:X\to Y$ if each $f_i:\bC_i\to \bD_i$ has finite displacement (resp.\ displacement $\Phi$) over $f$. If $f_\#$ has finite displacement over $f$, then Proposition~\ref{prop:uniform_displacement} guarantees $f_\#$ has displacement $\Phi$ over $f$ for some $\Phi$.
Given chain maps $f_\#,g_\#:\bC_\bullet\to \bD_\bullet$, a \emph{chain homotopy} from $f_\#$ to $g_\#$, denoted $f_\#\stackrel{h_\#}{\simeq}g_\#$ is a map $h_\#:\bC_\bullet\to \bD_{\bullet+1}$ such that $g_\#-f_\#=\partial h_\#+h_\#\partial$. A chain homotopy $h_\#:\bC_\bullet\to \bD_{\bullet+1}$ has \emph{finite displacement} (resp.\ \emph{displacement $\Phi$}) over $f$ if each $h_i:\bC_i\to \bD_{i+1}$ has finite displacement (resp.\ displacement $\Phi$) over $f$.

An \emph{augmentation} of an $R$-chain  complex $\bC_\bullet$ over $X$  consists of an $R$-module homomorphism $\epsilon:\bC_0\to R$  such that the composition
\[\bC_1\xrightarrow{\partial} \bC_0\xrightarrow{\epsilon}R\] is zero.
An \emph{augmented $R$-chain complex $(\bC_\bullet,\epsilon)$ over $X$}  is an $R$-chain complex $\bC_\bullet$ over $X$ equipped with an augmentation $\epsilon$.
Given  chain complexes $\bC_\bullet$ and $\bD_\bullet$ and   augmentations $\epsilon:\bC_0\to R$ and $\epsilon':\bD_0\to R$,  we say a chain map $f_\#:\bC_\bullet\to \bD_\bullet$ is \emph{augmentation-preserving} if $\epsilon'\circ f_0=\epsilon$. A \emph{reduced $k$-cycle} of $(\bC_\bullet,\epsilon)$ is a $k$-cycle  $\sigma\in \bC_k$ satisfying $\epsilon(\sigma)=0$ in the case $k=0$. We can thus define the \emph{reduced homology} $\widetilde{H}_k(\bC_\bullet)$ to be the $R$-module of reduced $k$-cycles modulo $k$-boundaries. When unambiguous, we will suppress $\epsilon$ from the notation, writing $\bC_\bullet$ instead of  $(\bC_\bullet,\epsilon)$. In the case where we have an augmented chain complex  whose augmentation has not been explicitly specified in the notation,  we will always use the letter $\epsilon$ to refer to this augmentation.

Let $\bM=(M,B,\delta,p,\cF)$ be an $R$-module over $X$. A geometric submodule $\bN\leq \bM$ is \emph{standard} if it is of the form $\bN=\cap_{b\in B'}\ker(\delta_b)$ for some $B'\subseteq B$. We note that each $\bM(i)$ is a standard submodule.
If $\bC_\bullet$ is an $R$-chain complex over $X$, we say a subcomplex $\bD_\bullet$ is \emph{standard} if each $\bD_k\leq\bC_k$ is standard.

\begin{defn}\label{defn:standard}
	Let $\bC_\bullet$ be an $R$-chain complex over $X$.
	For each $i\in \bbN$ and $Y\subseteq X$, we define $\bC_\bullet(i,Y)$ to be the maximal standard  subcomplex of $\bC_\bullet$ such that for each $k\in \bbN$, the following hold:
	\begin{enumerate}
		\item $\bC_k(i,Y)\subseteq \bC_k(i)$.\label{item:standard_1}
		\item If $\sigma\in \bC_k(i,Y)$, then $\supp(\sigma)\subseteq Y$.\label{item:standard_2}
	\end{enumerate}
\end{defn}

We always assume that when $\bC_\bullet$ is equipped with an augmentation $\epsilon:\bC_0\to R$, each $\bC_\bullet(i,Y)$ is equipped with the augmentation obtained by restricting $\epsilon$. The key property $\bC_\bullet(i,Y)$ satisfies is  the following:
\begin{lem}\label{lem:standard}
	Let $X$ be a metric space and suppose $\bC_\bullet$ is an $R$-metric complex over $X$ such that each boundary map has   displacement $\Phi$.
	\begin{enumerate}
		\item If $\sigma\in \bC_k(i,Y)$, then $\sigma\in \bC_k(i)$ and $\supp(\sigma)\subseteq Y$.\label{item:lemstandard1}
		\item If $\sigma\in \bC_k(i)$ and $\supp(\sigma)\subseteq Y$, then $\sigma\in \bC_k(\Phi^k(i),N_{2^k\Phi^k(i)}(Y))$.\label{item:lemstandard2}
	\end{enumerate}
\end{lem}
\begin{proof}
	(\ref{item:lemstandard1})  follows immediately from Definition~\ref{defn:standard}.

	(\ref{item:lemstandard2}): Suppose $\bC_k=(C_k,B_k,\delta^k,p_k,\{B_k(i)\})$ for each $k\in \bbN$. For each $Y\subseteq X$ and $i,k\in \bbN$, let
	\[\bD_k(i,Y)\coloneqq \cap \{\ker(\delta_b)\mid b\in B_k\setminus( B_k(i)\cap p_k^{-1}(Y))\},\] and note that $\bD_k(i,Y)$ consists precisely of the set of all $\sigma\in \bC_k(i)$ with $\supp(\sigma)\subseteq Y$. In particular, $\bC_k(i,Y)\subseteq \bD_k(i,Y)$ for each $k$.

	We will show $\bD_k(i,Y)\subseteq \bC_k(\Phi^k(i),N_{2^k\Phi^k(i)}(Y))$ for all $k,i\in\bbN$ and $Y\subseteq X$, which is equivalent to (\ref{item:lemstandard2}). We proceed by induction on $k$. For the base case, maximality of  $\bC_\bullet(i,Y)$ implies that $\bC_0(i,Y)= \bD_0(i,Y)$.
	For the inductive step,  we note that \begin{align*}
		\partial (\bD_k(i,Y)) & \subseteq\bD_{k-1}(\Phi(i),N_{\Phi(i)}(Y))\subseteq \bC_{k-1}(\Phi^k(i),N_{2^{k-1}\Phi^{k}(i)+\Phi(i)}(Y)) \\&\subseteq
		\bC_{k-1}(\Phi^k(i),N_{2^{k}\Phi^{k}(i)}(Y)),
	\end{align*} where the first inclusion follows because $\partial$ has displacement $\Phi$, and the second inclusion follows from the inductive hypothesis. Maximality of $\bC_{\bullet}(\Phi^k(i),N_{2^{k}\Phi^k(i)}(Y))$ implies that $\bD_k(i,Y)\subseteq \bC_{k}(\Phi^k(i),N_{2^{k}\Phi^k(i)}(Y))$ as required.
\end{proof}

\begin{defn}\label{defn:weak_uniform_acyc}
	Let $(\bC_\bullet, \epsilon)$ be an augmented $R$-metric complex over $X$. We say $\bC_\bullet$ is \emph{weakly uniformly $n$-acyclic} if there is a function $\Omega:\bbN\times \bbN\to \bbN$ such that \begin{enumerate}
		\item for all $i,r\in \bbN$, $k\leq n$ and $x\in X$, the map \[\widetilde{H}_k(\bC_\bullet(i,N_r(x)))\to \widetilde{H}_k(\bC_\bullet(\Omega(i,r),N_{\Omega(i,r)}(x)))\] induced by inclusion is zero;
		\item for each $x\in X$, there is an $m_x\in \bC_0(\Omega_0)$ with $\epsilon(m_x)=1$ and $\supp(m_x)\subseteq N_{\Omega_0}(x)$, where $\Omega_0\coloneqq \Omega(0,0)$.
	\end{enumerate}
	When the above hold, we say $\bC_\bullet$ is \emph{$\Omega$-weakly uniformly $n$-acyclic}. We say $\bC_\bullet$ is \emph{weakly uniformly acyclic} (resp.\ \emph{$\Omega$-weakly uniformly acyclic}) if $\bC_\bullet$ is \emph{weakly uniformly $n$-acyclic} (resp.\ \emph{$\Omega$-weakly uniformly $n$-acyclic}) for all $n$.
\end{defn}
\begin{rem}
	Weak uniform acyclicity should not be confused with the more restrictive condition of uniform acyclicity considered in Section~\ref{sec:unif_acyc}. We refer to Remark~\ref{rem:unif_acyc} for an explanation of the difference between these terms.
\end{rem}
We can reformulate Definition~\ref{defn:weak_uniform_acyc} as follows:
\begin{prop}\label{prop:equiv_weakuniformacyc}
	Let $(\bC_\bullet, \epsilon)$ be an augmented $R$-metric complex over $X$. The following are equivalent:
	\begin{enumerate}
		\item\label{item:equiv_weakuniformacyc1} $(\bC_\bullet, \epsilon)$ is weakly uniformly acyclic.
		\item\label{item:equiv_weakuniformacyc2} The following hold:
		      \begin{enumerate}
			      \item The augmented complex \[\dots \to\bC_2\to \bC_1\to \bC_0\xrightarrow{\epsilon} R\to 0\] is exact.
			      \item Each boundary map $\partial: \bC_n\to \bC_{n-1}$ has uniform preimages.
			      \item There is a constant $D$ such that for each $x\in X$, there is an $m_x\in \bC_0(D)$ with $\epsilon(m_x)=1$ and $\supp(m_x)\subseteq N_D(x)$.\label{item:equiv_weakuniformacyc_2c}
		      \end{enumerate}
	\end{enumerate}
\end{prop}
\begin{proof}
	Since $\bC_\bullet$ is staggered, Proposition~\ref{prop:uniform_displacement} ensures that if each boundary map $\partial: \bC_n\to \bC_{n-1}$ has uniform preimages, then there is an $\Omega$ such that each $\partial$ has $\Omega$-uniform preimages. The equivalence of the two conditions now follows easily by unpacking the definitions and applying Lemma~\ref{lem:standard}.
\end{proof}
\begin{rem}\label{rem:equiv_weakuniformacyc}
	A more  quantitive version of Proposition~\ref{prop:equiv_weakuniformacyc} holds:
	Suppose $(\bC_\bullet, \epsilon)$ is $\Omega$-weakly acyclic and each $\partial$ has displacement $\Phi$. Then each $\partial$ has $\Omega'$-uniform preimages, where $\Omega'=\Omega'(\Omega,\Phi)$, and the constant $D$ in (\ref{item:equiv_weakuniformacyc_2c}) depends only on $\Omega$ and $\Phi$, and vice versa for the (\ref{item:equiv_weakuniformacyc2}) implies (\ref{item:equiv_weakuniformacyc1}) direction.
\end{rem}

We are now ready to  formulate our key definition:
\begin{defn}\label{defn:proj_res}
	Let $X$ be a metric space. A \emph{projective (resp.\ free) $R$-resolution over $X$} is an augmented $R$-chain complex $(\bC_\bullet,\epsilon)$ over $X$  such that:
	\begin{enumerate}
		\item Each $\bC_i$ is a projective (resp.\ free) $R$-module over $X$.

		\item $(\bC_\bullet,\epsilon)$ is weakly uniformly acyclic.\label{defnitem_projres3}
	\end{enumerate}
\end{defn}
We note that every free $R$-resolution over $X$ is a projective $R$-resolution over $X$, but not conversely.

\begin{rem}
	A more accurate but verbose term for a projective $R$-resolution over $X$ might be \emph{a  resolution of $R$ by projective $R$-modules over $X$}. Since we do not consider projective resolutions of $R$-modules other than $R$, there is no harm in using the terminology projective $R$-resolution over $X$.
\end{rem}

As usual, when the  augmentation does not need to be explicitly stated, we frequently omit $\epsilon$ from the notation and simply say call the complex $\bC_\bullet$ a projective $R$-resolution over $X$.
Given an increasing function $\Phi:\bbN\to \bbN$, we say that $\bC_\bullet$ has \emph{displacement $\Phi$} if the following hold:
\begin{itemize}
	\item  each $\bC_i$ is a projective $R$-module of displacement $\Phi$;
	\item  each boundary map $\partial:\bC_i\to \bC_{i-1}$ has displacement $\Phi$.
\end{itemize}
It follows from  Propositions~\ref{prop:uniform_displacement} and~\ref{prop:uniform_displacement_projective}  that each $R$-resolution over $X$ has displacement $\Phi$ for some $\Phi$.

\begin{exmp}\label{exmp:metric_complex}
	In~\cite{kapovich2005coarse}, Kapovich--Kleiner define a class of spaces called \emph{metric simplicial complexes}. Recall that a metric simplicial complex  consists of the pair $(X^{(1)},X)$, where $X$ is  the geometric realisation of connected locally finite simplicial complex $X$, and where $X^{(1)}$ is the 1-skeleton of $X$ equipped with the induced path metric where edges have length one. If $(X^{(1)},X)$ is uniformly acyclic (defined in~\cite{kapovich2005coarse}), then the simplicial chain complex $C_\bullet(X)$ with coefficients in $R$ can be given the structure of a free $R$-resolution over $X^{(1)}$.
	Specifically, if $(X^{(1)},X)$ is a metric simplicial complex, then we define an $R$-chain complex over $X^{(1)}$ with \[\bC_k(X)=(C_k(X),\Sigma_k,\delta^k,p_k,\{\Sigma_k(i)\}),\] where:\begin{enumerate}
		\item $\Sigma_k$ is the set of $k$-simplices in $X$ and $C_k(X)$ is the free $R$-module with basis $\Sigma_k$;
		\item $\{\delta^k_\sigma\}_{\sigma\in \Sigma_k}$ is the dual basis to $\Sigma_k$;
		\item $p:\Sigma_k\to X^{(1)}$ is a function such that  $p_k(\sigma)\in \sigma$ for each simplex $\sigma\in \Sigma_k$;
		\item $\Sigma_k(i)=\begin{cases}
				      \Sigma_k  & \textrm{if}\;i\geq k, \\
				      \emptyset & \textrm{if}\;i<k.
			      \end{cases}$
	\end{enumerate}
	The requirement that $\Sigma_k(i)=\emptyset$ if $i<k$ is to ensure that $\{\bC_k(X)\}_k$ is staggered.

	More generally, Dru\c{t}u--Kapovich define \emph{metric cell complexes}~\cite{drutu2018geometric}. It can be shown similarly that the cellular chain complex of a bounded geometry uniformly acyclic metric cell complex can be endowed with the structure of free $R$-resolution over its 1-skeleton. Even more generally, if $(X,\bC_\bullet)$ is a  \emph{uniformly acyclic metric complex},  as defined in the appendix of~\cite{kapovich2005coarse}, then $\bC_\bullet$ is a free $R$-resolution over $X$.
\end{exmp}

\begin{prop}\label{prop:proj_res}
	Let $X$ and $Y$ be  metric spaces and let $f:X\to Y$ be $\Upsilon$-bornologous. Suppose $\bC_\bullet$ and $\bD_\bullet$ are  projective  $R$-resolutions over $X$ and $Y$ respectively, that  are  of displacement $\Phi$ and are  $\Omega$-weakly uniformly acyclic.
	\begin{enumerate}
		\item There exists an augmentation-preserving  chain map $f_\#:\bC_\bullet\to \bD_\bullet$ of displacement $\Phi'$ over $f$, where $\Phi'$ depends only on $\Phi$, $\Omega$ and $\Upsilon$.\label{item:proj_res1}
		\item Suppose  there exist  augmentation-preserving  chain maps $f_\#,g_\#:\bC_\bullet\to \bD_\bullet$  of displacement $\Psi$ over $f$. Then there exists a chain homotopy $f_\#\stackrel{h_\#}{\simeq}g_\#$ of displacement $\Psi'$ over $f$, where $\Psi'$ depends  only on $\Phi$, $\Psi$, $\Omega$ and $\Upsilon$.\label{item:proj_res2}
	\end{enumerate}
\end{prop}

We will prove Proposition~\ref{prop:proj_res} by induction.
The base case is stated separately, as it does not require the assumption that $f$ is bornologous:
\begin{lem}\label{lem:proj_res_basecase}
	Let $X$ and $Y$ be  metric spaces, and let $f:X\to Y$ be an arbitrary map. Suppose $\bC_\bullet$, $\bD_\bullet$, $\Phi$ and $\Omega$ are as in Proposition~\ref{prop:proj_res}.
	Then there exists an augmentation-preserving map $f_0:\bC_0\to \bD_0$ of displacement $\Phi'$ over $f$, where $\Phi'$ depends only on $\Phi$ and $\Omega$.
\end{lem}
\begin{proof}
	Set $\Omega_0\coloneqq \Omega(0,0)$.
	For each $y\in Y$, there is some $n_y\in \bD_0(\Omega_0)$ with $\supp_{\bD_0}(n_y)\subseteq N_{\Omega_0}(y)$ and $\epsilon(n_y)=1$.
	We suppose \begin{align*}
		\bC_0=(C_0,B,\delta,p,\cF)
	\end{align*}
	has projective basis $\{m_b\}$ of displacement $\Phi$.
	For each $b\in B$, set $q_b=\epsilon(m_b) n_{f(p(b))}$, noting that $q_b\in \bN(\Omega_0)$ and that $\supp_\bN(q_b)\subseteq N_{\Omega_0}(f(p(b)))$.  By Lemma~\ref{lem:define_fin_disp}, $m\mapsto \sum_{b\in B}\delta_b(m)q_b$ defines  a  map $f_0:\bC_0\to \bD_0$ of  displacement $\Omega_0$ over $f$, which is clearly augmentation-preserving by the choice of $\{q_b\}$.
\end{proof}

\begin{proof}[Proof of Proposition~\ref{prop:proj_res}]
	(\ref{item:proj_res1}): We define $f_k:\bC_k\to \bD_k$ by induction on $k$, where the base case $k=0$ follows from Lemma~\ref{lem:proj_res_basecase}.  If $f_i:\bC_i\to \bD_i$ has been defined for $i\leq k$, we define $f_{k+1}:\bC_{k+1}\to \bD_{k+1}$ as follows. Since $f_k$ has finite displacement over $f$ and $\partial:\bC_{k+1}\rightarrow\bC_k$ has finite displacement over $\id_X$, Lemma~\ref{lem:compos_findisp} implies $f_k\partial:\bC_{k+1}\to \bD_k$ has finite displacement over $f$.
	Exactness of $\bD_\bullet$ implies  $\im(f_k\partial)\subseteq \im(\bD_{k+1}\xrightarrow{\partial}\bD_k)$, where in the case $k=0$, we use the fact that $f_0$ is augmentation-preserving. As $\partial:\bD_{k+1}\to \bD_k$ has uniform preimages by Proposition~\ref{prop:equiv_weakuniformacyc}, we can apply Proposition~\ref{prop:projective} to define a map $f_{k+1}:\bC_{k+1}\to \bD_{k+1}$ of finite displacement over $f$,  with $\partial f_{k+1}=f_k\partial$.

	It follows inductively from Lemmas~\ref{lem:define_fin_disp} and~\ref{lem:compos_findisp}, Proposition~\ref{prop:projective}, and Remark~\ref{rem:equiv_weakuniformacyc} that the displacement of each $f_k$ depends only on  $\Upsilon$, $\Phi$ and $\Omega$. Thus Proposition~\ref{prop:uniform_displacement} and Remark~\ref{rem:uniformdisplacement} imply $f_\#$ has displacement $\Phi$, depending only on $\Upsilon$, $\Phi$ and $\Omega$.

	(\ref{item:proj_res2}):  Since $\epsilon\circ (f_0-g_0)=\epsilon-\epsilon=0$ on $\bC_0$, $\im(f_0-g_0)\subseteq \im(\bD_1\xrightarrow{\partial}\bD_0)$. We can thus use Proposition~\ref{prop:projective} as above to define a  map $h_0:\bC_0\to \bD_1$ of finite displacement over $f$ with $\partial h_0=f_0-g_0$. We now have $\partial (f_1-g_1-h_0\partial)=f_0\partial-g_0\partial -\partial h_0\partial=0$. Exactness of $\bD_2\to \bD_1\to \bD_0$ means $\im(f_1-g_1-h_0\partial)\subseteq \im(\bD_2\xrightarrow{\partial}\bD_1)$. We can thus apply Proposition~\ref{prop:projective} to define a map $h_1:\bC_1\to \bD_2$ such that $f_1-g_1=h_0\partial+\partial h_1$ of finite displacement over $f$. We continue in this way, inductively defining a chain homotopy $h_\#:\bC_{\bullet}\to \bD_{\bullet+1}$ between $f_\#$ and $g_\#$. A similar argument to that used in (\ref{item:proj_res1}) shows that $h_\#$ has displacement $\Psi'$, where $\Psi'$  depends  only on $\Phi$, $\Psi$, $\Omega$ and $\Upsilon$.
\end{proof}
Proposition~\ref{prop:proj_res} readily  implies that for any metric space $X$, a  projective  $R$-resolution over $X$ is unique up to a finite displacement chain homotopy.

\begin{cor}\label{cor:projres_indpt}
	Let $X$ be a metric space and let $\bC_\bullet$ and $\bD_\bullet$ be projective $R$-resolutions over $X$. Then there exists an augmentation-preserving finite displacement chain homotopy equivalence $f_\#:\bC_\bullet\to \bD_\bullet$, i.e.\ there exist augmentation-preserving finite displacement chain maps $f_\#:\bC_\bullet\to \bD_\bullet$ and $g_\#:\bD_\bullet\to \bC_\bullet$, and finite displacement chain homotopies  $\id\simeq g_\#f_\#$ and $\id\simeq f_\#g_\#$.
\end{cor}
\begin{proof}
	The existence of $f_\#$ and $g_\#$ follows by applying Proposition~\ref{prop:proj_res} (\ref{item:proj_res1}) twice to the identity $\id_X$, whilst the existence  of chain homotopies follows by applying Proposition~\ref{prop:proj_res} (\ref{item:proj_res2}) twice.
\end{proof}

We now give a criterion that allows one to truncate projective resolutions to obtain another projective resolution:
\begin{lem}\label{lem:retract_truncate}
	Let $\bC_\bullet$ be a projective $R$-resolution over a metric space $X$. If there exists a finite displacement retraction from $\bC_n$ onto $\im(\bC_{n+1}\to \bC_{n})$, then $\bC_n$ contains a projective submodule $\bP$ such that \[0 \to \bP\xrightarrow{\partial|_\bP} \bC_{n-1}\xrightarrow{\partial} \bC_{n-2}\to \cdots \] is a projective $R$-resolution over $X$.
\end{lem}
\begin{proof}
	Set $\bN\coloneqq \im(\bC_{n+1}\to \bC_{n})$. By Lemma~\ref{lem:comp} the  finite displacement retraction $\bC_n\to \bN$ implies the existence of a direct sum $\bC_n=\bN\oplus \bP$, where both $\bN$ and $\bP$ are projective $R$-modules over $X$, and both  projections  $\bC_n\to \bN$ and $\bC_n\to \bP$ of finite displacement. It follows that $0\to \bP\xrightarrow{\partial|_\bP} \bC_{n-1}\to \bC_{n-2}$ is exact.

	We pick $\Omega$ such that $\bC_n\to \bC_{n-1}$ has $\Omega$-uniform preimages. If $\sigma\in \im(\bC_n \to \bC_{n-1})$ with $\sigma\in \bC_{n-1}(i)$ and $\diam(\supp(\sigma))\leq D$, then there exists  $\omega\in \bC_n(j)$ with $\partial \omega=\sigma$ and $\supp(\omega)\subseteq N_j(\supp(\sigma))$, where $j=\Omega(i,D)$. If the projection $\pi:\bC_n\to \bP$ has displacement $\Psi$, then $\pi(\omega)\in \bP(\Psi(j))$ and $\supp(\pi(\omega))\subseteq N_{j+\Psi(j)}(\supp(\sigma))$ with $\partial \pi(\omega)=\sigma$.  It follows that $\bP\xrightarrow{\partial|_\bP} \bC_{n-1}$ has uniform preimages. It thus follows from Proposition~\ref{prop:equiv_weakuniformacyc} that \[0 \to \bP\xrightarrow{\partial|_\bP} \bC_{n-1}\xrightarrow{\partial} \bC_{n-2}\to \cdots \] is a projective $R$-resolution over $X$.
\end{proof}
Given a coarse embedding $f:X\to Y$ and an  $R$-projective resolution $\bC_\bullet$ over $Y$, we can consider the \emph{pullback chain complex} $f^*\bC_\bullet$, where the boundary maps are the pullbacks of the boundary maps of $\bC_\bullet$, as in Proposition~\ref{prop:induced_mod_maps}.  There is a \emph{canonical identification}  $\iota_\#:f^*\bC_\bullet\to \bC_\bullet$ as in Section~\ref{sec:rmod_metric}, which is just the identity map on underlying $R$-modules, and is clearly a chain map. We equip $f^*\bC_\bullet$ with an augmentation $\epsilon$ so that $\iota_\#$ is augmentation-preserving.

\begin{prop}\label{prop:pullback_res}
	Let $f:X\to Y$ be a coarse embedding and let $\bC_\bullet$ be a  projective $R$-resolution over $Y$. Then the pullback complex $f^*\bC_\bullet$ is a projective $R$-resolution over $X$.
\end{prop}
\begin{proof}
	Proposition~\ref{prop:induced_mod_maps} ensures each boundary map of $f^*\bC_\bullet$ has finite displacement and uniform preimages. Moreover, \[\dots \to f^*\bC_2\to f^*\bC_1\to f^*\bC_0\xrightarrow{\epsilon} R\to 0\] is exact since $f^*\bC_\bullet$ is isomorphic to $\bC_\bullet$ on the level of underlying $R$-modules. Each $f^*\bC_k$ is a projective $R$-module over $X$ by Proposition~\ref{prop:projective_operations_second}. The definition of the pullback of $\bM=(M,B,\delta,p,\{B(i)\})$ ensures that if $B(i)=\emptyset$, then $f^*B(i)=\emptyset$. This ensures that $f^*\bC_\bullet$ is staggered.

	All that remains is to show that condition (\ref{item:equiv_weakuniformacyc_2c}) of Proposition~\ref{prop:equiv_weakuniformacyc} holds, for then we can apply Proposition~\ref{prop:equiv_weakuniformacyc} to deduce  $f^*\bC_\bullet$ is weakly uniformly acyclic. To see this, suppose $f$ is a $(\Lambda, \Upsilon)$-coarse embedding. Since $\bC_\bullet$ is an $R$-projective resolution over $Y$, there exists  $D$ such that for each $y\in Y$, there exists $m_y\in \bC_0(D)$ with $\epsilon(m_y)=1$ and $\supp_{\bC_0}(m_y)\subseteq N_D(y)$. Let $\iota:f^*\bC_0\to \bC_0$ be the canonical identification, and let $n_x\coloneqq \iota^{-1}(m_{f(x)})$ for each $x\in X$. Then for each $x\in X$, we have $\iota(n_x)=m_{f(x)}\in \bC_0(D)$ and $\supp_{\bC_0}(m_{f(x)})\subseteq N_D(f(x))\subseteq N_D(f(X))$. Thus Lemma~\ref{lem:induced_mod} implies $n_x\in f^*\bC_0(D)$. Lemma~\ref{lem:induced_mod}  also implies  $f(\supp_{f^*\bC_0}(n_x))\subseteq N_D(\supp_{\bC_0}(m_{f(x)}))\subseteq N_{2D}(f(x))$, and so $\supp_{f^*\bC_0}(n_x)\subseteq N_{D'}(x)$, where $D'\coloneqq \widetilde{\Lambda}(2D)$. Hence condition (\ref{item:equiv_weakuniformacyc_2c}) of Proposition~\ref{prop:equiv_weakuniformacyc} holds as required.
\end{proof}

We now prove the existence of a projective $R$-resolution over a metric space.

\begin{prop}\label{prop:resln_exist}
	If $X$ is a metric space, then there exists a free (and hence projective)  $R$-resolution  $\bC_\bullet$ over $X$.  Moreover, $\bC_0$ has finite height and if $X$ is a proper metric space, we may assume $\bC_\bullet$ is proper.
\end{prop}
We prove Proposition~\ref{prop:resln_exist} via the following general construction:
\begin{defn}\label{defn:standard_res}
	Let $X$ be a metric space and let $R$ be a commutative ring. The \emph{standard $R$-resolution} $\bC_\bullet$ of $X$ is defined as follows:
	\begin{enumerate}
		\item For $j\geq 0$,  $C_j$ is the free $R$-module  generated by $B_j$, the set of  all $(j+1)$-tuples $[x_0,\dots, x_j]$ of $X$, and $C_j=0$ for $j<0$. Let $\cB_j=\{(b,\delta_{b})\}_{b\in B_j}$ be the corresponding projective basis of $C_j$.
		\item The boundary map $\partial:C_k\to C_{k-1}$ is given by  \[{[x_0,\dots, x_k]\mapsto \sum_{i=0}^k(-1)^i[x_0,\dots,\widehat{x_i},\dots,x_k],}\] and  as usual, $[x_0,\dots,\widehat{x}_i,\dots,x_k]$ denotes the $k$-tuple with $x_i$ removed.
		\item We define $p_j:B_j\to R$ by $[x_0,\dots,x_j]\mapsto x_0$.
		\item We define a filtration $\cF_j=\{B_j(i)\}_{i\in \bbN}$ by \[B_j(i)\coloneqq \begin{cases}
				      \{[x_0,\dots,x_j]\in B_j\mid d(x_k,x_l)\leq i \;\text{for all $0\leq k,l\leq j$}\} & \textrm{if}\;j\leq i, \\
				      \emptyset                                                                          & \textrm{if}\;j>i.
			      \end{cases}\]
		\item Let $\bC_j=(C_j,B_j,\delta,p,\cF_j,
			      \{b\}_{b\in B_j})$ be the projective $R$-module over $X$.
		\item We equip $\bC_\bullet$ with the augmentation $\epsilon:\bC_0\to R$ given by $\epsilon(b)=1$ for each $b\in B_0$.
	\end{enumerate}
\end{defn}

\begin{prop}\label{prop:standard_proj_res}
	The standard $R$-resolution $\bC_\bullet$ of a metric space $X$ is a free $R$-resolution over $X$.
\end{prop}
\begin{proof}
	We first verify each $\bC_j$ is actually a free $R$-module over $X$. Indeed, since  for each $b\in B_j(i)$, we have $m_b\in \bC_j(i)$ and $\supp_{\bC_j}(m_b)=\{p(b)\}$, we see that $\bC_j$ has displacement $\id_\bbN$ over $X$. We also observe that $\{\bC_j\}_j$ is staggered. We now note that if $b=[x_0,\dots,x_j]\in B_j(i)$, then $\supp_{\bC_{j-1}}(\partial b)=\{x_0,x_1\}\subseteq N_i(p(b))$. Thus the boundary maps have finite displacement over $\id_X$.

	For each $x\in X$, we define the map $h^x_\#:\bC_\bullet\to \bC_\bullet$ given by $h^x_\#([x_0,\dots,x_j])=[x,x_0,\dots,x_j]$. This map does \emph{not} have finite displacement unless  $X$ is bounded. It is easy to verify that $\partial h^x_\#\sigma+h^x_\#\partial \sigma=\sigma$ for all $\sigma\in \bC_j$ and $j>0$, and this  also holds when $j=0$ if we also assume  $\sigma\in \ker(\epsilon)$. This verifies that \[\cdots\to \bC_2\to\bC_1\to \bC_0\xrightarrow{\epsilon} R\to 0\] is exact. Moreover, suppose $\sigma\in \bC_j(i)$ is a reduced cycle with $x\in\supp(\sigma)$ and $D=\lceil\diam(\supp(\sigma))\rceil$. If $b=[x_0,\dots,x_j]\in B_j(i)$ with $\delta_b(\sigma)\neq 0$, then $x_0\in \supp(\sigma)$ and so $d(x,x_k)\leq D+i$ for all $0\leq k\leq j$. It follows that $h_\#^x(\sigma)\in \bC_{j+1}(i')$, where $i'=\max(j+1,i+D)$,  $\partial h_\#^x(\sigma)=\sigma$ and $\supp(h_\#^x(\sigma))=x\subseteq \supp(\sigma)$. This shows that each boundary map has uniform preimages.

	Finally, for each $x\in X$, we have $x\in B_0(0)$ and $\epsilon(x)=1$. By Proposition~\ref{prop:equiv_weakuniformacyc}, we deduce $\bC_\bullet$ is a projective $R$-resolution over $X$.
\end{proof}
\begin{proof}[Proof of Proposition~\ref{prop:resln_exist}]
	The first part of Proposition~\ref{prop:resln_exist} follows directly from Proposition~\ref{prop:standard_proj_res}. Moreover, since $B_0(0)=B_0$ in Definition~\ref{defn:standard_res}, for every standard resolution  $\bC_\bullet$, $\bC_0$ has finite height.

	Suppose $X$ is a proper metric space. Let $Y\subseteq X$ be a 1-net in $X$, i.e.\ a maximal subset $Y\subseteq X$ such that $d(y,y')\geq 1$ for all distinct $y,y'\in Y$.  Since $X$ is proper and is  covered by the collection of radius 1 open balls centred at a point of $Y$, it follows that   $\{y\in Y\mid d(x,y)\leq r\}$ is finite for each $x\in X$ and $r\in \bbR$. Thus the standard resolution $\bC_\bullet$ of $Y$ is proper. Letting  $f:X\to Y$ be a closest point projection, which is a coarse equivalence, it follows  from Propositions~\ref{prop:finiteness_operations_first} and~\ref{prop:pullback_res} that $f^*\bC_\bullet$  is a proper projective $R$-resolution over $X$.
	Moreover, Proposition~\ref{prop:finiteness_operations_first} ensures the pullback of $f^*\bC_0$  also has finite height.
\end{proof}
The above proof yields the following refinement of Proposition~\ref{prop:resln_exist} in the case $X$ is proper:
\begin{cor}\label{cor:standard_proper}
	If $X$ is a proper metric space, then there exists  a subspace $Y\subseteq X$ such $N_A(Y)=X$ for some $A\geq 0$, and the standard $R$-resolution of $Y$ is a proper free $R$-resolution over $X$.
\end{cor}

\begin{rem}\label{rem:subcomplex_standard}
	If $\bC_\bullet$ is the standard resolution of a metric space $X$, then for each $i\in \bbN$, $\bC_\bullet(i)$ is a subcomplex of finite-height free $R$-modules over $X$, since each $\bC_j(i)$ is the free summand of $\bC_j$ generated by $B_j(i)\subseteq B_j$. The analogous statement does not necessarily hold for an  arbitrary projective $R$-resolution $\bC_\bullet$ over $X$.
	Moreover, every submodule of the form $\bC_k(i,Y)$ is a free summand of  $\bC_k$, since it is a standard submodule of $\bC_k$.
\end{rem}

\begin{cor}\label{cor:projres_ginv}
	Let $G$ be a countable group equipped with a proper left-invariant metric. Then the standard $R$-resolution is a proper $R$-resolution $\bC_\bullet$ over $G$ such that  each $\bC_k$ is a  $G$-induced projective $R$-module over $X$, and all boundary maps are $G$-equivariant.
\end{cor}
\begin{proof}
	This follows immediately from the proof of Proposition~\ref{prop:resln_exist}, since the standard resolution of $G$ is constructed $G$-equivariantly.
\end{proof}

A related construction is the following:
\begin{defn}\label{defn:ordered_standard}
	Let $X$ be a metric space, and equip $X$ with an arbitrary total order $<$. Let $\bC_\bullet$ be the standard resolution of $X$. The \emph{ordered standard resolution} of $X$ is the subcomplex $\bD_\bullet$ of $\bC_\bullet$ in which each $\bD_k$ is the geometric submodule of $\bC_\bullet$ generated by $(k+1)$-tuples $(x_0,\dots,x_k)$
	with $x_0<x_1<\cdots<x_k$.
\end{defn}
An argument similar to the proof of Proposition~\ref{prop:standard_proj_res} yields the following:
\begin{lem}
	Let $X$ be a metric space, and let $\bD_\bullet$ be an ordered standard resolution of $X$. Then $\bD_\bullet$ is a projective $R$-resolution over $X$.
\end{lem}
One advantage of working with the (unordered) standard resolution is that if a group acts freely and isometrically on $X$, then it acts freely on the standard resolution $\bC_\bullet$ of $X$. However, ordered standard resolutions are useful as they are related to simplicial chain complexes of Rips complexes.
The following is a consequence of Definition~\ref{defn:ordered_standard}:
\begin{lem}\label{lem:ordered_res_rips}
	Let $X$ be a metric space, let $\bD_\bullet$ be the ordered  standard resolution of $X$. For each $j\in \bbN$, the subcomplex $\bD_\bullet(j)$ of $\bD_\bullet$ coincides in dimensions $k\leq j$ with the simplicial chain complex $\bC_\bullet(P_j(X);R)$ of the Rips complex $P_j(X)$  with coefficients in $R$, i.e.\ $\bD_k(j)=\bC_k(P_j(X);R)$ for all $k\leq j$, and the boundary maps coincide.
	Moreover, for each subspace $Y\subseteq X$ and $k,j\in \bbN$ with $k\leq j$, we have \[C_k(P_j(Y);R)\subseteq \bD_k(j,Y)\subseteq C_k(P_j(N_j(Y));R).\]
\end{lem}

We define the \emph{$n$-truncation}  $[\bC_\bullet]_n$  to be the chain complex $\bD_\bullet$ with \[\bD_k=\begin{cases}
		\bC_k & \textrm{if}\;k\leq n, \\
		0     & \textrm{if}\;k>n.
	\end{cases}\] We say $\bC_\bullet$ is \emph{$n$-dimensional} if $\bC_\bullet=[\bC_\bullet]_n$. A \emph{partial projective (resp.\ free) $R$-resolution of dimension $n$}  over $X$ is an augmented $n$-dimensional $R$-chain complex $(\bC_\bullet,\epsilon)$ over $X$ such that:
\begin{enumerate}
	\item Each $\bC_i$ is a projective (resp.\ free) $R$-module over $X$;
	\item  $(\bC_\bullet,\epsilon)$ is weakly uniformly $(n-1)$-acyclic.
\end{enumerate}
We now show that any partial projective resolution can be extended:
\begin{prop}\label{prop:extend_partial}
	Let $X$ be a metric space and let $\bC_\bullet$ be a partial projective (resp.\ free) $R$-resolution of dimensional $n$ over $X$. There exists a projective (resp.\ free) $R$-resolution $\bD_\bullet$ over $X$ with $[\bD_\bullet]_n=[\bC_\bullet]_n$.
\end{prop}
\begin{proof}
	We construct $\bD_k$ and the boundary map $\partial: \bD_k\to \bD_{k-1}$  inductively on $k$, setting  $\bD_k=\bC_k$ for $k\leq n$.
	For the inductive step, we assume that for some $k\geq n$  we have  a partial projective resolution  $\bD_\bullet$ over $X$ of dimension $k$.

	Let $\bE_\bullet$ be an arbitrary free $R$-resolution over $X$, which exists by Proposition~\ref{prop:resln_exist}. Applying Lemma~\ref{lem:proj_res_basecase} and then Proposition~\ref{prop:projective} inductively, we see there exist finite displacement  chain maps $f_\#:[\bD_\bullet]_k\to \bE_\bullet$ and $g_\#:[\bE_\bullet]_k\to \bD_\bullet$, and a finite displacement chain homotopy $h_\#:[\bD_\bullet]_{k-1}\to \bD_{\bullet+1}$ satisfying $\partial h_\#+h_\# \partial=g_\#f_\#-\id$ on $[\bD_\bullet]_{k-1}$.

	We consider  the following chain complex:
	\[\cdots \to 0 \to \bD_k\oplus \bE_{k+1}\xrightarrow{\partial_{k+1}} \bD_k\to \bD_{k-1}\to \dots \to \bD_0\to 0,\]where  $\partial_{k+1}$ is given by  $(d,e)\mapsto d+h_\#\partial d+g_\#\partial e-g_\#f_\#d$. It follows from Proposition~\ref{prop:projective_operations_second} that $\bD_k\oplus \bE_{k+1}$ is a projective $R$-module over $X$, and is a free $R$-module over $X$ when $\bD_k$ is. Since $f_\#$, $g_\#$ and  $h_\#$  have finite displacement, so does $\partial _{k+1}$. We also have $\partial\partial_{k+1}(d,e)=-h_\#\partial\partial d=0$ for all $(d,e)\in\bD_k\oplus \bE_{k+1}$. By Proposition~\ref{prop:equiv_weakuniformacyc}, all that remains is to show $\partial_{k+1}$ has uniform preimages, and that $\bD_{k}\oplus \bE_{k+1}\to \bD_k\to \bD_{k-1}$  is exact.

	To see this, suppose  $f_\#$, $g_\#$ and $h_\#$ have displacement $\Phi$, and that $\bE_\bullet$ has $\Omega$-uniform preimages. Suppose $\sigma\in \bD_k(i)$ is a reduced $k$-cycle with $\supp(\sigma)\subseteq N_r(x)$. Setting $i_1\coloneqq \Phi(i)$ and $r_1\coloneqq r+\Phi(i)$, we see $f_\#\sigma\in \bE_k(i_1)$ with $\supp(f_\#\sigma)\subseteq N_{r_1}(x)$. Setting $i_2\coloneqq \Omega(i_1,2r_1)$ and $r_2\coloneqq r_1+i_2$, we see that there is some $\omega\in \bE_{k+1}(i_2)$ with $\partial \omega=f_\#\sigma$ and $\supp(\omega)\subseteq N_{r_2}(x)$. We have $\partial_{k+1} (\sigma,\omega)=\sigma+h_\#\partial\sigma+g_\#\partial\omega-g_\#f_\#\sigma=\sigma$ with $(\sigma,\omega)\in (\bD_k\oplus\bE_{k+1})(i_2)$ and $\supp_{\bD_k\oplus\bE_{k+1}}(\sigma,\omega)\subseteq N_{r_2}(x)$ by Proposition~\ref{prop:dirsum_overX}, and we are done.
\end{proof}

Recall that if $\bC_\bullet$ and $\bD_\bullet$ are chain complexes, then $\bE_\bullet=\bC_\bullet\otimes \bD_\bullet$ is a chain complex where
\[\bE_k=\bigoplus_{i=0}^k\bC_i\otimes \bD_{k-i}\] and the boundary map $\partial$ is given by \[\partial (\sigma_1\otimes \sigma_2)=\partial^I \sigma_1\otimes \sigma_2+(-1)^i \sigma_1\otimes \partial^{II}\sigma_2\] for all $\sigma_1\in \bC_i$ and $\sigma_2\in \bD_{k-i}$, where $\partial^I$ and $\partial^{II}$ are the boundary maps of $\bC_\bullet$ and $\bD_\bullet$ respectively. If $\bC_\bullet$ and $\bD_\bullet$ are equipped with augmentations $\epsilon_\bC:\bC_0\to R$ and $\epsilon_\bD:\bD_0\to R$, then $\bC_\bullet$ is equipped with the augmentation  $\epsilon_\bC\otimes \epsilon_\bD:\bC_0\otimes \bD_0\to R$ given by $(\epsilon_\bC\otimes \epsilon_\bD)(\sigma\otimes \tau)=\epsilon_\bC(\sigma)\epsilon_\bD(\tau)$.

The following proposition is the starting point for the theory of cup products, cap products, and coarse Poincar\'e duality that will be developed in Sections~\ref{sec:cupcap} and~\ref{sec:coarsePDn}.
\begin{prop}\label{prop:tensor_res}
	Let  $\bC_\bullet$ and $\bD_\bullet$ be projective $R$-resolutions over $X$ and $Y$. Then $\bC_\bullet\otimes_R \bD_\bullet$ is a projective $R$-resolution over $X\times Y$.
\end{prop}
\begin{proof}
	Set $\bE_\bullet=\bC_\bullet\otimes_R \bD_\bullet$. Each $\bE_k$ is given the structure of an $R$-module over $X\times Y$ via Propositions~\ref{prop:dirsum_overX} and~\ref{prop:tensor_overX}. It follows from Propositions~\ref{prop:projective_operations_second} and~\ref{prop:tensor_overX} that each $\bE_k$ is a projective $R$-module over $X\times Y$,  that all boundary maps have finite displacement and that $\bE_\bullet$ is staggered. All that remains is to show that $\bE_\bullet$ is weakly uniformly acyclic. Pick $\Omega$ such that both $\bC_\bullet$ and $\bD_\bullet$ are $\Omega$-weakly uniformly acyclic.

	Let $\bC'_\bullet$ and $\bD'_\bullet$ be standard $R$-resolutions over $X$ and $Y$ respectively. By Corollary~\ref{cor:projres_indpt}, there exist augmentation-preserving finite displacement chain homotopy equivalences $f_\#:\bC_\bullet\to \bC'_\bullet$ and $g_\#:\bD_\bullet\to \bD'_\bullet$. Hence  $(f_\#\otimes g_\#):\bE_\bullet\to \bC'_\bullet\otimes \bD'_\bullet$ is also an augmentation-preserving finite displacement chain homotopy equivalence. Since being weakly uniformly acyclic is invariant under augmentation-preserving finite displacement chain homotopy equivalences, we can assume without loss of generality that both $\bC_\bullet$ and $\bD_\bullet$ are standard $R$-resolutions.

	By Remark~\ref{rem:subcomplex_standard}, for each $i\in \bbN$ and $Z\subseteq X$,  each subcomplex of the form $\bC_\bullet(i,Z)$ is a complex of free $R$-modules, and similarly for $\bD_\bullet$.
	Set $\Omega_0\coloneqq \Omega(0,0)$. Weak uniformly acyclicity of $\bC_\bullet$ and $\bD_\bullet$ ensure that for each $(x,y)\in X\times Y$, there exist $m_x\in \bC_0(\Omega_0)$ and $n_y\in \bD_0(\Omega_0)$ with  $(m_x\otimes n_y)\in \bE_0(\Omega_0)$, $(\epsilon_\bC\otimes \epsilon_\bD)(m_x\otimes n_y)=1$ and $\supp(m_x\otimes n_y)\subseteq N_{\Omega_0}(x)\times N_{\Omega_0}(y)$.

	We first argue in the case $R$ is a PID\@.
	The K\"unneth short exact sequence~\cite[Theorem 10.81]{rotman2009homalg} implies that for all $i,k,r\in \bbN$, $x\in X$ and $y\in Y$, we have a natural short exact sequence
	\begin{align*}
		0 & \to \bigoplus_{p+q=n}\widetilde H_p(\bC_\bullet(i,N_r(x))) \otimes_R \widetilde H_q(\bD_\bullet(i,N_r(y)))  \to \widetilde{H}_n(\bC_\bullet(i,N_r(x))\otimes_R \bD_\bullet(i,N_r(y))) \\
		  & \to \bigoplus_{p+q=n-1} \Tor_1^R(\widetilde H_p(\bC_\bullet(i,N_r(x))),\widetilde H_q(\bD_\bullet(i,N_r(y))))\to 0.
	\end{align*}
	Here, we use the fact that all subcomplexes of the form $\bC_\bullet(i,N_r(x))$ are complexes of free (hence flat) $R$-modules, so the K\"unneth theorem does apply.
	Naturality of this short sequence, coupled with $\Omega$-weak uniform acyclicity of $\bC_\bullet$ and $\bD_\bullet$, ensures that the maps
	\begin{align}
		\widetilde{H}_n(\bC_\bullet(i,N_r(x))\otimes_R \bD_\bullet(i,N_r(y)))\to \widetilde{H}_n(\bC_\bullet(j,N_{j}(x))\otimes_R \bD_\bullet(j,N_{j}(y))), \label{eqn:hom_inclusion}
	\end{align} induced by inclusion, are zero, where $j=\Omega(i,r)$. This implies that $\bC_\bullet\otimes_R \bD_\bullet$ is weakly uniformly acyclic.

	For a general commutative ring $R$ we argue similarly, using the K\"unneth spectral sequence~\cite[Theorem 10.90]{rotman2009homalg} instead of the K\"unneth short exact sequence.
	Recall this spectral sequence is a first quadrant homological  spectral sequence \begin{align*}
		E^2_{p,q}(P_\bullet,Q_\bullet)=\bigoplus_{s+t=1}\Tor_p^R(H_s(P_\bullet) ,H_t(Q_\bullet)) \Rightarrow_p H_n(P_\bullet\otimes_R Q_\bullet)
	\end{align*}
	for chain complexes $P_\bullet$ and $Q_\bullet$ of flat $R$-modules, and this spectral sequence is natural in $P_\bullet$ and $Q_\bullet$.
	Now weak acyclicity of $\bC_\bullet$ and $\bD_\bullet$ implies that the morphism of spectral sequences  \[	E^2_{p,q}(\bC_\bullet(i,N_r(x)),\bD_\bullet(i,N_r(y)))\to E^2_{p,q}(\bC_\bullet(j,N_j(x)),\bD_\bullet(j,N_j(y))), \] induced by inclusion, is zero, where $j=\Omega(i,r)$. The induced analogous morphism between $E^\infty$ terms  of these spectral sequences is also zero, which implies the   maps (\ref{eqn:hom_inclusion}), induced by inclusion, are zero.
\end{proof}

We can use a similar argument to prove a change of rings formula for projective resolutions over a metric space. If $f:R\to S$ is a ring homomorphism and $\bC_\bullet$ is a chain complex of $R$-modules, let $\bC_\bullet\otimes_R S$ denote the chain complex of $S$-modules with boundary maps given by $\partial\otimes \id_S:\bC_{k}\otimes_R S\to \bC_{k-1}\otimes_R S$.
\begin{prop}\label{prop:change_rings_res}
	Let $f:R\to S$ be a homomorphism between commutative rings, and let $\bC_\bullet$ be a projective $R$-resolution over $X$. Then $\bC_\bullet\otimes_R S$ is a projective $S$-resolution  over $X$.
\end{prop}
\begin{proof}
	It follows from Proposition~\ref{prop:change_rings} that $\bC_\bullet\otimes_R S$ is an $S$-chain complex over $X$ of projective $S$-modules over $X$. All that remains  is to show $\bC_\bullet\otimes_R S$ is weakly uniformly acyclic. This argument is virtually identical to the proof of Proposition~\ref{prop:tensor_res}: we first restrict to the case $\bC_\bullet$ is the standard resolution, and then apply the K\"unneth spectral sequence, thinking of $S$ as a chain complex of $R$-modules concentrated in dimension 0.
\end{proof}

\section{Coarse cohomology and homology}\label{sec:coarse_cohom}
\emph{For every metric space $X$ and   $R$-module $\bM$ over $X$, we define the \emph{coarse cohomology} $\coarse^k(X;\bM)$ and prove in Theorem~\ref{thm:functor_cohom} it is functorial with respect to coarse embeddings. We relate $\coarse^k(X;\bM)$ to Roe's coarse cohomology in Proposition~\ref{prop:roe_iso}. For a proper metric space $X$ and commutative ring $R$, we define coarse homology $\hcoarse_k(X;R)$. In Theorem~\ref{thm:univcoeff_ctblehyp}, we prove a version of the K\"unneth and Universal Coefficient Theorems relating coarse cohomology and homology, under the hypotheses of countable generation.
}
\vspace{.3cm}

For the remainder of article, we introduce the following \emph{sign convention} regarding the dual complex  of a chain or cochain complex. If $C_\bullet$ is a chain complex with boundary map $\partial:C_k\to C_{k-1}$ for each $k$, then the cochain complex $\Hom_R(C_\bullet,R)$ has coboundary map $\delta:\Hom_R(C_k,R)\to \Hom_R(C_{k+1},R)$ is  given by $\delta \alpha=(-1)^{k+1}\alpha\circ \partial$ for each $\alpha \in \Hom_R(C_k,R)$. We adopt a similar convention regarding the chain complex $\Hom_R(C^\bullet,R)$ that is the dual of  a cochain complex $C^\bullet$

We can now define  cohomology modules $\coarse^k(X;\bN)$ of a metric space.
If $\bC_\bullet$ is a projective $R$-resolution over $X$ and $\bN$ is an $R$-module over $X$, it follows from Proposition~\ref{prop:induced_hom} that since the boundary maps have finite displacement over $\id_X$, there is an induced  cochain complex $\Homfd(\bC_\bullet,\bN)$. We can now define coarse cohomology modules:
\begin{defn}
	Let $X$ be a metric space and let $\bN$ be an $R$-module over $X$. We define $\coarse^k(X;\bN)\coloneqq H^k(\Homfd(\bC_\bullet,\bN))$, where $\bC_\bullet$ is a projective $R$-resolution over $X$.
\end{defn}
It follows from Corollary~\ref{cor:projres_indpt} and Proposition~\ref{prop:induced_hom} that $\coarse^k(X;\bN)$ is well-defined, independent of the choice of  projective $R$-resolution over $X$. More generally, we have the following:

\begin{thm}\label{thm:functor_cohom}
	For any coarse embedding $f:X\to Y$ and $R$-module $\bN$ over $X$, there is an induced map $f^*:\coarse^k(Y;f_*\bN)\to \coarse^k(X;\bN)$ for every $k\in \bbN$, satisfying the following:
	\begin{enumerate}
		\item If $g:Y\to Z$ is a coarse embedding,  then  $(gf)^*=f^*g^*$.\label{propitem:functor_cohom1}
		\item If $Y=X$ and $f=\id_X$, then $f^*$ is the identity.\label{propitem:functor_cohom2}
		\item If $g:X\to Y$ is close to $f$, then $\coarse^k(Y;f_*\bN)=\coarse^k(Y;g_*\bN)$ and $f^*=g^*$.\label{propitem:functor_cohom3}
	\end{enumerate}
\end{thm}
\begin{proof}
	Let $\bC_\bullet$ and $\bD_\bullet$ be projective $R$-resolutions over $X$ and $Y$ respectively. By Proposition~\ref{prop:proj_res}, $f$ induces an augmentation-preserving chain map $f_\#:\bC_\bullet\to \bD_\bullet$ of finite displacement over $f$. By Proposition~\ref{prop:induced_hom}, $f_\#$ induces cochain maps $\Homfd(\bD_\bullet,f_*\bN)\to \Homfd(\bC_\bullet,\bN)$, which in turn induces the required maps 	$f^*:\coarse^k(Y;f_*\bN)\to \coarse^k(X;\bN)$ in cohomology. Proposition~\ref{prop:proj_res} also ensures that such a map $f_\#$ is unique up to a finite displacement chain homotopy. It follows $f^*$ does not depend on the choice of chain map $f_\#$.

	(\ref{propitem:functor_cohom1}): Let $\bE_\bullet$ be a projective $R$-resolution over $Z$ and let $g_\#:\bD_\bullet\to \bE_\bullet$ be a map of finite displacement over $g$.  Since the composition  $g_\#f_\#:\bC_\bullet\to \bE_\bullet$ has finite displacement over $gf$ by Lemma~\ref{lem:compos_findisp}, it follows that $(gf)^*=f^*g^*$ on the level of cohomology.

	(\ref{propitem:functor_cohom2}): This is clear, since the identity map $\bC_\bullet\to \bC_\bullet$ has finite displacement over $\id_X$.

	(\ref{propitem:functor_cohom3}): Although  $f_*\bN$ and $g_*\bN$ are not necessarily equal, Remark~\ref{rem:hom_close} implies  $\Homfd(\bD_\bullet,f_*\bN)=\Homfd(\bD_\bullet,g_*\bN)$. It follows  $\coarse^k(Y;f_*\bN)=\coarse^k(Y;g_*\bN)$. Since any map $f_\#:\bC_\bullet\to \bD_\bullet$ has finite displacement over $f$ if and only if it has finite displacement over $g$, it immediately follows that $f^*=g^*$.
\end{proof}

\begin{defn}
	If $X$ is a metric space and $N$ is an $R$-module, we define $\coarse^k(X;N)=\coarse^k(X;\bN)$, where $\bN$ is a constant $R$-module over $X$  with underlying $R$-module $N$.
\end{defn}
For any coarse embedding $f:X\to Y$, we observe that $f_*\bN$ is a constant $R$-module over $Y$ with underlying $R$-module $N$. It follows that $\coarse^k(Y;f_*\bN)=\coarse^k(Y;N)$. We can thus restate Theorem~\ref{thm:functor_cohom} in this setting:
\begin{cor}\label{cor:functor_cohom_bddsupport}
	For any coarse embedding $f:X\to Y$ and $R$-module $N$, there is an induced map $f^*:\coarse^k(Y;N)\to \coarse^k(X;N)$ for every $k\in \bbN$, satisfying (\ref{propitem:functor_cohom1}) --- (\ref{propitem:functor_cohom3}) of Theorem~\ref{thm:functor_cohom}.
\end{cor}

We now relate cohomology discussed above with Roe's coarse cohomology~\cite{roe1993coarse,roe2003lectures}.
First, we recall some key definitions from~\cite{roe2003lectures}, translated from the setting of abstract coarse spaces to metric spaces. If  $X$ is a metric space and $\Omega\subseteq X^{k+1}$, then:
\begin{itemize}
	\item $\Omega$ is \emph{controlled} if there exists a constant $A$ such that if $[x_0,\dots, x_k]\in \Omega$, then $d(x_i,x_j)\leq A$ for all $0\leq i,j\leq k$;
	\item $\Omega$ is \emph{bounded} if there is a bounded subset $F\subseteq X$ such that if $[x_0,\dots,x_k]\in \Omega$, then $x_j\in F$ for each $0\leq j\leq k$;
	\item $\Omega$ is \emph{cocontrolled} if, for every controlled $\Psi\subseteq X^{k+1}$, $\Omega\cap \Psi$ is bounded.
\end{itemize}
The \emph{coarse complex of $X$ with coefficients in $G$}  is then defined to be the cochain complex $CX^\bullet(X,G)$ where $CX^k(X,G)$ consists of functions $\phi:X^{k+1}\to G$ with cocontrolled support, equipped with the  coboundary map:
\[\delta \phi({[x_0,\dots, x_{k+1}]})=(-1)^{k+1}\sum_{i=0}^{k+1}(-1)^i\phi([x_0,\dots, \widehat{x_i},\dots, x_{k+1}]).\] Roe defines the \emph{coarse cohomology of $X$ with coefficients in $G$} to be \[HX^k(X,G)\coloneqq H^k(CX^\bullet(X,G)).\]

\begin{prop}\label{prop:roe_iso}
	Let $X$ be a metric space and let  $G$ be an abelian group. Then \[\coarse^k(X;G)\cong HX^k(X;G)\] for all $k\in \bbN$.
\end{prop}

\begin{proof}

	Let $\bC_\bullet$ be the standard $\bbZ$-resolution over $X$, and fix $k$.  Recall from Definition~\ref{defn:standard_res} that $\bC_k$ is a free abelian group with basis $B_k=X^{k+1}$. Moreover, if $i\geq k$, then $\bC_k(i)\leq \bC_k$ is the free abelian group generated by $B_k(i)\subseteq B_k$, where $B_k(i)$ consists of all $[x_0,\dots, x_k]$ with $d(x_j,x_l)\leq i$ for all $0\leq j,l\leq k$. In particular, we note that each $B_k(i)\subseteq X^{k+1}$ is controlled, and that every controlled set is contained in $\Omega\subseteq X^{k+1}$ is contained in $B_k(i)$ for $i$ sufficiently large.  Recall also that if $\sigma=[x_0,\dots,x_k]\in B_k$, then $\supp_{\bC_k}(\sigma)=\{x_0\}$.

	Let $\alpha\in \Hom(\bC_k,G)$ and $S(\alpha)\coloneqq\{\sigma\in X^{k+1}\mid \alpha(\sigma)\neq 0\}$. We claim $\alpha\in \Hom_c(\bC_k,G)$ if and only if $S(\alpha)$ is cocontrolled. Indeed, suppose $\alpha\in \Hom_c(\bC_k,G)$ and $\Psi\subseteq X^{k+1}$ is controlled. Then $\Psi\subseteq B_k(i)$ for some $i$. By Definition~\ref{defn:homc}, there is a bounded subset $F$ such that if $\sigma=[x_0,\dots, x_k]\in B_k(i)\subseteq \bC_k(i)$ with $x_0\in  X\setminus F$, then $\alpha(\sigma)=0$. Consequently, if $[x_0,\dots, x_k]\in S(\alpha)\cap \Psi$, then  $x_j\in N_i(F)$ for each $0\leq j\leq k$, showing that $S(\alpha)$ is cocontrolled.

	Conversely, suppose $S(\alpha)$ is cocontrolled and $i\in \bbN$. Then $S(\alpha)\cap B_k(i)$ is bounded. Thus there is a bounded subset $F\subseteq X$ such that if $[x_0,\dots,x_k]\in S(\alpha)\cap B_k(i)$, then $x_0\in F$. This implies if $\sigma\in \bC_k(i)$ with $\supp(\sigma)\subseteq X\setminus F$, then $\alpha(\sigma)=0$. By Definition~\ref{defn:homc}, this ensures $\alpha\in \Hom_c(\bC_k,G)$.

	The universal property of free abelian groups says that the functor $F:\texttt{Set}\to \texttt{Ab}$, taking $B$ to the free abelian group $F(B)$ with basis $B$, is  left-adjoint to the forgetful functor $U:\texttt{Ab}\to \texttt{Set}$. Moreover, for an abelian group $G$,  the natural isomorphism $\Hom_\texttt{Set}(B,U(G))\cong \Hom_\texttt{Ab}(F(B),G)$ is an isomorphism of abelian groups. We have thus shown that this natural isomorphism yields an isomorphism $CX^\bullet(X;G)\cong \Hom_c(\bC_\bullet,G)$ of cochain complexes.
\end{proof}

\textbf{For the remainder of this section, we restrict to the case that $X$ is a proper metric space, and hence assume that $\bC_\bullet$ is a proper projective $R$-resolution over $X$.}
We define $\bC^\bullet=\Hom_R(\bC_\bullet,R)$, equipped with the boundary map $\delta\alpha=(-1)^{k+1}\alpha\circ \partial$ for each $\alpha\in \bC^k$. The subcomplex $\Hom_c(\bC_\bullet,R)$ of $\bC^\bullet$ is denoted $\bC^\bullet_c$.
We define the $R$-chain complex $\bC_\bullet^{\lf}$ as follows. Each $\bC_k^{\lf}$ is defined as in Definition~\ref{defn:lf} and Proposition~\ref{prop:induced_proj_lf}, and we identify $\bC_k$ with a geometric submodule of $\bC_k^{\lf}$ as in Proposition~\ref{prop:induced_proj_lf}. The boundary map $\partial:\bC_k^{\lf}\to \bC_{k-1}^{\lf}$ is induced by the boundary map $\partial:\bC_k\to \bC_{k-1}$ as in Lemma~\ref{lem:lf_induced}. This boundary map is obtained by dualising the coboundary map $\delta:\bC^{k-1}_c\to \bC^k_c$ and  applying the sign convention.

\begin{defn}
	If $X$ is a proper metric space, we define $\hcoarse_k(X;R)\coloneqq H_k(\bC_\bullet^{\lf})$, where $\bC_\bullet$ is a proper projective $R$-resolution over $X$.
\end{defn}

\begin{prop}\label{prop:functor_hom}
	For any coarse embedding $f:X\to Y$ between proper metric spaces, there is an induced map $f_k:\hcoarse_k(X;R)\to \hcoarse_k(Y;R)$ for every $k\in \bbN$, satisfying the following:
	\begin{enumerate}
		\item If $g:Y\to Z$ is a coarse embedding,  then  $(gf)_*=g_*f_*$.\label{propitem:functor_hom1}
		\item If $Y=X$ and $f=\id_X$, then $f_*$ is the identity map.\label{propitem:functor_hom2}
		\item If $g:X\to Y$ is close to $f$, then $f_*=g_*$.\label{propitem:functor_hom3}
	\end{enumerate}
\end{prop}
\begin{proof}
	Let $\bC_\bullet$ and $\bD_\bullet$ be projective $R$-resolutions over $X$ and $Y$ respectively. By Proposition~\ref{prop:proj_res}, $f$ induces an augmentation-preserving map $f_\#:\bC_\bullet\to \bD_\bullet$ of finite displacement over $f$. By Lemma~\ref{lem:lf_induced}, $f_\#:\bC_\bullet\to \bD_\bullet$ extends to $f_\#^{\lf}:\bC_\bullet^{\lf}\to \bD_\bullet^{\lf}$, and this induces maps  $f_*:H_*(\bC_\bullet^{\lf})\to H_*(\bD_\bullet^{\lf})$ as required. Since $f_\#$ is obtained by dualising the map $f^\#:\Hom_c(\bD_\bullet,R)\to \Hom_c(\bC_\bullet,R)$, we can argue as in the proof of Theorem~\ref{thm:functor_cohom} to deduce           the required properties.
\end{proof}

If $X$ is a proper metric space and $\bC_\bullet$ is a proper projective $R$-resolution over $X$, then for each $k$, there is a pairing $H_k(\bC_\bullet^{\lf})\times H^k(\bC^\bullet_c)\to R$ given by $[\tau]\times [\alpha]=\tau(\alpha)$.
\begin{prop}\label{prop:pairing_indep}
	Let $X$ be a proper metric space. There is a well-defined pairing
	\[\hcoarse_*(X;R)\times \coarse^*(X;R)\to R\] given by the natural pairing $H_k(\bC_\bullet^{\lf})\times H^k(\bC^\bullet_c)\to R$ for some proper projective $R$-resolution $\bC_\bullet$ over $X$.
\end{prop}
\begin{proof}
	We need to show the pairing doesn't depend on the choice of $\bC_\bullet$. Indeed, suppose $\bC_\bullet$ and $\bD_\bullet$ are two proper projective $R$-resolutions over $X$. By Proposition~\ref{prop:proj_res}, there are finite displacement (over $\id_X$) chain maps $f_\#:\bC_\bullet\to \bD_\bullet$ and $g_\#:\bD_\bullet\to \bC_\bullet$, unique up to finite displacement chain homotopy. Moreover, there is a finite displacement chain homotopy $\id \stackrel{h_\#}{\simeq} g_\#f_\#$. Then $f_\#, g_\#$ and $h_\#$ induce maps  $f^\#$, $g^\#$, $f_\#^{\lf}$, $g_\#^{\lf}$, and $h_\#^{\lf}$ as in Proposition~\ref{prop:dual_fd} and Lemma~\ref{lem:lf_induced}, and these induce isomorphisms $f^*=(g^*)^{-1}$ and $f_*^{\lf}=(g_*^{\lf})^{-1}$ on homology and cohomology.

	Thus if $\tau\in \bC_n^{\lf}$ and $\alpha\in \bC^{n}_c$, we have
	\begin{align*}
		(f_\#^{\lf}\tau)(g^\#\alpha)=(g_\#^{\lf}f_\#^{\lf}\tau)(\alpha)=(\tau+h_\#\partial\tau+\partial h_\#\tau)(\alpha)=\tau(\alpha).
	\end{align*}
	This implies the pairings $H_k(\bC_\bullet^{\lf})\times H^k(\bC^\bullet_c)\to R$ and $H_k(\bD_\bullet^{\lf})\times H^k(\bD^\bullet_c)\to R$ coincide  when we identify $H_k(\bC_\bullet^{\lf})$ with $H_k(\bD_\bullet^{\lf})$ via $f_*^{\lf}$, and $H^k(\bC^\bullet_c)$ with $H^k(\bD^\bullet_c)$ via $g^*$.
\end{proof}

We now discuss versions of the K\"unneth and Universal Coefficient Theorem for coarse cohomology. These theorems do not hold in full generality, essentially because neither $\bC_\bullet^{\lf}$ nor $\bC^\bullet_c$ is a chain complex of projective $R$-modules.
We require the following lemma, which uses the language of inverse and direct limits. We refer the reader to Appendix~\ref{sec:inverse_limits} for background on these topics.
\begin{lem}\label{lem:dirinv_lim_space}
	Let $X$ be a proper metric space, and let $\bC_\bullet$ be a proper projective $R$-resolution of $X$, that is also the standard $R$-resolution of some subspace $Y\subseteq X$, which exists by Corollary~\ref{cor:standard_proper}.
	Remark~\ref{rem:subcomplex_standard} implies that for $i\in \bbN$, $\bC_\bullet(i)$ is a geometric subcomplex of $\bC_\bullet$ of finite-height projective $R$-modules over $X$.
	For each $i$, let $\bC^*_c(i)$ be the cochain complex $\Hom_c(\bC_\bullet(i),R)$.
	Then:
	\begin{enumerate}
		\item\label{item:inv_lim_space} $\bC^\bullet_c=\varprojlim \bC^\bullet_c(i)$, where the bonding maps $f_i^j:\bC^\bullet_c(j)\to \bC^\bullet_c(i)$ for $i\leq j$ are restriction maps. Moreover, the inverse system $\{\bC^\bullet_c(i)\}_i$ satisfies the Mittag-Leffler condition, and each $\bC^\bullet_c(i)$ is a chain complex of countably generated proper projective  $R$-modules over $X$.
		\item\label{item:dir_lim_space} $\bC_\bullet^{\lf}=\varinjlim \Hom_R(\bC^\bullet_c(i),R)$, where the bonding maps \[g_i^j:\Hom_R(\bC^\bullet_c(i),R)\to \Hom_R(\bC^\bullet_c(j),R)\] for $i\leq j$ are dual to $f_i^j$.
	\end{enumerate}
\end{lem}
\begin{proof}
	(\ref{item:inv_lim_space}): Proposition~\ref{prop:duality_ht} and Remark~\ref{rem:subcomplex_standard} imply that as each $\bC_k(i)$ is a finite-height proper projective $R$-module over $X$, so is each  $\bC^k_c(i)$. In particular, each $\bC^k_c(i)$ is a countably generated $R$-module.
	Remark~\ref{rem:subcomplex_standard} also says  each  $\bC_k(i)$ is a direct summand of $\bC_k$. Therefore, for each  $i\leq j$, it follows that $\bC_k(i)$ is a direct summand of $\bC_k(j)$. This implies that each restriction map  $f_i^j$ is surjective. Thus  $\{\bC^\bullet_c(i)\}_i$ is an inverse system with surjective bonding maps, hence satisfies the Mittag-Leffler condition.

	Now we have \[\Hom_R(\bC_\bullet,R)=\Hom_R(\varinjlim\bC_\bullet(i),R)\cong \varprojlim\Hom_R(\bC_\bullet(i),R),\]  and the projection $\Hom_R(\bC_\bullet,R)\to \Hom_R(\bC_\bullet(i),R)$ is the restriction map. Definition~\ref{defn:homc} implies that $\phi\in \Hom_R(\bC_\bullet,R)$ has locally bounded support if and only if the restriction of $\phi$ to each geometric submodule of the form $\bC_\bullet(i)$ has locally bounded support. This implies $\varprojlim \bC^\bullet_c(i)= \bC^\bullet_c$.

	(\ref{item:dir_lim_space}): We remark that as $\bC_k^{\lf}$ is an $R$-module over $X$, it has a filtration  $\{\bC_k^{\lf}(i)\}$, and so $\varinjlim \bC_\bullet^{\lf}(i)=\bC_\bullet^{\lf}$. As the geometric submodule $\bC_k(i)\leq \bC_k$ has finite height, we have $\Ann(\bC_k(j))=0$ for  $j$ sufficiently large. This ensures that $\bC_k^{\lf}(i)=\Hom_R(\bC^k_c(i),R)$ for all $k$ and $i$, yielding the desired equality.
\end{proof}
Combined with results from Appendix~\ref{sec:inverse_limits}, we deduce:
\begin{thm}\label{thm:univcoeff_ctblehyp}
	Let $X$ be a proper metric space.  Assume that $R$ is either a field or a countable PID\@, and that $\coarse^k(X;R)$ is a countably generated $R$-module for $k\leq n+1$. Let $\bC_\bullet$ be a proper projective $R$-resolution over $X$.
	\begin{enumerate}
		\item\label{item:univcoeff_ctblehyp1} Let $f:R\to S$ be a ring homomorphism and let $\bD_\bullet\coloneqq \bC_\bullet\otimes_R S$, which is a proper projective $S$-resolution over $X$ by Proposition~\ref{prop:change_rings_res}.  For each $k<n$, we have a natural short exact sequence \[0\to  H^k(\bC^\bullet_c) \otimes_R S\xrightarrow{\phi} H^k(\bD_c^\bullet)\to \Tor_1^R(H^{k+1}(\bC_c^\bullet),S)\to 0\]
		      and an injection $H^n(\bC^\bullet_c) \otimes_R S\xrightarrow{\phi} H^n(\bD_c^\bullet)$.
		      Moreover, the injections $\phi$ are given by $[\alpha]\otimes s\mapsto [\alpha\otimes s]$.
		\item\label{item:univcoeff_ctblehyp2}
		      We have surjections $H_k(\bC_\bullet^{\lf})\xrightarrow{} \Hom_R(H^{k}(\bC_c^\bullet), R)$ for $k\leq n$, and also for $k=n+1$ in the case $R$ is a field. These surjections $\psi$ are given by $\psi([\phi])([\alpha])=\phi(\alpha)$.
	\end{enumerate}
\end{thm}
\begin{proof}
	By Corollary~\ref{cor:projres_indpt},  we can assume without loss of generality that $\bC_\bullet$ satisfies the hypothesis of  Lemma~\ref{lem:dirinv_lim_space}.
	The result  is now an immediate consequence of  Lemma~\ref{lem:dirinv_lim_space} combined  with Corollary~\ref{cor:stable_invsystem} and Theorem~\ref{thm:truncated_kunneth_chain}.
\end{proof}

\section{Coarse cohomological dimension of a metric space}\label{sec:coarse_dim}
\emph{Given a metric space $X$ and a commutative ring $R$, we define the \emph{coarse cohomology} $\ccd_R(X)$ of $X$ with respect to $R$, and give several equivalent formulations of $\ccd_R(X)$ in Proposition~\ref{prop:ccd_chars}. We prove $\ccd_R(X)$ is monotone under coarse embeddings in Theorem~\ref{thm:ccd_monotonicity}. We show in Theorem~\ref{thm:gp_ccd} that $\ccd_R(G)$ coincides with $\cd_R(G)$ when $G$ is a countable group and $\cd_R(G)<\infty$.
}
\vspace{.3cm}

Recall a chain complex $\bC_\bullet$ is \emph{$n$-dimensional} if $\bC_i=0$ for all $i>n$.
\begin{defn}
	Let $X$ be a metric space. We define the \emph{coarse cohomological dimension} $\ccd_R(X)$ of $X$ to be the least $n$ such that there exists an $n$-dimensional projective  $R$-resolution $\bC_\bullet$ over $X$. If no such $n$ exists, then $\ccd_R(X)=\infty$. We write $\ccd_\bbZ(X)$ as $\ccd(X)$.
\end{defn}

\begin{prop}\label{prop:ccd_chars}
	Let $X$ be a metric space and $n\in \bbN$. The following are equivalent:
	\begin{enumerate}

		\item If $\bC_\bullet$ is any projective $R$-resolution over $X$, then there exists a projective submodule $\bP\leq \bC_n$   such that
		      \[\cdots \to 0\to \bP\xrightarrow{\partial|_\bP} \bC_{n-1}\to \bC_{n-2}\to \bC_{n-3}\to \cdots \]  is a projective $R$-resolution over $X$.\label{item:propccdchars_2}
		\item $\ccd_R(X)\leq n$.\label{item:propccdchars_1}
		\item $\coarse^{k}(X,\bN)=0$ for all $k>n$ and all $R$-modules $\bN$ over $X$.\label{item:propccdchars_3.0}
		\item $\coarse^{n+1}(X,\bN)=0$ for all $R$-modules $\bN$ over $X$.\label{item:propccdchars_3}
		\item If $\bC_\bullet$ is a projective $R$-resolution over $X$, then there exists a finite displacement retraction from $\bC_n$ to $\im(\bC_{n+1}\to \bC_{n})$.\label{item:propccdchars_4}
	\end{enumerate}
\end{prop}
\begin{proof}
	$(\ref{item:propccdchars_2})\implies(\ref{item:propccdchars_1})$ is obvious.

	$(\ref{item:propccdchars_1})\implies(\ref{item:propccdchars_3.0})$: Suppose $\bC_\bullet$ is a projective $R$-resolution over $X$ with $\bC_k=0$ for $k>n$. Then for any $R$-module $\bN$ over $X$, the homology of $\Homfd(\bC_\bullet,\bN)$ vanishes in dimensions greater than $n$.

	$(\ref{item:propccdchars_3.0})\implies(\ref{item:propccdchars_3})$ is obvious.

	$(\ref{item:propccdchars_3})\implies(\ref{item:propccdchars_4})$:  We define $\bN\coloneqq\im (\bC_{n+1}\to \bC_{n})$. By hypothesis, we have  $H^{n+1}(\Homfd(\bC_\bullet,\bN))=0$.
	We consider map $\phi:\bC_{n+1}\to\bN$ given by $\phi(m)=\partial m$ for all $m\in \bC_{n+1}$.
	It is clear that $\phi(\im(\bC_{n+2}\to \bC_{n+1}))=0$, so that $\delta\phi=0$ and thus $\phi$ is a cocycle. Hence there is a finite displacement map $\psi:\bC_n\to \bN$ with $\delta\psi=(-1)^{n+1}\phi$. For each $n\in \bN$, we have $n=\partial m$ for some $m\in \bC_{n+1}$, so that $\psi(n)=\psi(\partial m)=(-1)^{n+1}(\delta\psi)(m)=\phi(m)=\partial m=n$. Thus $\psi$ is a finite displacement retraction onto $\bN$.

	$(\ref{item:propccdchars_4})\implies(\ref{item:propccdchars_2})$ is the content of Lemma~\ref{lem:retract_truncate}.
\end{proof}
In the following theorem, the relation $\leq$ is defined on $\bbN\cup \{\infty\}$ in the obvious way.
\begin{thm}\label{thm:ccd_monotonicity}
	If $R$ is a commutative ring,  $X$ and $Y$ are metric spaces, and $f:X\to Y$ is a coarse embedding, then $\ccd_R(X)\leq \ccd_R(Y)$.
\end{thm}
\begin{proof}
	There is nothing to show if $\ccd_R(Y)=\infty$, so assume  $n=\ccd_R(Y)<\infty$ and let $\bC_\bullet$ be an $R$-projective resolution over $Y$ with $\bC_k=0$ for $k>n$. By Proposition~\ref{prop:pullback_res}, $f^*\bC_\bullet$ is an $R$-projective resolution over $X$. As $f^*\bC_k=0$ for all $k>n$, it follows that $\ccd_R(X)\leq n$ as required.
\end{proof}
Theorem~\ref{thm:ccd_monotonicity} readily implies:
\begin{cor}\label{cor:ccd_coarseinv}
	If $X$ and $Y$ are coarsely equivalent, then $\ccd_R(X)=\ccd_R(Y)$.
\end{cor}

We can also compare coarse cohomological dimension between different rings:
\begin{prop}
	Let $R$ and $S$ be commutative rings such that there is a homomorphism $f:R\to S$. If $X$ is any metric space, then $\ccd_S(X)\leq \ccd_R(X)$. In particular, $\ccd_S(X)\leq \ccd(X)$ for every commutative ring $S$.
\end{prop}
\begin{proof}
	If $\ccd_R(X)=\infty$, there is nothing to show, so assume $\ccd_R(X)=n<\infty$. Let $\bC_\bullet$ be a projective $R$-resolution over $X$ with $\bC_k=0$ for $k>n$. Then $\bC_k\otimes_R S$ is a projective $S$-resolution over $X$ by Proposition~\ref{prop:change_rings_res}, and satisfies $\bC_k\otimes_R S=0$ for $k>n$. Thus $\ccd_S(X)\leq n$ by Proposition~\ref{prop:ccd_chars}. For every commutative ring $S$, there is a homomorphism $\bbZ\to S$, showing $\ccd_S(X)\leq \ccd_\bbZ(X)=\ccd(X)$ as required.
\end{proof}
The following criterion yields a large class of groups of finite coarse cohomological dimension:
\begin{prop}\label{prop:ccd_action}
	Let $G$ be a countable group acting properly and cocompactly  on a contractible simplicial complex $X$ of dimension $n$. Then $\ccd(G)\leq n$.
\end{prop}
\begin{proof}
	Let $(X^{(1)},X)$ be the associated metric simplicial complex as in Example~\ref{exmp:metric_complex}. Since $X$ is contractible and $G$ acts cocompactly on $X$, $X$ is uniformly acyclic, and so the simplicial chain complex $C_\bullet(X)$ can be given the structure of a projective $R$-resolution over $X^{(1)}$ as in Example~\ref{exmp:metric_complex}. As $X$ is $n$-dimensional, $\ccd(X)\leq n$.  Since $G$ acts properly and cocompactly on $X^{(1)}$, it is quasi-isometric to $X^{(1)}$, and so Corollary~\ref{cor:ccd_coarseinv} implies $\ccd(G)=\ccd(X)\leq n$.
\end{proof}

A priori, $\ccd_R(X)$ may be hard to compute for general metric spaces. Fortunately, when $G$ is a countable group, $\ccd_R(G)$ agrees with the (virtual) cohomological dimension of $G$ when finite:
\begin{thm}\label{thm:gp_ccd}
	Let $G$ be a countable group. If $\cd_R(G)<\infty$, then $\ccd_R(G)=\cd_R(G)$.
\end{thm}
Before proving Theorem~\ref{thm:gp_ccd}, we need to prove some  results relating group cohomology to coarse cohomology.
Let $R$ be a commutative ring, let $G$ be a group and let $N$ be an $R$-module.
Then $RG\otimes_{R} N$ admits both a left $G$-action given by $g\cdot(k\otimes n)=gk\otimes n$, and a right $G$-action given by $(k\otimes n)\cdot g=kg\otimes n$, for all $g,k\in G$ and $n\in N$.
If $M$ is a left $RG$-module and $N$ is  an $R$-module, then  $\Hom_{RG}(M,RG\otimes_{R} N)$ consists of the collection of all $G$-equivariant $R$-module homomorphisms with respect to the left $G$-action on $M$ and on $RG\otimes_R N$. Thus $\Hom_{RG}(M,RG\otimes_{R} N)$ admits a right $G$-action given by $(\phi \cdot g)(m)=\phi(m)\cdot g$.

If $\bM$ is a $G$-induced projective $R$-module over $G$ and $N$ is an $R$-module, then $\Hom_c(\bM,N)$ admits a right $G$-action given by $(\alpha\cdot g)(m)=\alpha(gm)$ for all $g\in G$, $\alpha\in \Hom_c(\bM,N)$ and $m\in \bM$. We now have the following:
\begin{lem}\label{lem:gptospace}
	Let $G$ be a countable group and
	let  $\bM$ be a  $G$-induced projective $R$-module over $G$ with underlying $R$-module $M$. Let $N$ be an $R$-module. Then there is an isomorphism $\Lambda_M=\Lambda:\Hom_{RG}(M,RG\otimes_{R} N)\cong \Hom_c(\bM,N)$ as right $RG$-modules. Moreover, $\Lambda$ is natural in the following sense: If $\bL$ is a $G$-induced projective  $R$-module over $G$ and  $\xi:\bL\to \bM$ is an $RG$-module homomorphism, then the diagram in Figure~\ref{fig:natural_iso_gp} commutes.
	\begin{figure}[ht]
		\[\begin{tikzcd}
				{\Hom_{RG}(M,RG\otimes_{R} N)} && {\Hom_{RG}(L,RG\otimes_{R} N)} \\
				\\
				{\Hom_{c}(\bM,N)} && {\Hom_{c}(\bL,N)}
				\arrow["{\phi\mapsto \phi\circ\xi}"', from=1-1, to=1-3]
				\arrow["{\xi^*}", from=1-1, to=1-3]
				\arrow["{\Lambda_M}"', from=1-1, to=3-1]
				\arrow["{\Lambda_L}"', from=1-3, to=3-3]
				\arrow["{\xi^*}", from=3-1, to=3-3]
				\arrow["{\psi\mapsto\psi\circ\xi}"', from=3-1, to=3-3]
			\end{tikzcd}\]
		\caption{Naturality of the isomorphism $\Lambda_M$.}\label{fig:natural_iso_gp}
	\end{figure}
\end{lem}

\begin{proof}
	Suppose $\bM=(M,B,\delta,p,\{B(i)\}, \{m_b\})$ and $\phi\in \Hom_{RG}(M,RG\otimes_R N)$. Then there exists a family $\{\phi_g:M\to N\}_{g\in G}$ of $R$-module homomorphisms such that  $\phi(m)=\sum_{g\in G}g\otimes \phi_{g}(m)$ for all $m\in M$. For each $m\in M$, there are only finitely many $g\in G$  such that $\phi_{g}(m)\neq 0$.

	For each $h\in G$  and $m\in M$, we have:
	\begin{align*}
		\sum_{g\in G}g\otimes \phi_g(hm)=\phi(hm)=h\phi(m)=	\sum_{g\in G}hg\otimes \phi_g(m)=\sum_{g\in G}g\otimes \phi_{h^{-1}g}(m).
	\end{align*} This implies that $\phi_{g}(m)=\phi_{1}(g^{-1}m)$ for all $g\in G$ and $m\in M$. We claim that for each $i\in \bbN$, \[Z_{i}\coloneqq \{b\in B(i)\mid \phi_1(m_b)\neq 0\}\] is finite, which by Proposition~\ref{prop:bdd_support}, will imply that $\phi_1\in \Hom_c(\bM,N)$.  To see this, we  fix $i$ and recall that $G$ acts freely on $B(i)$ with finitely many orbits. Let  $\{b_1,\dots, b_n\}$ be a set of $G$-transversals of $B(i)$.
	As noted above, the set $R_i\coloneqq \{g\in G\mid \phi_{g}(m_{b_j})\neq 0 \text{ for some $1\leq j\leq n$}\}$  is  finite. Thus if $\phi_1(m_b)\neq 0$ for some $b=gb_j\in B(i)$, then  $0\neq\phi_{1}(m_b)=\phi_{1}(gm_{b_j})=\phi_{g^{-1}}(m_{b_j})$, which implies $g^{-1}\in R_i$. Hence there are only finitely many such $g$, and so $Z_i$ is finite.

	We thus define the map $\Lambda:\Hom_{RG}(M,RG\otimes_R N)\to \Hom_c(\bM,N)$ by $\Lambda(\phi)=\phi_1$. It is easy to see that $\Lambda$ is an $R$-module homomorphism.
	Now for $h\in G$, we have $(\phi\cdot h)(m)=\sum_{g\in G}gh\otimes \phi_g(m)=\sum_{g\in G}g\otimes \phi_{gh^{-1}}(m)$, and  so $\Lambda(\phi\cdot h)=\phi_{h^{-1}}$. Since $\phi_{h^{-1}}(m)=\phi_1(hm)=(\phi_1\cdot h)(m)$, we see $\Lambda(\phi\cdot h)=\phi_1\cdot h$, and so $\Lambda$ is a homomorphism of right $RG$-modules.
	Moreover, if $\Lambda(\phi)=0$, then $\phi_{1}=0$ and so,  since $\phi_{g}(m)=\phi_{1}(g^{-1}m)=0$ for all $m\in \bM$ and $g\in G$, we see that $\phi=0$. Thus $\Lambda$ is injective.

	We now show $\Lambda$ is surjective.  For $\psi\in \Hom_c(\bM,N)$
	we define $\phi:\bM\to RG\otimes N$ by $\phi(m)=\sum_{g\in G}g\otimes \psi(g^{-1}m)$, noting that the condition $\psi\in \Hom_c(\bM,N)$ ensures that all but finitely many terms of the sum defining $\phi(m)$ are zero. Applying the preceding arguments in reverse, it can be shown that $\phi\in \Hom_{RG}(M,RG\otimes_R N)$ and $\Lambda(\phi)=\psi$.

	Finally, to see naturality of $\Lambda$, suppose we have a  homomorphism $\phi\in \Hom_{RG}(\bM,RG\otimes_R N)$ and a $G$-equivariant map $\xi:\bL\to \bM$. We note that for each $l\in L$, we have
	\begin{align*}
		(\xi^*\Lambda_\bM(\phi))(l) & =(\Lambda_\bM(\phi))(\xi(l))=\phi_1(\xi(l))=(\phi\circ\xi)_1(l) \\&=\Lambda_\bL(\phi\circ \xi)(l)=\Lambda_\bL(\xi^*(\phi))(l)
	\end{align*} as required.
\end{proof}
Let $G$ be a countable group and let $\bC_\bullet$ be a projective $R$-resolution over $G$ as in Corollary~\ref{cor:projres_ginv}, in which each $\bC_k$ is a $G$-induced projective $R$-module over $G$ and the boundary maps are $G$-equivariant. By Lemma~\ref{lem:gptospace}, for each $R$-module $N$,  we have an isomorphism $\Hom_{RG}(\bC_\bullet,RG\otimes_{R} N)\cong \Hom_c(\bC_\bullet,N)$ of cochain complexes that is $G$-equivariant with respect to the right $G$-actions. The cohomology of these cochain complexes thus inherit the structure of a right $G$-module, and we deduce the following:
\begin{cor}\label{cor:induced_mod}
	Let $G$ be a countable group and let  $N$  be an $R$-module. Then  $\coarse^k(G;N)\cong H^k(G,RG\otimes_R N)$ as right $RG$-modules for all $k\in \bbZ$.
\end{cor}

\begin{proof}[Proof of Theorem~\ref{thm:gp_ccd}]
	Let $n=\cd_R(G)<\infty$.  By~\cite[Proposition 5.1]{bieri1981homological}, we have $H^n(G,F)\neq 0$ for some free $RG$-module $F$. We now note that $F=RG \otimes_R N$ for some free $R$-module $N$, so Corollary~\ref{cor:induced_mod} implies that  $\coarse^n(G;N)\cong H^n(G,F)\neq 0$. It follows from Proposition~\ref{prop:ccd_chars}  that $\ccd_R(G)\geq \cd_R(G)=n$.
	We now show $\ccd_R(G)\leq \cd_R(G)$. In the case $R$ is Noetherian, any projective resolution of the trivial $RG$-module $R$  yields a projective $R$-resolution over $G$ by Lemma~\ref{lem:ginduced_produce} and Propositions~\ref{prop:ginduced} and~\ref{prop:ginv_unifpreim}. This  implies the inequality $\ccd_R(G)\leq \cd_R(G)$ when applied to a length $n$ projective resolution of  the trivial $RG$-module $R$. For more general $R$, Proposition~\ref{prop:ginv_unifpreim} cannot be applied, so we require another argument.

	We pick $\bC_\bullet$ as in the discussion preceding Corollary~\ref{cor:induced_mod}. We claim there is a $G$-equivariant retraction $\bC_n\to \bC_n$ onto $\im(\bC_{n+1}\to \bC_n)$.   By  Proposition~\ref{prop:ginduced}, such a map will have finite displacement, so  Proposition~\ref{prop:ccd_chars} will imply  $\ccd_R(G)\leq n=\cd_R(G)$. It remains to prove the claim.  We first assume $n>1$.
	By~\cite[Proposition 4.1b)]{bieri1981homological},  $\ker(\bC_{n-1}\xrightarrow{\partial} \bC_{n-2})=\im(\bC_n\xrightarrow{\partial} \bC_{n-1})$ is projective as an $RG$-module.
	This implies that the short exact sequence \[1\to \ker(\bC_{n}\to \bC_{n-1})\to \bC_n\to \im(\bC_n\to \bC_{n-1})\to 1\] splits $G$-equivariantly, so there is a $G$-equivariant retraction $\bC_n\to \bC_n$ onto $\ker(\bC_{n}\xrightarrow{\partial} \bC_{n-1})=\im(\bC_{n+1}\xrightarrow{\partial} \bC_{n})$, proving the claim when $n>1$.
	For $n=0$ (resp. $n=1$) the argument is essentially identical, except that~\cite[Proposition 4.1b)]{bieri1981homological} tells us that the trivial $R$-module $R$ (resp. $\ker(\bC_0\xrightarrow{\epsilon} R)=\im(\bC_1\to \bC_0)$) is a projective $R$-module. In both cases, we obtain a $G$-equivariant retraction $\bC_n\to \bC_n$ onto $\im(\bC_{n+1}\to \bC_n)$ as required.
\end{proof}

We now classify spaces of coarse cohomological dimension zero:
\begin{prop}\label{prop:bdd_0dim}
	If $R$ is a commutative ring and $X$ is a metric space, then $\ccd_R(X)=0$ if and only if $X$ is bounded.
\end{prop}
\begin{proof}
	If $X$ is bounded, then it is coarsely equivalent to a point $Y=\{y\}$. Let $C_0=Rb\cong R$ be the free $R$-module with basis $b$.  Let $\bC_0=(C_0,\{b\},\delta,p,\{B(i)\})$, with $\delta_b(rb)=r$, $p(b)=y$ and $B(i)=\{b\}$ for all $i$. Thus $\bC_0$ is a free $R$-module over $Y$. Letting $\bC_i=0$ for $i\neq 0$, and defining the augmentation $\epsilon:\bC_0\cong R$ by $\epsilon(b)=1$, we have $(\bC_\bullet,\epsilon)$ is a projective $R$-resolution over $Y$, so $\ccd_R(Y)=0$. It follows from Corollary~\ref{cor:ccd_coarseinv} that $\ccd_R(X)=\ccd_R(Y)=0$.

	Conversely, suppose $\ccd_R(X)=0$. Let $\bC_\bullet$ be the standard resolution of $X$. By Proposition~\ref{prop:ccd_chars}, there is a finite displacement retraction from $\bC_0$ onto $\bN\coloneqq \im(\bC_1\to \bC_0)$. By Lemma~\ref{lem:comp}, $\bN$ has a complement $\bM$ such that the projection $\phi:\bC_0=\bN\oplus \bM\to\bM$ has finite displacement $\Phi$. Moreover, as  $\epsilon$ vanishes on $\bN$,  $\phi$ is augmentation-preserving and $\epsilon$ restricts to an isomorphism $\bM\cong \bC_0/\bN\cong R$.
	Thus there is a unique $m_0\in \bM$ such that $\epsilon(m_0)=1$.

	Set  $L\coloneqq \supp_{\bC_0}(m_0)$, which  is bounded.
	Recall from Definition~\ref{defn:standard_res} that $\bC_0$ has a basis $B_0=B_0(0)=\{[x]\}_{x\in X}$ with $\epsilon([x])=1$ and $\supp_{\bC_0}([x])=\{x\}$ for all $x\in X$. Set $m_x=(x)$ for each $x$. Since $\phi([x])\in \bM$ and  $\epsilon(\phi([x]))=\epsilon([x])=1$, it follows $\phi([x])=m_0$ for all $x\in X$. Then $L= \supp_{\bC_0}(\phi([x]))\subseteq N_{\Phi(0)}(\supp_{\bC_0}([x]))=N_{\Phi(0)}(x)$ for all $x\in X$. It follows that $X\subseteq N_{\Phi(0)}(L)$, and so $X$ is bounded as required.
\end{proof}

\section{Uniform acyclicity and finite type resolutions}\label{sec:unif_acyc}
\emph{In this section we study spaces that are coarsely uniformly $(n-1)$-acyclic over $R$. In Proposition~\ref{prop:acyc_res}, we show $X$ is coarsely uniformly $(n-1)$-acyclic over $R$ if and only if there exists a projective $R$-resolution $\bC_\bullet$ over $X$ such that $\bC_i$ is of finite height for $i\leq n$. The latter characterisation is a metric space analogue of being of type $FP_n^R$. In Theorem~\ref{thm:ccd_finit_type_max} we show that for spaces of coarse finite type, $\ccd_R(X)$ is the largest $n$ for which $\coarse^n(X;R)\neq 0$, generalising an analogous result in group cohomology. We apply this result to compare the coarse cohomological and asymptotic dimensions (Proposition~\ref{prop:ccd_vs_asdim}), and to compute the coarse cohomological dimension of all hyperbolic and $\CAT(0)$ groups in terms of their boundaries (Corollary~\ref{cor:bdry}).
}
\vspace{.3cm}

We characterise coarse uniform acyclicity of $X$ in terms of projective resolutions over $X$. To do this, we require the following:
\begin{defn}\label{defn:unif_acyc_complex}
	Let $R$ be a commutative ring and let $X$ be a metric space.
	Suppose  $\lambda:\bbN\rightarrow \bbN$ and $\mu:\bbN\times \bbN \rightarrow \bbN$  are increasing functions such that $\lambda(i)\geq i$ and $\mu(i,r)\geq r$ for all $i,r\in \bbN$.
	Let $\bC_\bullet$ be a projective $R$-resolution over $X$ augmented $R$-metric complex over $X$.
	We say that $\bC_\bullet$ is \emph{$(\lambda,\mu)$-uniformly $n$-acyclic} if for all $i,r\in \bbN$, $k\leq n$ and $x\in X$, the map \[\widetilde{H}_k(\bC_\bullet(i,N_r(x)))\to \widetilde{H}_k(\bC_\bullet(\lambda(i),N_{\mu(i,r)}(x))),\] induced by inclusion, is zero. We say $\bC_\bullet$ is \emph{uniformly $n$-acyclic} if it is $(\lambda,\mu)$-uniformly $n$-acyclic for some $(\lambda,\mu)$, and that $\bC_\bullet$ is \emph{uniformly acyclic} if it is uniformly $n$-acyclic for all $n$.
\end{defn}

\begin{rem}\label{rem:unif_acyc}
	To avoid confusion between Definitions~\ref{defn:weak_uniform_acyc} and~\ref{defn:unif_acyc_complex}, we briefly explain the differences between weak uniform acyclicity and uniform acyclicity.
	If $\bC_\bullet$ is a projective $R$-resolution over $X$, then weak uniform acyclicity of  $\bC_\bullet$, which is always assumed by  Definition~\ref{defn:proj_res}, ensures  there exist functions $\lambda,\mu:\bbN\times \bbN\to  \bbN$ such that for all $i,r\in \bbN$, $k\leq n$ and $x\in X$, the map \[\widetilde{H}_k(\bC_\bullet(i,N_r(x)))\to \widetilde{H}_k(\bC_\bullet(\lambda(i,r),N_{\mu(i,r)}(x))),\] induced by inclusion, is zero. The crucial part of being uniformly $n$-acyclic in Definition~\ref{defn:unif_acyc_complex} is that there exists  $\lambda$ as above that does not depend on $r$.
\end{rem}

\begin{prop}\label{prop:acyc_res}
	Let $X$ be a metric space, let $R$ be a commutative ring  and let $n\in \bbN$. The following are equivalent:
	\begin{enumerate}
		\item\label{item:acyc_res1}
		      $X$ is coarsely uniformly $(n-1)$-acyclic over $R$.
		\item\label{item:acyc_res2}
		      There exists a projective $R$-resolution over $X$ that is uniformly $(n-1)$-acyclic.
		\item\label{item:acyc_res3}
		      Every projective $R$-resolution over $X$ is uniformly $(n-1)$-acyclic.
		\item\label{item:acyc_res4}
		      There exists a free $R$-resolution $\bC_\bullet$ over $X$ such that $\bC_k$ has finite height for each $k\leq n$.
		\item\label{item:acyc_res5}
		      There exists a projective $R$-resolution $\bC_\bullet$ over $X$ such that $\bC_k$ has finite height for each $k\leq n$.
	\end{enumerate}
	Moreover, if $X$ is proper and is coarsely uniformly $(n-1)$-acyclic over $R$, then there exists a proper free $R$-resolution $\bC_\bullet$ over $X$ such that $\bC_k$ has finite height for each $k\leq n$.
\end{prop}

\begin{proof}
	(\ref{item:acyc_res1}) $\implies$ (\ref{item:acyc_res2}): Let $\bD_\bullet$ be the ordered standard $R$-resolution over $X$. Suppose that $X$ is uniformly $(\lambda,\mu)$-coarsely uniformly $(n-1)$-acyclic over $R$. Without loss of generality, we may assume $\lambda$ and $\mu$ restrict to functions $\bbN\to \bbN$ and $\bbN\times\bbN\to \bbN$ respectively. We claim that for all $k< n$ and $i,r\in \bbN$, the map \begin{align}\label{eqn:acyc_res1}
		\widetilde{H}_k(\bD_\bullet(i,N_r(x)))\to \widetilde{H}_k(\bD_\bullet(\lambda(i),N_{\mu(i,r+i)}(x))),
	\end{align} induced by inclusion, is zero. In the case $k>i$, we have $\bD_k(i,N_r(x))\subseteq \bD_k(i)=0$ and there is nothing to show, so assume $k\leq i$. Then Lemma~\ref{lem:ordered_res_rips} ensures that (\ref{eqn:acyc_res1}) factors through the map \[\widetilde{H}_k(C_\bullet(P_i(N_{r+i}(x));R))\to \widetilde{H}_k(C_\bullet(P_{\lambda(i)}(N_{\mu(i,r+i)}(x));R)),\] which is zero by Definition~\ref{defn:unif_acyc}.

	(\ref{item:acyc_res2}) $\implies$ (\ref{item:acyc_res3}): Let $\bD_\bullet$ and $\bC_\bullet$ be projective $R$-resolutions over $X$, and assume that $\bC_\bullet$ is $(\lambda,\mu)$-uniformly $(n-1)$-acyclic. By Proposition~\ref{prop:proj_res},  there exist augmentation-preserving finite displacement  chain maps $f_\#:\bC_\bullet\to \bD_\bullet$ and $g_\#:\bD_\bullet\to \bC_\bullet$ and  a finite displacement chain homotopy $\id\stackrel{h_\#}{\simeq}f_\#g_\#$. Using Lemma~\ref{lem:standard},  there is a function $\Phi:\bbN\to \bbN$ such that for all $x\in X$ and $i,r,k\in \bbN$ with $k\leq n$,
	\begin{align}\label{eqn:disp}
		f_\#(\bC_k(i,N_r(x)))\subseteq \bD_k(\Phi(i),N_{r+\Phi(i)}(x)),\quad g_\#(\bD_k(i,N_r(x)))\subseteq \bC_k(\Phi(i),N_{r+\Phi(i)}(x))
	\end{align}
	and
	\begin{align*}
		h_\#(\bD_k(i,N_r(x)))\subseteq \bD_{k+1}(\Phi(i),N_{r+\Phi(i)}(x)).
	\end{align*}
	This implies that for all $x\in X$ and $i,r,k\in \bbN$ with $k< n$, the map  \begin{align*}
		\widetilde{H}_k(\bD_\bullet(i,N_r(x)))\to \widetilde{H}_k(\bD_\bullet(\Phi(\lambda(\Phi(i))),N_{\Phi(\lambda(\Phi(i)))+\mu(\Phi(i),r+\Phi(i))}(x))),
	\end{align*}
	induced by inclusion, factors through
	\begin{align*}
		\widetilde{H}_k(\bC_\bullet(\Phi(i),N_{r+\Phi(i)}(x)))\to \widetilde{H}_k(\bC_\bullet(\lambda(\Phi(i)),N_{\mu(\Phi(i),r+\Phi(i))}(x))),
	\end{align*} which is zero because $\bC_\bullet$ is $(\lambda,\mu)$-uniformly $(n-1)$-acyclic.

	(\ref{item:acyc_res3}) $\implies$ (\ref{item:acyc_res1}): Let $\bD_\bullet$ be the ordered standard $R$-resolution over $X$. By assumption,  $\bD_\bullet$ is $(\lambda,\mu)$-uniformly $(n-1)$-acyclic. Lemma~\ref{lem:ordered_res_rips} ensures that  for all $x\in X$ and $i,r,k\in \bbN$ with $k<n$, the map \[\widetilde{H}_k(C_\bullet(P_i(N_{r}(x));R))\to \widetilde{H}_k(C_\bullet(P_{\lambda(i)}(N_{\lambda(i)+\mu(i,r)}(x));R)),\] induced by inclusion, factors through the map  \begin{align*}
		\widetilde{H}_k(\bD_\bullet(i,N_r(x)))\to \widetilde{H}_k(\bD_\bullet(\lambda(i),N_{\mu(i,r)}(x))),
	\end{align*} hence is zero. Thus $X$ is coarsely uniformly $(n-1)$-acyclic over $R$.

	(\ref{item:acyc_res3}) $\implies$ (\ref{item:acyc_res4}):  We prove this by induction on $n$. For the base case, we  note that for any metric space $X$,  if  $\bC_\bullet$ is the standard resolution of $X$, then $\bC_0$ has finite height.
	For the inductive hypotheses, we assume the standard resolution $\bC_\bullet$ of $X$ is $(\lambda,\mu)$-uniformly $(n-1)$-acyclic and that $\bD_\bullet$ is some free $R$-resolution over $X$  such that $\bD_k$ has finite height for $k<n$. By Proposition~\ref{prop:proj_res},  there exist augmentation-preserving finite displacement  chain maps $f_\#:\bC_\bullet\to \bD_\bullet$, $g_\#:\bD_\bullet\to \bC_\bullet$, and a finite displacement chain homotopy $\id\stackrel{h_\#}{\simeq} f_\#g_\#$. By Lemma~\ref{lem:standard}, we can pick $\Phi$ such that each of these maps has displacement $\Phi$, and that (\ref{eqn:disp}) holds.

	Pick $i_0$ such that $\bD_{n-1}=\bD_{n-1}(i_0)$. Set $i_1\coloneqq \Phi(i_0)$ and $i_2\coloneqq \lambda(i_1)$. We set $\bE_n\coloneqq \bD_{n-1}\oplus \bC_n(i_2)$, which is a finite height free $R$-module over $X$ by Remark~\ref{rem:subcomplex_standard} and Propositions~\ref{prop:finiteness_operations_first} and~\ref{prop:projective_operations_second}. A boundary map $\partial: \bE_n\to \bD_{n-1}$ is defined by \[\partial (\sigma,\omega)=\sigma+h_\#\partial\sigma+f_\#\partial\omega-f_\#g_\#\sigma.\] We note that $\partial$  has finite displacement and $\partial\partial(\sigma,\omega)=\partial\sigma+\partial h_\#\partial\sigma-f_\#g_\#\partial\sigma=-h_\#\partial^2\sigma=0$.

	Let $\sigma\in \bD_{n-1}(i_0,N_r(x))$ be a reduced cycle, so that $g_\#(\sigma)$ is a reduced cycle in $\bC_{n-1}(i_1,N_{r+i_1}(x))$. Thus there is a chain $\omega\in \bC_n(i_2,N_{\mu(i_1,r+i_1)}(x))$ with $\partial\omega=g_\#\sigma$.  It follows that $\partial(\sigma,\omega)=\sigma+h_\#\partial\sigma+f_\#g_\#\sigma-f_\#g_\#\sigma=\sigma$. Thus \[\bE_n\to \bD_{n-1}\to \dots \to \bD_0\to 0\] is a partial free resolution over $X$ of dimension $n$ in which $\bE_n$ has finite height. By Proposition~\ref{prop:extend_partial}, we can extend this partial free $R$-resolution over $X$ to a  free $R$-resolution over $X$, proving the inductive hypothesis.

	We now prove the ``moreover'' part of Proposition~\ref{prop:acyc_res}. In the case $X$ is proper, we let $Y\subseteq X$ be a subspace as in Corollary~\ref{cor:standard_proper}.
	We carry out the preceding argument for $Y$, assuming as part of our inductive hypothesis that $\bD_\bullet$ is proper, and noting the standard resolution $\bC_\bullet$ over $Y$ is proper.
	This allows us to construct a proper projective $R$-resolution $\bE_\bullet$ over $Y$ such that $\bE_k$ has finite height for each $k\leq n$.
	If $f:X\to Y$ is a closest point projection, Propositions~\ref{prop:finiteness_operations_first} and~\ref{prop:pullback_res} ensure the pullback $f^*\bE_\bullet$ is proper projective $R$-resolution  over $X$ such that $f^*\bE_k$ has finite height for each $k\leq n$.

	(\ref{item:acyc_res4}) $\implies$ (\ref{item:acyc_res5}) is obvious.

	(\ref{item:acyc_res5}) $\implies$ (\ref{item:acyc_res2}):
	Let $\bC_\bullet$ be a projective $R$-resolution over $X$ such that $\bC_k$ has height $i_0<\infty$ for each $k\leq n$. Thus $\bC_k(i,Y)=\bC_k(i_0,Y)$ for all $i\geq i_0$ and $Y\subseteq X$. Since $\bC_\bullet$ a  projective $R$-resolution over $X$, it is $\Omega$-weakly uniformly acyclic for some $\Omega$.  Defining $\lambda:\bbN\to \bbN$ by $\lambda(i)=\max(i,i_0)$, it follows that for all $x\in X$ and  $i,r,k\in \bbN$ with $k< n$, the map \[\widetilde{H}_k(\bC_\bullet(i,N_r(x)))\to \widetilde{H}_k(\bC_\bullet(\lambda(i),N_{\Omega(i,r)}(x))),\] induced by inclusion, is zero. This shows $\bC_\bullet$ is uniformly $(n-1)$-acyclic (see Remark~\ref{rem:unif_acyc}).
\end{proof}
Using  Proposition~\ref{prop:acyc_res} and its proof, we deduce the following:
\begin{cor}\label{cor:unif_acyc}
	Let $X$ be a metric space and let $R$ be a commutative ring. The following are equivalent:
	\begin{enumerate}
		\item\label{item:acyc_res_inf1}
		      $X$ is coarsely uniformly acyclic over $R$.
		\item There exists a finite-height free $R$-resolution over $X$.
		\item There exists a finite-height projective $R$-resolution over $X$.
	\end{enumerate}
\end{cor}

We now consider the following class of metric spaces:
\begin{defn}
	A metric space $X$ is said to be of \emph{coarse finite type over $R$} if there exists a finite-dimensional finite-height projective $R$-resolution over $X$.
\end{defn}

\begin{prop}\label{prop:finite_type}
	Let $X$ be a metric space that is coarsely uniformly $(n-1)$-acyclic over $R$ with  $\ccd_R(X)\leq n$. Then there exists a finite height $n$-dimensional projective $R$-resolution $\bC_\bullet$ over $X$, which is proper when $X$ is proper. In particular, $X$ has coarse finite type over $R$.
\end{prop}
\begin{proof}
	By Proposition~\ref{prop:acyc_res}, there exists a projective $R$-resolution $\bC_\bullet$ with $\bC_k$ of finite height for $k\leq n$, which is proper when $X$ is proper. By Proposition~\ref{prop:ccd_chars}, there exists a projective submodule $\bP\leq \bC_n$ such that $0\to \bP\xrightarrow{\partial|_\bP} \bC_{n-1}\to \dots\to \bC_0\to 0$ is the required finite height projective $R$-resolution over $X$, noting that since $\bC_n$ has finite height, so does $\bP$.
\end{proof}
It is not known to the author if the following is true:
\begin{ques}
	If $X$ is  coarsely uniformly $(n-1)$-acyclic over $R$ and $\ccd_R(X)\leq n$, does there exist a  finite-height $n$-dimensional  \emph{free}  $R$-resolution over $X$?
\end{ques}
This question is a coarse analogue of the following open problem regarding groups: is every group of type $FP$ also of type $FL$?

The following lemma shows that certain coarse cohomology modules of coarsely uniformly acyclic metric spaces are countably generated:
\begin{prop}\label{prop:ctble_cohomology_fingen}
	Let $R$ be a PID\@.
	Let $X$ be a proper coarsely uniformly $(n-1)$-acyclic metric space over $R$. Then $\coarse^k(X;R)$ is a countably generated $R$-module for $k\leq n$.
\end{prop}
\begin{proof}
	By Proposition~\ref{prop:acyc_res}, there exists a  proper projective $R$-resolution $\bC_\bullet$ over $X$ such that $\bC_k$ has finite height for each $k\leq n$. By Proposition~\ref{prop:duality_ht}, $\bC^k_c$ is a finite height proper projective $R$-module over $X$ for each $k\leq n$. Thus $\bC^k_c$ is countably generated for $k\leq n$. Since $\coarse^k(X;R)\cong H^k(\bC^\bullet_c)$ is isomorphic to a quotient of a submodule of $\bC^k_c$, the result follows by Lemma~\ref{lem:submodule_ctble}.
\end{proof}

\begin{defn}\label{defn:chom}
	A metric space $X$ is \emph{coarsely homogeneous} if there exist $\Lambda$, $\Upsilon$ such that  for every $x_1,x_2\in X$, there exists a $(\Lambda, \Upsilon)$-coarse equivalence $f:X\to X$ with $f(x_1)=x_2$.
\end{defn}
For instance, any metric space admitting a cocompact isometric group action, or more generally a cobounded quasi-action, is coarsely homogeneous.
\begin{prop}\label{prop:unif_preimage_cobdry}
	Assume $R$ is a Noetherian ring. Suppose $X$ is a coarsely homogeneous proper  metric space. Let $\bC_\bullet$ be a proper projective $R$-resolution over $X$ such that $\bC_k$ is of finite height for $k\leq n$.
	\begin{enumerate}
		\item\label{item:unif_preimage_cobdry}  For $k\leq n$, the coboundary map $\delta:\bC^{k-1}_c\to \bC^{k}_c$ has uniform preimages.
		\item\label{item:full_support} Suppose that  $H^k(\bC^\bullet_c)$ is a finitely generated $R$-module for some $k\leq n$. Then there  is a constant $D$ such for every $x\in X$ and class $[\alpha]\in H^k(\bC^\bullet_c)$, there is a cocycle $\alpha_x\in \bC^k_c$ such that $[\alpha_x]=[\alpha]$ and $\supp(\alpha_x)\subseteq N_D(x)$.
	\end{enumerate}
\end{prop}

\begin{proof}
	We  fix $i_0$ such that $\bC_k(i_0)=\bC_k$ for all $k\leq n$, and fix $\Lambda,\Upsilon$ as in Definition~\ref{defn:chom}. By Proposition~\ref{prop:duality_ht}, this implies that $\bC^k_c(i_0)=\bC^k_c$ for all $k\leq n$.
	Fix a basepoint $x_0\in X$ and for each $x\in X$, choose a $(\Lambda, \Upsilon )$-coarse equivalence $f_x:X\to X$ with  $f_x(x_0)=x$.
	Setting $A\coloneqq \Upsilon(0)$, we see that for each $x\in X$ there exists a  $(\Lambda, \Upsilon)$-coarse equivalence  $g_x:X\to X$ that is an $A$-coarse inverse to $f_x$.  By applying  Propositions~\ref{prop:proj_res} and~\ref{prop:duality_ht}, there exists  $\Phi:\bbN\to\bbN$ such that:
	\begin{enumerate}
		\item For every $x\in X$, there exist augmentation-preserving chain homotopy-equivalences $(f_x)_\#,(g_x)_\#:\bC_\bullet\to \bC_\bullet$ of displacement $\Phi$ over $f_x$ and $g_x$.
		\item For each $x\in X$,  there exists a chain homotopy $\id\stackrel{(h_{x})_\#}{\simeq} (g_x)_\#(f_x)_\#$ of displacement $\Phi$ over $\id_X$.
		\item The induced dual maps $(f_x)^\#$, $(g_x)^\#$ and $(h_{x})^\#$ have  displacement $\Phi$ over $g_x$, $f_x$ and $\id_X$, when restricted to $\bC^k_c$ for $k\leq n$.
	\end{enumerate}
	We define $i_1\coloneqq \Phi(i_0)$

	(\ref{item:unif_preimage_cobdry}): There is nothing to prove when $k=0$, so fix $1\leq k\leq n$. For each $r\geq 0$, let $B_r$ be the submodule of $\bC^{k}_c$ generated by \[\{\alpha\in \bC^k_c\mid \supp(\alpha)\subseteq N_{i_1+\Upsilon(r)+A}(x_0)\}.\] Since $R$ is Noetherian and  $\bC^k_c$ is proper and of finite height, $B_r$ is a finitely generated $R$-module.

	Since $R$ is Noetherian,  $B_r\cap \im(\bC^{k-1}_c\to \bC^k_c)$ is also a  finitely generated $R$-module, so we can pick $s_r\geq r$ such that for every  $\alpha\in B_r\cap \im(\bC^{k-1}_c\to \bC^k_c)$, there exists $\beta\in \bC^{k-1}_c$ with $\supp(\beta)\subseteq N_{s_r}(x_0)$ and $\delta\beta=\alpha$.  We pick an increasing function $\Omega:\bbN\to\bbN$ such that
	\[\Omega(r)\geq i_1+\Upsilon(s_r)+r.\]
	We claim $\delta:\bC^{k-1}_c\to \bC^{k}_c$ has $\Omega$-uniform preimages. To see this,
	we fix $\alpha\in \im(\bC^{k-1}_c\to \bC^k_c)$ with $\supp(\alpha)\subseteq N_r(x)$, and will show there exists $\beta\in \bC^{k-1}_c$ with $\supp(\alpha)\subseteq N_{\Omega(r)}(x)$ and $\delta\beta=\alpha$.
	We first observe that \begin{align*}
		\supp((f_x)^\#(\alpha))\subseteq N_{i_1}(g_x(\supp(\alpha)))\subseteq N_{i_1+\Upsilon(r)+A}(x_0),
	\end{align*}
	so that $(f_x)^\#(\alpha)\in B_r\cap \im(\bC^{k-1}_c\to \bC^k_c)$.
	Thus there exists $\beta'\in \bC^{k-1}_c$ with $\supp(\beta')\subseteq N_{s_r}(x_0)$ and $\delta\beta'=(f_x)^\#\alpha$. We set $\beta=(g_x)^\#\beta'-(h_x)^\#\alpha$, so that \[\supp(\beta)\subseteq N_{i_1}(f_{x}(\supp(\beta')))\cup N_{i_1+r}(x)\subseteq N_{\Omega(r)}(x)\]
	and \[\delta\beta=(g_x)^\#(f_x)^\#\alpha-\delta(h_x)^\#\alpha =\alpha+(h_x)^\#\delta\alpha=\alpha.\] Thus the coboundary map $\delta:\bC^{k-1}_c\to \bC^{k}_c$ has $\Omega$-uniform preimages.

	(\ref{item:full_support}): Pick a collection $\alpha_1,\dots,\alpha_m$ of cocycles in $\bC^k_c$ such that $[\alpha_1], \dots, [\alpha_m]$ generate the finitely generated $R$-module $H^k(\bC^\bullet_c)$. We can thus pick $D_0$ large enough so that $\supp(\alpha_i)\subseteq N_{D_0}(x_0)$ for all $1\leq i\leq m$. Set $D\coloneqq i_1+\Upsilon(D_0)$. Then for each $x\in X$ and $1\leq i\leq m$, we have
	\begin{align*}
		\supp(g^\#_x\alpha_i)\subseteq N_{i_1}(f_x(\supp(\alpha_i)))\subseteq N_{i_1}(f_x(N_{D_0}(x_0)))\subseteq N_{D}(x).
	\end{align*}

	Now for each $x\in X$, as $g_\#^x$ is a chain homotopy equivalence, it follows that $g^\#_x:\bC^\bullet_c\to \bC^\bullet_c$ induces isomorphisms $g^*_x:H^k(\bC^\bullet_c)\to H^k(\bC^\bullet_c)$. Therefore, $[g^\#_x\alpha_1], \dots, [g^\#_x\alpha_m]$ generate $H^k(\bC^\bullet_c)$. It follows that  each class $[\alpha]\in H^k(\bC^\bullet_c)$ can be represented by a cocycle $\alpha_x$ that is an $R$-linear combination of $g^\#_x\alpha_1, \dots, g^\#_x\alpha_m$. By Lemma~\ref{lem:support_sum}, we deduce that $\supp(\alpha_x)\subseteq N_D(x)$ as required.
\end{proof}

A basic result and very useful result in group cohomology is that if $G$ is of type $FP$, then $\cd(G)=\max\{n\mid H^n(G,\bbZ G)\neq 0\}$~\cite[Proposition VIII.6.7]{brown1982cohomology}.
The following theorem is a generalisation of this result. However, our proof is quite different  from  the proof of the analogous result in group cohomology, which makes use of the fact that for a group $G$ of type $FP_n$, the functor $H^n(G,-)$ commutes with direct limits. This result has no immediate analogue for  general metric spaces, so we appeal to a more hands-on proof.
\begin{thm}\label{thm:ccd_finit_type_max}
	Assume $R$ is a Noetherian ring. Suppose $X$ is a coarsely homogeneous proper  metric space of coarse finite type over $R$. Then $\ccd(X)=\max\{\coarse^n(X;R)\neq 0\}$.
\end{thm}
\begin{proof}
	Let $n=\ccd_R(X)$ and $m=\max\{\coarse^n(X;R)\neq 0\}$. It follows from Proposition~\ref{prop:ccd_chars} that $m\leq n$. For the converse, we assume for contradiction that $\coarse^n(X;R)=0$.
	Let $\bC_\bullet$ be a proper finite-height $n$-dimensional projective $R$-resolution over $X$.  Since $\bC_n$ is $n$-dimensional, the assumption $\coarse^n(X;R)=0$ ensures the coboundary map $\bC^{n-1}_c\xrightarrow{\delta}\bC^n_c$  is surjective. By Proposition~\ref{prop:unif_preimage_cobdry}, this coboundary map has $\Omega$-uniform preimages for some $\Omega$. We pick $\Psi$ such that $\bC_n$, $\bC_{n-1}$ and  $\partial:\bC_n\to \bC_{n-1}$ have displacement $\Psi$.

	Suppose that $\bC_n=(C_n,B, \delta, p, \cF)$ and $\bC_{n-1}=(C_{n-1},C, \epsilon, q, \cF')$ both have height $i_0$. By Lemma~\ref{lem:proj_basis_finheight},  $\bC_n$ and $\bC_{n-1}$ have projective bases $\{m_b\}_{b\in B}$ and  $\{n_c\}_{c\in C}$ respectively, of displacement $\Psi$ and satisfying $m_b=n_c=0$ for all $b\in B\setminus B(i_0)$ and $c\in C\setminus C(i_0)$.
	For each $b\in B$, we have $\delta_b\in \bC^n_c$ by Remark~\ref{rem:coord_boundedsupp}. Moreover,  Proposition~\ref{prop:duality_ht} ensures that $\supp_{\bC^n_c}(\delta_b)\subseteq N_{\Psi(i_0)}(p(b))$. Thus $\diam(\supp_{\bC^n_c}(\delta_b))\leq 2\Psi(i_0)$. Setting $r\coloneqq \Omega(i_0,2\Psi(i_0))+\Psi(i_0)$, we see there exists $\mu_b\in \bC^{n-1}_c$ such  that $\delta \mu_b=\delta_b$ and $\supp_{\bC^{n-1}_c}(\mu_b)\subseteq N_r(p(b))$. Since $\delta_{b}=0$ for $b\in B\setminus B(i_0)$, we may assume $\mu_b=0$ for all $b\in B\setminus B(i_0)$.

	We claim $\pi:\bC_{n-1}\to \bC_{n-1}$ given by $\pi(\tau)=\sum_{b\in B}\mu_b(\tau)\partial m_b$ for each $\tau\in \bC_{n-1}$ is a finite displacement retraction onto $\bN\coloneqq \im(\bC_{n}\xrightarrow{\partial}\bC_{n-1})$. By Proposition~\ref{prop:ccd_chars}, this claim would imply $\ccd_R(X)\leq n-1$, which is the desired contradiction.

	First let $\tau\in \bC_{n-1}$ and  suppose  $\mu_b(\tau)\neq 0$ for some $b\in B$. Then $b\in B(i_0)$,  and  Proposition~\ref{prop:duality_ht} implies that $d(\supp_{\bC^{n-1}_c}(\mu_b),\supp_{\bC_{n-1}}(\tau))\leq \Psi(i_0)$. This implies $p(b)\in N_{\Psi(i_0)+r}(\supp_{\bC_{n-1}}(\tau))$. Since $\bC_{n-1}$ is proper, we see that for each $\tau\in\bC_{n-1}$, there are only finitely many $b\in B$ with $\mu_b(\tau)\neq0$, and so $\pi$ is well-defined. Moreover,
	\begin{align*}\supp_{\bC_{n-1}}(\pi(\tau)) & \subseteq \bigcup_{\substack{b\in B \\ \mu_b(\tau)\neq 0}}\supp_{\bC_{n-1}}(\partial m_b)\\
                                           & \subseteq \bigcup_{\substack{b\in B \\ \mu_b(\tau)\neq 0}}N_{2\Psi(i_0)}(p(b))\subseteq N_{3\Psi(i_0)+r}(\supp_{\bC_{n-1}}(\tau)),\end{align*}
	showing that as $\bC_{n-1}$ has finite height, $\pi$ has finite displacement.
	Finally, if $\tau=\partial\sigma$ for some $\sigma\in \bC_n$, we have
	\begin{align*}
		\pi(\tau) & =\pi(\partial\sigma)=\sum_{b\in B}\mu_b(\partial \sigma)\partial m_b=\partial\left(\sum_{b\in B}(\delta\mu_b)( \sigma) m_b\right) \\
		          & =\partial\left(\sum_{b\in B}\delta_b( \sigma) m_b\right)=\partial \sigma=\tau
	\end{align*} showing that $\pi$ is a retraction onto $\bN$.
\end{proof}

Theorem~\ref{thm:ccd_finit_type_max} allows us to compare $\ccd$ with asymptotic dimension $\asdim(X)$ of a metric space. For background about asymptotic dimension, we refer the reader to~\cite{belldranishnikov2011asymptotic}. It was shown by Dranishnikov that if $G$ is a group of type $FP$, then $\vcd(G)\leq \asdim(G)$~\cite[Proposition 5.10]{dranishnikov2009cohom}. We generalise this to metric spaces:
\begin{prop}\label{prop:ccd_vs_asdim}
	Let $X$ be a proper coarsely homogeneous metric space of coarse finite type over $\bbZ$. Then $\ccd(X)\leq \asdim(X)$.
\end{prop}
\begin{proof}
	Let $n=\max\{k\mid HX^k(X;\bbZ)\neq0\}$, where $HX^k(X;\bbZ)$ is Roe's coarse cohomology.  Dranishnikov showed that $n\leq \asdim(X)$~\cite[Propositions 5.1 \& 5.6]{dranishnikov2009cohom}. By Theorem~\ref{thm:ccd_finit_type_max} and Proposition~\ref{prop:roe_iso}, $n=\ccd_\bbZ(X)$.
\end{proof}
Specializing to the case of groups, we deduce:
\begin{cor}
	Let $G$ be a group of type $FP_\infty$ with $\ccd_\bbZ(G)<\infty$. Then $\ccd_\bbZ(G)\leq \asdim(G)$.
\end{cor}

We can now strengthen Theorem~\ref{thm:univcoeff_ctblehyp} in the case that $X$ is uniformly acyclic over $R$.
\begin{thm}\label{thm:truncated_kunneth}
	Let $X$ be a metric space,  let $R$ be a PID, and let $n\in \bbN$. Assume $X$ is coarsely uniformly $(n-1)$-acyclic over $R$.  Suppose $\bC_\bullet$ is a proper projective $R$-resolution over $X$.
	\begin{enumerate}
		\item\label{item:truncated_kunneth1} Let $f:R\to S$ be a ring homomorphism and let $\bD_\bullet\coloneqq \bC_\bullet\otimes_R S$, which is a proper projective $S$-resolution over $X$ by Proposition~\ref{prop:change_rings_res}.  For each $k<n$, we have a natural short exact sequence \[0\to  H^k(\bC^\bullet_c) \otimes_R S\xrightarrow{\phi} H^k(\bD_c^\bullet)\to \Tor_1^R(H^{k+1}(\bC_c^\bullet),S)\to 0\]
		      and an injection $H^n(\bC^\bullet_c) \otimes_R S\xrightarrow{\phi} H^n(\bD_c^\bullet)$.
		      Moreover, the injections $\phi$ are given by $[\alpha]\otimes s\mapsto [\alpha\otimes s]$.
		\item\label{item:truncated_kunneth2}  For each $k<n$, we have a natural short exact sequence \[0\to \Ext^1_R(H^{k+1}(\bC_c^\bullet),R) \to H_k(\bC_\bullet^{\lf})\xrightarrow{\psi} \Hom_R(H^{k}(\bC_c^\bullet), R)\to 0\]
		      and a surjection $H_n(\bC_\bullet^{\lf})\xrightarrow{\psi} \Hom_R(H^{n}(\bC_c^\bullet), R)$. Moreover, the surjections $\psi$ are given by $\psi([\phi])([\alpha])=\phi(\alpha)$.
	\end{enumerate}
\end{thm}
This result should be thought of as a generalisation of the group-theoretic  Universal Coefficient Theorems for groups of type $FP_\infty$~\cite[\S 3]{bieri1981homological}. While there is significant overlap between Theorem~\ref{thm:univcoeff_ctblehyp} and Theorem~\ref{thm:truncated_kunneth} in light of  Proposition~\ref{prop:ctble_cohomology_fingen}, neither one implies the other, and it is necessary to use both to obtain the strongest results in Section~\ref{sec:coarsePDn}.
\begin{proof}
	By Corollary~\ref{cor:projres_indpt}, we may assume without loss of generality that $\bC_\bullet$ satisfies the hypothesis of  Lemma~\ref{lem:dirinv_lim_space}.
	We show that for each $k\leq n$, the inverse systems $\{H^k(\bC^\bullet_c(i))\}_i$ are \emph{stable} as defined in Definition~\ref{defn:stable}. After proving this, the  result will follow from Lemma~\ref{lem:dirinv_lim_space} and Theorem~\ref{thm:truncated_kunneth_chain}.

	By Proposition~\ref{prop:acyc_res}, there exists a projective $R$-resolution $\bE_\bullet$ over $X$ with $\bE_k$ of finite height for $k\leq n$. By Corollary~\ref{cor:projres_indpt}, there are chain maps $f_\#:\bC_\bullet\to \bE_\bullet$ and $g_\#:\bE_\bullet\to \bC_\bullet$ of displacement $\Phi$ and  chain homotopies $\id\stackrel{h_\#}{\simeq}g_\#f_\#$ and $\id\stackrel{}{\simeq}f_\#g_\#$ of displacement $\Phi$.

	For $i\leq j$, we let $m_\#^{i,j}:\bC_\bullet(i)\to \bC_\bullet(j)$ be the inclusion map. Since $\bE_k$ has finite height for $k\leq n$, we can pick $i_0$ such that $g_\#(\bE_k)\subseteq \bC_k(i_0)$ for $k\leq n$, and set $j_0\coloneqq\Phi(i_0)$. Let $f_\#^i$ and $h_\#^i$ be the restrictions  of $f_\#$ and $h_\#$ to $\bC_\bullet(i)$. For arbitrary $i$, set $j\coloneqq\Phi(i)$, and consider the map  \[m_\#^{i,j}, m_\#^{i_0,j}g_\#f_\#^i:\bC_\bullet(i)\to \bC_\bullet(j).\] We observe that since  $m_\#^{i_0,j}g_\#f_\#^i- m_\#^{i,j}=\partial h_\#^i+h_\#^i\partial$, these two maps
	induce identical maps \[f^*_i\circ g^*\circ m^*_{i_0,j}=m^*_{i,j}:H^k(\bC^\bullet_c(j))\to H^k(\bC^\bullet_c(i))\] on cohomology.
	We claim that $m^*_{i_0,i}$ restricts to an isomorphism \[\im(m^*_{i,j})\cong \im(m^*_{i_0,j_0}).\] Since this holds for arbitrary large $i$, it follows  that $\{H^k(\bC^\bullet_c(i))\}_i$ is a stable inverse system.

	To see this, first suppose $[\alpha]\in H^k(\bC^\bullet_c(j))$ with $m^*_{i_0,j}([\alpha])=0$. Then  $m^*_{i,j}([\alpha])=f^*_i g^* m^*_{i_0,j}([\alpha])=0$, showing that $m^*_{i_0,i}$ restricted to ${\im(m^*_{i,j})}$ is injective.
	Now suppose $[\beta]\in H^k(\bC^\bullet_c(j_0))$. Then \begin{align*}m^*_{i_0,j_0}({[\beta]})=f^*_{i_0} g^*m^*_{i_0,j_0}({[\beta]}) = m^*_{i_0,i}m^*_{i,j}f^*_j g^*m^*_{i_0,j_0}({[\beta]}),\end{align*} where last inequality follows because $f_\#^j m_\#^{i,j}m_\#^{i_0,i}=f_\#^{i_0}$. Since $[\beta]\in H^k(\bC^\bullet_c(j_0))$ was arbitrary, this shows $m^*_{i_0,i}$ maps ${\im(m^*_{i,j})}$  onto $\im(m^*_{i_0,j_0})$ as required.
\end{proof}

The $\cZ$-boundary was developed by Bestvina for certain torsion-free groups~\cite{bestvina1996zboundary}, and later generalised by Dranishnikov~\cite{dranishnikov2006bestvinamess}. In this article, a \emph{$\cZ$-boundary} will always refer to a $\cZ$-boundary in the sense of Dranishnikov~\cite{dranishnikov2006bestvinamess}.
For a hyperbolic group $G$, its $\cZ$-boundary is  its Gromov boundary. For a $\CAT(0)$ group $G$, its  $\cZ$-boundary is the visual boundary $\partial X$ of a proper $\CAT(0)$ space $X$ admitting a proper cocompact $G$-action.

If $Z$ is a topological space, let $\dim(X)$ be the Lebesgue covering dimension of $X$, and let $\dim_R(X)$ be the cohomological dimension of $X$ with coefficients in $R$, i.e.\ the largest $n$ such that there exists a closed subset $A\subseteq X$ such that $\check{ H}^n(X,A;R)\neq 0$, where $\check{ H}$ is  \v Cech cohomology. A theorem of Alexandroff says if $X$ is a compact metric space with $\dim(X)<\infty$, then $\dim(X)=\dim_\bbZ(X)$~\cite[Theorem 1.4]{dranishnikov2009cohom}.
\begin{prop}[{cf.~\cite[Proposition 3.9]{margolis2019codim1}}]\label{prop:zbdry}
	Suppose that $R$ is a commutative ring and $G$ is a group  such that  $\ccd_R(G)<\infty$. If  $G$ admits a $\cZ$-boundary $Z$, then $\dim_R(Z)+1=\ccd_R(G)$. Moreover, $\dim(Z)+1=\ccd(G)$.
\end{prop}
\begin{proof}
	Let $(\overline{X},Z)$ be a $\cZ$-structure on $G$ in the sense of~\cite{dranishnikov2006bestvinamess}. In particular, $G$ acts cocompactly and properly discontinuously on $X\coloneqq \overline{X}\setminus Z$. Since $X$ is uniformly contractible~\cite[Proposition 1]{dranishnikov2006bestvinamess},  Theorem~\ref{thm:type_FPn} implies $G$ is of type $FP_\infty$.
	As $\overline{X}$ is contractible,  the long exact sequence in cohomology yields isomorphisms $H^{n+1}_c(X;R)\cong H^{n}(Z;R)$ for all $n$, where $H^*_c$ is cohomology with compact supports; we refer to~\cite[Proposition 1.5]{bestvina1996zboundary} for details. Since $X$ is uniformly contractible, we have $H^{n+1}_c(X;R)\cong HX^{n+1}(X;R)\cong \coarse^{n+1}(X;R)$  by Proposition~\ref{prop:roe_iso} and~\cite[Proposition 3.33]{roe1993coarse}.  It follows that   $H^{n+1}(G,RG)\cong\coarse^{n+1}(G;R)\cong H^{n}(Z;R)$ as $R$-modules.  Theorem~\ref{thm:ccd_finit_type_max} and the main result of~\cite{dranishnikov2006bestvinamess} imply  that $\dim_R(Z)+1= \ccd_R(G)$ for any PID $R$. Moreover, Moran showed that the topological dimension of $Z$ is finite~\cite{moran2016finite}, so Alexandroff's theorem implies that $\dim(Z)+1= \ccd(G)$.
\end{proof}

\begin{cor}\label{cor:bdry}
	\bdry{}
\end{cor}
\begin{proof}
	By Proposition~\ref{prop:zbdry} and the discussion preceding it, it is sufficient to show $\ccd(G)<\infty$, which we do in both cases by appealing to  Proposition~\ref{prop:ccd_action}. In the case $G$ is hyperbolic, it acts geometrically on each Rips complex $P_r(G)$, which is contractible for $r$ sufficiently large~\cite[{Theorem III.$\Gamma$.3.21}]{bridson1999metric}. In the case $G$ is $\CAT(0)$, we know it acts properly and cocompactly on some finite-dimensional contractible simplicial complex~\cite[{Remark {III.$\Gamma$.3.27}}]{bridson1999metric}.
\end{proof}

\section{Spaces of coarse cohomological dimension one}\label{sec:ccd1}
\emph{In this section we prove Theorem~\ref{thm:ccd1}, characterising quasi-geodesic metric spaces $X$ of coarse cohomological dimension one as unbounded  quasi-trees.
}
\vspace{.3cm}

\begin{lem}\label{lem:geod-acyc}
	Let $R$ be a commutative ring with 1. If $X$ is a quasi-geodesic metric space, then $X$ is coarsely uniformly $0$-acyclic over $R$.
\end{lem}
\begin{proof}
	Suppose $X$ is quasi-isometric to a geodesic metric space $Y$. Then for each $y\in Y$ and $r\geq 0$, $N_r(y)$ is path-connected, and so the Rips complex $P_i(N_r(y))$ is connected for all $i$. This ensures the map $\widetilde{H}_0(P_i(N_r(x));R)\to \widetilde{H}_0(P_{i}(N_{r}(x));R) $ induced by inclusion is zero, and so $Y$ is coarsely uniformly $0$-acyclic over $R$. Since being coarsely uniformly $0$-acyclic over $R$ is invariant under coarse equivalence, $X$ is also coarsely uniformly $0$-acyclic over $R$.
\end{proof}

\begin{thm}\label{thm:ccd1}
	Let $X$ be a quasi-geodesic metric space. Then $\ccd_R(X)=1$ if and only if $X$ is quasi-isometric to an unbounded simplicial tree.
\end{thm}
We think of a simplicial tree as a metric space by taking the induced path metric in which each edge has length 1.
We prove Theorem~\ref{thm:ccd1} by appealing to Manning's bottleneck criterion:
\begin{thm}[{\cite[Theorem 4.6]{manning2005pseudo}}]\label{thm:bottleneck}
	Let $X$ be a geodesic metric space. Then $X$ is quasi-isometric to a simplicial tree if and only if there exists a constant $\Delta>0$ such that the following condition holds.  For all $x,y\in X$, there exists some $m\in X$ with $d(x,m)=d(y,m)=\frac{1}{2}d(x,y)$ such that every path from $x$ to $y$ intersects $B_\Delta(m)$.
\end{thm}
In order to prove Theorem~\ref{thm:ccd1}, we require a few lemmas.
The following lemma relates 1-chains in the standard resolution to paths:
\begin{lem}\label{lem:chain-path}
	Let $X$ be a metric space and let $\bC_\bullet$ be the ordered standard  $R$-resolution over $X$. If $x,y\in X$ and  $\partial \sigma=y-x$ for some  $\sigma\in \bC_1(i)$, then there exists a sequence $\{x_0,x_1,\dots, x_n\}\subseteq \supp(\sigma)\cup\{x,y\}$ with $x_0=x$, $x_n=y$ and $d(x_j,x_{j+1})\leq i$ for all $j$.
\end{lem}
\begin{proof}
	We note that $\sigma\in\bC_1(i)$ can be expressed as a finite sum $\sigma=\sum_{j=0}^k r_j[z_j,y_j]$, where for each $j$, we have $r_j\in R\setminus \{0\}$, $z_j,y_j\in X$ and  $d(z_j,y_j)\leq i$.
	By Lemma~\ref{lem:ordered_res_rips}, the chain complex $\bC_1(i)\to \bC_0(i)\to 0$ can be identified with the  simplicial chain complex of the 1-skeleton of $P_i(X)$, the Rips complex of $X$ with parameter $i$. A standard argument implies that $\sigma$ must be a sum of a path in $P_i(X)$ from $x$ to $y$, and a linear combination of closed loops in $P_i(X)$. Since $\supp(\sigma)=\{z_0,\dots, z_k\}$, this completes the proof.
\end{proof}
We also show that a path can be approximated by a 1-chain in the following sense:
\begin{lem}\label{lem:path_approx}
	Let $X$ be a metric space and $\bC_\bullet$ the ordered standard  $R$-resolution over $X$. Suppose  $\bC_\bullet$ is $(\lambda,\mu)$-uniformly $0$-acyclic over $R$ and   $p:[0,L]\to X$ is a  path from $x$ to $y$. Then there exists a 1-chain $\omega\in \bC_1(\lambda(0))$ with $\partial \omega=y-x$ and $\supp(\omega)\subseteq N_{\mu(0,1)}(\im(p))$.
\end{lem}
\begin{proof}
	Since $p$ is uniformly continuous, there is a constant $\delta>0$ such that $d(p(s),p(t))\leq 1$ whenever $|s-t|\leq \delta$. We thus pick a sequence $0=t_0<t_1<\cdots<t_N=L$ with $t_j-t_{j-1}\leq\delta$ for each $j$.
	Setting $x_j=p(t_j)$ for each $j$, we have a sequence $x=x_0,\dots, x_N=y$ with $d(x_{j-1},x_j)\leq 1$ for all $j$.
	Note that by Lemma~\ref{lem:standard}, each $x_{j}-x_{j-1}$ is a reduced 0-cycle in $\bC_\bullet(0,N_1(x_j))$. Thus there exists  some $\omega_j\in \bC_1(\lambda(0),N_{\mu(0,1)}(x_j))$ with $\partial \omega_j=x_j-x_{j-1}$. Therefore  $\supp(\omega_j)\subseteq N_{\mu(0,1)}(x_j)\subseteq N_{\mu(0,1)}(\im(p))$. Setting $\omega\coloneqq \sum_{j=1}^N\omega_j$, we see $\supp(\omega)\subseteq N_{\mu(0,1)}(\im(p))$ and $\partial\omega=y-x$ as required.
\end{proof}

\begin{proof}[Proof of Theorem~\ref{thm:ccd1}]
	First suppose $X$ is quasi-isometric to an unbounded simplicial tree $T$. It follows from Example~\ref{exmp:metric_complex} that the simplicial chain complex of $T$ with coefficients in $R$  can be endowed with the structure of projective $R$-resolution over $T$, hence $\ccd_R(T)\leq 1$. Since $T$ is unbounded, Proposition~\ref{prop:bdd_0dim} implies $\ccd_R(T)=1$.  As $\ccd_R(-)$ is invariant under quasi-isometries, it follows $\ccd_R(X)= 1$.

	Conversely, suppose $\ccd_R(X)=1$. Replacing $X$ by geodesic metric space quasi-isometric to $X$ if necessary, we can assume $X$ is itself a geodesic metric space. Let $\bC_\bullet$ be the standard $R$-resolution over $X$. By Proposition~\ref{prop:acyc_res} and Lemma~\ref{lem:geod-acyc}, we can choose $(\lambda,\mu)$ such that $\bC_\bullet$ is $(\lambda,\mu)$-uniformly $0$-acyclic over $R$.  By Lemma~\ref{lem:comp} and Proposition~\ref{prop:ccd_chars}, $\im(\bC_2\to \bC_1)\leq \bC_1$ has a complement $\bP\leq \bC_1$, and there is a finite displacement retraction $r_\#$ from $\bC_1$ to $\bP$. Let $\Phi$ be the displacement of $r_\#$. Since $\bP$ is a complement of $\im(\bC_2\to\bC_1)=\ker(\bC_1\to \bC_0)$,  $\partial|_\bP:\bP\to \bC_0$ is injective and  $\partial m=\partial r_\#m$ for each $m\in \bC_1$. Set $i_0\coloneqq \lambda(0)$, $i_1=\Phi(i_0)$, $D\coloneqq \mu(0,1)$, $D_1\coloneqq \Phi(i_0)+D$,  $D_2\coloneqq 3D_1+i_1$ and $\Delta\coloneqq D_1+D_2$.

	Let $x,y\in X$ and let $\gamma$ be a geodesic from $x$ to $y$ with midpoint $m$.   By Lemma~\ref{lem:path_approx}, there exists $\omega\in \bC_1(i_0)$ with $\supp(\omega)\subseteq N_D(\im(\gamma))$ and $\partial\omega=y-x$.
	Therefore, we deduce that $r_\#(\omega)\in \bC_1(i_1)$, that $\partial r_\#\omega=\partial \omega=y-x$ and that
	$\supp(r_\#\omega)\subseteq N_{\Phi(i_0)}(\supp(\omega))\subseteq N_{D_1}(\im(\gamma))$. It follows from Lemma~\ref{lem:chain-path} that there exists a sequence $x=x_0,\dots, x_n=y$ in $\supp(r_\#\omega)\cup \{x,y\}\subseteq N_{D_1}(\im(\gamma))$ with $d(x_{j-1},x_j)\leq i_1$ for all $j$.
	Let $\pi:X\to \im(\gamma)$ be a closest point projection, and set $y_j\coloneqq \pi(x_j)$ for each $j$. Then $d(y_{j-1},y_j)\leq  2D_1+i_1$. Let $\gamma_j$ be the geodesic segment $[y_{j-1},y_j]_\gamma$ along $\gamma$. Since $\gamma$ is a geodesic and the concatenation $\gamma_1 \cdots  \gamma_n$ is path from $x$ to $y$ contained in $\im(\gamma)$, it follows $m\in \gamma_j$ for some $j$. Thus \[d(x_j,m)\leq d(x_j,y_j)+d(y_j,m)\leq d(x_j,y_j)+d(y_j,y_{j-1})\leq 3D_1+i_1=D_2,\] and so $m\in N_{D_2}(\supp(r_\#\omega)\cup \{x,y\})$.

	Now let $p$ be any path from $x$ to $y$. By Lemma~\ref{lem:path_approx}, there exists $\tau\in \bC_1(i_0)$ with $\supp(\tau)\subseteq N_D(\im(p))$ and $\partial\tau=y-x$. Thus $\supp(r_\#\tau)\cup\{x,y\}\subseteq N_{D_1}(\im(p))$. However, since  $\partial r_\#\tau=\partial r_\#\omega=y-x$ and $\partial|_\bP$ is injective, we deduce $r_\#\tau=r_\#\omega$. Therefore, \[m\in N_{D_2}(\supp(r_\#\omega)\cup \{x,y\})\subseteq N_{D_2}(\supp(r_\#\tau)\cup \{x,y\}) \subseteq N_{\Delta}(\im(p)).\] In other words, every path $p$ from $x$ to $y$ must intersect the ball $B_\Delta(m)$. By Theorem~\ref{thm:bottleneck}, this condition guarantees that $X$ is quasi-isometric to a simplicial tree.
\end{proof}

\section{Cap and cup products}\label{sec:cupcap}
\emph{In this section, we develop a notion of cup and cap product for proper metric spaces, whose properties are described in Lemma~\ref{lem:cupcap_properties} and Proposition~\ref{prop:cupcapinf_properties}. In applications, both in Section~\ref{sec:coarsePDn} and in upcoming work, the utility of the theory  lies not in simply looking at the induced maps on the level of homology/cohomology, but on  the precise quantitative bounds on supports on the chain/cochain level.
}
\vspace{.3cm}

Throughout this section, we fix a  metric space $X$ and a proper projective $R$-resolution $\bC_\bullet$ over $X$. We equip $X\times X$ with the $\ell_1$-metric. We define the \emph{diagonal embedding} $\gamma:X\to X\times X$ by $\gamma(x)=(x,x)$, so that $d_{X\times X}(\gamma(x),\gamma(x'))=2d_X(x,x')$.
\begin{defn}
	A \emph{diagonal approximation} of $\bC_\bullet$ is a chain map $\Delta:\bC_\bullet\to \bC_\bullet\otimes \bC_\bullet$ that is  augmentation-preserving and of finite displacement over the  diagonal embedding $\gamma$.
\end{defn}
Proposition~\ref{prop:proj_res} implies that:
\begin{prop}\label{prop:diagonal_approx}
	The diagonal approximation exists and is unique up to a finite displacement chain homotopy equivalence.
\end{prop}

We now fix a diagonal approximation $\Delta$ and pick $\Phi$ large enough such that $\bC_\bullet$ and $\bC_\bullet\otimes \bC_\bullet$ have displacement $\Phi$, and $\Delta$ has displacement $\Phi$ over $\gamma$.
We suppose $\bC_k=(C_k,B_k,\delta^k,p_k,\{B_k(i)\})$ and each $\bC_k$ is equipped with a projective basis $\{m_b^k\}_{b\in B_k}$ satisfying the conclusion of Lemma~\ref{lem:proj_basis_finheight} when $\bC_k$ has finite height.

\begin{defn}
	Given $\alpha\in \bC^j$ and $\beta\in \bC^k$, we define the \emph{cup product} $\alpha\smile \beta\in \bC^{j+k} $ by \[(\alpha\smile \beta)(\sigma)=(\alpha\otimes\beta)\Delta(\sigma)\] for each $\sigma\in \bC_{j+k}$.
	If $\tau\in \bC_n$ and $\alpha\in \bC^j$ with $j\leq n$, we define the \emph{cap product} $\tau\frown\alpha\in \bC_{n-j}$ by $(\alpha\otimes 1)\Delta\tau$.
\end{defn}

\begin{rem}
	For explicitness, we note that if $\alpha\in \bC^j$,  $\beta\in \bC^k$ and $\sigma_1\otimes \sigma_2\in \bC_{l}\otimes \bC_{j+k-l}$, then $(\alpha\otimes \beta)(\sigma_1\otimes \sigma_2)$ is equal to $\alpha(\sigma_1)\beta(\sigma_2)$ if $l=j$ and is equal to zero otherwise. Similarly,  $(\alpha\otimes 1)(\sigma_1\otimes \sigma_2)$ is equal to $\alpha(\sigma_1)\sigma_2$ if $l=j$ and is equal to zero otherwise.
\end{rem}

The following properties of cup and cap products will be used extensively:
\begin{lem}\label{lem:cupcap_properties}
	There exists an increasing function  $\Psi:\bbN\to\bbN$, depending only on $\Phi$, such that for all  $\alpha\in \bC^j$, $\beta\in \bC^k$ and $\tau\in \bC_n$ with $n\geq j$,  the following hold:
	\begin{enumerate}
		\item\label{item:cupcap_properties_bdry_cup}
		      $\delta(\alpha\smile \beta)=(-1)^k\delta\alpha\smile\beta+\alpha\smile\delta\beta$.
		\item  $\partial(\tau\frown\alpha)=\tau\frown\delta\alpha+(-1)^j\partial\tau\frown\alpha$.\label{item:cupcap_properties_bdry_cap}
		\item\label{item:cupcap_properties_cup}
		      For each $i\in \bbN$, \begin{align*}
			      \supp(\alpha\smile \beta,i)\subseteq N_{\Psi(i)}(\supp(\alpha,\Psi(i)))\cap N_{\Psi(i)}(\supp(\beta,\Psi(i))).
		      \end{align*}
		\item\label{item:cupcap_properties_cap}
		      If $\tau\in \bC_n(i)$, then $\tau\frown\alpha\in \bC_{n-j}(\Psi(i))$ and \[\supp(\tau\frown \alpha)\subseteq N_{\Psi(i)}(\supp(\tau))\cap N_{\Psi(i)}(\supp(\alpha,\Psi(i))).\]
		\item\label{item:cupcap_properties_aug}
		      If $j=n$, $\partial\tau=0$ and $\delta\alpha=0$, then $\epsilon(\tau\frown\alpha)=\alpha(\tau)$.

	\end{enumerate}
\end{lem}
\begin{rem}
	The sign conventions adopted mean the formulae (\ref{item:cupcap_properties_bdry_cup}) and (\ref{item:cupcap_properties_bdry_cap}) are somewhat non-standard. They were  chosen specifically so that in the case where $\partial \tau=0$, the map $g_\#:\bC^{n-\bullet}\to \bC_\bullet$ defined by $\alpha\mapsto \tau\frown \alpha$ is a chain map. This fact will be used extensively in Section~\ref{sec:coarsePDn}.
\end{rem}
\begin{proof}
	Let $\bD_\bullet=\bC_\bullet\otimes \bC_\bullet$.
	For $\sigma\in \bD_n$, let $\sigma_j$ denote the component of $\sigma$ contained in $\bC_j\otimes \bC_{n-j}$.  For each $i,s\in \bbN$, we let $i_s\coloneqq \Phi^s(i)$.

	(\ref{item:cupcap_properties_bdry_cup}): Let $\sigma\in \bC_{j+k+1}$. Then \begin{align*}
		\delta(\alpha\smile\beta)(\sigma) & =(-1)^{j+k+1}(\alpha\smile\beta)(\partial\sigma)=
		(-1)^{j+k+1}(\alpha\otimes\beta)\Delta(\partial\sigma)                                                                                                      \\
		                                  & =(-1)^{j+k+1}(\alpha\otimes\beta)((\partial\otimes 1) (\Delta\sigma)_{j+1}+(-1)^j(1\otimes \partial)(\Delta\sigma)_{j}) \\
		                                  & =(-1)^{k}(\delta\alpha\otimes \beta)(\Delta\sigma)_{j+1}+(\alpha\otimes \delta\beta)(\Delta\sigma)_j                    \\
		                                  & =\left((-1)^k\delta\alpha\otimes \beta +\alpha\otimes \delta\beta\right)(\Delta\sigma)                                  \\
		                                  & =\left((-1)^k\delta\alpha\smile\beta+\alpha\smile\delta\beta\right)(\sigma)
	\end{align*} as required.

	(\ref{item:cupcap_properties_bdry_cap}): \begin{align*}
		\partial\tau\frown\alpha & =(\alpha\otimes 1)\Delta\partial\tau=(\alpha\otimes 1)\left((\partial\otimes 1) (\Delta\tau)_{j+1}+(-1)^j(1\otimes \partial)(\Delta\tau)_{j}\right) \\
		                         & =(-1)^{j+1}(\delta\alpha\otimes 1)(\Delta\tau)+(-1)^j\partial(\alpha\otimes 1)(\Delta\tau)                                                          \\
		                         & =(-1)^j\left(-\tau\frown\delta\alpha+\partial(\tau\frown\alpha)\right)
	\end{align*}
	so that $\partial(\tau\frown\alpha)=\tau\frown\delta\alpha+(-1)^j\partial\tau\frown\alpha$.

	(\ref{item:cupcap_properties_cup}):
	Suppose $x\in \supp(\alpha\smile \beta,i)$. Then there is some $b\in p_{j+k}^{-1}(x)\cap B_{j+k}(i)$ with $(\alpha\smile\beta)(m^{j+k}_b)=(\alpha\otimes \beta)(\Delta (m^{j+k}_b))\neq 0$. Thus $(\alpha\otimes \beta)(\Delta (m^{j+k}_b)_j)\neq 0$.
	Since $m^{j+k}_b\in \bC(i_1)$, we have $	\Delta(m^{j+k}_b)\in \bD_{j+k}(i_2)$ and so Proposition~\ref{prop:tensor_overX} implies
	\begin{align*}
		\Delta(m^{j+k}_b)_j=\sum_{(b_1,b_2)\in B_j(i_2)\times B_k(i_2)}(\delta^j_{b_1}\otimes \delta^k_{b_2})(\Delta (m^{j+k}_b))m^j_{b_1}\otimes m^k_{b_2}.\label{eqn:cup_prod_splt}
	\end{align*}
	As $(\alpha\otimes \beta)(\Delta (m^{j+k}_b)_j)\neq 0$, there exist $b_1\in B_j(i_2)$ and $b_2\in B_k(i_2)$ such that:
	\begin{itemize}
		\item $(\delta_{b_1}\otimes \delta_{b_2})(\Delta (m^{j+k}_b)_j)\neq 0$;
		\item $(\alpha\otimes\beta)(m^j_{b_1}\otimes m^k_{b_2})=\alpha(m^j_{b_1})\beta(m^k_{b_2})\neq 0$.
	\end{itemize}
	Setting $x_1\coloneqq p_j(b_1)$ and $x_2\coloneqq p_k(b_2)$, this implies $(x_1,x_2)\in \supp(\Delta(m^{j+k}_b))$, $x_1\in \supp(\alpha,i_2)$ and $x_2\in \supp(\beta,i_2)$.
	As $\Delta$ has finite displacement over $\gamma$, \begin{align*}
		\supp(\Delta(m^{j+k}_b))\subseteq N_{i_2}(\gamma(\supp(m^{j+k}_b)))\subseteq N_{i_2}(\gamma(N_{i_1}(x)))\subseteq N_{i'}(x,x),
	\end{align*} where $i'\coloneq i_2+2i_1$. This implies $x_1,x_2\in N_{i'}(x)$, and so \[\supp(\alpha\smile \beta,i)\subseteq N_{i'}(\supp(\alpha,i_2))\cap N_{i'}(\supp(\beta,i_2)).\]

	(\ref{item:cupcap_properties_cap}):
	Suppose $\tau\in \bC_n(i)$. Since $\Delta$ has displacement $\Phi$, we can apply Proposition~\ref{prop:tensor_overX} to see that $\Delta(\tau)\in \bD_n(i_1)\subseteq \bC_\bullet(i_2)\otimes \bC_\bullet(i_2)$. This implies \[\tau\frown\alpha=(\alpha\otimes 1)(\Delta(\tau))\in (\alpha\otimes 1)\left(\bC_{j}(i_2)\otimes \bC_{n-j}(i_2)\right)\subseteq \bC_{n-j}(i_2)\]  as required.

	Now suppose $x\in \supp(\tau\frown \alpha)=\supp((\alpha\otimes 1)\Delta\tau)$.
	Since $\tau\in \bC_n(i)$, we deduce $\tau=\sum_{b\in B_n(i)}\delta^n_b(\tau)m^n_b$. Thus $x\in \supp((\alpha\otimes 1)\Delta m^n_b)$ for  some $b\in B_n(i)$ with $\delta^n_b(\tau)\neq 0$.
	Since $m^n_b\in \bC_\bullet(i_1)$, we see $\Delta(m^n_b)\in \bD_n(i_2)$ and so
	\begin{align*}
		\Delta(m^n_b)_j=\sum_{(b_1,b_2)\in B_j(i_2)\times B_{n-j}(i_2)}(\delta^j_{b_1}\otimes \delta^{n-j}_{b_2})(\Delta (m^n_b))m^j_{b_1}\otimes m^{n-j}_{b_2}.
	\end{align*}
	Since $x\in \supp((\alpha\otimes 1)\Delta (m^n_b))$, there exist $b_1\in B_j(i_2)$ and $b_2\in B_{n-j}(i_2)$ such that:
	\begin{itemize}
		\item $(\delta^j_{b_1}\otimes \delta^{n-j}_{b_2})(\Delta (m^n_b))\neq 0$,
		\item $x\in \supp((\alpha\otimes 1)(m^j_{b_1}\otimes m^{n-j}_{b_2}))=\supp(\alpha(m^j_{b_1})m^{n-j}_{b_2})$. In particular, $\alpha(m^j_{b_1})\neq 0$ and $x\in \supp(m^{n-j}_{b_2})$.
	\end{itemize}
	We now set $x_0=p_n(b)$, $x_1=p_{j}(b_1)$ and $x_2=p_{n-j}(b_2)$, observing that $x_0\in \supp(\tau)$, that $x_1\in \supp(\alpha,i_2)$, and that $x\in \supp(m_{b_2}^{n-j})\subseteq N_{i_3}(x_2)$. Moreover,
	\begin{align*}
		(x_1,x_2) & \in \supp(\Delta(m_b^n))\subseteq N_{i_2}(\gamma(\supp(m_b^n)))   \\
		          & \subseteq N_{i_2}(\gamma(N_{i_1}(x_0)))\subseteq N_{i'}(x_0,x_0),
	\end{align*}
	where $i'\coloneqq i_2+2i_1$. Thus $x_1,x_2\in N_{i'}(x_0)$. We conclude that
	\begin{align*} x\in N_{i_3+i'}(x_0) & \subseteq N_{i_3+i'}(x_0)\cap N_{i_3+2i'}(x_1)                       \\
                                    & \subseteq N_{i_3+i'}(\supp(\tau))\cap N_{i_3+2i'}(\supp(\alpha,i_2))\end{align*}
	as required.

	(\ref{item:cupcap_properties_aug}): Since $\Delta$ is augmentation-preserving, $\epsilon=(\epsilon\otimes\epsilon)\circ \Delta=\epsilon\circ (1\otimes \epsilon)\circ \Delta$ on $\bC_0$. Thus $f_\#\coloneqq (1\otimes \epsilon)\circ \Delta:\bC_\bullet\to \bC_\bullet$ is an augmentation-preserving chain map of finite displacement over the identity, hence there is a finite displacement chain homotopy $\id\stackrel{h_\#}{\simeq} f_\#$.
	Thus \begin{align*}
		\epsilon(\tau\frown\alpha) & =\epsilon((\alpha\otimes 1)\Delta\tau)=\alpha((1\otimes \epsilon)\Delta\tau)=\alpha (f_\#\tau) \\
		                           & =\alpha(\tau+\partial h_\#\tau+h_\#\partial\tau)=\alpha(\tau),
	\end{align*} where the last equality holds because  $\delta\alpha=0$ and $\partial\tau=0$.\end{proof}
Proposition~\ref{prop:cupcapinf_properties} implies:
\begin{cor}
	The cup and cap product in Proposition~\ref{prop:cupcapinf_properties} descend to a  cup  product \[\smile:H^j(\bC^\bullet)\times H^k(\bC^\bullet)\xrightarrow{[\alpha]\times [\beta]\to [\alpha\smile \beta]} H^{j+k}(\bC^\bullet)\] and a cap product
	\[\frown:H_n(\bC_\bullet)\times H^j(\bC^\bullet)\xrightarrow{[\tau]\times [\alpha]\to [\tau\frown \alpha]} H_{n-j}(\bC_\bullet).\]
\end{cor}

We recall from Proposition~\ref{prop:induced_proj_lf} that we identify $\bC_\bullet$ as a submodule of $\bC_\bullet^{\lf}$ via the double dual map. We thus show that the cap products extends as follows:
\begin{prop}\label{prop:cupcapinf_properties}
	The cap product  $\bC_n\times \bC^j\to \bC_{n-j}$ extends linearly to a cap product \[\bC_n^{\lf}\times \bC^j\to \bC_{n-j}^{\lf}\] such the following holds. There exists a function  $\Psi:\bbN\to\bbN$, depending only on $\Phi$, such that for all  $\alpha\in \bC^j$ and $\tau\in \bC_n^{\lf}$ with $n\geq j$,  the following hold:
	\begin{enumerate}
		\item\label{item:cupcapinf_properties_bdr} $\partial(\tau\frown\alpha)=\tau\frown\delta\alpha+(-1)^j\partial\tau\frown\alpha$.
		\item\label{item:cupcapinf_properties_cap}
		      If $\tau\in \bC_n^{\lf}(i)$, then $\tau\frown\alpha\in \bC_{n-j}^{\lf}(\Psi(i))$ and \[\supp(\tau\frown \alpha)\subseteq N_{\Psi(i)}(\supp(\tau))\cap N_{\Psi(i)}(\supp(\alpha,\Psi(i))).\]
		\item\label{item:cupcapinf_properties_finitesum}
		      If $\alpha\in \bC^j_c$, then $\tau\frown \alpha\in \bC_{n-j}$.
		\item\label{item:cupcapinf_properties_augment}
		      Suppose that $j=n$, that $\alpha\in \bC^n_c$, that $\partial\tau=0$ and that $\delta\alpha=0$.  Then $\epsilon(\tau\frown\alpha)= \alpha(\tau)$.
	\end{enumerate}
\end{prop}
\begin{proof}
	Pick $\Psi$ such that Lemma~\ref{lem:cupcap_properties} hold. If $\tau\in \bC_n^{\lf}$, then $\tau(\Ann(\bC_n(i)))=0$ for some $i$. Then Proposition~\ref{prop:induced_proj_lf} ensures $\tau$ is given by the formal sum $\tau=\sum_{b\in B_n(i)}\delta^n_b(\tau)m^n_b$, interpreted as in Notation~\ref{not:infsum}.
	We  define \[\tau\frown\alpha\coloneqq \sum_{b\in B_n(i)}\delta^n_b(\tau)\left(m^n_b\frown \alpha\right).\]
	In order to show this is a well-defined element of $\bD_{n-j}^{\lf}$, we need to show the following.
	\begin{enumerate}[label= (\roman*)]
		\item\label{item:lf_cap_define1}
		      For each $\beta\in \bC^{n-j}_c$, there are only  finitely $b\in B_n(i)$ such that $\delta^n_b(\tau)\beta (m^n_b\frown \alpha)\neq 0$.  Therefore, $\tau\frown \alpha$ defines an element of $\Hom_R(\bC^{n-j}_c,R)$ given by \[\beta\mapsto \sum_{b\in B_n(i)}\delta^n_b(\tau)\beta (m^n_b\frown \alpha).\]
		\item\label{item:lf_cap_define2}
		      There is some $i_1\in \bbN$ such $(\tau\frown\alpha)(\Ann(\bC_{n-j}(i_1)))=0$.
	\end{enumerate} For~\ref{item:lf_cap_define1}, we note that if $b\in B_n(i)$, then $m^n_b\in \bC_n(\Phi(i))$, so that by Lemma~\ref{lem:cupcap_properties} we know that $m^n_b\frown \alpha\in \bC_{n-j}(i_1)$ and that $\supp(m^n_b\frown \alpha)\subseteq N_{i_1}(p_n(b))$, where $i_1\coloneqq \Psi(\Phi(i))+\Phi(i)$. We now fix $\beta\in \bC^{n-j}_c$. We see that if $\beta(m^n_b\frown \alpha)\neq 0$, then by Lemma~\ref{lem:pairing_support_finite}, $\emptyset\neq \supp(\beta,i_1)\cap \supp(m^n_b\frown \alpha)$, and so $p_n(b)\in N_{i_1}(\supp(\beta,i_1))$. Since $\bC_n$ is proper and $\supp(\beta,i_1)$ is finite by Proposition~\ref{prop:bdd_support}, we conclude that there are only finitely many $b\in B_n(i)$ with $\beta(m^n_b\frown \alpha)\neq 0$, proving~\ref{item:lf_cap_define1}.

	For~\ref{item:lf_cap_define2}, we note that if $\beta\in \Ann(\bC_{n-j}(i_1))$ with $i_1$ as above, then for each $b\in B_n(i)$, we have $m^n_b\frown \alpha\in \bC_{n-j}(i_1)$, and so $\beta(m^n_b\frown\alpha)=0$. This implies $(\tau\frown\alpha)(\Ann(\bC_{n-j}(i_1)))=0$ as required.

	Property (\ref{item:cupcapinf_properties_bdr}) now follows by linearity and Lemma~\ref{lem:cupcap_properties}.  For (\ref{item:cupcapinf_properties_cap}), we first note that since $(\tau\frown\alpha)(\Ann(\bC_{n-j}(i_1)))=0$, we have $\tau\frown\alpha\in \bC_{n-j}^{\lf}(\Phi(i_1))$ by Proposition~\ref{prop:induced_proj_lf}. Moreover, \begin{align*}
		\supp(\tau\frown\alpha) & =\left\{p_{n-j}(b')\,\middle\vert\, b'\in B_{n-j}, \sum_{b\in B_n(i)}\delta^n_b(\tau)\delta_{b'}^{n-j}\left(m^n_b\frown \alpha\right)\neq0 \right\} \\
		                        & \subseteq \bigcup_{\substack{b\in B_n(i)                                                                                                            \\p_n(b)\in \supp(\tau)}}\supp(m^n_b\frown\alpha)
	\end{align*}
	Since $\supp(m^n_b\frown\alpha)\subseteq N_{i_1}(p_n(b))\cap N_{i_1}(\supp(\alpha,i_1))$ by Lemma~\ref{lem:cupcap_properties} and the choice of $i_1$, it follows that \[\supp(\alpha\frown \tau)\subseteq N_{i_1}(\supp(\tau))\cap N_{i_1}(\supp(\alpha,i_1))\] as required for (\ref{item:cupcapinf_properties_cap}).

	If $\alpha\in \bC^j_c$, then Proposition~\ref{prop:bdd_support} implies $\supp(\alpha,i_1)$ is finite, and s Lemma~\ref{lem:cupcap_properties} implies that all but finitely many summands of  $\tau\frown\alpha= \sum_{b\in B_n(i)}\delta^n_b(\tau)\left(m^n_b\frown \alpha\right)$ are zero. This ensures that $\tau\frown \alpha\in \bC_{n-j}$ as required for (\ref{item:cupcapinf_properties_finitesum}).

	It remains to prove (\ref{item:cupcapinf_properties_augment}). Let $f_\#$ and $h_\#$ be as in the proof of  (\ref{item:cupcap_properties_aug}) of Lemma~\ref{lem:cupcap_properties}. In particular, we have $\alpha(f_\#(\sigma))=\epsilon(\sigma\frown \alpha)$ for all $\sigma\in \bC_n$, and $\partial h_\#+h_\#\partial =f_\#-\id$. Applying Lemma~\ref{lem:lf_induced}, we see that $f_\#$ can be extended to a chain map $f_\#^{\lf}:\bC_\bullet^{\lf}\to \bC_\bullet^{\lf}$, and $h_\#$ extends to a chain homotopy $h_\#^{\lf}:\bC_\bullet^{\lf}\to \bC_{\bullet+1}^{\lf}$ such that  $\partial h_\#^{\lf}+h_\#^{\lf}\partial =f_\#^{\lf}-\id$.
	Since $\alpha\in \bC^n_c$ and $\bC_n$ is proper, it follows that there are only finitely many $b\in B_n(i)$ with $m^n_b\frown \alpha\neq 0$. Thus \begin{align*}
		\epsilon(\tau\frown\alpha) & =\epsilon\left(\sum_{b\in B_n(i)}\delta^n_b(\tau)\left(m^n_b\frown \alpha\right)\right)=
		\sum_{b\in B_n(i)}\delta^n_b(\tau)\epsilon\left(m^n_b\frown \alpha\right)                                             \\
		                           & =\sum_{b\in B_n(i)}\delta^n_b(\tau)\alpha(f_\#(m^n_b))
		=\alpha\left(f_\#^{\lf}\left(\sum_{b\in B}\delta^n_b(\tau)m^n_b\right)\right)
		=\alpha(f_\#^{\lf}(\tau))                                                                                             \\
		                           & =\alpha(\tau+\partial h_\#^{\lf}\tau+h_\#^{\lf}\partial\tau)=\alpha(\tau)
	\end{align*}
	since $\partial\tau=0$ and $\delta\alpha=0$. This proves (\ref{item:cupcapinf_properties_augment}).
\end{proof}

\section{Coarse Poincar\'e duality spaces}\label{sec:coarsePDn}
\emph{In this section, we recall Kapovich--Kleiner's notion of  coarse Poincar\'e duality spaces~\cite{kapovich2005coarse}. We give several characterisations of these spaces, and prove Theorems~\ref{thm:main_field_intro} and~\ref{thm:main_finitenss_space_intro} from the introduction.
}
\vspace{.3cm}

We recall from Proposition~\ref{prop:duality_ht} that if $\bM$ is a proper projective $R$-module over $X$ of finite height, then so is $\bM^*_c$. Thus if $\bC_\bullet$ is a finite height $R$-chain complex over $X$, the notion of a chain map $f_\#:\bC_\bullet\to \bC^{n-\bullet}_c$ being of finite displacement makes sense. We thus define a coarse Poincar\'e duality group as follows:
\begin{defn}\label{defn:coarsePD_n}
	Given a natural number $n\geq 1$ and a commutative ring $R$, a proper metric space $X$ is a \emph{coarse $R$-Poincar\'e duality space of dimension $n$}, or coarse $PD_n^R$ space for short, if there exists a proper finite height projective $R$-resolution $\bC_\bullet$ over $X$ such that the following holds. There are finite displacement chain maps $f_\#:\bC_\bullet\to \bC^{n-\bullet}_c$ and $\overline{f}_\#:\bC^{n-\bullet}_c\to \bC_{\bullet}$ and finite displacement chain homotopies $\overline f_\#\circ f_\#\stackrel{h_\#}{\simeq} \id$ and $ f_\#\circ \overline f_\#\stackrel{\overline h_\#}{\simeq} \id$. We call any such  $\bC_\bullet$ a \emph{$PD_n^R$-complex over $X$}.
\end{defn}
\begin{lem}\label{lem:coarse_pdc_coarseequiv}
	Let $X$ be a coarse $PD_n^R$ space and let $\bC_\bullet$ be a coarse $PD_n^R$-complex over $X$. Suppose $X$ is coarsely equivalent to $Y$ and $\bD_\bullet$ is any proper finite height projective $R$-resolution over $Y$. Then $\bD_\bullet$ is  a coarse $PD_n^R$-complex over $Y$, and so $Y$ is a coarse $PD_n^R$ space.
	In particular, setting $X=Y$, we observe that any proper finite height projective $R$-resolution over $X$ is a coarse $PD_n^R$ complex.
\end{lem}
\begin{proof}
	This follows easily from Proposition~\ref{prop:proj_res}, using an argument similar to~\cite[Proposition 4.3]{margolis2018quasi}.
\end{proof}

\begin{exmp}
	We compare the notion of Poincar\'e duality spaces in Definition~\ref{defn:coarsePD_n} with the definitions considered by Kapovich--Kleiner~\cite{kapovich2005coarse}: \begin{itemize}
		\item Recall that a metric simplicial complex is a pair $(X^{(1)},X)$, where $X$ is a simplicial  complex and $X^{(1)}$ is its 1-skeleton endowed with the induced path metric~\cite{kapovich2005coarse}. In Definition 6.1 of~\cite{kapovich2005coarse}, Kapovich--Kleiner  define a Poincar\'e duality space to be a uniformly acyclic  metric simplicial complex $(X^{(1)},X)$ whose simplicial chain complex $C_\bullet(X)$ satisfies  properties analogous to those defined in Definition~\ref{defn:coarsePD_n}. By Example~\ref{exmp:metric_complex}, the simplicial  chain complex $C_\bullet(X)$ can be given the structure of a projective $\bbZ$-resolution over $X^{(1)}$, which is hence  a coarse $PD_n^\bbZ$ complex. Thus, for any coarse Poincar\'e duality space $(X^{(1)},X)$ of dimension $n$ in the sense of~\cite[Definition 6.1]{kapovich2005coarse}, the underlying metric space $X^{(1)}$ is a coarse $PD_n^R$ space in the sense of Definition~\ref{defn:coarsePD_n}.
		\item More generally, in $\S 11$ of~\cite{kapovich2005coarse}, Kapovich--Kleiner define \emph{metric complexes} $(X,\bC_\bullet)$, where $X$ is a metric space and $\bC_\bullet$ is a chain complex. Up to some freedom in the choice of filtrations, it is straightforward to check that there is a correspondence between proper finite-height free $\bbZ$-resolutions over $X$, and the chain complex $\bC_\bullet$ of a finite-type uniformly acyclic metric complex in the sense of $\S 11$ of~\cite{kapovich2005coarse}.
		      In $\S 11$ of~\cite{kapovich2005coarse}, Kapovich--Kleiner define a \emph{Poincar\'e duality space}
		      to be a metric complex  $(X,\bC_\bullet)$,
		      where $\bC_\bullet$ satisfies properties analogous to those defined in Definition~\ref{defn:coarsePD_n}. In particular, it is easy to deduce  $X$ is a coarse $PD_n^\bbZ$ space in the sense of Definition~\ref{defn:coarsePD_n} if and only if there is a metric complex $(X,\bC_\bullet)$ that is a coarse Poincar\'e duality space of dimension $n$ in the sense of~\cite[$\S 11$]{kapovich2005coarse}.
	\end{itemize}
\end{exmp}

We describe some immediate properties of coarse $PD_n^R$ spaces:
\begin{prop}\label{prop:pdn_properties}
	If $X$ is a coarse $PD_n^R$ space, the following hold:
	\begin{enumerate}
		\item\label{item:prop_pdn_properties2}
		      For each $k\in \bbZ$, \[\coarse^k(X;R)\cong\begin{cases}
				      R & \textrm{if}\;k=n,     \\
				      0 & \textrm{if}\;k\neq n.
			      \end{cases}\]
		\item\label{item:prop_pdn_properties1}
		      $X$ has coarse finite type over $R$  and $\ccd_R(X)=n$.
	\end{enumerate}
\end{prop}
\begin{proof}Let $\bC_\bullet$ be a coarse $PD_n^R$ complex and pick $f_\#$, $\overline f_\#$, $h_\#$ and $\overline h_\#$ as in Definition~\ref{defn:coarsePD_n}.

	(\ref{item:prop_pdn_properties2}):  Since $\overline f_\#:\bC^{n-\bullet}_c\to \bC_{\bullet}$ is a chain homotopy equivalence and \[H_0(\bC_\bullet)\cong\begin{cases}
			R & \textrm{if}\;k=0,     \\
			0 & \textrm{if}\;k\neq 0,
		\end{cases}\] the result follows.

	(\ref{item:prop_pdn_properties1}): The fact that $X$ is coarsely uniformly acyclic over  $R$ follows from Proposition~\ref{prop:acyc_res} and the fact that $\bC_\bullet$  is a  finite-height projective $R$-resolution over $X$. Applying Lemma~\ref{lem:ccd_pdn} below to the maps $f_\#:\bC_\bullet\to \bC^{n-\bullet}_c$,  $g_\#=\overline{f}_\#:\bC^{n-\bullet}_c\to\bC_\bullet$ and $\id\stackrel{h_\#}{\simeq}g_\#f_\#$, we conclude that $\ccd_R(X)\leq n$. Since $\coarse^n(X;R)=R\neq 0$, Proposition~\ref{prop:ccd_chars} implies $\ccd_R(X)=n$. By Proposition~\ref{prop:finite_type}, $X$ has coarse finite type over $R$.
\end{proof}

\begin{defn}
	A countable group $G$ is a  \emph{coarse $PD_n^R$ group} if, when equipped with a proper left-invariant metric, it is a coarse $PD_n^R$ space.
\end{defn}
We recall the following class of groups introduced by Bieri--Eckmann:
\begin{defn}[{\cite{bierieckmann1973duality,bierieckmann1974finiteness}}]\label{defn:pdn_group}
	A group  $G$ is a \emph{Poincar\'e duality group of dimension $n$ over the ring $R$}, or a  $PD_n^R$ group for short, if it is of type $FP_\infty^R$, $\cd_R(G)<\infty$ and \[H^k(G;RG)\cong\begin{cases}
			R & \textrm{if}\;k=n,     \\
			0 & \textrm{if}\;k\neq n.
		\end{cases}\]
\end{defn}

We now state some useful properties of coarse $PD_n^R$ groups and their relation to $PD_n^R$ groups.
\begin{prop}\label{prop:pdngroup_properties}
	Let $G$ be coarse $PD_n^R$ group. Then the following hold:
	\begin{enumerate}
		\item\label{item:prop_pdngp_properties2}
		      For each $k\in \bbZ$, \[H^k(G;RG)\cong\begin{cases}
				      R & \textrm{if}\;k=n,     \\
				      0 & \textrm{if}\;k\neq n.
			      \end{cases}\]
		\item\label{item:prop_pdngp_properties1}
		      $G$ is of type $FP_\infty^R$ and $\ccd_R(G)=n$.
		\item\label{item:prop_pdngp_properties3} Every $PD_n^R$ group is a coarse $PD_n^R$ group.
		\item\label{item:prop_pdngp_properties4} $G$ is a Poincar\'e duality group of dimension $n$ over the ring $R$ if and only if $\cd_R(G)<\infty$.
	\end{enumerate}
\end{prop}
\begin{proof}
	Let $\bC_\bullet$ be a proper projective $R$-resolution over $G$. We deduce (\ref{item:prop_pdngp_properties2}) and (\ref{item:prop_pdngp_properties1}) from Proposition~\ref{prop:pdn_properties},  Corollary~\ref{cor:induced_mod} and Theorem~\ref{thm:type_FPn}. Kapovich--Kleiner showed (\ref{item:prop_pdngp_properties3})~\cite[Example 11.6]{kapovich2005coarse}, whilst (\ref{item:prop_pdngp_properties4}) follows readily from  (\ref{item:prop_pdngp_properties2}), (\ref{item:prop_pdngp_properties1}) and Definition~\ref{defn:pdn_group}.
\end{proof}
\begin{rem}\label{rem:coarse_pdn}
	In~\cite[Example 11.6]{kapovich2005coarse}, Kapovich--Kleiner further clarify the relationship between coarse $PD_n^R$ groups and $PD_n^R$ groups. They observe that if $G$ is a coarse $PD_n^R$ group, then $G$ is a $PD_n^R$ group precisely when there exists a coarse $PD_n^R$ complex $\bC_\bullet$ as in Definition~\ref{defn:coarsePD_n}, admitting a free $G$-action, such that the maps $f_\#$, $g_\#$, $h_\#$ and $\overline h_\#$   are all $G$-equivariant.
\end{rem}

\begin{lem}\label{lem:ccd_pdn}
	Let $\bC_\bullet$ be a projective $R$-resolution over a metric space $X$. Suppose that for some $n\in \bbN$,  there exists a non-negative cochain complex  $D^\bullet$ and chain maps $f_\#:\bC_\bullet\to D^{n-\bullet}$ and $g_\#:D^{n-\bullet}\to \bC_\bullet$ such that the composition $g_\#f_\#:\bC_\bullet\to \bC_\bullet$ is augmentation-preserving and of finite displacement over $\id_X$.   Then $\ccd_R(X)\leq n$.
\end{lem}
In  Lemma~\ref{lem:ccd_pdn}, $D^\bullet$ is only assumed to be a complex of $R$-modules, not $R$-modules over $X$. Thus $f_\#$ and $g_\#$ are only required to be chain maps, not finite displacement chain maps.
\begin{proof}
	We first assume $n>0$. Since $g_\#f_\#$ has finite displacement and is augmentation-preserving, Proposition~\ref{prop:proj_res} implies there is a finite displacement chain homotopy $\id\stackrel{h_\#}{\simeq}g_\#f_\#$.
	Let $r\coloneqq g_n f_n+\partial h_n:\bC_n\to\bC_n$, which has finite displacement. For each $\sigma\in \bC_n$, we have \[\partial r(\sigma)=\partial g_n f_n\sigma+\partial^2 h_n\sigma=g_n\delta f_n\sigma=0,\] where the last equality follows because $f_n\sigma\in D^0$ and $\delta:D^0 \to D^{-1}=0$ is the zero map. Thus $\im(r)\subseteq \ker(\bC_{n}\xrightarrow{\partial} \bC_{n-1})=\im(\bC_{n+1}\xrightarrow{\partial} \bC_{n})$.
	Moreover, if $\sigma\in \im(\bC_{n+1}\xrightarrow{\partial} \bC_{n})$, then  \[\sigma=g_\#f_\#\sigma+h_\#\partial\sigma+\partial h_\#\sigma=r(\sigma).\]
	We thus see that  $r$ is a finite displacement retraction onto  $\im(\bC_{n+1}\xrightarrow{\partial} \bC_{n})$. Proposition~\ref{prop:ccd_chars} now implies $\ccd_R(X)\leq n$.

	In the case $n=0$, we suppose $g_\#f_\#$ has displacement $\Phi$. We know from weak uniform acyclicity of $\bC_\bullet$ that there is a constant $\Omega_0$ such that for each $x\in X$, there is some $m_x\in \bC_0(\Omega_0)$ with $\supp(m_x)\subseteq N_{\Omega_0}(x)$ and $\epsilon(m_x)=1$. Therefore $\supp(g_\#f_\#m_x)\subseteq N_{\Phi(\Omega_0)+\Omega_0}(x)$ for all $x\in X$. Moreover, $\supp(g_\#f_\#m_x)$ is non-empty since $\epsilon(g_\#f_\#m_x)=\epsilon(m_x)=1$ and so $g_\#f_\#m_x\neq 0$. For each $x,y\in X$, we have $m_x-m_y=\partial \omega$ for some $\omega\in \bC_1$, and so $f_\#(\omega)\in D^{-1}=0$. Therefore $f_\#(m_x-m_y)=\delta f_\#(\omega)=0$, and so $f_\#(m_x)=f_\#(m_y)$ for all $x,y\in X$. This implies \[\emptyset\neq \supp(g_\#f_\#m_x)=\supp(g_\#f_\#m_y)\subseteq N_{\Phi(\Omega_0)+\Omega_0}(x)\cap N_{\Phi(\Omega_0)+\Omega_0}(y).\] Thus $d(x,y)\leq 2\Phi(\Omega_0)+2\Omega_0$ for all $x,y\in X$, and so $X$ is bounded. Therefore, $\ccd_R(X)=0$ by Proposition~\ref{prop:bdd_0dim}.
\end{proof}

In the remainder of this section, we give several conditions that imply a space is a coarse $PD_n^R$ space. These all rely on the following  criterion. It is important to note that in the following lemma, we do not assume $X$ is uniformly acyclic.
\begin{lem}\label{lem:pdn_criterion_main}
	Let $R$ be a Noetherian ring, let $X$ be a coarsely homogeneous proper metric space and let $(\bC_\bullet,\epsilon)$ be a proper projective $R$-resolution over $X$. Let $n\geq 1$ and suppose $H^k(\bC^\bullet_c)=0$ for $k<n$ and there exists  $[\alpha]\in H^n(\bC^\bullet_c)$ and a chain map $f_\#:\bC_\bullet\to \bC^{n-\bullet}_c$ satisfying the following:
	\begin{enumerate}
		\item For each $\sigma\in \bC_0$, $[f_\#(\sigma)]=\epsilon(\sigma)[\alpha]$ in $H^n(\bC^\bullet_c)$.
		\item There is a function $\Phi:\bbN\to \bbN$ such that if $i\in \bbN$ and $\sigma\in \bC_k(i)$, then \[\supp(f_\#(\sigma),i)\subseteq N_{\Phi(i)}(\supp(\sigma)).\]
	\end{enumerate}
	Suppose also there exists some $[\tau]\in H_n(\bC^{\lf}_\bullet)$ such that $\tau(\alpha)=1$. Then $X$ is a coarse $PD_n^R$ space.
\end{lem}
\begin{proof}
	Since the proof is fairly long, we split it into three  steps.

	\textbf{Step 1: $X$ is coarsely uniformly acyclic over $R$.}\\
	We pick $\Phi_0:\bbN\to \bbN$ such that  $\bC_\bullet$ has displacement $\Phi_0$. Fix a diagonal approximation $\Delta:\bC_\bullet\to \bC_\bullet\otimes \bC_\bullet$  and choose $\Psi:\bbN\to \bbN$ large enough so that Lemma~\ref{lem:cupcap_properties} and Proposition~\ref{prop:cupcapinf_properties} hold. Pick $\Omega:\bbN\times \bbN\to \bbN$ such that $\bC_\bullet$ is $\Omega$-weakly uniformly acyclic.
	Pick $j_0$ large enough so that $\tau\in \bC^{\lf}_n(j_0)$, and set $j_1\coloneqq \Psi(j_0)$. We define $g_\#:\bC^{n-\bullet}_c\to \bC_\bullet$ by $g_\#(\alpha)=\tau\frown\alpha$. Since $\partial\tau=0$, Proposition~\ref{prop:cupcapinf_properties} ensures that $g_\#$ is a chain map.

	We   claim that the composition $g_\#f_\#$ is an augmentation-preserving chain map of finite displacement over $\id_X$. Firstly, suppose $\sigma\in \bC_0$.
	Then using  Proposition~\ref{prop:cupcapinf_properties} and the hypotheses on $\tau$ and on $f_\#$, we see \begin{align*}
		\epsilon(g_\#f_\#(\sigma))=\epsilon(\tau\frown f_\#(\sigma))=\tau(f_\#(\sigma))=\epsilon(\sigma)\tau(\alpha)=\epsilon(\sigma).
	\end{align*}
	For each $i\in \bbN$ and $\sigma\in \bC_k(i)$,  Proposition~\ref{prop:cupcapinf_properties} implies:
	\begin{itemize}
		\item $g_\#f_\#(\sigma)\in \bC_k(j_1)$;
		\item $\supp(g_\#f_\#(\sigma))\subseteq N_{j_1}(\supp(f_\#\sigma,j_1))\subseteq N_{j_2}(\supp(\sigma))$,
	\end{itemize}
	where $j_2\coloneqq j_1+\Phi(j_1)$.
	This shows that $g_\#f_\#$ has finite displacement. Proposition~\ref{prop:proj_res} now implies that there is finite-displacement chain homotopy $\id\stackrel{h_\#}{\simeq} g_\#f_\#$, which we may assume is of displacement $\Phi_1$.

	We define \begin{enumerate}
		\item $\lambda(i)\coloneqq  \max(j_1,\Phi_1(i))$ and
		\item $\mu(i,r)\coloneqq \max(\Phi_1(i)+r,j_2+\Omega(i,r)+r)$.
	\end{enumerate}
	We  show that for each  $x\in X$,  $i,r\in \bbN$, and reduced cycle $\sigma\in \bC_k(i,N_r(x))$, there is some $\omega\in \bC_{k+1}(\lambda(i))$ with $\supp(\omega)\subseteq N_{\mu(i,r)}(x)$ such that $\partial \omega=\sigma$. By Lemma~\ref{lem:standard} and Proposition~\ref{prop:acyc_res}, this will imply that $X$ is uniformly acyclic.
	We thus fix some reduced  cycle $\sigma\in \bC_k(i,N_r(x))$.
	As $\bC_\bullet$ is $\Omega$-weakly uniformly acyclic, there exists some $\omega'\in \bC_{k+1}(\Omega(i,r),N_{\Omega(i,r)}(x))$ with $\partial \omega'=\sigma$.
	We set $\omega\coloneqq g_\#f_\#\omega'-h_\#\sigma$ and first note that \begin{align*}
		\partial \omega=\partial g_\#f_\#\omega'-\partial h_\#\sigma=g_\#f_\#\sigma-\partial h_\#\sigma=\sigma+h_\#\partial\sigma=\sigma.
	\end{align*}
	By using the properties of $g_\#f_\#$ described above, we have $g_\#f_\#(\omega')\in \bC_{k+1}(j_1)\subseteq \bC_{k+1}(\lambda(i))$ and \[\supp(g_\#f_\#(\omega'))\subseteq N_{j_2}(\supp(\omega'))\subseteq N_{j_2+\Omega(i,r)}(\supp(\sigma))\subseteq N_{\mu(i,r)}(x)\]
	Moreover, as $h_\#$ has displacement $\Phi_1$, we see $h_\#(\sigma)\in \bC_{k+1}(\Phi_1(i))\subseteq \bC_{k+1}(\lambda(i))$  and
	\[\supp(h_\#\sigma)\subseteq N_{\Phi_1(i)}(\supp(\sigma))\subseteq N_{\Phi_1(i)+r}(x)\subseteq N_{\mu(i,r)}(x).\] We have thus shown $\partial \omega=\sigma$ with $\omega\in \bC_{k+1}(\lambda(i))$ and $\supp(\omega)\subseteq N_{\mu(i,r)}(x)$ as required. This implies $X$ is coarsely uniformly acyclic over $R$.

	\textbf{Step 2: Reduction to the case $\bC_\bullet$ is $n$-dimensional and of finite height.}\\
	Applying Lemma~\ref{lem:ccd_pdn} to the maps $f_\#$ and  $g_\#$, we conclude that $\ccd_R(X)\leq n$. By Proposition~\ref{prop:finite_type}, there exists an $n$-dimensional finite-height proper $R$-resolution $\bD_\bullet$ of $X$. We want to show all the hypotheses of the lemma we are trying to prove, which we assume hold for a projective $R$-resolution $\bC_\bullet$, also hold for $\bD_\bullet$.

	By Proposition~\ref{prop:proj_res}, there are finite displacement augmentation-preserving chain maps $l_\#:\bC_\bullet\to\bD_\bullet$ and $m_\#:\bD_\bullet\to \bC_\bullet$ and chain homotopy equivalences $\id\stackrel{i_\#}{\simeq} m_\#l_\#$ and $\id\stackrel{ j_\#}{\simeq} l_\#m_\#$, and we assume all these maps have displacement $\Phi_2$. By increasing $\Phi_2$, we may assume $\bC_\bullet$ and $\bD_\bullet$  also have displacement $\Phi_2$.

	We define $\hat f_\#:\bD_\bullet\to \bD^{n-\bullet}_c$ to be the composition
	\[\bD_\bullet\xrightarrow{m_\#} \bC_\bullet\xrightarrow{f_\#}\bC^{n-\bullet}_c\xrightarrow{m^\#_c}\bD^{n-\bullet}_c ,\] and define $[\hat\alpha]\coloneqq [m^\#_c\alpha]\in H^n(\bD^n_c)$.
	For each $\sigma\in \bD_0$, we have \begin{align*}
		[\hat f_\#(\sigma)]=[m^\#_c f_\#m_\#\sigma] = m^*_c([f_\#m_\#\sigma])=m^*_c(\epsilon(m_\#\sigma)[\alpha])=\epsilon(\sigma)[\hat \alpha]
	\end{align*} as required.
	For each $\sigma\in \bC_D(i)$, we apply Lemma~\ref{lem:local_support_disp} to deduce
	\begin{align*}
		\supp(\hat f_\#(\sigma),i) & =\supp(m^\#_c f_\#m_\#(\sigma),i)                                        \\
		                           & \subseteq N_{\Phi_2^2(i)+\Phi_2(i)}(\supp(f_\#m_\#(\sigma),\Phi_2^2(i))) \\ &\subseteq N_{\Phi_2^2(i)+\Phi_2(i)+\Phi(\Phi_2^2(i))}(\supp(m_\#(\sigma)))\\
		                           & \subseteq N_{\hat\Phi(i)}(\supp(\sigma)),
	\end{align*}
	where $\hat\Phi(i)\coloneqq\Phi_2^2(i)+2\Phi_2(i)+\Phi(\Phi_2^2(i))$. Moreover, setting $\hat\tau\coloneqq l^{\lf}_\#(\tau)$, the proof of  Proposition~\ref{prop:pairing_indep} implies $\hat\tau(\hat\alpha)=\tau(\alpha)=1$. Finally, since $m_\#$ is a finite displacement chain homotopy equivalence, it induces isomorphism $H^k(\bC^\bullet_c)\cong H^k(\bD^\bullet_c)$ for each $k$, so that $H^k(\bD^\bullet_c)=0$ for $k<n$.  We have thus shown that all the hypotheses of the lemma hold after replacing $\bC_\bullet$, $f_\#$, $[\alpha]$, $[\tau]$ and $\Phi$ with $\bD_\bullet$, $\hat f_\#$, $[\hat \alpha]$, $[\hat \tau]$ and $\hat \Phi$.

	\textbf{Step 3: Proof in the  case $\bC_\bullet$ is $n$-dimensional and of finite height.}\\
	By Step 2, it is sufficient to complete the proof when $\bC_\bullet$ is $n$-dimensional and of finite height. We assume $\bC_\bullet$ satisfies these properties, and consider the  chain map  $g_\#:\bC^{n-\bullet}_c\to \bC_\bullet$ and the finite displacement  chain homotopy $\id\stackrel{h_\#}{\simeq} g_\#f_\#$ as in Step 1.
	The maps $f_\#$ and $g_\#$ are shown in the commutative diagram in Figure~\ref{fig:commutative}.

	\begin{figure}[htb]
		\[\begin{tikzcd}
				0 & {\mathbf{C}_{n}} & \cdots & {\mathbf{C}_{1}} & {\mathbf{C}_{0}} & 0 \\
				0 & {\mathbf{C}^0_c} & \cdots & {\mathbf{C}^{n-1}_c} & {\mathbf{C}^{n}_c} & 0 \\
				0 & {\mathbf{C}_{n}} & \cdots & {\mathbf{C}_{1}} & {\mathbf{C}_{0}} & 0
				\arrow["\partial", from=1-1, to=1-2]
				\arrow[from=1-1, to=2-1]
				\arrow["\partial", from=1-2, to=1-3]
				\arrow["{f_\#}", from=1-2, to=2-2]
				\arrow["\partial", from=1-3, to=1-4]
				\arrow["\partial", from=1-4, to=1-5]
				\arrow["{f_\#}", from=1-4, to=2-4]
				\arrow["\partial", from=1-5, to=1-6]
				\arrow["{f_\#}", from=1-5, to=2-5]
				\arrow[from=1-6, to=2-6]
				\arrow["\delta", from=2-1, to=2-2]
				\arrow[from=2-1, to=3-1]
				\arrow["\delta", from=2-2, to=2-3]
				\arrow["{g_\#}", from=2-2, to=3-2]
				\arrow["\delta", from=2-3, to=2-4]
				\arrow["\delta", from=2-4, to=2-5]
				\arrow["{g_\#}", from=2-4, to=3-4]
				\arrow["\delta", from=2-5, to=2-6]
				\arrow["{g_\#}", from=2-5, to=3-5]
				\arrow[from=2-6, to=3-6]
				\arrow["\partial", from=3-1, to=3-2]
				\arrow["\partial", from=3-2, to=3-3]
				\arrow["\partial", from=3-3, to=3-4]
				\arrow["\partial", from=3-4, to=3-5]
				\arrow["\partial", from=3-5, to=3-6]
			\end{tikzcd}\]
		\caption{A commutative diagram showing the chain maps $f_\#$ and $g_\#$.}\label{fig:commutative}
	\end{figure}

	Since $\bC_\bullet$ is a finite-height resolution, Proposition~\ref{prop:duality_ht} implies $\bC^{n-\bullet}_c$ is a finite-height $R$-chain complex over $X$.  Thus Remark~\ref{rem:height_support}, combined with the hypothesis on $f_\#$ and  Proposition~\ref{prop:cupcapinf_properties}, ensures that $f_\#$ and $g_\#$ have finite displacement over $\id_X$. We thus have finite-displacement chain maps $f_\#:\bC_\bullet\to \bC^{n-\bullet}_c$ and $g_\#:\bC^{n-\bullet}_c\to \bC_\bullet$, and a finite displacement chain homotopy $\id\stackrel{}{\simeq} g_\#f_\#$.
	By Definition~\ref{defn:coarsePD_n}, it remains to show there exists a finite displacement  chain homotopy $\id\stackrel{}{\simeq} f_\#g_\#$. To do this, we will construct an augmentation $\kappa:\bC^n_c\to R$ such that $f_\#$ and $g_\#$ are augmentation-preserving, and $(\bC^{n-\bullet}_c,\kappa)$ is a projective $R$-resolution over $X$. The result will then follow from Proposition~\ref{prop:proj_res}.

	We first claim $H^n(\bC^\bullet_c)\cong R$ and is generated by $[\alpha]$.
	Proposition~\ref{prop:duality_ht} implies $\bC_k\cong (\bC^k_c)^*_c$ for each $k\in \bbZ$, and that  we can dualise $f_\#$, $g_\#$ and $h_\#$ to obtain finite displacement  cochain maps
	$\bC^\bullet_c\xrightarrow{g^\#_c} \bC_{n-\bullet}\xrightarrow{f^\#_c}\bC^\bullet_c$
	and a cochain homotopy $\id\stackrel{h^\#_c}{\simeq}{f^\#_c}{g^\#_c}$.
	We now observe that:
	\begin{itemize}
		\item The composition $H^n(\bC^\bullet_c)\xrightarrow{g^*_c} H_0(\bC_{\bullet})\xrightarrow{f^*_c}H^n(\bC^\bullet_c)$ is the identity map.
		\item $H^n(\bC^\bullet_c)\neq 0$, since there exists $[\alpha]\in H^n(\bC^\bullet_c)$ and $[\tau]\in H_n(\bC_\bullet^{\lf})$ with $\tau(\alpha)=1$.
		\item $H_0(\bC_\bullet)\cong R$, since the projective $R$-resolution $\bC_\bullet$ is weakly uniformly acyclic.
	\end{itemize} These facts imply that $H^n(\bC^\bullet_c)\cong R$. Let $\beta\in \bC^n_c$ be a cocycle such that $[\beta]$ generates $H^n(\bC^\bullet_c)\cong R$. Thus there is some $r\in R$ with $[\alpha]=r[\beta]$. Since $1=\tau(\alpha)=\tau(r\beta)=r\tau(\beta)$, we see that $r$ is a unit, proving the claim.

	We  define  $\kappa:\bC^n_c\to R$ by $\kappa\coloneqq \epsilon\circ g_\#$, where  $\epsilon:\bC_0\to R$ is the augmentation of $\bC_\bullet$. We first observe that \begin{align*}
		(\kappa\circ\delta)(\bC^{n-1}_c)=(\epsilon\circ g_\#\circ\delta)(\bC^{n-1}_c)=(\epsilon\circ\partial\circ g_\#)(\bC^{n-1}_c)=0
	\end{align*} since $\epsilon\circ \partial=0$; thus $\kappa$ is an augmentation of $\bC^{n-\bullet}_c$.
	Pick now pick some $\sigma\in \bC_0$ with $\epsilon(\sigma)=1$, so  that $[f_\#\sigma]=\epsilon(\sigma)[\alpha]=[\alpha]$. Thus \[\kappa(\alpha)=\kappa(f_\#\sigma)=\epsilon(g_\#f_\#\sigma)=\epsilon(\sigma)=1\] since  $g_\#f_\#$ is augmentation-preserving. Thus for each $\sigma'\in \bC_0$, we have $\kappa(f_\#\sigma')=\kappa(\epsilon(\sigma')\alpha)=\epsilon(\sigma')$, so that $f_\#$ is augmentation-preserving. By the definition of $\kappa$, it is clear that $g_\#$ is augmentation-preserving.

	Since $(\bC^{n-\bullet}_c,\kappa)$ is an $R$-chain complex of projective $R$-modules over $X$ by Proposition~\ref{prop:duality_ht},   all that remains is to show $(\bC^{n-\bullet}_c,\kappa)$ is weakly uniformly acyclic.
	Suppose $\beta\in \bC^n_c$ satisfies $\kappa(\beta)=0$. As $\bC^{n+1}_c=0$, we note  that $\delta\beta=0$. Since $\kappa(\beta)=0$,  $\kappa(\alpha)=1$ and $[\alpha]$ generates $H^n(\bC^\bullet_c)$, it follows that $[\beta]=0$, showing that $\im(\bC^{n-1}_c\xrightarrow{\delta}\bC^n_c)=\ker(\kappa)$. Since $H^k(\bC^\bullet_c)=0$ for $k<n$, this implies that:
	\[0\to \bC^0_c\to\cdots\to \bC^n_c\xrightarrow{\kappa} R\to 0\] is exact. We know from Proposition~\ref{prop:unif_preimage_cobdry} that the coboundary maps $\delta:\bC^{k-1}_c\to \bC^k_c$ have uniform preimages.
	As $\bC_\bullet$ is weakly uniformly acyclic, there exists a constant $D\in \bbN$ such that for each $x\in X$, there exists some $\sigma_x\in \bC_0(D)$ with $\supp(\sigma_x)\subseteq N_D(x)$ and $\epsilon(\sigma_x)=1$. For each $x\in X$, set $\alpha_x\coloneqq f_\#(\sigma_x)$.
	Since $f_\#$ is augmentation-preserving and of finite displacement,
	there exists a constant $D'$ such that for all $x\in X$, $\kappa(\alpha_x)=1$, $\supp(\alpha_x)\subseteq N_{D'}(x)$ and $\alpha_x\in \bC^n_c(D')$. By Proposition~\ref{prop:equiv_weakuniformacyc}, $(\bC^{n-\bullet}_c,\kappa)$ is weakly uniformly acyclic, and we are done.
\end{proof}

We now give conditions under which the hypothesis of Lemma~\ref{lem:pdn_criterion_main} can be satisfied. In the case that we assume $X$ satisfies appropriate finiteness conditions, we obtain the following:
\begin{thm}\label{thm:main_finiteness}
	Let $R$ be a PID and let $n\geq 1$. Suppose $X$ is a coarsely homogeneous proper metric space that is coarsely uniformly $(n-1)$-acyclic over $R$. Assume that:
	\begin{enumerate}
		\item $\coarse^k(X;R)=0$ for $k<n$.
		\item $\coarse^n(X;R)$ is a non-zero finitely generated $R$-module.
	\end{enumerate}  Then $X$ is a coarse $PD_n^R$ space.
\end{thm}
\begin{proof}
	By Proposition~\ref{prop:acyc_res}, there exists a proper projective $R$-resolution $\bC_\bullet$ over $X$ such that $\bC_k$ has finite height for $k\leq n$. We pick $i_0$ such that $\bC_k(i_0)=\bC_k$ for $k\leq n$. By Proposition~\ref{prop:duality_ht}, for each $k\leq n$, $\bC^k_c$ is a  projective $R$-module over $X$ of height at most $i_0$.  We first complete the proof under the assumption that $H^n(\bC^\bullet_c)\cong \coarse^n(X;R)$ is a torsion-free $R$-module.

	We proceed by showing that the hypotheses of Lemma~\ref{lem:pdn_criterion_main} are satisfied. Since $H^n(\bC^\bullet_c)$ is a finitely generated torsion-free module over a PID, it is free. Let $[\alpha]\in H^n(\bC^\bullet_c)$ be a non-zero primitive element of $H^n(\bC^\bullet_c)$, so that there is a homomorphism $\overline\tau:H^n(\bC^\bullet_c)\to R$ with $\overline\tau([\alpha])=1$. It thus follows from Theorem~\ref{thm:truncated_kunneth} that there exists some $[\tau]\in H_n(\bC_\bullet^{\lf})$ with $\tau(\alpha)=1$. It remains to construct a map $f_\#:\bC_\bullet\to \bC^{n-\bullet}_c$ satisfying the hypotheses of Lemma~\ref{lem:pdn_criterion_main}.

	We consider $\bZ^n_c\coloneqq \ker(\bC^n_c\to \bC^{n+1}_c)$ and $\bB^n_c=\im(\bC^{n-1}_c\to \bC^n_c)$ as geometric submodules of $\bC^n_c$. Let $\bL\coloneqq \bB^{n}_c+R\alpha$ be the geometric submodule of $\bC^n_c$ generated by $\bB^n_c$ and the cyclic submodule $R\alpha$. Since $[\alpha]$ is primitive in $H^n(\bC^\bullet_c)$, the quotient map $\bZ^n_c\to \bZ^n_c/\bB^n_c= H^n(\bC^\bullet_c)$ restricts to a surjection $j:\bL\to R$ with $j(\alpha)=1$ and $j(\bB^n_c)=0$. By construction, we have  that  $[\beta]=j(\beta)[\alpha]$ for all $\beta\in \bL$, and that \begin{align}\label{eqn:exact_aug}
		\bC^{n-1}_c\xrightarrow{\delta}\bL\xrightarrow{j} R\to 0
	\end{align}
	is exact.

	\begin{figure}[htbp]
		\[\begin{tikzcd}
				{\bC_{n+1}} & {\bC_n} & {\bC_{n-1}} & \cdots & {\bC_1} & {\bC_0} & 0 \\
				0 & {\bC^0_c} & {\bC^1_c} & \cdots & {\bC^{n-1}_c} & {\bC^n_c}
				\arrow[from=1-1, to=1-2]
				\arrow["0", dashed, from=1-1, to=2-1]
				\arrow["\partial", from=1-2, to=1-3]
				\arrow["{f_n}", dashed, from=1-2, to=2-2]
				\arrow["\partial", from=1-3, to=1-4]
				\arrow["{f_{n-1}}", dashed, from=1-3, to=2-3]
				\arrow["\partial", from=1-4, to=1-5]
				\arrow["\partial", from=1-5, to=1-6]
				\arrow["{f_1}", dashed, from=1-5, to=2-5]
				\arrow["0", from=1-6, to=1-7]
				\arrow["{f_0}", dashed, from=1-6, to=2-6]
				\arrow[from=2-1, to=2-2]
				\arrow["\delta", from=2-2, to=2-3]
				\arrow["\delta", from=2-3, to=2-4]
				\arrow["\delta", from=2-4, to=2-5]
				\arrow["\delta", from=2-5, to=2-6]
			\end{tikzcd}\]
		\caption{A commutative diagram showing the map $f_\#:\bC_\bullet\to \bC^{n-\bullet}_c$.}\label{fig:dualmap}
	\end{figure}

	By Proposition~\ref{prop:unif_preimage_cobdry}, there is a constant $D$ such that for each $x\in X$, there is a cocycle $\alpha_x\in \bC^n_c$ with $[\alpha_x]=[\alpha]$ and $\supp(\alpha_x)\subseteq N_D(x)$. In particular, we note that $\alpha_x\in \bL$ and $j(\alpha_x)=1$. Since $\bC^n_c$ has height at most $i_0$, we have $\alpha_x\in \bC^n_c(i_0)$ for all $x\in X$.
	Suppose that $\bC_0=(C_0,B,\delta,p,\cF)$ with projective basis $\{m_b\}$.  We can therefore use Lemma~\ref{lem:define_fin_disp} to construct a finite displacement map $f_0:\bC_0\to \bC^{n}_c$ such that $f_0(m_b)=\epsilon(m_b)\alpha_{p(b)}$ for all $b\in B$. Since $\im(f)\subseteq \bL$ and  $j(f_0(m_b))=\epsilon(m_b)j(\alpha_{p(b)})=\epsilon(m_b)$ for all $b\in B$, it follows that $j\circ f_0$ and $\epsilon$ agree on $\bC_0$.
	Thus for each $\sigma\in \bC_0$,  \[{[f_0(\sigma)]}=j(f_0(\sigma))[\alpha]=\epsilon(\sigma)[\alpha].\] Therefore, the first condition in the statement of Lemma~\ref{lem:pdn_criterion_main} is satisfied.

	We now extend $f_0$ to a chain map $f_\#:\bC_\bullet\to \bC^{n-\bullet}_c$ as shown in Figure~\ref{fig:dualmap}. Indeed, it follows from Proposition~\ref{prop:unif_preimage_cobdry} that all the maps in the bottom row of Figure~\ref{fig:dualmap} have uniform preimages.
	Since $j\circ f_0=\epsilon$ and  (\ref{eqn:exact_aug}) is exact, it follows that  $\im(\bC_1\xrightarrow{f_0\circ \partial}\bC^{n}_c)\subseteq \im(\bC^{n-1}_c\xrightarrow{\delta}\bC^n_c)$. We can thus apply Proposition~\ref{prop:projective} to define a finite displacement map $f_1:\bC_1\to \bC^{n-1}_c$ with $\delta f_1=f_0\partial$. Using the hypothesis $H^k(\bC^\bullet_c)=0$ for $k<n$, we repeatedly apply Proposition~\ref{prop:projective} to define a finite displacement chain map $f_\#:\bC_\bullet\to \bC^{n-\bullet}_c$ as in Figure~\ref{fig:dualmap}.
	Since $f_\#$ is of finite displacement and each $\bC_k$ and $\bC^k_c$ has finite height for $k\leq n$, it follows from Remark~\ref{rem:height_support} that $f_\#$ satisfies the second condition in the statement of Lemma~\ref{lem:pdn_criterion_main}. Thus all the hypotheses of Lemma~\ref{lem:pdn_criterion_main} are satisfied, so we conclude that $X$ is a coarse $PD_n^R$ space. This completes the proof under the assumption $H^n(\bC^\bullet_c)$ is assumed to be torsion-free. In particular, this completes the proof in the case $R$ is a field.

	We now assume for contradiction that $H^n(\bC^\bullet_c)$ is not torsion-free.  The structure theory for finitely generated modules over a PID implies $H^n(\bC^\bullet_c)$  contains a direct summand isomorphic to $R/p^j R$ for some  $j\geq 1$ and prime $p\in R$. We thus consider the field  $\bbF\coloneqq R/pR$. Our choice of $p$ ensures that $\Tor_1^R(H^n(\bC^\bullet_c),\bbF)$ and $H^n(\bC^\bullet_c)\otimes_\bbF$ are both non-zero, being the kernel and cokernel respectively of the map $H^n(\bC^\bullet_c)\xrightarrow{\times p}H^n(\bC^\bullet_c)$.
	Let $\bD_\bullet=\bC_\bullet \otimes_R \bbF$, which is a proper projective $\bbF$-resolution over $X$ by Proposition~\ref{prop:change_rings_res}. As $\bD_k$ has finite height for $k\leq n$ by Proposition~\ref{prop:change_rings}, it follows that $X$ is coarsely uniformly $(n-1)$-acyclic over $\bbF$ by Proposition~\ref{prop:acyc_res}.

	Since $\coarse^k(X;\bbF)\cong H^k(\bD^\bullet_c)$ and $H^k(\bC^\bullet_c)=0$ for $k<n$,  Theorem~\ref{thm:truncated_kunneth} implies that $\coarse^k(X;\bbF)=0$ for $k<n-1$, and $\coarse^{n-1}(X;\bbF)\cong \Tor_1^R(H^n(\bC^\bullet_c),\bbF)$ is non-zero and finite-dimensional. As we have already completed the proof in the  case where the coefficients are a  field, we apply the result to conclude $X$ is a coarse $PD_{n-1}^\bbF$ space. In particular, this implies that $\coarse^n(X;\bbF)=0$ by Proposition~\ref{prop:pdn_properties}. However, we know from Theorem~\ref{thm:truncated_kunneth} that $0\neq H^n(\bC^\bullet_c)\otimes_R\bbF$ is isomorphic to a submodule of $\coarse^n(X;\bbF)\cong H^n(\bD^\bullet_c)$, which yields the desired contradiction.
\end{proof}
Combining Theorem~\ref{thm:main_finiteness} with Proposition~\ref{prop:pdn_properties}, we conclude:
\begin{cor}\label{cor:main_finiteness_space}
	\mainspace{}
\end{cor}

We would like to strengthen Theorem~\ref{thm:main_finiteness} by removing the hypothesis that $X$ is coarsely uniformly $(n-1)$-acyclic. Unfortunately, one of the key points in the proof of Theorem~\ref{thm:main_finiteness} relies on the fact the coboundary maps $\delta:\bC^{k-1}_c\to \bC^k_c$ in the bottom row of Figure~\ref{fig:dualmap} have the uniform preimage property, and so Proposition~\ref{prop:projective} can be applied. In the case where the $\bC_k$ do not have finite height, it is not clear what the analogue of the uniform preimage property should be.

However, we can make progress if we restrict ourselves to the case where $X$ is a countable group $G$ equipped, as always, with a proper left-invariant metric. The key point is that we can choose $\bC_\bullet$ so the coboundary maps   $\delta:\bC^{k-1}_c\to \bC^k_c$ are $G$-equivariant. We can then define a chain map $f_\#:\bC_\bullet\to \bC^{n-\bullet}_c$ as in the statement of Lemma~\ref{lem:pdn_criterion_main}, by requiring that each $f_\#(\bC_k)$ is contained in a finitely generated $G$-submodule of $\bC^{n-k}_c$; see Proposition~\ref{prop:construct_map_weak_fin_disp} for a precise statement. The condition that $\im(f_\#)$ be contained in a finitely generated $G$ maps is a surrogate for the uniform preimage condition, and allows us to construct the desired map $f_
	\#$. This  idea allows us to prove the following  far-reaching generalisation of a result of Farrell, who proved a similar theorem in the case that $G$ satisfies additional finiteness properties~\cite[Theorem 1]{farrell1975pdgroups}.
\begin{thm}\label{thm:main_field}
	\maingroup{}
\end{thm}

Before proving Theorem~\ref{thm:main_field}, we briefly recall some details regarding the standard resolution $\bC_\bullet$ of a countable group $G$. As noted in Corollary~\ref{cor:projres_ginv}, each $\bC_k=(C_k,B_k,\delta^k,p_k,\{B_k(i)\})$ is a proper $G$-induced $R$-module over $G$, and all boundary maps are $G$-invariant. We recall that $\bC_k$ is a free $R$-module over $G$ with basis $B_k=\{[g_0,\dots, g_k]\mid g_0,\dots,g_k\in G\}$ and $p_k([g_0,\dots, g_k])=g_0$. Since the $G$-action on $B_k$ is given by $g[g_0,\dots, g_k]=[gg_0,\dots, gg_k]$, we observe that \[\Sigma_k=\{[g_0,g_1,\dots,g_n]\in B_k\mid g_0=1\}\] is a basis of  $\bC_k$ as an $RG$-module. Moreover, for each $i\in \bbN$, $\Sigma_k(i)\coloneqq \Sigma_k\cap \bC_k(i)$ is a finite basis of $\bC_k(i)$ as an $RG$-module.

The left $G$-action on $\bC_\bullet$ satisfies $g\supp(\sigma)=\supp(g\sigma)$ for all $\sigma\in \bC_\bullet$ and $g\in G$.
There is a right $G$-action on $\bC^\bullet_c$ given by $(\alpha\cdot g)(\sigma)=\alpha(g\sigma)$ for all $k\in \bbN$, $\alpha\in \bC^k_c$, $\sigma\in \bC_k$ and $g\in G$. It is straightforward to verify that $\supp(\alpha\cdot g,i)=g^{-1}\supp(\alpha,i)$ for all $i,k\in \bbN$, $\alpha\in \bC^k_c$ and $g\in G$. Moreover, since the coboundary maps $\delta:\bC^k_c\to \bC^{k+1}_c$ are $G$-equivariant, each cohomology module $H^k(\bC^\bullet_c)$ admits a right $G$-action given by $[\alpha]\cdot g=[\alpha\cdot g]$.

We can now formulate  the new ingredient needed to prove Theorem~\ref{thm:main_field}:
\begin{prop}\label{prop:construct_map_weak_fin_disp}
	Let $R$ be a commutative ring and let $n\in \bbN$. Let $G$ be a countable group and let $\bC_\bullet$ be the standard $R$-resolution of $G$. Suppose $H^k(\bC^\bullet_c)=0$ for $k<n$, and there exists a non-zero class $[\alpha]\in H^n(\bC^\bullet_c)$ contained in a   $G$-invariant submodule $N\leq H^n(\bC^\bullet_c)$ that is finitely generated as an $R$-module. Then there exists a chain map $f_\#:\bC_\bullet\to \bC^{n-\bullet}_c$ such that:
	\begin{enumerate}
		\item For each $\sigma\in \bC_0$, $[f_\#(\sigma)]=\epsilon(\sigma)[\alpha]$ in $H^n(\bC^\bullet_c)$.
		\item There is a function $\Phi:\bbN\to \bbN$ such that if $i\in \bbN$ and $\sigma\in \bC_k(i)$, then \[\supp(f_\#(\sigma),i)\subseteq N_{\Phi(i)}(\supp(\sigma)).\]
	\end{enumerate}
\end{prop}
To prove Proposition~\ref{prop:construct_map_weak_fin_disp}, we will make repeated use of the following lemma:
\begin{lem}\label{lem:equiv-uniformity}
	Let $G$ be a countable group and let $\bC_\bullet$ be the standard $R$-resolution over $G$.
	Assume there is a  homomorphism $f:\bC_k\to \bC^{n-k}_c$ such that the following holds.
	For each $i\in \bbN$, there is a finitely generated submodule $M_i\leq  \bC^{n-k}_c$ such that for all $g\in G$, $f(g\Sigma_k(i))\subseteq M_i g^{-1}$.
	Then:
	\begin{enumerate}
		\item\label{item:equiv-uniformity1}
		      $f(\sigma)\subseteq \langle M_i \supp(\sigma)^{-1}\rangle$ for all $\sigma\in \bC_k(i)$, where $\langle M_i \supp(\sigma)^{-1}\rangle$ is the $R$-submodule generated by $M_i \supp(\sigma)^{-1}$.
		\item\label{item:equiv-uniformity2}
		      There exists a function $\Phi:\bbN\to \bbN$ such that if $i\in \bbN$ and $\sigma\in \bC_k(i)$, then \[\supp(f(\sigma),i)\subseteq N_{\Phi(i)}(\supp(\sigma)).\]
	\end{enumerate}
\end{lem}
\begin{proof}
	(\ref{item:lemstandard1}): For each $b\in B_k(i)$  we have $b\in p_k(b)\Sigma_k(i)$, and so $f(b)\in M_i p_k(b)^{-1}$. For $\sigma\in \bC_k(i)$, we have $\sigma=\sum_{b\in B_k(i)}\delta_b(\sigma)b$. Since $\delta_b(\sigma)\neq 0$ if and only if $p_k(b)\in \supp(\sigma)$, we have
	\[f(\sigma)\in \sum_{p_k(b)\in \supp(\sigma)} M_i p_k(b)^{-1}= \langle M_i \supp(\sigma)^{-1}\rangle\] as required.

	(\ref{item:lemstandard2}):
	Since each $M_i$ is finitely generated, we can pick $\Phi$ such that for each $i\in \bbN$ and $\alpha\in M_i$, $\supp(\alpha,i)\subseteq N_{\Phi(i)}(1)$.
	Thus for each $\alpha\in M_i g^{-1}$, we have $\supp(\alpha,i)\subseteq N_{\Phi(i)}(g)$. Consequently, for each $\alpha\in  \langle M_i \supp(\sigma)^{-1}\rangle$, we have $\supp(\alpha,i)\subseteq N_{\Phi(i)}(\supp(\sigma))$. Combining this with (\ref{item:lemstandard1}) yields the desired conclusion.
\end{proof}

\begin{proof}[Proof of Propsotion~\ref{prop:construct_map_weak_fin_disp}]
	Pick $\Phi$ such that $\bC_\bullet$ has displacement $\Phi$. We recall $\Sigma_k(i)$  is a free $RG$-basis of each $\bC_k(i)$, and that $B_k(i)=G\Sigma_k(i)$ is a free $R$-basis of each $\bC_k(i)$.  Moreover, $\epsilon(b)=1$ for all $g\in G$ and $b\in B_0(i)$.
	We will define a chain map $f_\#:\bC_\bullet\to \bC^{n-\bullet}_c$ and a collection  $\{M_{k,i}\}_{k,i\in \bbN}$, where each $M_{k,i}$ is a finitely generated $R$-submodule of $\bC^{n-k}_c$ satisfying $f_\#(g\Sigma_k(i))\subseteq M_{k,i} g^{-1} $ for all $k,i\in \bbN$ and $g\in G$. Furthermore, we will also show that $[f_\#(\sigma)]=\epsilon(\sigma)[\alpha]$ for all $\sigma\in \bC_0$, and so applying Lemma~\ref{lem:equiv-uniformity} will prove the result.

	We will define $f_k:\bC_k\to \bC^{n-k}_c$ and $\{M_{k,i}\}_{i\in \bbN}$ by induction on $k$.
	First, we prove the base case $k=0$.
	Let $[\alpha_1], \dots, [\alpha_m]$ be a finite generating set  of $N$.
	For each $i\in \bbN$, let $M_{0,i}=\langle\alpha_1,\dots,\alpha_m\rangle$ be the $R$-submodule of $\bC^n_c$ generated by $\alpha_1,\dots,\alpha_m$.
	Since $N$ is $G$-invariant, for each $g\in G$, the image of $M_{0,i}g^{-1}$ in $H^n(\bC^\bullet_c)$ is equal to  $N$. As $[\alpha]\in N$, for each $g\in G$, there is some $\alpha_g\in  M_{0,i}g^{-1}$ with $[\alpha_g]=[\alpha]$.

	We note $\Sigma_0=\Sigma_0(i)=\{[1]\}$ for all $i$, and $B_0=B_0(i)=G\Sigma_0(i)=\{[g]\mid g\in G\}$ for all $i$.
	We can thus define an $R$-module homomorphism $f_0:\bC_0\to \bC^n_c$ by $f_0([g])=\alpha_{g}$ for all $g\in G$. By construction, this map satisfies the property that \[f_\#(g\Sigma_0(i))\subseteq M_{0,i}g^{-1}\] for all $i\in \bbN$ and $g\in G$.
	Moreover, for each $\sigma\in \bC_0$, we have $\sigma=\sum_{g\in G}\delta_{[g]}(\sigma)[g]$ with $\epsilon(\sigma)=\sum_{g\in G}\delta_{[g]}(\sigma)$. Thus
	\begin{align*}
		[f_\#(\sigma)]=\sum_{g\in G}\delta_{[g]}(\sigma)[f_\#({[g]})]=\sum_{g\in G}\delta_{[g]}(\sigma)[\alpha_g]=\sum_{g\in G}\delta_{[g]}(\sigma)[\alpha]=\epsilon(\sigma)[\alpha]
	\end{align*}
	as required. This proves the base case.

	For inductive hypotheses,  assume we have a chain map  $f_\#:[\bC_\bullet]_k\to \bC^{n-\bullet}_c$ and  $\{M_{k,i}\}_{i\in \bbN}$, where each $M_{k,i}$ is finitely generated $R$-submodule of $\bC^{n-k}_c$ such that $f_\#(g\Sigma_k(i))\subseteq M_{k,i}g^{-1}$ for all $i\in \bbN$ and $g\in G$.
	For each $i\in \bbN$, let $F_i\coloneqq N_{\Phi(i)}(1)\subseteq G$.
	Since $F_i$ is finite and $M_{k,i}$ is finitely generated, so is $\langle M_{k,i}F_i^{-1}\rangle$. Since $R$ is Noetherian,  $T_{k,i}\coloneqq \im(\bC^{n-k-1}_c\xrightarrow{\delta}\bC^{n-k}_c)\cap \langle M_{k,i}F_i^{-1}\rangle$ is also finitely generated. Hence, we can pick a finitely generated submodule $M_{k+1,i}\subseteq \bC^{n-k-1}_c$ such that $\delta (M_{k+1,i})=T_{k,i}$.

	We now fix some  $\sigma\in \Sigma_{k+1}(i)$ with $g\in \supp(\sigma)$. Since  $\supp(g\sigma)=\{g\}$ and $\partial$ has displacement $\Phi$, we see  $\supp(\partial g\sigma)\subseteq N_{\Phi(i)}(g)=gF_i$.  By the inductive hypothesis and Lemma~\ref{lem:equiv-uniformity}, we deduce $f_\#(\partial g\sigma)\subseteq \langle M_{k,i}F_i^{-1}g^{-1}\rangle=\langle M_{k,i}F_i^{-1}\rangle g^{-1}$.  We now show  $f_\#(\partial g\sigma)$ is a coboundary. Indeed, we first note that $\delta f_\#(\partial g\sigma)=f_\#(\partial^2 g\sigma)=0$ so that $f_\#(\partial g\sigma)$ is a cocycle. Since  $H^j(\bC^\bullet_c)=0$ for $j<n$, we have $f_\#(\partial g\sigma)\in \im(\delta)$ when $k>0$. In the case $k=0$, we have $[f_\#(\partial g\sigma)]=\epsilon(\partial g\sigma)[\alpha]=0$, again showing $f_\#(\partial g\sigma)$ is a coboundary.

	We have thus shown that for all $\sigma\in \Sigma_{k+1}(i)$ and $g\in G$, $f_\#(\partial g\sigma)\in\langle M_{k,i}F_i^{-1}\rangle g^{-1}$ and $f_\#(\partial g\sigma)$ is a coboundary, so that $f_\#(\partial b)\in T_{k,i}g^{-1}$. By the definition of $M_{k+1,i}$ and $G$-equivariance of $\delta$,  we can choose some $n_{g\sigma}\in M_{k+1,i}g^{-1}$ with $\delta n_{g\sigma} =f_\#(\partial g\sigma)$ for each $\sigma\in \Sigma_{k+1}(i)$ and $g\in G$. Defining $f_\#(g\sigma)=n_{g\sigma}$ allows us to extend $f_\#$ to a map $f_{k+1}:\bC_{k+1}\to \bC^{n-k-1}_c$ with \[f_\#(g\Sigma_{k+1}(i))\subseteq M_{k+1,i}g^{-1}\] for all $g\in G$ and $i\in \bbN$. This proves the inductive step.
\end{proof}

\begin{proof}[Proof of Theorem~\ref{thm:main_field}]
	Let $\bC_\bullet$ be the standard projective $\bbF$-resolution over $G$. By Corollary~\ref{cor:induced_mod}, there is an isomorphism $H^k(G,\bbF G)\cong H^k(\bC^\bullet_c)$ as right $RG$-modules. The hypotheses ensure there exists a non-zero class $[\alpha]\in H^n(\bC^\bullet_c)$ contained  in a non-zero finite-dimensional $G$-invariant subspace. By Theorem~\ref{thm:univcoeff_ctblehyp}, there exists an element $\tau\in H_n(\bC_\bullet^{\lf})$ such that $\tau(\alpha)=1$. Moreover, we can apply Proposition~\ref{prop:construct_map_weak_fin_disp} to construct a map $f_\#:\bC_\bullet\to \bC^{n-\bullet}_c$ satisfying the  hypotheses of Lemma~\ref{lem:pdn_criterion_main}. We conclude by applying Lemma~\ref{lem:pdn_criterion_main} to deduce $X$ is a coarse $PD_n^\bbF$ space as required.
\end{proof}

We now focus on the case the coefficient ring is not a field:
\begin{thm}\label{thm:main_PID}
	Let $R$ be a countable PID, let $G$ be a countable group and let $n\in \bbN$. Suppose:
	\begin{enumerate}
		\item $H^k(G,R G)=0$ for $k<n$.
		\item $H^n(G,R G)$ is a non-zero finitely generated  $R$-module.
		\item $H^{n+1}(G,R G)$ is a countably generated $R$-module.
	\end{enumerate} Then  $G$ is a coarse $PD_{n}^{R}$ group.
\end{thm}
\begin{rem}
	Theorem~\ref{thm:main_finiteness} shows that  if  $G$ of type $FP_n^R$, we can remove the hypothesis that  $H^{n+1}(G,R G)$ is countably generated from the preceding theorem.
\end{rem}
\begin{proof}
	We will use  Corollary~\ref{cor:induced_mod} implicitly throughout this proof.
	Suppose for contradiction that $M\coloneqq H^n(G,R G)$ contains torsion. As in the proof of Theorem~\ref{thm:main_finiteness}, we see that for some prime $p\in R$, both $\Tor_1^R(M,\bbF)$ and $M\otimes_R \bbF$ are non-zero, where $\bbF$ is the field $R/pR$. It follows from Theorem~\ref{thm:univcoeff_ctblehyp} and Corollary~\ref{cor:induced_mod} that $H^k(G,\bbF G)=0$ for $k<n-1$,  that $H^{n-1}(G,\bbF G)\cong \Tor_1^R(M,\bbF)$, and that $M\otimes_R \bbF$ is a submodule of $H^n(G,\bbF G)$. Since $H^{n-1}(G,\bbF G)\cong \Tor_1^R(M,\bbF)$ is finite dimensional as $M$ is finitely generated, Theorem~\ref{thm:main_field} implies $G$ is a coarse $PD_{n-1}^\bbF$ group. This contradicts Proposition~\ref{prop:pdn_properties} and  the fact that $0\neq M\otimes_R \bbF$ is a submodule of $H^n(G,\bbF G)$, so we conclude that $M$ is torsion-free.

	Since $M$ is torsion-free and finitely generated, it is free. Let $\bC_\bullet$ be a proper projective $R$-resolution over $G$, and  let $[\alpha]\in H^n(\bC^\bullet_c)\cong M$ be a non-zero primitive element. By Theorem~\ref{thm:univcoeff_ctblehyp}, there is some class $[\tau]\in H_n(\bC_\bullet^{\lf})$ with $\tau(\alpha)=1$. We can thus apply Proposition~\ref{prop:construct_map_weak_fin_disp} and Lemma~\ref{lem:pdn_criterion_main} to conclude that $X$ is a coarse $PD_n^R$ space as required.
\end{proof}

This implies the following characterisation of coarse $PD_n^R$ groups.
\begin{cor}\label{cor:char_pdn}
	Suppose $G$ is a countable group, $n\in \bbN$ and $R$ is either a field or a countable PID\@. Then the following are equivalent:
	\begin{enumerate}
		\item $G$ is a coarse $PD_n^R$ group.
		\item $H^k(G;RG)\cong\begin{cases}
				      R & \textrm{if}\;k=n,     \\
				      0 & \textrm{if}\;k\neq n.
			      \end{cases}$
	\end{enumerate}
\end{cor}
Crucially, Corollary~\ref{cor:char_pdn} does not assume $G$ satisfies any finiteness conditions.

\appendix

\section{Some results about derived limits}\label{sec:inverse_limits}
\emph{In this appendix, we deduce some useful facts about inverse and derived limits. The key result is Corollary~\ref{cor:stable_invsystem}, which states  that under mild hypotheses, if $C^\bullet=\varprojlim C^\bullet_i$ is an inverse limit and $H^k(C^\bullet)$ is countably generated for all $k$, then the inverse systems $\{H^k(C^\bullet_i)\}_i$ are stable.
}
\vspace{.3cm}

In this section, we assume $R$ is a PID\@.
We make use of the following facts about modules over $R$:
\begin{lem}[{\cite[Theorem B-2.28]{rotman2015advanced}}]\label{lem:submodule_ctble}
	Let $R$ be a PID\@.
	\begin{enumerate}
		\item If $F$ is a free $R$-module and $F'\leq F$, then $F'$ is a free $R$-module with $\rank(F')\leq \rank(F)$, where these ranks may be infinite cardinals.
		\item Any submodule of a  countably generated $R$-module is countably generated.
	\end{enumerate}
\end{lem}
\begin{proof}
	(1) is~\cite[Theorem B-2.28]{rotman2015advanced}. For (2), let $M$ be a countable generated $R$-module, and let $\pi:F\to M$ be a surjection, with $F$ a countably generated free $R$-module. Then $F'\coloneqq \pi^{-1}(N)$ is a submodule of $F$ with $\rank(F')\leq \rank(F)$ by (1). Thus $F'$ is countably generated, and since $\pi|_{F'}:F'\to N$ is surjective, it follows that $N$ is countably generated.
\end{proof}

We recall an \emph{inverse system} consists of  a sequence   $\{M_i\}_{i\in \bbN}$ of $R$-modules equipped with \emph{bonding maps} $f_i^j:M_j\to M_i$ for each $i\leq j$, satisfying the following:
\begin{itemize}
	\item $f_i^i=\id$
	\item $f_i^j\circ f_j^k=f_i^k$ for all $i\leq j\leq k$.
\end{itemize}

The \emph{Eilenberg map} of the inverse system $\{M_i\}$ is the map $\Delta:\prod_i M_i\to \prod_i M_i$ given by $(m_i)\mapsto (m_i-f_i^{i+1}(m_{i+1}))$. We define the \emph{inverse limit} of $\{M_i\}$ to be $\varprojlim M_i\coloneqq \ker(\Delta)$ and the \emph{derived limit} to be $\varprojlim ^1M_i\coloneqq \coker(\Delta)$. We note that there are \emph{projections} $f_i:\varprojlim M_i\to M_i$ that commute with the bonding maps. It can be shown that $\varprojlim M_i$ is characterised by the following \emph{universal property}: If there exists an $R$-module $N$ and a family of maps $\phi_i:N\to M_i$ such that $f^j_i\circ\phi_j=\phi_i$ for all $j\geq i$, then there is a unique map $\phi:N\to \varprojlim M_i$ such that $f_i\circ\phi=\phi_i$ for all $i$.

Suppose $0\to \{L_i\}\to \{M_i\}\to \{N_i\}\to 0$ is a short exact sequence of inverse systems, i.e.\ for each $i$ there is a short exact sequence $0\to L_i\to M_i\to N_i\to 0$, and these short exact sequence commute with the bonding maps. Applying the snake lemma yields the following \emph{six-term short exact sequence}:
\begin{align*}
	0\to \varprojlim  L_i\to \varprojlim  M_i\to \varprojlim  N_i\to \varprojlim{}^1 L_i\to \varprojlim{}^1 M_i\to \varprojlim{}^1 N_i\to 0
\end{align*}

An inverse system is said to satisfy the \emph{Mittag-Leffler condition} if for each $i$, there exists $j\geq i$ such that $\im(M_j\xrightarrow{f_i^j}M_i)=\im(M_k\xrightarrow{f_i^k}M_i)$ for all $k\geq j$.

\begin{prop}\label{prop:cardinality_invlimit}
	Assume $R$ is either a field or a countable PID\@.
	Let $\{M_i\}$ be an inverse system of countably generated $R$-modules.
	\begin{enumerate}
		\item\label{item:cardinality_invlimit1}
		      $\varprojlim{}^1 M_i=0$ if and only if $\{M_i\}$ satisfies the  Mittag-Leffler condition.
		\item\label{item:cardinality_invlimit2}
		      $\varprojlim{}^1 M_i$ is either zero or an uncountably generated $R$-module.
		\item\label{item:cardinality_invlimit3}
		      If $\varprojlim{} M_i$ is countably generated, then  there exists an $i_0$ such that the projection $f_k:\varprojlim{} M_i\to M_k$ is injective for all $k\geq i_0$.
	\end{enumerate}
\end{prop}
In order to prove Proposition~\ref{prop:cardinality_invlimit}, we require two lemmas.
The following is a variant of a result of Gray~\cite{gray1966spaces}:
\begin{lem}\label{lem:cardinality_invlimit}
	Let $R$ be as in Proposition~\ref{prop:cardinality_invlimit}.
	Let $M$ be an $R$-module and let $M= M_0\supseteq M_1 \supseteq M_2 \supseteq \dots$ be a descending sequence of submodules that does not stabilise, i.e.\ for each $j$,  $\cap_i M_i\neq M_j$. Then $\varprojlim M/M_i$ is an uncountably generated $R$-module.
\end{lem}
\begin{proof}
	For each $i$, we consider the short exact sequence \[1\to M_i/M_{i+1}\to M/M_{i+1}\to M/M_{i}\to 1.\]
	In the case $R$ is a field, the above short exact sequence splits as  $M/M_{i+1}\cong M/M_{i}\oplus M_i/M_{i+1}$. This ensures that  $\varprojlim M/M_i\cong \prod_i M_i/M_{i+1}$, which is uncountably generated as required.

	In the case $R$ is countable, we note that an $R$-module is countably generated if and only if it is countable. We can pick set-theoretic sections of the maps  $M/M_{i+1}\to M/M_{i}$ for each $i$, and hence identify $M/M_{i+1}$ with $M/M_{i}\times  M_i/M_{i+1}$ as sets. Thus $\varprojlim M/M_i$ can be identified with $\prod_i M_i/M_{i+1}$ as  sets, showing that $\varprojlim M/M_i$ is an uncountable set, hence uncountably generated as an $R$-module.
\end{proof}

\begin{lem}\label{lem:not_mittag}
	Let $R$ be as in Proposition~\ref{prop:cardinality_invlimit}.
	If $\{M_i\}$ is an inverse sequence of countably generated $R$-modules that does not satisfy the Mittag-Leffler condition, then $\varprojlim^1M_i$ is an uncountably generated $R$-module, and in particular, is non-trivial.
\end{lem}
\begin{proof}
	We assume for contradiction that   $\{M_i\}$ does not satisfy the Mittag-Leffler condition and $\varprojlim^1M_i$ is countably generated. We can  pick $i_0$  such that $\im(M_j\to M_{i_0})\neq \im(M_{j+1}\to M_{i_0})$ for infinitely many $j$. Set $W_i=\im(M_i\to M_{i_0})$. For each $i\geq i_0$, we have a short exact sequence $1\to K_i\to M_i\to W_i\to 1$, where $K_i=\ker(M_i\to M_{i_0})$.
	By the six-term exact sequence, we have a surjection $\varprojlim^1 M_i\to \varprojlim^1 W_i$, hence $\varprojlim^1 W_i$ is also countably generated.

	Let $W=W_{i_0}$ and note that the sequence $W=W_{i_0}\supseteq W_{i_0+1}\supseteq \dots$ does not stabilise by the choice of $i_0$. We now consider the short exact sequence $1\to \{W_i\}\to \{W\}\to \{W/W_i\}\to 1$.  The six-term exact sequence  yields an exact sequence $W\to \varprojlim W/W_i\to \varprojlim^1W_i\to 0$, noting that $\varprojlim^1 W=0$ as $\{W\}$ is the constant inverse system. Since $W$ and $\varprojlim^1W_i$ are countably generated, so is $\varprojlim W/W_i$. This contradicts  Lemma~\ref{lem:cardinality_invlimit} and the choice of $i_0$.
\end{proof}

\begin{proof}[Proof of Proposition~\ref{prop:cardinality_invlimit}]
	The  ``if'' direction of (\ref{item:cardinality_invlimit1}) is well-known, see~\cite[Proposition 3.5.7]{weibel1994introduction}, while the ``only if'' direction follows from Lemma~\ref{lem:not_mittag}. We deduce (\ref{item:cardinality_invlimit2}) from (\ref{item:cardinality_invlimit1}) and Lemma~\ref{lem:not_mittag}.
	It remains to show (\ref{item:cardinality_invlimit3}).
	Let $N=\varprojlim{} M_i$ and let $f_i:N\to M_i$ be the projection. Let $N_i\coloneq \ker(f_i)$, so that $N_0\supseteq N_1\supseteq \cdots$ is a descending chain of submodules of $N$.

	The universal property of inverse limits implies there is an isomorphism $\pi:N\to \varprojlim N/N_i$ such that postcomposing $\pi$ with the projection to $N/N_i$ is the quotient map.
	Since we are assuming that $\varprojlim N/N_i\cong  \varprojlim_{i} M_i$ is countably generated, it follows from Lemma~\ref{lem:cardinality_invlimit} that  the sequence $(N_i)$ stabilises, i.e.\ there exists an $i_0$ such that $N_k=N_{i_0}$ for all $k\geq i_0$. Since $N=\varprojlim{} M_i= \varprojlim N/N_i=N/N_{i_0}$, it follows that $N_{i_0}=0$ and thus $f_k:\varprojlim{} M_i\to M_k$ is injective for all $k\geq i_0$.
\end{proof}

\begin{defn}\label{defn:stable}
	An inverse system is said to be \emph{stable} if it satisfies the Mittag-Leffler condition and the projection $\varprojlim M_i\to M_i$ is injective for  $i$ sufficiently large.
\end{defn} Stable inverse systems have the following characterisation, which makes them easy to work with:
\begin{prop}[{\cite[Lemma 12.31.5.]{stacks-project}}]\label{prop:stable_limit}
	Let $\{M_i\}_i$ be a stable inverse system of $R$-modules, and let $M=\varprojlim M_i$. Then there exists an $i_0$ such that for all $i\geq i_0$, there is a direct sum decomposition $M_i\cong M\oplus Z_i$ such that the following hold:
	\begin{enumerate}
		\item For all $j\geq i\geq i_0$, the bonding maps $f^j_i:M\oplus Z_j\to M\oplus Z_i$ preserve the direct sum decomposition, and are of the form $f^j_i(m,z_j)=(m,f^j_i(z_j))$ for each $m\in M$ and $z_j\in Z_j$.
		\item For all $i\geq i_0$, there is some $j\geq i$ such that $f^j_i(Z_j)=0$.
		\item Each projection $f_i:\varprojlim M_i\to M_i=M\oplus Z_i$ is of the form $f_j(m)=(m,0)$.
	\end{enumerate}
\end{prop}
\begin{proof}
	As $\{M_i\}$ is stable, we can pick $k$ large enough such that $f_{k}:M\to M_k$ is injective. We pick $i_0\geq k$ such that $\im(M_{i_0}\xrightarrow{f_{k}^{i_0}} M_{k})=\im(M_{i}\xrightarrow{f_{k}^{i}} M_{k})$ for all $i\geq i_0$. Since $\im(f_{k})=\im(f_{k}^{i_0})$, there is homomorphism $\beta:M_{i_0}\to M$ such that $f_k^{i_0}=f_k\circ\beta$.
	For all $i\geq i_0$, we see that $f_k\beta f^i_{i_0}f_i=f_k$. Since $f_k=f_k^i f_i$ is injective, we deduce $f_i\beta f^i_{i_0}:M_i\to M_i$ is a projection onto $\im(f_i)$.  Therefore,  $M_{i}=\im(f_i)\oplus Z_i$, where $Z_i\coloneqq \ker(f_i\beta f^i_{i_0})$. Identifying $M$ with $\im(f_i)$ via $f_i$, we have $M_i\cong M\oplus Z_i$ and the bonding map respect these direct product decompositions. It is straightforward to verify each $f^j_i:M\oplus Z_j\to M\oplus Z_i$ is of the form $(m,z_j)=(m,f^j_i(z_j))$, and the projection $f_i:M\to M\oplus Z_i$ is of the form $m\mapsto (m,0)$. Moreover, since $\{M_i\}$ satisfies the Mittag-Leffler condition,  for each $i\geq i_0$ there is some $j\geq i$ with $\im(f_i^j)=\im(f_i)$. Thus $f_i^j(Z_j)=0$ as required.
\end{proof}
\begin{cor}\label{cor:invlimit_tensortor_stable}
	Let $\{M_i\}$ be a stable inverse system of $R$-modules with inverse limit $M$. For each $R$-module $N$, the induced inverse systems $\{M_i\otimes_R N\}$ and $\{\Tor_1^R(M_i,N)\}$ are stable with $\varprojlim{}(M_i\otimes_R N)=M\otimes_R N$ and $\varprojlim{}(\Tor_1^R(M_i,N))=\Tor_1^R(M,N)$.
\end{cor}
\begin{proof}
	This follows from the isomorphisms $M_i\cong (\varprojlim_{} M_i)\oplus Z_i$ from Proposition~\ref{prop:stable_limit} and the fact that  $-\otimes_R N$ and $\Tor_1^R(-,N)$ commute with direct sums.
\end{proof}

Under mild assumptions, the cohomology of an inverse limit of chain complexes can be computed using the Milnor exact sequence:
\begin{prop}\label{prop:milnor_exact}
	Let $\{C^\bullet_i\}_i$ be an inverse system of chain complexes such that each $C^\bullet_i$ is a chain complex of countably generated projective $R$-modules, and $\{C^\bullet_i\}$ satisfies the Mittag-Leffler condition.  Let $C^\bullet=\varprojlim C^\bullet_i$.
	For each $k$, there is a short exact sequence
	\[0\to \varprojlim_i{}^1 H^{k-1}(C_i^\bullet)\to H^k(C^\bullet)\xrightarrow{\psi}\varprojlim_i{} H^{k}(C_i^\bullet)\to 0.\] Moreover, the projection $H^k(C^\bullet)\xrightarrow{\psi}\varprojlim_i{} H^{k}(C_i^\bullet)\to H^k(C_i^\bullet)$ is induced by the projection $C^\bullet\to C^\bullet_i$.
\end{prop}
\begin{proof}
	Since $\{C^\bullet_i\}_i$ satisfies the Mittag-Leffler condition, Proposition~\ref{prop:cardinality_invlimit} implies the Eilenberg map $\Delta:\prod_i C^\bullet_i\to \prod_i C^\bullet_i$ is surjective. As $\ker(\Delta)=C^\bullet$, this yields the short exact sequence \[0\to C^\bullet\to \prod_i C^\bullet_i\xrightarrow{\Delta }\prod_i C^\bullet_i \to 0. \] The corresponding long exact sequence in cohomology is
	\[\cdots \prod_i H^{k-1}(C^\bullet_i)\xrightarrow{\Delta^{k-1}} \prod_i H^{k-1}(C^\bullet_i)\to H^k(C^\bullet)\to \prod_i H^{k}(C^\bullet_i) \xrightarrow{\Delta^k} %
		\cdots,  \] where $\Delta^*$ are the appropriate Eilenberg maps. The required short exact sequence follows from the observation that  $\varprojlim_i{}^1 H^{k-1}(C_i^\bullet)\cong \coker(\Delta^{k-1})$ and $\varprojlim_i H^{k}(C_i^\bullet)\cong \ker(\Delta^{k})$.
\end{proof}
Combining Lemma~\ref{lem:submodule_ctble}, Propositions~\ref{prop:cardinality_invlimit} and~\ref{prop:milnor_exact}, we conclude the following:
\begin{cor}\label{cor:stable_invsystem}
	Suppose that $\{C^\bullet_i\}_i$ and $C^\bullet$ are as in Proposition~\ref{prop:milnor_exact} and $R$ is either a countable PID or a field. Suppose $H^k(C^\bullet)$ is countably generated for $k\leq n$. Then  the inverse system  $\{H^k(C^\bullet_i)\}$ is stable for each $k< n$.
\end{cor}

Recall there is a dual notion of a direct system. A \emph{direct system} consists of  a sequence   $\{M_i\}_{i\in \bbN}$ of $R$-modules equipped with \emph{bonding maps} $f_i^j:M_i\to M_j$ for each $i\leq j$, satisfying the following:
\begin{itemize}
	\item $f_i^i=\id$; and
	\item $f_j^k\circ f_i^j =f_i^k$ for all $i\leq j\leq k$.
\end{itemize}
The \emph{direct limit} $\varinjlim M_i$ is an $R$-module $M$ equipped with \emph{canonical maps} $f_i:M_i\to M$ such that $f_i=f_j\circ f_i^j$ for all $i\leq j$, and satisfying the following universal property: If $N$ is any $R$-module such that there are maps $g_i:M_i\to N$ satisfying $g_i=g_j\circ f_i^j$ for all $i\leq j$, there is a unique map $g:M\to N$ such that $g_i=g\circ f_i$ for all $i$.

If $\{M_i\}$ is an \emph{inverse} system of $R$-modules and $N$ is an $R$-module, then applying the contravariant functors $\Hom_R(-,N)$ and $\Ext_R^1(-,N)$ to $\{M_i\}$ yields the induced direct systems $\{\Hom_R(M_i,N)\}$ and $\{\Ext_R^1(M_i,N)\}$. We have the following analogue of Corollary~\ref{cor:invlimit_tensortor_stable}:
\begin{cor}\label{cor:dirlimit_homext_stable}
	Let $\{M_i\}$ be a stable inverse system of $R$-modules with inverse limit $M$. Then the induced inverse systems $\{\Hom_R(M_i,N)\}$ and $\{\Ext_R^1(M_i,N)\}$ have direct limits $\varinjlim{}\Hom_R(M_i,N)=\Hom_R(M,N)$ and $\varinjlim{}\Ext_R^1(M_i,N)=\Ext_R^1(M,N)$.
\end{cor}
\begin{proof}
	This follows from  Proposition~\ref{prop:stable_limit} and the fact that  $\Hom_R(-,N)$ and $\Ext_R^1(-,N)$ commute with finite direct sums.
\end{proof}
Unlike inverse limits, direct limits of $R$-modules are exact: if  $0\to \{L_i\}\to \{M_i\}\to \{N_i\}\to 0$ is a short exact sequence of direct systems, then there is a short exact sequence $0\to \varinjlim L_i\to \varinjlim M_i\to  \varinjlim N_i\to 0$. Moreover, homology commutes with direct limits: if $C_\bullet=\varinjlim_i C_\bullet^i$ is a direct limit of chain complexes, then $H_n(C_\bullet)=\varinjlim_i(H_n(C_\bullet^i))$.
\begin{lem}\label{lem:dirlim_pair}
	Let $\{M_i\}$ be an inverse sequence of countably generated $R$-modules and let $N$ be an $R$-module.
	\begin{enumerate}
		\item\label{item:dirlim_pair1}  There  is a pairing \[\Phi:\varinjlim \Hom_R(M_i,N)\to \Hom( \varprojlim M_i,N)\] such that $\Phi(g_i(\alpha))(\sigma)=\alpha(f_i(\sigma))$ for each $i\in \bbN$, $\alpha\in \Hom_R(M_i,N)$ and $\sigma\in \varprojlim M_i$, where  $f_i:\varprojlim M_i\to M_i$ and $g_i:\Hom_R(M_i,N)\to\varinjlim \Hom_R(M_i,N)$ are the canonical maps.
		\item\label{item:dirlim_pair2} In the case where $R$ is a field and $ \varprojlim M_i$ has countable dimension, then $\Phi$ is surjective.
	\end{enumerate}
\end{lem}
\begin{proof}
	(\ref{item:dirlim_pair1}): Let $f_i^j:M_j\to M_i$ be the bonding maps of $\{M_i\}$. Since the maps  $f_i^*:\Hom_R(M_i,N)\to \Hom( \varprojlim M_i,N)$ dual to $f_i: \varprojlim M_i\to M_i$  satisfy $f_j^*(f_i^j)^*=f_i^*$ for all $i\leq j$, the universal property of direct limits yields the required map $\Phi$.

	(\ref{item:dirlim_pair2}): Since $\varprojlim M_i$ has countable dimension, Proposition~\ref{prop:cardinality_invlimit} implies there is some $j$ such that $f_j:\varprojlim M_i\to M_j$ is injective. Let $\phi\in \Hom( \varprojlim M_i,N)$. Picking a complement of $\im(f_j)$ in  $M_j$, we can define a map $\widetilde\phi:M_j\to N$ such that $\widetilde\phi\circ f_j=\phi$. Thus we have $\Phi(g_j(\widetilde{\phi}))=\widetilde{\phi}\circ f_j=\phi$, showing $\Phi$ is indeed surjective.
\end{proof}

We now prove versions of the K\"unneth and Universal Coefficient Theorems for inverse limits:
\begin{thm}\label{thm:truncated_kunneth_chain}
	Let $R$ be a PID and let $N$ be an $R$-module.
	Suppose that $\{C^\bullet_i\}_i$ and $C^\bullet$ are as in Proposition~\ref{prop:milnor_exact}. Assume also that each $C^k_i$ is a free $R$-module.
	Assume that for $k\leq n$, $\{H^k(C^\bullet_i)\}$ is a stable inverse sequence.
	\begin{enumerate}
		\item\label{item:truncated_kunneth_chain1} Let $D^\bullet_i\coloneqq C^\bullet_i\otimes_R N$ for each $i$, and let $D^\bullet\coloneqq\varprojlim_i D^\bullet_i$.
		      For each $k<n$, we have a natural short exact sequence \[0\to  H^k(C^\bullet) \otimes_R N\xrightarrow{\phi} H^k(D^\bullet)\to \Tor_1^R(H^{k+1}(C^\bullet),S)\to 0,\]
		      and an injection $H^n(C^\bullet) \otimes_R N\xrightarrow{\phi} H^n(D^\bullet)$.
		      Moreover, the injections $\phi$ are given by $[\alpha]\otimes s\mapsto [\alpha\otimes s]$.
		\item\label{item:truncated_kunneth_chain2} Let $E_\bullet\coloneqq \varinjlim_i \Hom_R(C^\bullet_i,N)$.  For each $k<n$, we have a natural short exact sequence \[0\to \Ext^1_R(H^{k+1}(C^\bullet),N) \to H_k(E_\bullet)\xrightarrow{\psi} \Hom_R(H^{k}(C^\bullet), N)\to 0,\]
		      and a surjection $H_n(E_\bullet)\xrightarrow{\psi} \Hom_R(H^{n}(C^\bullet),N)$.   Moreover,  these  surjections $\psi$ are given by the pairing in Lemma~\ref{lem:dirlim_pair}.
		\item\label{item:truncated_kunneth_chain3}
		      In the case where $R$ is a field and $H^{n+1}(C^\bullet)$ is countable dimensional, there is a surjection $H_{n+1}(E_\bullet)\xrightarrow{\psi} \Hom_R(H^{n+1}(C^\bullet),N)$ given by Lemma~\ref{lem:dirlim_pair}.
	\end{enumerate}
\end{thm}
\begin{rem}\label{rem:ctbly_gen_invlimit_rem}
	Corollary~\ref{cor:invlimit_tensortor_stable} yields many situations in which Theorem~\ref{thm:truncated_kunneth_chain} can be applied.
\end{rem}
\begin{proof}
	Propositions~\ref{prop:cardinality_invlimit} and~\ref{prop:milnor_exact} imply that for $k\leq n+1$, we have  $\varprojlim H^k(C^\bullet_i)=H^k(C^\bullet)$.

	(\ref{item:truncated_kunneth_chain1}):
	Since each $C^\bullet_i$ is chain complex of free $R$-modules and $R$ is a PID, we have short exact sequences \begin{align}
		0\to H^k(C^\bullet_i)\otimes_R N\to H^k(D^\bullet_i )\to \Tor_1^R(H^{k+1}(C^\bullet_i),N)\to 0\label{eqn:kunneth_ses_chain_1}
	\end{align}for each $k$, and these are natural in $C^\bullet_i$. If $k\leq n$, then Corollary~\ref{cor:invlimit_tensortor_stable} implies  both $H^k(C^\bullet_i)\otimes_R N$ and $\Tor_1^R(H^{k}(C^\bullet_i),N)$ are stable, with inverse limits $H^k(C^\bullet)\otimes_R N$ and $\Tor_1^R(H^{k}(C^\bullet),N)$ respectively. Applying the six-term exact sequence to (\ref{eqn:kunneth_ses_chain_1}) for each $k<n$, we deduce  that $\varprojlim_{}^1H^k(D^\bullet_i )=0$ and there is a short exact sequence \begin{align}
		0\to H^k(C^\bullet)\otimes_R N\to \varprojlim  _i H^k(D^\bullet_i )\to \Tor_1^R(H^{k+1}(C^\bullet),N)\to 0\label{eqn:kunneth_ses_chain_2}
	\end{align} for each $k<n$.
	Propositions~\ref{prop:cardinality_invlimit} and~\ref{prop:milnor_exact} then imply  \[\varprojlim_i  H^k(D^\bullet_i)\cong H^k(D^\bullet)\] for $k\leq n$. Combining this with (\ref{eqn:kunneth_ses_chain_2}) yields the desired short exact sequence.
	Applying the six-term exact sequence to (\ref{eqn:kunneth_ses_chain_1}) with $k=n$ yields an injection $H^n(C^\bullet)\otimes_R N\to \varprojlim _i H^n( D^\bullet_i)\cong H^n(D^\bullet)$.

	(\ref{item:truncated_kunneth_chain2}):
	Since each $C^\bullet_i$ is chain complex of free $R$-modules and $R$ is a PID, we have short exact sequences \begin{align}
		0\to \Ext_R^1(H^{k+1}(C^\bullet_i),N)\to H_k(\Hom_R(C^\bullet_i,N))\to \Hom_R(H^{k}(C^\bullet_i),N)\to 0\label{eqn:kunneth_ses_chain_3}
	\end{align}for each $k$.
	Corollary~\ref{cor:dirlimit_homext_stable} implies that for each $k\leq n$, the direct systems $\{ \Ext_R^1(H^{k}(C^\bullet_i),N)\}$ and $\{\Hom_R(H^{k}(C^\bullet_i),N)\}$ are stable with direct limits $\Ext_R^1(H^{k}(C^\bullet),N)$ and $\Hom_R(H^{k}(C^\bullet),S)$ respectively.
	Thus applying the exact functor $\varinjlim_i$ to (\ref{eqn:kunneth_ses_chain_3}) with $k\leq n$, and using the fact the homology commutes with $\varinjlim_i$,   we deduce  that  there is a short exact sequence \begin{align*}
		0\to \varinjlim_i\Ext_R^1(H^{k+1}(C^\bullet_i),N)\to  H_k(E_\bullet)\to \Hom_R(H^{k}(C^\bullet),N)\to 0\label{eqn:kunneth_ses_chain_4}
	\end{align*} for each $k\leq n$.
	In the case $k<n$, we have \[\varinjlim_i\Ext_R^1(H^{k+1}(C^\bullet_i),N)\cong \Ext_R^1(H^{k+1}(C^\bullet),N)\] as required.

	(\ref{item:truncated_kunneth_chain3}) is just a restatement of Lemma~\ref{lem:dirlim_pair} in the present context.
\end{proof}

\bibliography{bibliography/bibtex}
\bibliographystyle{alpha}
\end{document}